\documentclass[12pt,a4paper,twoside,openany]{amsart}
\usepackage{a4,latexsym,amscd,etex}

\usepackage[main=english, ngerman]{babel}
\usepackage[utf8]{inputenc}%
\usepackage{bm} 

\usepackage{setspace}
\usepackage{graphicx}
\usepackage{subfigure}
\usepackage{float}		
\usepackage{tikz}		
\usepackage{tikz-cd}	
\usepackage{xcolor} 
\colorlet{shadecolor}{gray!22} 

\usepackage[
labelfont={sf},		
font={small},		
]{caption}

\newcommand{\executeiffilenewer}[3]{\ifnum\pdfstrcmp{\pdffilemoddate{#1}}{\pdffilemoddate{#2}}>0{\immediate\write18{#3}}\fi} 

\usepackage{multirow}	
\usepackage{tabularx}	
\usepackage{ragged2e}
\newcolumntype{L}[1]{>{\raggedright\arraybackslash}p{#1}}
\newcolumntype{C}[1]{>{\centering\arraybackslash}p{#1}}
\newcolumntype{R}[1]{>{\raggedleft\arraybackslash}p{#1}}
\newcolumntype{J}[1]{>{\justifying\arraybackslash}p{#1}}
\usepackage{rotating}
\newcommand\tabrotate[1]{\begin{turn}{90}\rlap{#1}\end{turn}}

\usepackage{amsmath}
\usepackage{amstext}
\usepackage{amsthm}
\usepackage{amssymb}
\usepackage{mathtools}
\usepackage{stackengine}
\usepackage{scalerel}
\usepackage{pst-all}

\usepackage{enumerate}	
\usepackage{autobreak}	
\usepackage{verbatim}	
\usepackage{hyperref}	

\usepackage[
backend=biber,
style=alphabetic, 
sorting=nyt, 
abbreviate=true, 
firstinits=true, 
maxbibnames=3, 
isbn=false,	
url=true, 
doi=true, 
hyperref=true, 
backref=false, 
]{biblatex}

\nocite{*} 
\DeclareFieldFormat[article]{title}{{#1}} 

\DeclareSourcemap{
	\maps[datatype=bibtex]{
		\map{
			\step[fieldsource=issue, match=\regexp{\A[0-9]+\Z}, final]
			\step[fieldset=number, origfieldval]
			\step[fieldset=issue, null]
		}
	}
}

\DeclareBibliographyDriver{article}{%
	\printnames{author}%
	\addcomma\space
	\printfield{title}%
	\addcomma\space
	\printfield{journaltitle}%
	\space
	\textbf{\printfield{volume}}%
	\space
	\iffieldundef{number}
		{}
		{(\printfield{number})\space}%
	(\printfield{year})%
	\iffieldundef{pages}
		{\adddot}
		{\addcomma\space\printfield{pages}\adddot}%
	\space
	\printfield{doi}%
}

\DeclareBibliographyDriver{book}{%
	\printnames{author}%
	\addcomma\space
	\printfield{title}%
	\addcomma\space
	\iffieldundef{edition}
		{}
		{\printfield{edition}\addcomma\space}%
	\printlist{publisher}%
	\addcomma\space
	\printlist{location}%
	\space
	(\printfield{year})
	\adddot\space
	\printfield{doi}%
}

\DeclareBibliographyDriver{online}{%
	\printnames{author}%
	\addcomma\space
	\printfield{title}%
	\iffieldundef{year}
		{}
		{\space(\printfield{year})}%
	\addcomma\space
	\printfield{eprint}%
	\adddot\space
	\printfield{note}%
}

\theoremstyle{plain}
\newtheorem{theo}{Theorem}[section]
\newtheorem{prop}[theo]{Proposition}
\newtheorem{lem}[theo]{Lemma}
\newtheorem{cor}[theo]{Corollary}
\newtheorem*{thmA}{Theorem A}
\newtheorem*{thmB}{Theorem B}

\theoremstyle{definition}

\newtheorem*{rem}{Remark}

\newcommand{\ia}{\hspace*{2mm}{\rm({\it i\/}\rm)}\hspace{2mm}}
\newcommand{\ii}{\hspace*{2mm}{\rm({\it ii\/}\rm)}\hspace{2mm}}

\newcommand{\br}{\\[0.3em]}

\usepackage{enumitem}

\newcommand{\R}{\mathbb{R}}

\newcommand{\Z}{\mathbb{Z}}

\newcommand{\s}{\mathbb{S}}
\newcommand{\M}{\mathbb{M}}
\newcommand{\E}{\mathbb{E}}
\renewcommand{\H}{\mathbb{H}}

\newcommand{\J}{\mathcal{J}}

\newcommand{\Berger}{\s^3_\text{Berg}}
\newcommand{\EKT}{\E(\kappa,\tau)}

\newcommand{\HR}{\H^2\times\R}

\newcommand{\Nildrei}{\Nil_3}

\newcommand{\SR}{\s^2\times\R}
\newcommand{\SKR}{\s^2(\kappa)\times\R}
\newcommand{\MK}{\M^2(\kappa)}

\newcommand{\PSLzwei}{\widetilde{\PSL}_2(\R)}
\newcommand{\Soldrei}{\Sol_3}
\newcommand{\XKT}{X_a(\kappa,\tau)}

\newcommand{\norm}[1]{\left\lVert#1\right\rVert}

\DeclareMathOperator{\Isom}{Isom}
\DeclareMathOperator{\Nil}{Nil}
\DeclareMathOperator{\PSL}{PSL}

\DeclareMathOperator{\sgn}{sgn}

\DeclareMathOperator{\Sol}{Sol}

\DeclareMathOperator{\sn}{sn_\kappa}
\DeclareMathOperator{\sns}{sn^2_\kappa}

\DeclareMathOperator{\cs}{cs_\kappa}

\DeclareMathOperator{\ct}{ct_\kappa}
\DeclareMathOperator{\cts}{ct^2_\kappa}

\DeclareMathOperator{\arcs}{arcs_\kappa}
\DeclareMathOperator{\arct}{arct_\kappa}

\DeclareMathOperator{\tn}{tn_\kappa}

\newcommand{\Jmin}{J_\mathrm{min}}
\newcommand{\Jmax}{J_\mathrm{max}}
\newcommand{\Jtube}{J_\mathrm{tube}}
\newcommand{\family}{\mathcal{F}_{a,H}}

\renewcommand{\mod}{\operatorname{mod}}

\renewcommand{\epsilon}{\varepsilon}
\renewcommand{\theta}{\vartheta}
\renewcommand{\phi}{\varphi}

\newcommand{\terma}{\dfrac{\sns(r)}{\sns(r)+\left(4\tau\sns\!\left(\frac{r}{2}\right)-a\right)^2}} 
\newcommand{\termb}{\dfrac{\sns(r)+\left(4\tau\sns\!\left(\frac{r}{2}\right)-a\right)^2}{\sns(r)}} 
\newcommand{\termbs}{\frac{\sns(r)+\left(4\tau\sns\!\left(\frac{r}{2}\right)-a\right)^2}{\sns(r)}} 

\title[Screw motion CMC surfaces in homogeneous 3-manifolds]{Screw motion surfaces of constant mean curvature in homogeneous 3-manifolds}
\author{Philipp Käse}
\address{Technische Universität Darmstadt, Fachbereich Mathematik, AG Geometrie und Approximation, Schlossgartenstr. 7, 64289 Darmstadt, Germany}
\email{kaese@mathematik.tu-darmstadt.de}

\begin{document}

\maketitle

\begin{abstract}
	We study the geometry of non-minimal surfaces of supercritical constant mean curvature invariant under screw motions in the homogeneous 3-manifolds $\EKT$ including the space-forms of non-negative curvature. We give a complete classification, thereby unifying and extending various previous results. We give the first classification for the Berger sphere case, and we exhibit a new family of screw motion CMC surfaces, called tubes.
\end{abstract}

\section{Introduction}

Surfaces of constant mean curvature (CMC) play an important role in differential geometry and arise in various real world problems. For example, CMC surfaces are closely related to the isoperimetric problem, they arise naturally in general relativity, and they model interfaces.

The study of CMC surfaces invariant under a one-parameter group of isometries has received particular attention. In 1841 Delaunay described surfaces of mean curvature $H=1$ in Euclidean space invariant under rotation: cylinder, unduloids, sphere, and nodoids \cite{delaunay}. We refer to these surfaces as Delaunay surfaces. More than a century later, DoCarmo and Dajczer generalized this classification to CMC surfaces invariant under a screw motion \cite{docarmo_dajczer}, that is a rotation composed with a translation along the rotational axis.

In place of $\R^3$ it is natural to consider other ambient 3-dimensional Riemannian manifolds whose isometry groups contain screw motions, such as the space-forms $\s^3$ and $\H^3$, or more generally, the two-parameter family of spaces $\left\lbrace\EKT\colon(\kappa,\tau)\in\R^2\right\rbrace$. All $\EKT$-spaces are simply connected homogeneous manifold and there exists a Riemannian fibration
\begin{equation*}
	\EKT\overset{\pi}{\longrightarrow}\MK
\end{equation*}
over a simply connected 2-dimensional manifold $\MK$ of constant curvature $\kappa$ with geodesic fibers. We call $\kappa$ the \textit{base curvature} and $\tau\coloneqq\frac{1}{2}\,g\big([E_1,E_2],E_3\big)$ the \textit{bundle curvature}, where $(E_1,E_2,E_3)$ is orthonormal on $\EKT$ such that $(\pi(E_1),\pi(E_2))$ is an orthonormal frame on $\MK$. That is, $E_1$ and $E_2$ are \textit{horizontal} and $E_3$ is \textit{vertical} $(d\pi(E_3)=0)$.

For $\kappa-4\tau^2\neq0$ the isometry group is 4-dimensional, while for $\kappa-4\tau^2=0$ we recover the space-forms $\R^3$ and $\s^3(\kappa)$ with their 6-dimensional isometry group. The space-form $\H^3$ is not included since it does not admit a Riemannian fibration. In all cases the isometry group contains screw motions with respect to the fibers. The following geometries can be identified with $\EKT$ for different choices of $\kappa$ and $\tau$:

\begin{table}[H]
	\label{table:ekt}
	\renewcommand{\arraystretch}{2.0}
	\begin{tabular}{C{20mm}C{20mm}C{20mm}C{20mm}}
		\hline
		& $\kappa<0$ & $\kappa=0$ & $\kappa>0$ \\
		\hline
		$\tau=0$ & $\HR$ & $\big(\R^3\big)$ & $\SR$ \\
		$\tau\neq0$ & $\PSLzwei$ & $\Nildrei$ & $\Berger$ \\
		\hline
	\end{tabular}
\end{table}

All these manifolds are Thurston geometries \cite[Chap.~3.8]{thurston}, except for the Berger spheres $\Berger$, which do not satisfy the maximality condition on the isometry group. For a detailed account we refer to \cite[Chap.~4]{scott}.

Rotational and screw motion CMC surfaces in $\EKT$ are natural generalizations of the Delaunay surfaces which have been studied extensively over the past decades. Hsiang and Hsiang \cite{hsiang_hsiang} studied rotational CMC surfaces in the product space $\H^2\times\R$. The generalization to screw motion surfaces was done by Montaldo and Onnis \cite{montaldo_onnis}, and in a different setup by Sa Earp and Toubiana \cite{saearp_toubiana}, who also treated the case of positive base curvature $\s^2\times\R$. The closely related case $\s^1\times\MK$ has been studied by Pedrosa and Ritor\'{e} \cite{pedrosa_ritore}. For Heisenberg space $\Nildrei$, the screw motion surfaces were completely described by Figueroa, Mercuri, and Pedrosa \cite{figueroa_mercuri_pedrosa}, while the rotational case was first described by Tompter \cite{tompter} and Caddeo, Piu, and Ratto \cite{caddeo_piu_ratto}. Rotational and screw motion surfaces in $\PSLzwei$ have been studied by Pe\~{n}afiel \cite{penafiel12,penafiel15}. Rotational surfaces in Berger spheres $\Berger$ have been studied by Torralbo \cite{torralbo}. Moreover, there is a recent result by Manzano \cite{manzano} dealing with screw motions with horizontal orbits in $\Berger$, where new interesting examples arise. However, by our knowledge there is no work about the classification of screw motion CMC surfaces in $\s^3(\kappa)$ and $\Berger$.

The above mentioned results were phrased as follows: All non-minimal screw motion CMC surfaces with supercritical mean curvature $(4H^2+\kappa>0)$ are of generalized Delaunay type. However, this classification is not complete as we will prove in this paper. We mention that for critical $(4H^2+\kappa=0)$ and subcritical mean curvature $(4H^2+\kappa<0)$ the geometry is quite different, see \cite{montaldo_onnis,saearp_toubiana,penafiel12,penafiel15}, and so we exclude this case.

There are two approaches: Instead of looking at the 2-dimensional surface in the ambient 3-dimensional space, the invariance under screw motion can be used to consider the problem either in the 2-dimensional abstract quotient space or in the $xz$-plane for a particular model of $\EKT$ diffeomorphic to $\R^3$ by solving an ordinary differential equation. The authors of \cite{hsiang_hsiang,tompter,figueroa_mercuri_pedrosa,montaldo_onnis,torralbo} follow the first approach and consider parameterized curves in the quotient space which generate the CMC surfaces. This procedure is often referred to as the reduction procedure \cite{back_docarmo_hsiang}. The authors of \cite{docarmo_dajczer,saearp_toubiana,penafiel12,penafiel15} use the latter approach and consider a graph in the $xz$-plane. This leads to a non-parametric description. 

The present paper provides a complete classification of non-minimal screw motion CMC surfaces with supercritical mean curvature $(4H^2+\kappa>0)$ in the homogeneous 3-manifolds $\EKT$, which are understood to include the space-forms of non-negative curvature $\R^3$ and $\s^3(\kappa)$. The proof is based on the reduction procedure. Rather than considering just the canonical cases $\kappa=-1,0,1$ and $\tau=0,1$, we provide a unified treatment of all parameters $(\kappa,\tau)\in\R^2$. Therefore, our results can be understood to give a moduli space picture of the screw motion CMC surfaces in all $\EKT$-spaces. In particular, we present the first classification of screw motion CMC surfaces in $\s^3(\kappa)$ and Berger spheres:

\begin{thmA}
	The screw motion invariant surfaces in $\EKT$ of non-minimal supercritical constant mean curvature form a one-parameter family and are of one of the following six types: vertical cylinder, unduloid type, sphere type, nodoid type I, nodoid type II, or tube.
\end{thmA}

The precise statement can be found in Theorem~\ref{theo:classification} and a moduli space is shown in Figure~\ref{fig:modulispace}.  This classification includes a new family of screw motion CMC surfaces in $\EKT$, which we call \textit{tubes}. In particular, this shows that previous classifications were not complete:

\begin{thmB}
	For certain choices of $\kappa$, $\tau$ and $H$, there exist examples whose profile curves are simple loops and generate screw motion CMC surfaces, which are topological cylinders or tori, different from Delaunay type surfaces.
\end{thmB}

For the precise statement see Theorem \ref{theo:tube_existence} and Table \ref{tab:tube_existence}. An example of a tube can be found in Figure \ref{fig:nodoidfamily}. Our result includes the new examples recently found by Manzano \cite{manzano}.

\begin{figure}[H]
	\centering
	\tiny
	\def\svgwidth{0.95\textwidth}\executeiffilenewer{nodo_family.svg}{nodo_family.pdf}{inkscape -z -D --file=nodo_family.svg 	--export-pdf=nodo_family.pdf --export-latex}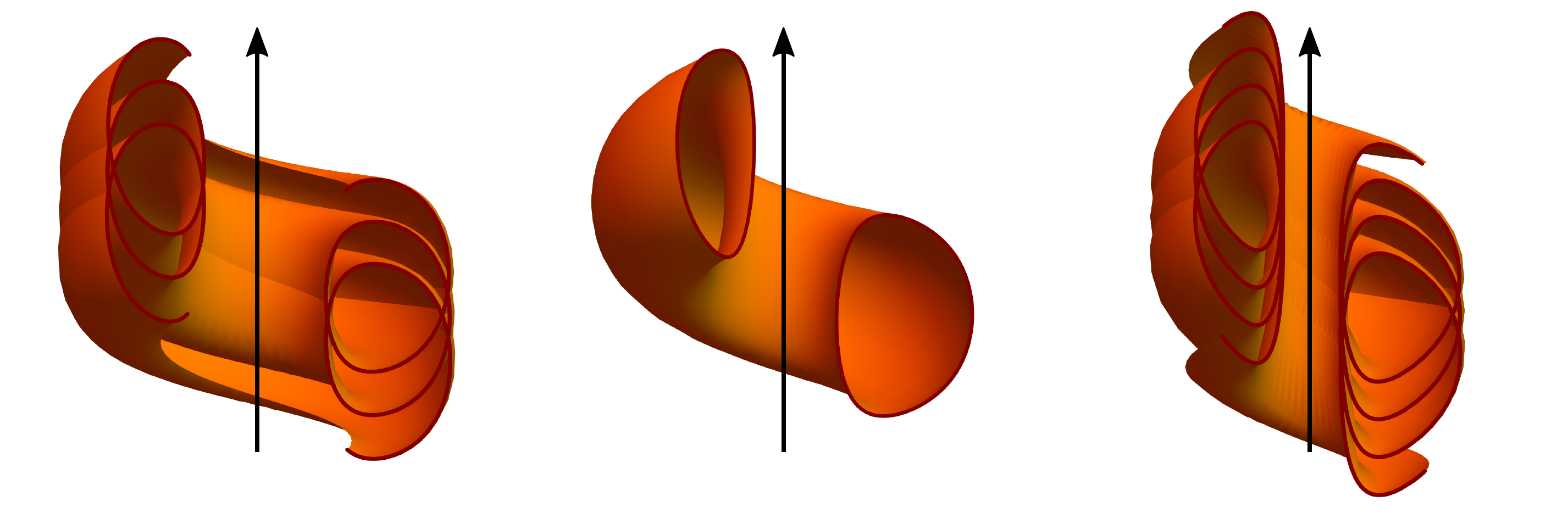
	\caption{Numerical plots of screw motion CMC surfaces in $\E(0,0.5)\cong\Nildrei$ for $H=0.5$. \textit{Left}: Nodoid type I. \textit{Center}: Tube. \textit{Right}: Nodoid type II.}
	\label{fig:nodoidfamily}
\end{figure}

\subsubsection*{Acknowledgments} The author would like to thank Karsten Große-Brauckmann and Francisco Torralbo for valuable comments and inspiring discussions. This paper is part of the authors PhD project.

\section{The ODE of constant mean curvature}

A common framework for all $\EKT$-spaces was originally introduced by Cartan \cite{cartan}, see also \cite{daniel_hauswirth_mira}. But in order to study screw motion surfaces, it is advantageous to use geodesic polar coordinates $(r, \theta, z)$ such that $r$ measures the arc-length of a geodesic in $\MK$. Therefore, we use the following model
\begin{equation}\label{model}
	I_\kappa\times[0,2\pi]\times\R\ni(r,\theta,z),
	\hspace{4mm}
	g\coloneqq dr^2
	+\sns\left(r\right)d\theta^2
	+\left(4\tau\sns\!\left(\frac{r}{2}\right)d\theta-dz\right)^2,
\end{equation}

where
\begin{equation*}
	I_\kappa\coloneqq\left\lbrace\begin{array}{ll}
		(0,\frac{\pi}{\sqrt{\kappa}}) & \text{ for } \kappa>0, \\[0.5em]
		(0,\infty) & \text{ for } \kappa\leq0.
	\end{array}\right.
\end{equation*}

We use here and in the following the generalized trigonometric functions
\begin{equation*}
	\sn(r)\coloneqq\left\lbrace\begin{array}{ll}
		\frac{1}{\sqrt{\kappa}}\sin\left(\sqrt{\kappa}r\right) \\[0.5em]
		r \\[0.5em]
		\frac{1}{\sqrt{-\kappa}}\sinh\left(\sqrt{-\kappa}r\right)
	\end{array}\right.
	\hspace{8mm}
	\cs(r)\coloneqq\left\lbrace\begin{array}{ll}
		\cos\left(\sqrt{\kappa}r\right) & \text{ for } \kappa>0, \\[0.5em]
		1 & \text{ for } \kappa=0, \\[0.5em]
		\cosh\left(\sqrt{-\kappa}r\right) \hspace{8mm}\, & \text{ for } \kappa<0,
	\end{array}\right.
\end{equation*}
as well as $\tn(r)\coloneqq\frac{\sn(r)}{\cs(r)}$ and $\ct(r)\coloneqq\frac{\cs(r)}{\sn(r)}$, all depending continuously on the base curvature $\kappa$.

For $\kappa\leq0$ this is a global model of $\EKT$ with the fiber at $r=0$ removed. For $\kappa>0$ it is a model for the fiber-wise universal covering of $\EKT$ without the fibers through the poles of the underlying $\s^2(\kappa)$ at $r=0$ and $r=\frac{\pi}{\sqrt{\kappa}}$. In particular, this applies to the case $\kappa-4\tau^2=0$ of the space-form $\s^3(\kappa)$. However, the regular part of $\EKT$ with respect to screw motion does not include the fiber at $r=0$, and no problem arises.

In geodesic polar coordinates, a screw motion with pitch $a\in\R$ by an angle $\phi\in\R$ is given by $L_{a,\phi}(r,\theta,z)=(r,\theta+\phi,z+a\phi)$. The set $G_a\coloneqq\lbrace L_{a,\phi}\,|\,\phi\in\R\rbrace$ of screw motions with fixed pitch is a closed one-parameter subgroup of $\Isom(\EKT)$ and therefore a Lie subgroup. This includes the rotational case for $a=0$ with closed orbits. Note however, for $\tau\neq0$ the orbits of the rotation are not horizontal with respect to the Riemannian fibration. The axis of the screw motion is the set of points invariant under $G_a$. It is the fiber over $r=0$.

It is a common procedure in equivariant geometry to use symmetry to reduce the dimension of the geometric problem: Instead of studying surfaces in $\EKT$, we consider the generating curves in the 2-dimensional quotient space $\EKT/G_a$. This so-called reduction procedure goes back to Back, DoCarmo and Hsiang, belatedly published in \cite{back_docarmo_hsiang}. It is also used in \cite{hsiang_hsiang,tompter,figueroa_mercuri_pedrosa,montaldo_onnis,torralbo}.

Since $G_a$ is a Lie group acting on $\EKT$ by isometries, the metric on the quotient space $\XKT\coloneqq\EKT/G_a$ can be constructed explicitly: There is one linearly independent Killing field $\xi_a=\frac{\partial}{\partial\theta} +a\frac{\partial}{\partial z}$ associated with the $G_a$-action. The quotient $\XKT$ can be locally parameterized by the $G_a$-invariant functions $r$ and $h\coloneqq z-a\theta$ of this Killing field, which induces a unique metric:

\begin{lem}\label{lem:quotient}
	Consider the model \eqref{model} and the one-parameter group of screw motions $G_a$ associated with the Killing field $\xi_a=\frac{\partial}{\partial\theta} +a\frac{\partial}{\partial z}$. The quotient space $\XKT=\EKT/G_a$ with the quotient metric $g_a$ is given by
	\begin{equation*}
		\XKT=I_\kappa\times\R\ni(r,h), \hspace{8mm}
		g_a= dr^2 +\terma\,dh^2.
	\end{equation*}
\end{lem}

\begin{proof}
	As stated in \cite[Prop.~2.1]{back_docarmo_hsiang} the quotient metric is given by $(g_a)_{ij}=h^{ij}$ with $h_{ij}=g(\nabla f_i,\nabla f_j)$, where in our case the invariant functions are $f_1=r$ and $f_2=h=z-a\theta$. A direct computation yields the stated result.
\end{proof}

In $\SKR$ or $\Berger$ the screw motion has two axes leading to further symmetries of the the quotient space. A screw motion around the fiber over the north pole $r=0$ agrees with the screw motion around the fiber over the south pole $r=\frac{\pi}{2\sqrt{\kappa}}$, provided we change the pitch from $a$ to $\tilde{a}=\frac{4\tau}{\kappa}-a$. For pitch $a=\frac{2\tau}{\kappa}$ with horizontal orbits we have $\tilde{a}=a$. This transformation is represented by reflection at the equatorial line $\lbrace\frac{\pi}{2\sqrt{\kappa}}\rbrace\times\R$:

\begin{lem}\label{lem:isometry}
	Suppose $\kappa>0$. For $\tilde a\coloneqq\frac{4\tau}{\kappa}-a$ the map
	\begin{equation*}
		\Psi_a\!:\quad \XKT\to X_{\tilde a}(\kappa,\tau),\quad (r,h)\mapsto\left(\frac{\pi}{\sqrt{\kappa}}-r,h\right)
	\end{equation*}
	is an isometry.
\end{lem}

\begin{proof}
	Straightforward computation using various trigonometric identities.
\end{proof}

Now let $\gamma(t)=(r(t),h(t))$ be a unit-speed curve in the quotient space $\XKT$ and let $\Sigma_\gamma$ be the surface generated by $\gamma$ under $G_a$. Furthermore, let $\sigma(t)\in\R$ be the angle between the coordinate direction $\frac{\partial}{\partial r}$ and the tangent vector $\gamma'(t)$. We obtain the following ODE which $\gamma$ must satisfy in order for $\Sigma_\gamma$ to have constant mean curvature $H$:

\begin{theo}\label{theo:ode}
	Let $\gamma=(r,h)$ be a unit-speed curve in the orbit space $\XKT=\EKT/G_a$ and $\Sigma_\gamma=G_a(\gamma)$ the screw motion surface generated by $\gamma$ under $G_a$. Then $\Sigma_\gamma$ has constant mean curvature $H$ if and only if
	\begin{equation}\label{ode3}
		2H =\ct(r)\sin\sigma+\sigma'.
	\end{equation}
\end{theo}

\begin{proof}
	The Reduction Theorem of Back, DoCarmo and Hsiang \cite[Prop.~4.1]{back_docarmo_hsiang} gives the differential equation $2H=k_\mathrm{geod}-\partial_{n}\ln\norm{\xi_a}_g$, where $k_\mathrm{geod}$ is the geodesic curvature of $\gamma$ and $\partial_{n}\ln\norm{\xi_a}_g=g_a(\nabla\ln\norm{\xi_a}_g,n)$ is the normal derivative. A straightforward computation gives the stated ODE.
\end{proof}

Equation \eqref{ode3} together with the unit-speed condition $\norm{\gamma'}_{g_a}=1$ gives the following system of ODEs for $\gamma$:
\begin{equation}\label{ode}\tag{ODE}
	\left\lbrace\begin{array}{l}
		r' =\cos\sigma, \\[1.0em]
		h'=\sqrt{\termb}\,\sin\sigma, \\[1.6em]
		2H =\ct(r)\sin\sigma+\sigma'.
	\end{array}\right.
\end{equation}\vspace{2mm}

The ODE is Lipschitz for $r\in I_\kappa$ and solutions are unique for given initial data. A trivial solution is $r(t)=r_0, \sigma(t)=\frac{\pi}{2}$, which generates a vertical cylinder of CMC $H=\frac{1}{2}\ct(r_0)$.

We gather some elementary observations about \eqref{ode}, which follow from uniqueness:

\begin{lem}\label{lem:properties_of_ode}
	Let $\gamma=(r,h,\sigma)$ be a solution curve of \eqref{ode} for mean curvature $H$.\vspace{-2mm}
	\begin{enumerate}[leftmargin=10mm]\setlength{\itemsep}{2mm}
		\item Any vertical translate of $\gamma$ is again a solution curve for $H$.
		\item Any reflection of $\gamma$ across a line $h=h_0$ is a solution curve for $-H$.
		\item Any linear reparametrization $\gamma\circ\phi$ with $\phi(t)=\epsilon t+b$, where $\epsilon\in\lbrace\pm1\rbrace$ and $b\in\R$, is a solution curve for $\epsilon H$.
		\item If $\gamma$ is defined for $t\in(t_0-\epsilon,t_0]$ with $r'(t_0)=0$, then $\gamma$ can be extended to a solution curve defined on the interval $(t_0-\epsilon,t_0+\epsilon)$ by reflection across the line $h=h(t_0)$ and linear reparametrization.
	\end{enumerate}
\end{lem}

Because of the invariance under vertical translation by Lemma \ref{lem:properties_of_ode}, there exists a conserved quantity according to the Noether Theorem:
\begin{prop}\label{prop:J}
	The surface $\Sigma_\gamma$ generated by $\gamma(t)=(r(t),h(t),\sigma(t))$ has constant mean curvature $H$ if and only if the energy
	\begin{equation}\label{J}
		J(r(t),\sigma(t))\coloneqq\left\lbrace\begin{array}{ll}
			\dfrac{2H}{\kappa}\Big(\!\cs(r(t))-1\Big) +\sn(r(t))\sin\sigma(t) \hspace{6mm}\, & \text{ for } \kappa\neq0 \\[1.4em]
			-Hr^2+r\sin\sigma & \text{ for } \kappa=0
		\end{array}\right.
	\end{equation}
	is a constant function of $t$.
\end{prop}

\begin{rem}
	The energy $J$ extends to $\kappa=0$ continuously by taking the limit $\kappa\to0$. For simplicity, we only address the case $\kappa\neq0$ in the following with the understanding that the case $\kappa=0$ is obtained by taking the limit.
\end{rem}

From now on we restrict our considerations to non-minimal $(H>0)$ and supercritical mean curvature $(4H^2+\kappa>0)$ as explained in the introduction.

\begin{lem}\label{lem:J}
	For non-minimal and supercritical mean curvature $H$ the energy \eqref{J} is bounded from above and
	\begin{equation*}
		J(r(t),\sigma(t))\leq\frac{\sqrt{4H^2+\kappa}-2H}{\kappa}\eqqcolon \Jmax
	\end{equation*}
	holds with equality precisely for the vertical cylinder $r(t)=\arct(2H)$ and $\sigma(t)=\frac{\pi}{2}(\mod 2\pi)$.
\end{lem}

\begin{proof}
	The inequality
	\begin{equation*}
		J =\frac{2H}{\kappa}\Big(\!\cs(r)-1\Big) +\sn(r)\sin\sigma
		\leq \frac{2H}{\kappa}\Big(\!\cs(r)-1\Big) +\sn(r),
	\end{equation*}
	holds with equality if and only if $\sigma=\frac{\pi}{2}(\mod 2\pi)$. By differentiating the right hand side we see that it attains its maximum if and only if $r=\arct(2H)$. Thus,
	\begin{equation*}
		J \leq \frac{2H}{\kappa}\Big(\!\cs(\arct(2H))-1\Big) +\sn(\arct(2H))
		=\frac{\sqrt{4H^2+\kappa}-2H}{\kappa}.
		\qedhere
	\end{equation*}
\end{proof}

The isometry $\Psi_a$ from Lemma \ref{lem:isometry} implies a symmetry for the energy:

\begin{lem}\label{lem:symmetry}
	Suppose $\kappa>0$. For every $G_a$-invariant CMC-$H$ surface $\Sigma_\gamma$ with energy $J$ there is a $G_a$-invariant CMC-$H$ surface $\Sigma_{\tilde\gamma}$ with energy $\tilde J=-J-\frac{4H}{\kappa}$. The relation between the generating curves $\gamma=(r,h,\sigma)$ and $\tilde\gamma=(\tilde r,\tilde h,\tilde\sigma)$ is given via $\tilde r=\frac{\pi}{\sqrt{\kappa}}-r$ and $\tilde\sigma=\sigma+\pi$. Moreover, if $\tau=0$ then there exists a constant $h_0\in\R$ such that $\tilde h=-h+h_0$.
\end{lem}

\begin{proof}
	Suppose $\tilde\gamma=(\tilde r,\tilde h,\tilde\sigma)$ with $\tilde r=\frac{\pi}{\sqrt{\kappa}}-r$ and $\tilde\sigma=\sigma+\pi$. Then
	\begin{equation*}
		\begin{aligned}
			J\left(\frac{\pi}{\sqrt{\kappa}}-r,\sigma+\pi\right)
			&=2H\cos(\pi-\sqrt{\kappa}r) +\sqrt{\kappa}\sin(\pi-\sqrt{\kappa}r)\sin(\sigma+\pi) \\
			&=-2H\cos(\sqrt{\kappa}r) -\sqrt{\kappa}\sin(\sqrt{\kappa}r)\sin\sigma \\
			&=-J-\frac{4H}{\kappa}.
		\end{aligned}
	\end{equation*}
	
	Furthermore, \eqref{ode} yields $\tilde H=H$ and $\tilde h'=-h'$ if $\tau=0$.
\end{proof}

\begin{rem}
	The previous lemma establishes a one-to-one correspondence between surfaces with $J>-\frac{2H}{\kappa}$ and surfaces with $J<-\frac{2H}{\kappa}$ for $\kappa>0$ due to the change of axis. Therefore it suffices to restrict our discussion to $J\geq-\frac{2H}{\kappa}$. It further implies that the energy is also bounded from below by $\Jmin\coloneqq-\Jmax-\frac{4H}{\kappa}$. However, for $\kappa\leq0$ the energy is not bounded from below and $\Jmin\to-\infty$ as $\kappa\searrow0$.
\end{rem}

\section{Classification of screw motion CMC surfaces}
\label{sec:solutions}

We classify the screw motion CMC surfaces in $\EKT$ including $\R^3$ and $\s^3(\kappa)$ by classifying the solution curves of \eqref{ode} in terms of the energy $J$. Our result includes and extends the works \cite{delaunay,docarmo_dajczer,figueroa_mercuri_pedrosa,hsiang_hsiang,montaldo_onnis,penafiel12,penafiel15,saearp_toubiana,tompter,torralbo}, and generalizes \cite[Thm.~1.1]{manzano_torralbo}:

\begin{theo}[Classification]\label{theo:classification}
	The non-minimal surfaces in $\EKT$ of supercritical constant mean curvature $H$, which are invariant under screw motions of pitch $a\in\R$, form a continuous one-parameter family
	\begin{equation}\label{eq:family}
		\family\coloneqq \Big\lbrace J\mapsto\Sigma_{a,H,J}\colon
		J\in(-\infty,\Jmax] \text{ if } \kappa\leq0 \text{ resp. } J\in[-\tfrac{2H}{\kappa},\Jmax] \text{ if } \kappa>0 \Big\rbrace
	\end{equation}
	parameterized by the energy $J$, where $\Jmax\coloneqq(\sqrt{4H^2+\kappa}-2H)/\kappa$. All surfaces are cylindrically bounded. The geometry of the surface $\Sigma_{a,H,J}$ with profile curve $\gamma=(r,h,\sigma)$ depends on the value of $J$:\vspace{-2mm}
	\begin{enumerate}[leftmargin=10mm]\setlength{\itemsep}{2mm}
		\item $J=\Jmax$: The profile curve is a vertical straight line and the surface is a \textbf{vertical round cylinder} of radius $r=\arct(2H)$.
		\item $J\in(0,\Jmax)$: The profile curve has $h$ increasing and $r$ and $\sigma$ are periodic. The surface is topologically an embedded cylinder $(a=0)$ or an immersed plane $(a\neq0)$ and we call it an \textbf{unduloid type surface}.
		\item $J=0$: The profile curve is homeomorphic to a semi-circle touching the screw motion axis and $h$ is increasing. The surface is topologically a doubly punctured sphere $(a=0)$ or an immersed incomplete strip $(a\neq0)$ and we call it a \textbf{sphere type surface}.
		\item $J<0$: The profile curve has $\sigma$ increasing and $r$, $h'$, $\sigma(\mod 2\pi)$ are periodic. We distinguish further according to the sign of the vertical period $\Delta\coloneqq h(\sigma+2\pi)-h(\sigma)$:\vspace{2mm}
		\begin{enumerate}\setlength{\itemsep}{2mm}
			\item $\Delta>0$: For positive period the surface is topologically an immersed cylinder $(a=0)$ or an immersed plane $(a\neq0)$ and we call it a \textbf{nodoid type I surface}.
			\item $\Delta=0$: For vanishing period the profile curve is a simple loop. The surface is topologically an embedded or immersed torus or cylinder and we call it a \textbf{tube}.
			\item $\Delta<0$: For negative period the surface is topologically an immersed cylinder $(a=0)$ or an immersed plane $(a\neq0)$ and we call it a \textbf{nodoid type II surface}.
		\end{enumerate}
	\end{enumerate}
	The profile curves are visualized in Figure \ref{fig:classification}. An example of the moduli space is displayed in Figure \ref{fig:modulispace}.
\end{theo}

\begin{figure}[h]
	\centering
	\tiny
	\def\svgwidth{\textwidth}\executeiffilenewer{curve_family.svg}{curve_family.pdf}{inkscape -z -D --file=curve_family.svg 	--export-pdf=curve_family.pdf --export-latex}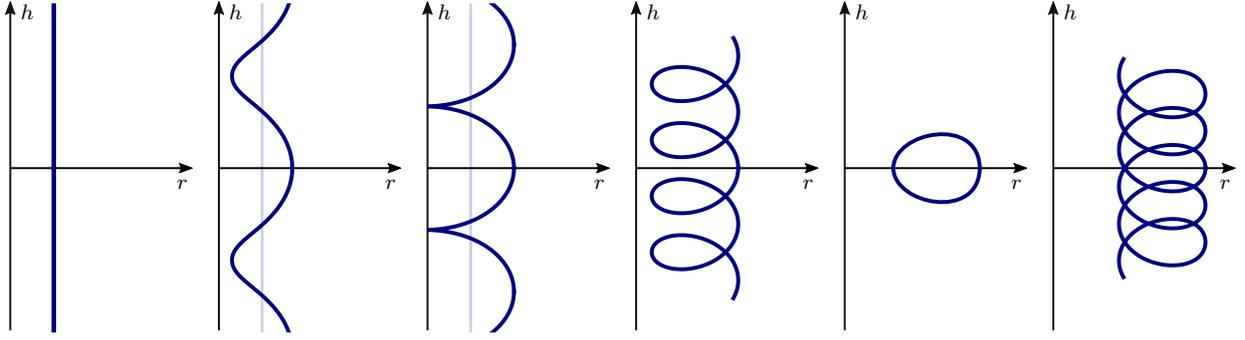
	\caption{Profile curves of the screw motion CMC surfaces in $\EKT$ according to Theorem \ref{theo:classification} (left to right): vertical cylinder, unduloid type, sphere type, nodoid type I, tube, nodoid type II. In the second and third picture the vertical cylinder is drawn for reference, while in the other cases the position depends on the choice of parameters.}
	\label{fig:classification}
\end{figure}

\begin{figure}[h]
	\centering
	\scriptsize
	\def\svgwidth{0.85\textwidth}\executeiffilenewer{Modulraum.svg}{Modulraum.pdf}{inkscape -z -D --file=Modulraum.svg 	--export-pdf=Modulraum.pdf --export-latex}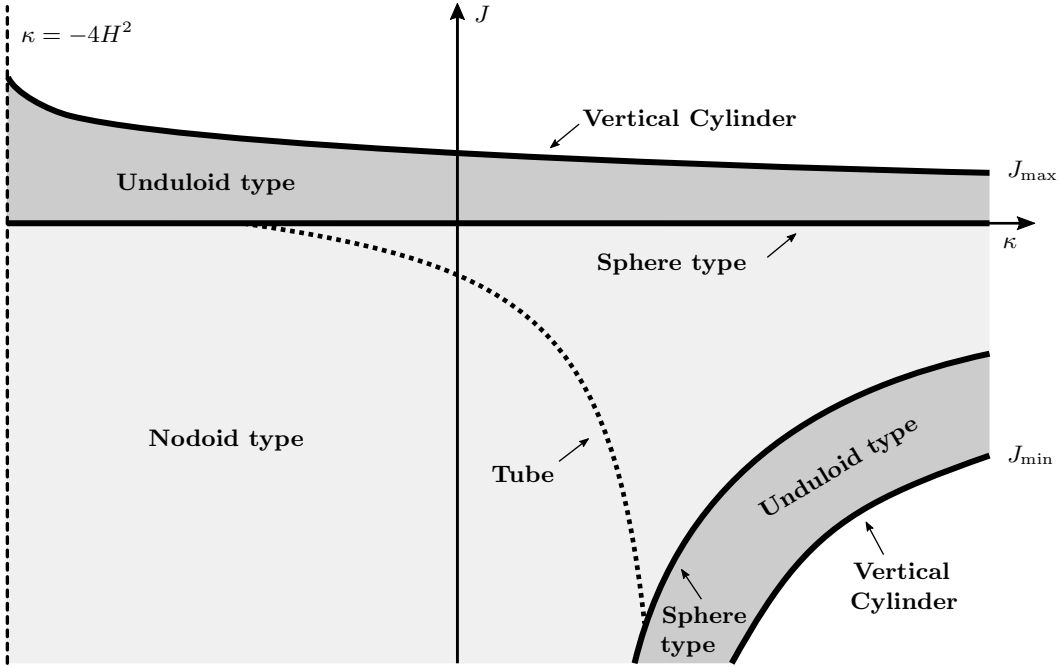
	\caption{Moduli space of the family $\family$ of screw motion CMC surfaces in $\E(\kappa,\tau)$ depending on the energy $J\in(-\infty,\Jmax]$ resp. $J\in[\Jmin,\Jmax]$ and the base curvature $\kappa$ for fixed mean curvature $H$, $\tau\neq0$ (chose pitch $a>\frac{1}{2\tau}$), according to Theorems \ref{theo:classification} and \ref{theo:tube_existence}. The dotted line for the tube is obtained numerically, see Corollary \ref{cor:energy_range} and the remark afterwards.}
	\label{fig:modulispace}
\end{figure}

Our terminology for the surfaces is based on their rotational invariant counterpart. While the profile curves are closely related, the topology of the surface can be different. The term \textit{tube} refers to the fact that the surface may be a torus or a cylinder. However, there is no additional geometric meaning associated in this context, i.e., they are not Riemannian tubes. Numerical examples of profile curves are shown in Figures \ref{fig:curves_skr}, \ref{fig:curves_hkr}, \ref{fig:curves_berger} for the rotational case ($a=0$). Varying the pitch does not change the picture qualitatively.

To prove Theorem \ref{theo:classification} we study the behavior of the generating profile curves at different energy levels. Section \ref{sec:positiveenergy} covers positive energy and Section \ref{sec:negativeenergy} negative energy. Each section is divided into two parts: First we find bounds on $r$ and study the monotonicity of~$h$ and~$\sigma$, see Lemmas \ref{lem:unduloid} and \ref{lem:nodoid}. Second we choose initial conditions followed by a qualitative description of the profile curve, see Propositions \ref{prop:unduloid} and \ref{prop:nodoid}. For negative energy we further discuss the existence of tubes in Theorem \ref{theo:tube_existence}. The case of zero energy then follows as a limiting case, see Proposition \ref{prop:sphere}. Our reasoning generalizes the analysis given in \cite{tompter} for the rotational case $(a=0)$ in Heisenberg space $\Nildrei$.

\subsection{Solution curves of positive energy}
\label{sec:positiveenergy}

The profile curves are given as solutions of \eqref{ode}. For $J=\Jmax$ the solution curve $\gamma$ is a vertical straight line and the generated surface a vertical cylinder of radius $\arct(2H)$, see Lemma \ref{lem:J}. For $J\in(0,\Jmax)$ the solution curve is not a straight line:

\begin{lem}\label{lem:unduloid}
	Let $\gamma=(r,h,\sigma)$ be a solution curve with $J\in(0,\Jmax)$. Then: \br
	\ia $h$ is everywhere strictly increasing and $\sigma(t)$ can be assumed to be in $(0,\pi)$ for all $t$. \br
	\ii $r$ attains values exactly in the interval $[r_-,r_+]$, where
	\begin{equation*}
		\begin{aligned}
			r_\pm&=\arct\!\left(2H\right) \pm\arcs\!\left(\frac{\kappa J+2H}{\sqrt{4H^2+\kappa}}\right).
		\end{aligned}
	\end{equation*}
	The minimal or maximal value are attained if and only if $\sigma(t)=\frac{\pi}{2}$.
\end{lem}

\begin{proof}
	(\textit{i}) The definition of the energy $J$ yields
	\begin{equation*}
		\sin\sigma=\frac{\kappa J-2H(\cs(r)-1)}{\kappa\sn(r)}
		>\frac{0-2H(\cs(r)-1)}{\kappa\sn(r)}
		=2H\left(\frac{1-\cs(r)}{\kappa\sn(r)}\right) >0.
	\end{equation*}
	
	Thus, $h'=\sqrt{\termbs}\,\sin\sigma>0$ and without loss of generality $\sigma\in(0,\pi)$.
	
	(\textit{ii}) Rewrite the energy using trigonometric identities as
	\begin{equation*}
		J=\frac{1}{\kappa}\sqrt{4H^2+\kappa\sin^2\sigma}\, \cs\!\left(r-\arct\!\left(\frac{2H}{\sin\sigma}\right)\!\right)-\frac{2H}{\kappa}.
	\end{equation*}
	
	Solving this expression for $r$ gives
	\begin{equation}\label{radius}
		r=\arct\!\left(\frac{2H}{\sin\sigma}\right)\pm\arcs\!\left(\frac{\kappa J+2H}{\sqrt{4H^2+\kappa\sin^2\sigma}}\right).
	\end{equation}
	This expression yields the maximal and minimal radius. On one hand
	\begin{equation*}
		\arct\!\left(\frac{2H}{\sin\sigma}\right)
		+\arcs\!\left(\frac{\kappa J+2H}{\sqrt{4H^2+\kappa\sin^2\sigma}}\right)
		\leq \arct\!\left(2H\right)
		+\arcs\!\left(\frac{\kappa J+2H}{\sqrt{4H^2+\kappa}}\right)\eqqcolon r_+
	\end{equation*}
	and on the other hand
	\begin{equation*}
		\arct\!\left(\frac{2H}{\sin\sigma}\right)
		-\arcs\!\left(\frac{\kappa J+2H}{\sqrt{4H^2+\kappa\sin^2\sigma}}\right)
		\geq \arct\!\left(2H\right)
		-\arcs\!\left(\frac{\kappa J+2H}{\sqrt{4H^2+\kappa}}\right) \eqqcolon r_-
	\end{equation*}
	with equality in both cases for $\sigma=\frac{\pi}{2}$.
\end{proof}

This allows us to choose initial conditions without loss of generality:

\begin{prop}\label{prop:unduloid}
	Let $\gamma=(r,h,\sigma)$ be a solution curve with $J\in(0,\Jmax)$ for initial data $r(0)=r_+, h(0)=0, \sigma(0)=\frac{\pi}{2}$. Then there exist $0<t_1<t_2<\infty$ such that:
	\begin{itemize}
		\item $\sigma'(t)>0$ for $t\in[0,t_1)$, \\[-0.5em]
		\item $\sigma'(t_1)=0$, \\[-0.5em]
		\item $\sigma'(t)<0$ for $t\in(t_1,t_2]$, \\[-0.5em]
		\item $\sigma(t_2)=\frac{\pi}{2}$ and $r(t_2)=r_-$.
	\end{itemize}
	The solution extends to a curve on $\R$ by successive reflections at heights $kh(t_2), k\in\Z$, see Figure \ref{fig:generatingcurve} (left). In particular, $r$ and $\sigma$ are $2t_2$-periodic. The generated surface is of unduloid type.
\end{prop}

\begin{figure}[h]
	\centering
	\scriptsize
	\def\svgwidth{0.7\textwidth}\executeiffilenewer{profilecurve_proof.svg}{profilecurve_proof.pdf}{inkscape -z -D --file=profilecurve_proof.svg 	--export-pdf=profilecurve_proof.pdf --export-latex}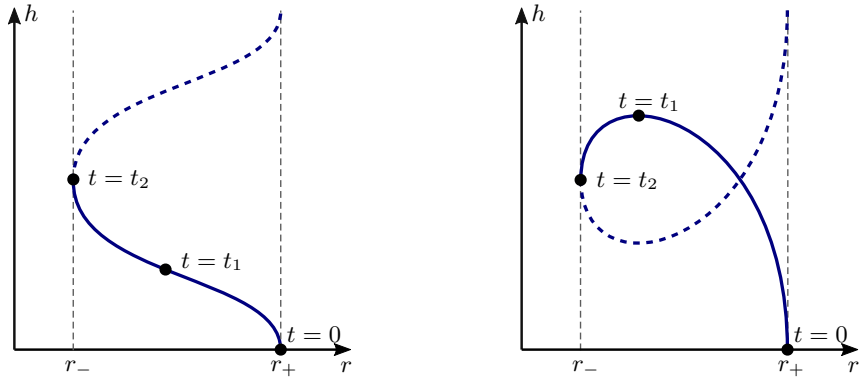
	\caption{Solution curves of \eqref{ode} for initial data $r(0)=r_+, h(0)=0, \sigma(0)=\frac{\pi}{2}$. \textit{Left:} Unduloid type for $0<J<\Jmax$. \textit{Right:} Nodoid type for $J<0$.}
	\label{fig:generatingcurve}
\end{figure}

\begin{proof}
	We divide the proof into several steps:
	
	(\textit{i}) $\sigma'(0)>0$. Note, that $r_+$ as a function of $J$ is strictly decreasing. Therefore,
	\begin{equation*}
		\sigma'(0)=2H-\ct\!\big(r_+(J)\big)
		>2H-\ct\!\big(r_+(\Jmax)\big)=2H-\ct\big(\!\arct(2H)\big)=0.
	\end{equation*}
	
	(\textit{ii}) \textit{There exists $t_1>0$ such that $\sigma'(t_1)=0$.}
	
	On the contrary, assume $\sigma'\neq0$ for all $t>0$. Then $\sigma'>0$ by (\textit{i}) and continuity, and $\sigma>\frac{\pi}{2}$ by initial data. Since on the other hand $\sigma<\pi$, this implies $r'=\cos\sigma$ is everywhere negative and bounded away from $0$ for $t>\epsilon>0$. Thus, at some point $r<r_-$, which is a contradiction. Without loss of generality we choose $t_1>0$ to be the smallest possible with this property.
	
	(\textit{iii}) A critical point of $\sigma$ is a minimum if $\sigma\in(0,\frac{\pi}{2})$ and a maximum if $\sigma\in(\frac{\pi}{2},\pi)$.
	
	The second derivative of $\sigma$ is given by
	\begin{equation*}
		\sigma''=\big(2H-\ct(r)\sin\sigma\big)'
		=\big(\kappa+\cts(r)\big)\cos\sigma\sin\sigma-\sigma'\ct(r)\cos\sigma.
	\end{equation*}
	
	At a critical point $(\sigma'=0)$ this becomes $\sigma''=(\kappa+\cts(r))\cos\sigma\sin\sigma$. Since $(\kappa+\cts(r))>0$, this implies $\sigma''>0$ if and only if $\sigma\in(0,\frac{\pi}{2})$ and $\sigma''<0$ if and only if $\sigma\in(\frac{\pi}{2},\pi)$. In particular, $\sigma(t_1)$ is a maximum.
	
	(\textit{iv}) \textit{There exists $t_2>t_1$ such that $\sigma(t_2)=\frac{\pi}{2}$ and $r(t_2)=r_-$.}
	
	For $t>t_1$ the angle $\sigma$ starts decreasing ($\sigma$ has a maximum at $t_1$). Since there is no minimum for $\sigma>\frac{\pi}{2}$ by (\textit{iii}), it is decreasing as long as it is greater than $\frac{\pi}{2}$. We claim, this implies there exists $t_2>t_1$ such that $\sigma(t_2)=\frac{\pi}{2}$. If not, $\sigma$ must converge to some $\alpha\geq\frac{\pi}{2}$ from above. If $\alpha>\frac{\pi}{2}$, then $r'=\cos\sigma$ is everywhere negative and bounded away from~$0$. Thus, $r<r_-$ at some point, which is a contradiction. If $\alpha=\frac{\pi}{2}$, then $r\to \hat r$ for some $\hat r<r(t_1)< \arct(2H)$. Then \eqref{ode} implies $\sigma'=2H-\ct(r)\sin\sigma\to 2H-\ct(\hat r)<0$, i.e., $\sigma'$ is bounded away from $0$, which would contradict convergence.	Without loss of generality we choose again $t_2>t_1$ to be the smallest possible with this property. Lemma \ref{lem:unduloid} states that $r$ attains at $t_2$ its minimum $r_-$ or its maximum $r_+$. But $\sigma(t)>\frac{\pi}{2}$ for all $t\in(0,t_2)$, and so $r(t_2)<r_+$. Therefore, $r(t_2)=r_-$.
	
	By Lemma \ref{lem:properties_of_ode} we can extend the solution curve from $[0,t_2]$ to $\R$. The full solution fulfills $r(t+2t_2)=r(t)$ and $\sigma(t+2t_2)=\sigma(t)$ for all $t\in\R$. Hence, it is periodic in $r$ and $\sigma$.
\end{proof}

\subsection{Solution curves of negative energy}
\label{sec:negativeenergy}

\begin{lem}\label{lem:nodoid}
	Let $\gamma=(r,h,\sigma)$ be a solution curve with $J\in(-\infty,0)$ if $\kappa\leq0$ resp. $J\in[-\frac{2H}{\kappa},0)$ if $\kappa>0$. Then: \br
	\ia $\sigma$ is everywhere strictly increasing. \br
	\ii $r$ attains values exactly in the interval $[r_-,r_+]$, where
	\begin{equation*}
		\begin{aligned}
			r_\pm&=\pm\arct\!\left(2H\right)+\arcs\!\left(\frac{\kappa J+2H}{\sqrt{4H^2+\kappa}}\right).
		\end{aligned}
	\end{equation*}
	The maximal value is attained if and only if $\sigma(t)=\frac{\pi}{2} (\mod 2\pi)$ and the minimal value is attained if and only if $\sigma(t)=\frac{3\pi}{2} (\mod 2\pi)$. 
\end{lem}

\begin{proof}
	(\textit{i}) Assume $\kappa\leq0$. Then $\ct(r)>0$. Solving $J=\frac{2H}{\kappa}(\cs(r)-1) +\sn(r)\sin\sigma<0$ for $\sin\sigma$ and substituting in \eqref{ode} gives
	\begin{equation*}
		\begin{aligned}
			\sigma'&=2H-\ct(r)\sin\sigma
			>2H-\ct(r)\,\frac{2H(1-\cs(r))}{\kappa\sn(r)} 
			=2H\,\left(\frac{1-\cs(r)}{\kappa\sns(r)}\right)>0.
		\end{aligned}
	\end{equation*}
	
	Now assume $\kappa>0$. If $r\in(0,\frac{\pi}{2\sqrt{\kappa}})$, then $\ct(r)>0$ and the above proof carries over. But for $r\in[\frac{\pi}{2\sqrt{\kappa}},\frac{\pi}{\sqrt{\kappa}})$ it holds $\ct(r)\leq0$. Solving $J=\frac{2H}{\kappa}(\cs(r)-1) +\sn(r)\sin\sigma>-\frac{2H}{\kappa}$ for $\sin\sigma$ and substituting in \eqref{ode} gives
	\begin{equation*}
		\begin{aligned}
			\sigma'&=2H-\ct(r)\sin\sigma 
			\geq2H+\ct(r)\,\frac{2H\cs(r)}{\kappa\sn(r)} 
			=2H\,\left(1+\frac{\cts(r)}{\kappa}\right)>0.
		\end{aligned}
	\end{equation*}
	Thus, $\sigma'>0$ everywhere.
	
	(\textit{ii}) Analogously to the proof of Lemma \ref{lem:unduloid} we obtain
	\begin{equation}\label{radius2}
		r=\arct\!\left(\frac{2H}{\sin\sigma}\right)+\arcs\!\left(\frac{\kappa J+2H}{\sqrt{4H^2+\kappa\sin^2\sigma}}\right).
	\end{equation}
	Note that the solution of \eqref{radius} with negative sign corresponds to $r<0$ and thus does not apply here. We obtain the maximal and minimal radius by estimates similar to before. On the one hand 
	\begin{equation*}
		\arct\!\left(\frac{2H}{\sin\sigma}\right)+\arcs\!\left(\frac{\kappa J+2H}{\sqrt{4H^2+\kappa\sin^2\sigma}}\right)
		\leq \arct\!\left(2H\right)+\arcs\!\left(\frac{\kappa J+2H}{\sqrt{4H^2+\kappa}}\right)
		\eqqcolon r_+,
	\end{equation*}
	
	while on the other hand
	\begin{equation*}
		\arct\!\left(\frac{2H}{\sin\sigma}\right)+\arcs\!\left(\frac{\kappa J+2H}{\sqrt{4H^2+\kappa\sin^2\sigma}}\right)
		\geq -\arct\!\left(2H\right)+\arcs\!\left(\frac{\kappa J+2H}{\sqrt{4H^2+\kappa}}\right)
		\eqqcolon r_-.
		\qedhere
	\end{equation*}
\end{proof}

This again allows us to choose initial conditions without loss of generality:

\begin{prop}\label{prop:nodoid}
	Let $\gamma=(r,h,\sigma)$ be a solution curve with $J\in(-\infty,0)$ if $\kappa\leq0$ resp. $J\in[-\frac{2H}{\kappa},0)$ if $\kappa>0$ for initial data $r(0)=r_+$, $h(0)=0$, $\sigma(0)=\frac{\pi}{2}$. Then there exist $0<t_1<t_2<\infty$ such that:
	\begin{itemize}
		\item $h'(t)>0$ for $t\in[0,t_1)$, \\[-0.5em]
		\item $\sigma(t_1)=\pi$ and $h'(t_1)=0$, \\[-0.5em]
		\item $h'(t)<0$ for $t\in(t_1,t_2]$, \\[-0.5em]
		\item $\sigma(t_2)=\frac{3\pi}{2}$ and $r(t_2)=r_-$.
	\end{itemize}
	The solution extends to a curve on $\R$ by successive reflections at heights $kh(t_2), k\in\Z$, see Figure \ref{fig:generatingcurve} (right). In particular, $r$ and $\sigma(\mod 2\pi)$ are $2t_2$-periodic. The generated surface is of nodoid type or a tube.
\end{prop}

\begin{proof}
	We divide the proof into several steps:
	
	(\textit{i}) \textit{There exists $t_1>0$ such that $\sigma(t_1)=\pi$.}
	
	Assume $\sigma\neq\pi$ everywhere. Then $\sigma\in(\frac{\pi}{2},\pi)$ for $t>0$ by continuity and initial data, since $\sigma$ is increasing. Then $r'=\cos\sigma$ is everywhere negative and bounded away from $0$ for $t>\epsilon>0$. Thus, at some point $r<r_-$, which is a contradiction. Without loss of generality we choose $t_1>0$ to be the smallest possible with this property.
	
	(\textit{ii}) \textit{There exists $t_2>t_1$ such that $\sigma(t_2)=\frac{3\pi}{2}$.}
	
	Assume $\sigma\neq\frac{3\pi}{2}$ everywhere. Then $\sigma\in(\pi,\frac{3\pi}{2})$ for $t>t_1$ since $\sigma$ is increasing. Thus, $\sigma$ must converge to some $\alpha\in(\pi,\frac{3\pi}{2}]$ from below. If $\alpha<\frac{3\pi}{2}$, then $r'=\cos\sigma$ is everywhere negative and bounded away from~$0$. Thus, $r<r_-$ at some point, which is a contradiction. If $\alpha=\frac{3\pi}{2}$, then $r\to \hat r$ for some $\hat r$. Then \eqref{ode} implies $\sigma'=2H-\ct(r)\sin\sigma \to 2H+\ct(\hat r)>0$, i.e., $\sigma'$ is bounded away from $0$, contradicting convergence. Without loss of generality we choose $t_2>t_1$ to be the smallest possible with this property.
	
	(\textit{iii}) $r(t_2)=r_-$. This follows directly from Lemma \ref{lem:nodoid} because $\sigma(t_2)=\frac{3\pi}{2}$.
	
	(\textit{iv}) \textit{$h'(t)>0$ for $t\in[0,t_1)$ and $h'(t)<0$ for $t\in(t_1,t_2]$.}
	
	From \eqref{ode} we have $h'=\sqrt{\termbs}\,\sin\sigma$. Since $\sigma(t)\in[\frac{\pi}{2},\pi)$ for $t\in[0,t_1)$ and $\sigma(t)\in(\pi,\frac{3\pi}{2}]$ for $t\in(t_1,t_2]$, the statement follows.
	
	By Lemma \ref{lem:properties_of_ode} we can extend the solution curve from $[0,t_2]$ to $\R$. The full solution fulfills $r(t+2t_2)=r(t)$ and $\sigma(t+2t_2)=\sigma(t)(\mod 2\pi)$ for all $t\in\R$. Hence, it is periodic in $r$ and $\sigma(\mod 2\pi)$.
\end{proof}

By the previous proposition, $r$ and $\sigma(\mod 2\pi)$ are $2t_2$-periodic. Thus, let us focus on the arc for $t\in[0,2t_2]$. In order to create a profile curve of a nodoid type surface, this arc must not be closed, meaning that $\Delta\coloneqq h(t+2t_2)-h(t)\neq0$. Due to periodicity it suffices to compare $h_2\coloneqq h(t_2)$ and $h_0\coloneqq h(0)$, i.e., prove $h_2>h_0$ (nodoid type~I) or $h_2<h_0$ (nodoid type~II). If on the other hand $h_2=h_0$, the curve is a simple loop, and the generated surface is a tube.

Now consider the 1-parameter family $\family$ of \eqref{eq:family} for fixed pitch $a\in\R$ and constant mean curvature $H$. If we consider $J\in(-\infty,0)$ for $\kappa\leq0$ resp. $J\in(-\frac{4H}{\kappa},0)$ for $\kappa>0$, then $\family$ consists only of nodoid type surfaces and tubes by Proposition \ref{prop:nodoid}. The following theorem states conditions for $\family$ to contain a tube:

\begin{theo}[Existence of tubes]\label{theo:tube_existence}
	Suppose $\kappa-4\tau^2\neq0$ and define $\epsilon\coloneqq\sgn(\kappa-4\tau^2)$. The family $\family$ contains a tube if\vspace{-1mm}
	\begin{enumerate}[leftmargin=12mm]\setlength{\itemsep}{2mm}
		\item $\kappa\leq0$, $a\tau\epsilon\in\left(-\infty,\dfrac{\epsilon}{2}\right)$ and $H^2>\dfrac{2\tau^2-a\tau\kappa}{4a\tau-2}$, or
		\item $\kappa>0$, $a\tau\epsilon\in\left[\dfrac{2\tau^2\epsilon}{\kappa},\dfrac{\epsilon}{2}\right)$ and $H^2>\dfrac{2\tau^2-a\tau\kappa}{4a\tau-2}$, or
		\item $\kappa>0$, $a\tau\epsilon\in\left(\dfrac{4\tau^2\epsilon}{\kappa}-\dfrac{\epsilon}{2},\dfrac{2\tau^2\epsilon}{\kappa}\right]$ and $H^2>\dfrac{2\tau^2-a\tau\kappa}{4a\tau+2-\frac{16\tau^2}{\kappa}}$.
	\end{enumerate}
	If the product $a\tau\epsilon$ is not contained in any of the above intervals, the family $\family$ does not contain a tube regardless of the value of the mean curvature $H$.
\end{theo}

This existence result is summarized in Table \ref{tab:tube_existence}. An example of a tube in $\Nildrei$ is shown in Figure \ref{fig:nodoidfamily}. Only few examples of tubes in $\EKT$ have been known so far: Rotational tubes in $\SR$ were described by Pedrosa and Ritoré \cite{pedrosa_ritore,pedrosa}. Screw motion tubes in $\SR$ were described by Vr\v{z}ina \cite{vrzina}. And recently Manzano described tubes with horizontal pitch $a=\frac{2\tau}{\kappa}$ in $\Berger$ \cite{manzano}. Moreover, tubes are expected to exist in $\Soldrei$ because of numerical experiments done by López \cite{lopez}.\vspace{4mm}

\begin{table}[H]
	\centering
	\caption{Existence of screw motion CMC tubes by Theorem \ref{theo:tube_existence}: Does the family $\family$ contain a tube?}
	\renewcommand{\arraystretch}{1.0}
	\begin{tabular}{C{4mm}|C{32mm}|C{32mm}|C{52mm}|}
		& \centering $\kappa<0$ & \centering $\kappa=0$ & \multicolumn{1}{c|}{$\kappa>0$} \\
		\hline
		\multirow{4}{*}{\tabrotate{$\tau=0$}} 
		& \rule{0pt}{20pt}\large$\HR$ & \large$\R^3$ & \large$\SR$ \\[3mm]
		& no
		& no
		& yes for $a\in\R$ \\[2mm]
		\hline
		\multirow{4}{*}{\tabrotate{$\tau\neq0$}} & \rule{0pt}{20pt}\large$\PSLzwei$ & \large$\Nildrei$ & \large$\Berger$ \\[3mm]
		& yes for $a>\frac{1}{2\tau}$
		& yes for $a>\frac{1}{2\tau}$
		& yes for $a\epsilon\in(\frac{4\tau\epsilon}{\kappa}-\frac{\epsilon}{2\tau}, \frac{\epsilon}{2\tau})$ \\[2mm]
		\hline
	\end{tabular}
	\label{tab:tube_existence}
\end{table}\vspace{2mm}

The proof of Theorem \ref{theo:tube_existence} is based on the intermediate value theorem for $J\mapsto(h_2-h_0)$. The following lemma establishes conditions for the sign of $h_2-h_0$:

\begin{lem}\label{lem:height_inequalities}
	Suppose $\kappa-4\tau^2\neq0$ and $J\in(-\infty,0)$ for $\kappa\leq0$ resp. $J\in(-\frac{4H}{\kappa},0)$ for $\kappa>0$. Define
	\begin{equation*}
		\J_1\coloneqq\frac{2H}{\kappa-4\tau^2}\,(2a\tau-1)
		\quad\text{and}\quad
		\J_2\coloneqq\J_1-\frac{2\tau^2-a\tau\kappa}{H(\kappa-4\tau^2)}.
	\end{equation*}
	
	The following inequalities hold, where $\epsilon\coloneqq\sgn(\kappa-4\tau^2)$.\vspace{-1mm}
	\begin{enumerate}[leftmargin=12mm]\setlength{\itemsep}{2mm}
		\item If $2\tau^2-a\tau\kappa>0$, then $\epsilon\J_1>\epsilon\J_2$. Furthermore, for $\epsilon J\geq\epsilon\J_1$ it holds $h_2>h_0$ and for $\epsilon J\leq\epsilon\J_2$ it holds $h_2<h_0$.
		\item If $2\tau^2-a\tau\kappa<0$, then $\epsilon\J_1<\epsilon\J_2$. Furthermore, for $\epsilon J\geq\epsilon\J_2$ it holds $h_2>h_0$ and for $\epsilon J\leq\epsilon\J_1$ it holds $h_2<h_0$.
		\item If $2\tau^2-a\tau\kappa=0$, then $\J_1=\J_2$. Furthermore, for $\epsilon J>\epsilon\J_1$ it holds $h_2>h_0$ and for $\epsilon J<\epsilon\J_1$ it holds $h_2<h_0$.
	\end{enumerate}
\end{lem}

\begin{proof}
	Recall from Proposition \ref{prop:nodoid} that $h'(t)>0$ and $\sigma(t)\in[\frac{\pi}{2},\pi)$ for $t\in[0,t_1)$ as well as $h'(t)<0$ and $\sigma(t)\in(\pi,\frac{3\pi}{2}]$ for $t\in(t_1,t_2]$. Together with $\sigma'>0$, there exists for every $\hat t\in[0,t_1)$ exactly one $\tilde t\in(t_1,t_2]$ such that $\sin\sigma(\hat t)=-\sin\sigma(\tilde t)\geq0$. For short notation we write $\hat\sigma$ for $\sigma(\hat t)$ and $\tilde\sigma$ for $\sigma(\tilde t)$.
	
	The pointwise condition $-\frac{dh}{d\sigma}(\tilde\sigma)<\frac{dh}{d\sigma}(\hat\sigma)$ for all pairs $(\tilde\sigma,\hat\sigma)$ is sufficient for $h_2>h_0$. In the same way $-\frac{dh}{d\sigma}(\tilde\sigma)>\frac{dh}{d\sigma}(\hat\sigma)$ is sufficient for $h_2<h_0$. Since $\sigma$ is strictly increasing, both are still sufficient if we exclude $(\tilde\sigma,\hat\sigma)=(\frac{3\pi}{2},\frac{\pi}{2})$. The advantages are strict inequalities in the following as $|\sin\sigma|\in(0,1)$. From \eqref{ode} we obtain	
	\begin{equation*}
		\frac{dh}{d\sigma}
		=\frac{\sqrt{\sns(r)+\left(4\tau\sns\!\left(\frac{r}{2}\right)-a\right)^2}}{2H\sn(r)-\cs(r)\sin\sigma}\,\sin\sigma.
	\end{equation*}
	
	After extracting $r$ from the energy $J=\frac{2H}{\kappa}\left(\cs(r)-1\right) + \sn(r)\sin\sigma$, substituting in the above expression and a long, but straightforward computation we can rewrite this as
	\begin{equation*}
		\frac{dh}{d\sigma}
		=\frac{\sqrt{f(\sigma)}}{(4H^2+\kappa\sin^2\!\sigma)\sqrt{\sin^2\!\sigma-\kappa J^2-4HJ}}\,\sin\sigma
	\end{equation*}
	with
	\begin{equation*}
		f(\sigma)\coloneqq
		\underbrace{C_1\sin^4\!\sigma+C_2\sin^2\sigma+C_3}_{\eqqcolon\psi_1(\sigma)}
		+\underbrace{\left(C_4\sin^2\!\sigma+C_5\right)}_{\eqqcolon\psi_2(\sigma)}
		\underbrace{\sqrt{\sin^2\!\sigma-4HJ}\sin\sigma}_{\eqqcolon\psi_3(\sigma)},
	\end{equation*}
	where $C_i=C_i(\kappa,\tau,a,H,J)$ are coefficients depending on $\kappa,\tau,a,H,J$. In particular,
	\begin{equation*}
		C_4=8\tau^2-4a\tau\kappa
		\qquad\text{and}\qquad
		C_5=-16\tau^2HJ-16a\tau H^2+4\kappa HJ+8H^2.
	\end{equation*}
	
	Note that $-\frac{dh}{d\sigma}(\tilde\sigma)<\frac{dh}{d\sigma}(\hat\sigma)$ is equivalent to $f(\tilde\sigma)<f(\hat\sigma)$ and  $-\frac{dh}{d\sigma}(\tilde\sigma)>\frac{dh}{d\sigma}(\hat\sigma)$ is equivalent to $f(\tilde\sigma)>f(\hat\sigma)$. Since even powers of $\sin\sigma$ are the same for both $\tilde\sigma$ and $\hat\sigma$, but odd powers differ by sign, it holds $\psi_1(\tilde\sigma)=\psi_1(\hat\sigma)$ and $\psi_2(\tilde\sigma)=\psi_2(\hat\sigma)$, but $\psi_3(\tilde\sigma)=-\psi_3(\hat\sigma)<0$. Thus, $f(\tilde\sigma)<f(\hat\sigma)$ is equivalent to $\psi_2(\tilde\sigma)>0$, and $f(\tilde\sigma)>f(\hat\sigma)$ is equivalent to $\psi_2(\tilde\sigma)<0$. All together:
	\begin{equation*}
		\psi_2(\tilde\sigma)\gtrless0 \text{ for all } \tilde\sigma\in\left(\pi,\tfrac{3\pi}{2}\right)
		\quad\implies\quad
		h_2\gtrless h_0.
	\end{equation*}
	
	It is more convenient to derive a condition independent of~$\sigma$. On this account, we consider the following rather rough estimates: For (\textit{i}), i.e., $2\tau^2-a\tau\kappa>0$, it holds $C_4>0$ and $C_5<C_4\sin^2\sigma+C_5<C_4+C_5$. Thus, $C_5\geq0$ implies $\psi_2>0$ and $C_4+C_5\leq0$ implies $\psi_2<0$.
	For (\textit{ii}), i.e., $2\tau^2-a\tau\kappa>0$, it holds $C_4<0$ and $C_5>C_4\sin^2\sigma+C_5>C_4+C_5$. Thus, $C_4+C_5\geq0$ implies $\psi_2>0$ and $C_5\leq0$ implies $\psi_2<0$.
	For (\textit{iii}), i.e., $2\tau^2-a\tau\kappa=0$, it holds $C_4\sin^2\sigma+C_5=C_5$. Thus, $C_5\gtrless0$ is equivalent to $\psi_2\gtrless0$. Therefore, we are going to study the (in)equalities $C_5\gtreqless0$ and $C_4+C_5\gtreqless0$.
	
	The (in)equality $C_5\gtreqless0$ is equivalent to
	\begin{equation*}
		J(\kappa-4\tau^2)-2H(2a\tau-1)\gtreqless0
		\quad\Leftrightarrow\quad
		\epsilon J\gtreqless\epsilon \J_1,
	\end{equation*}
	and $C_4+C_5\gtreqless0$ is equivalent to
	\begin{equation*}
		J(\kappa-4\tau^2)-2H(2a\tau-1)+\frac{2\tau^2-a\tau\kappa}{H}\gtreqless0
		\quad\Leftrightarrow\quad
		\epsilon J\gtreqless\epsilon \J_2. \qedhere
	\end{equation*}
\end{proof}

We can now use Lemma \ref{lem:height_inequalities} to prove Theorem \ref{theo:tube_existence}:

\begin{proof}[Proof of Thm. \ref{theo:tube_existence}]
	Suppose $\kappa<0$.	If $a\tau>\frac{1}{2}$, then $2\tau^2-a\tau\kappa>0$ and $\J_1<\J_2$. If $H$ fulfills the curvature bound, then $\J_2<0$. Therefore, there exist $J_1, J_2\in(-\infty,0)$ such that $J_1\leq\J_1<\J_2\leq J_2$ and by Lemma \ref{lem:height_inequalities} this corresponds to $h_2>h_0$ resp. $h_2<h_0$. By the intermediate value theorem there exists a $\Jtube\in(\J_1,\J_2)$ such that $h_2=h_0$, i.e., the family $\family$ contains a tube.
	
	If $a\tau<\frac{1}{2}$, but $a>\frac{2\tau}{\kappa}$, then still $2\tau^2-a\tau\kappa>0$ and $\J_1<\J_2$, but $\J_1>0$. Thus, $J<\J_1$ for all $J\in(-\infty,0)$ and it always holds $h_2>h_0$. If instead $a\leq\frac{2\tau}{\kappa}$, then $2\tau^2-a\tau\kappa\leq0$ and $\J_1\geq\J_2$, but $\J_2>0$. Thus, $J<\J_2$ for all $J\in(-\infty,0)$ and it always holds $h_2>h_0$. Thus, for $a\tau<\frac{1}{2}$ the family $\family$ does not contain a tube.
	
	The remaining cases $\kappa=0$ and $\kappa>0$ are proven by using similar arguments.
\end{proof}

As a direct consequence of the proof we obtain a range for the tube energy:

\begin{cor}\label{cor:energy_range}
	Suppose $\family$ contains a tube and denote by $\Jtube$ its energy.
	\vspace{-1mm}
	\begin{enumerate}[leftmargin=12mm]\setlength{\itemsep}{2mm}
		\item If $2\tau^2-a\tau\kappa\neq0$, then $\Jtube$ lies in the open interval between $\J_1$ and $\J_2$.
		\item If $2\tau^2-a\tau\kappa=0$, then $\Jtube=-\frac{2H}{\kappa}$.
	\end{enumerate}
\end{cor}

\begin{rem}
	Theorem \ref{theo:tube_existence} only states the existence of tubes and says nothing about uniqueness. For the two special cases of $\SKR$ with arbitrary pitch $a\in\R$ an $\Berger$ with horizontal pitch $a=\frac{2\tau}{\kappa}$ it holds $2\tau^2-a\tau\kappa=0$ and Corollary \ref{cor:energy_range} implies uniqueness. In general, the tools used here do not allow us to prove uniqueness (we would need monotonicity of $h_2$ as a function of $J$). Nevertheless, numerical computations indicate that the tubes are indeed unique.
\end{rem}

For completeness let us also write down the existence result for tubes in $\R^3$ and $\s^3(\kappa)$:

\begin{theo}\label{theo:tube_existence_spaceform}
	Suppose $\kappa-4\tau^2=0$ and $J\in(-\infty,0)$ for $\kappa=0$ resp. $J\in(-\frac{4H}{\kappa},0)$ for $\kappa>0$. \vspace{-1mm}
	\begin{enumerate}[leftmargin=12mm]\setlength{\itemsep}{2mm}
		\item If $a\tau<\frac{1}{2}$, all surfaces are of nodoid type I. 
		\item If $a\tau=\frac{1}{2}$, all surfaces are tubes.
		\item If $a\tau>\frac{1}{2}$, all surfaces are of nodoid type II.
	\end{enumerate}
\end{theo}

As expected, there are no tubes in $\R^3$ as $\tau=0$. For $\s^3(\kappa)$ recall, that the nodoid type~I and nodoid type~II surfaces coincide by Lemma \ref{lem:isometry} and \ref{lem:symmetry}. The tubes in $\s^3(\kappa)$ are the well-known distance tori.

\begin{proof}
	The proof of Lemma \ref{lem:height_inequalities} carries over except for the last paragraph. If $a\tau\lesseqgtr\frac{1}{2}$, then $C_{4,5}\gtreqless0$ and therefore $\psi_2\gtreqless$ and $h_2\gtreqless0$.
\end{proof}

\subsection{Solution curves of zero energy}
\label{sec:zeroenergy}

At last we consider vanishing energy. Instead of directly analyzing the solutions of $\eqref{ode}$ for $J=0$, we take the limit $J\searrow0$ of the unduloid type solutions. Alternatively, one can also consider the limit $J\nearrow0$ of the nodoid type solutions.

We first turn our attention to the minimal and maximal radius. From Lemma \ref{lem:unduloid} we obtain $r_-=0$ and $r_+=2\arct(2H)$. Note that $r=0$ does not lie in the regular orbit space but on its boundary. However, the remaining part of Lemma \ref{lem:unduloid} continues to hold whenever $r>0$: $h$ is everywhere strictly increasing and $\sigma(t)$ can be assumed to be in $(0,\pi)$ for all $t$ with $r(t)>0$. But as $r\to0$ we only have $h'\to0$ for $a=0$. Thus, only in the rotational case the tangent vector $\gamma'$ becomes perpendicular to the axis of screw motion. Based on these considerations, we can state the following analogue to Proposition \ref{prop:unduloid}. Most of the proof carries over.

\begin{prop}\label{prop:sphere}
	Let $\gamma=(r,h,\sigma)$ be a solution curve with $J=0$ for initial data $r(0)=r_+, h(0)=0, \sigma(0)=\frac{\pi}{2}$. Then there exists $0<t_1<\infty$ such that $r(t)\to0$ for $t\to t_1$ and
	\begin{itemize}
		\item $\sigma'(t)>0$ for $t\in[0,t_1)$, \\[-0.5em]
		\item $h'(t)>0$ for $t\in[0,t_1)$, \\[-0.5em]
		\item $h'(t)\to |a|H$ as $t\to t_1$.
	\end{itemize}
	A maximal solution curve is obtained by extension to the axis and successive reflections at the heights $h=kh(t_1), k\in\Z$. In particular, the solution curve can be extended to all of $\,\R$, and $r$ as well as $\sigma$ are $2t_1$-periodic, but $\sigma$ is not continuous.
\end{prop}

\begin{proof}
	We focus on the arguments that do not carry over from Proposition \ref{prop:unduloid}. Vanishing energy $J=0$ implies
	\begin{equation*}
		\sin\sigma=\frac{2H}{\kappa}\frac{1-\cs(r)}{\sn(r)}.
	\end{equation*}
	Therefore,
	\begin{equation*}
		h'=\sqrt{\termb}\,\sin\sigma
		=\frac{2H}{\kappa}\sqrt{\sns(r)+\left(4\tau\sns\!\left(\frac{r}{2}\right)-a\right)^2}\,\cdot\frac{1-\cs(r)}{\sns(r)},
	\end{equation*}
	and
	\begin{equation*}
		\lim\limits_{r\to0}h'
		=\frac{2H}{\kappa}\sqrt{a^2}\,\cdot\lim\limits_{r\to0}\left(\frac{1-\cs(r)}{\sns(r)}\right)
		=\frac{2H}{\kappa}\sqrt{a^2}\cdot\frac{\kappa}{2}=|a|H.
		\qedhere
	\end{equation*}
\end{proof}

\newpage


\begin{figure}[H]
	\centering
	\begin{minipage}[t]{0.16\textwidth}
		\includegraphics[width=\textwidth]{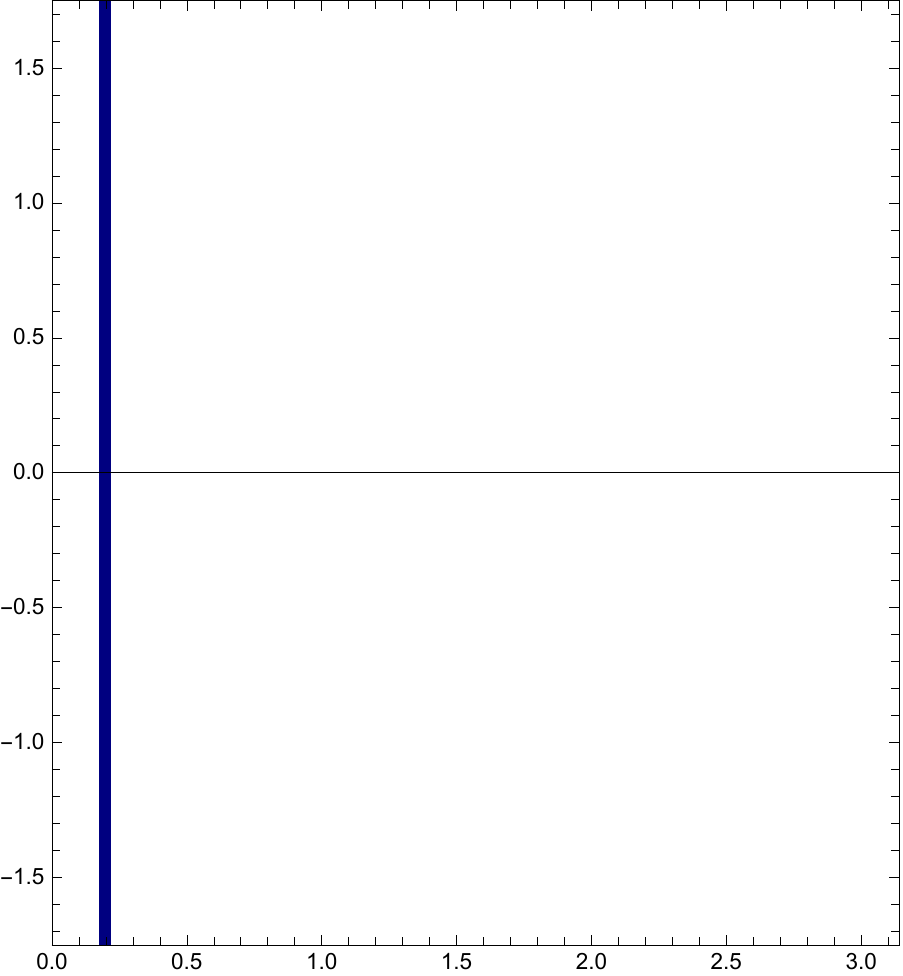} \\
		\includegraphics[width=\textwidth]{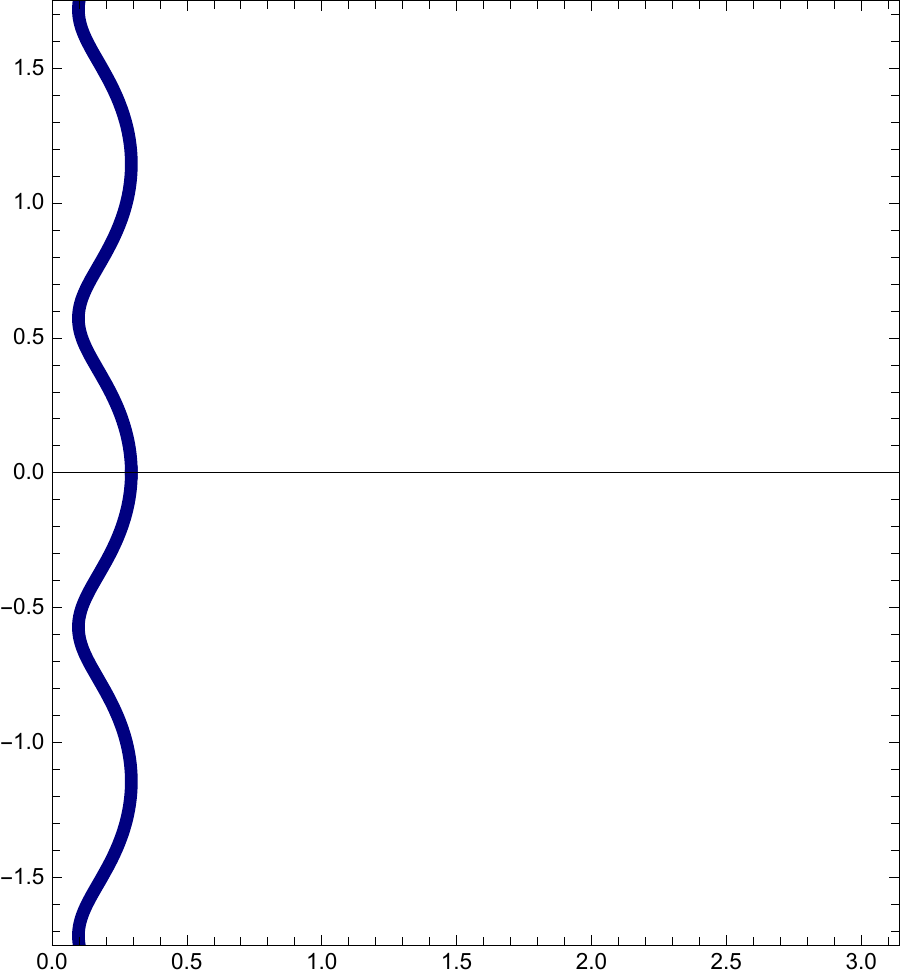} \\
		\includegraphics[width=\textwidth]{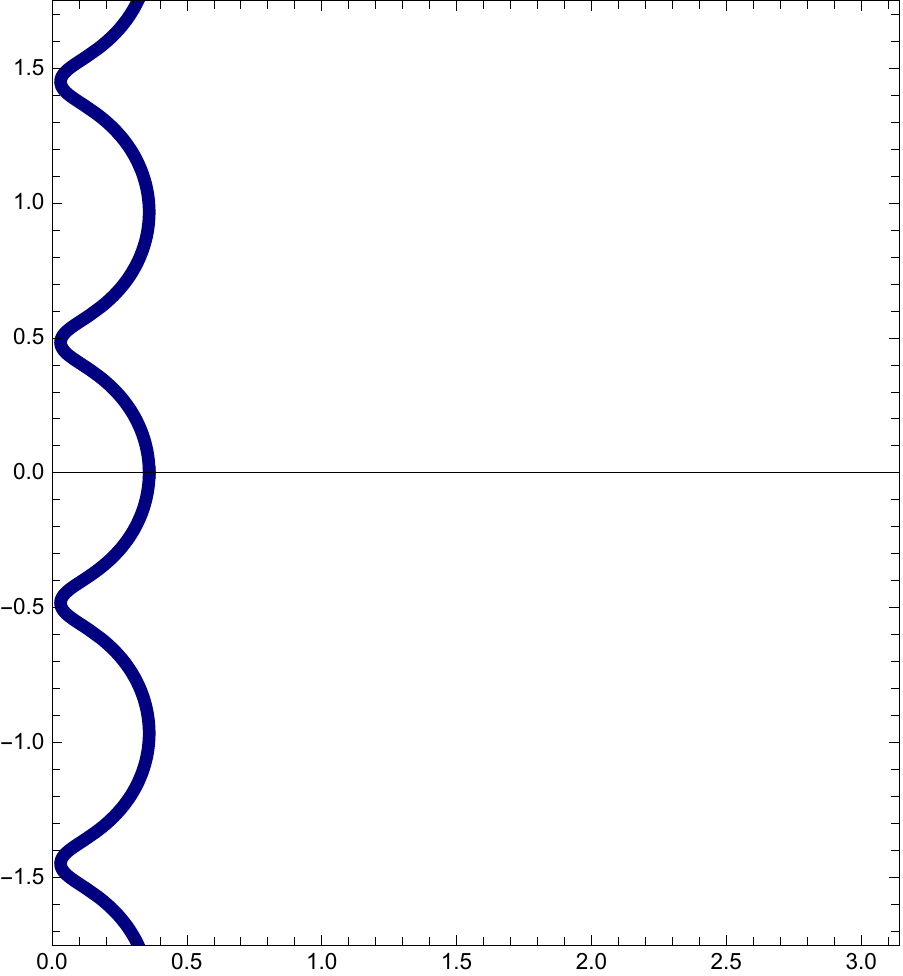} \\
		\includegraphics[width=\textwidth]{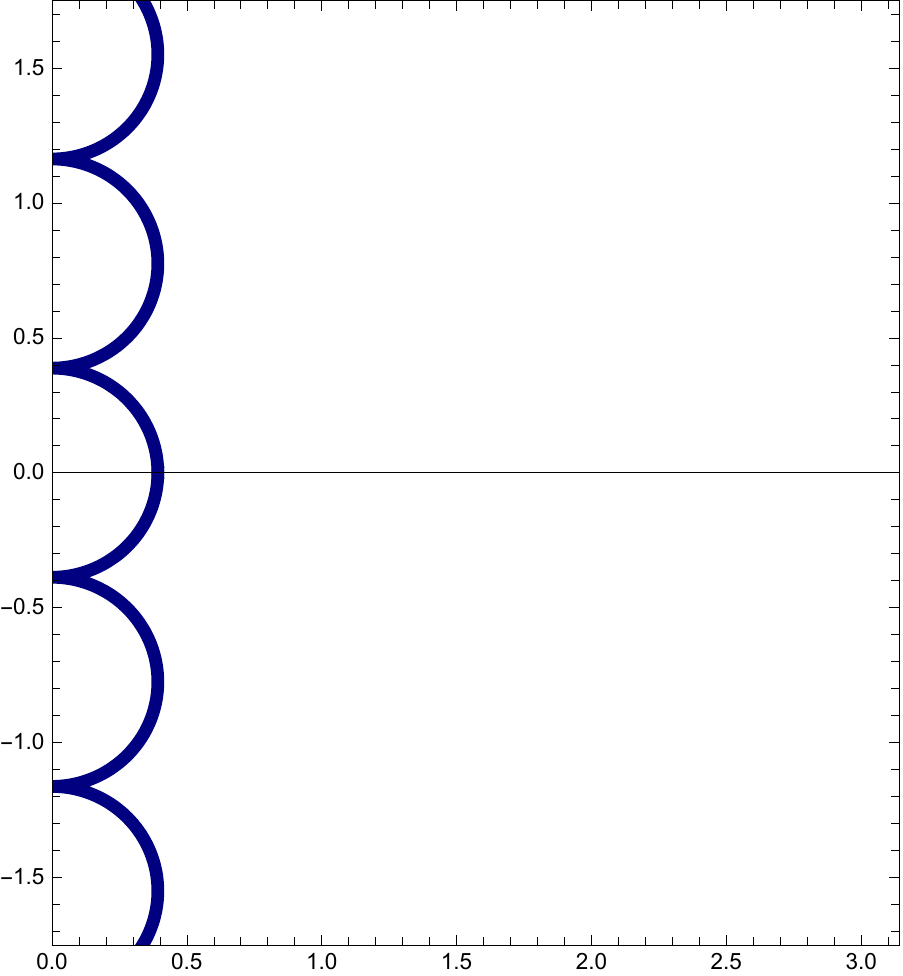} \\
		\includegraphics[width=\textwidth]{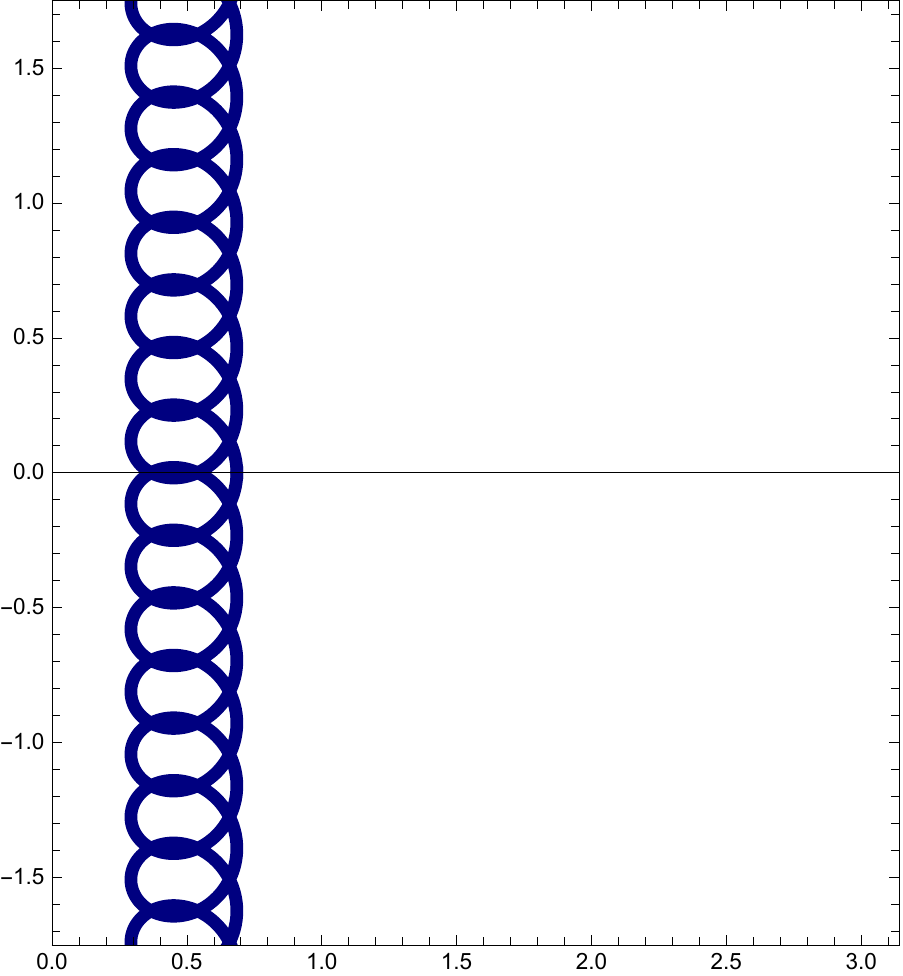} \\
		\includegraphics[width=\textwidth]{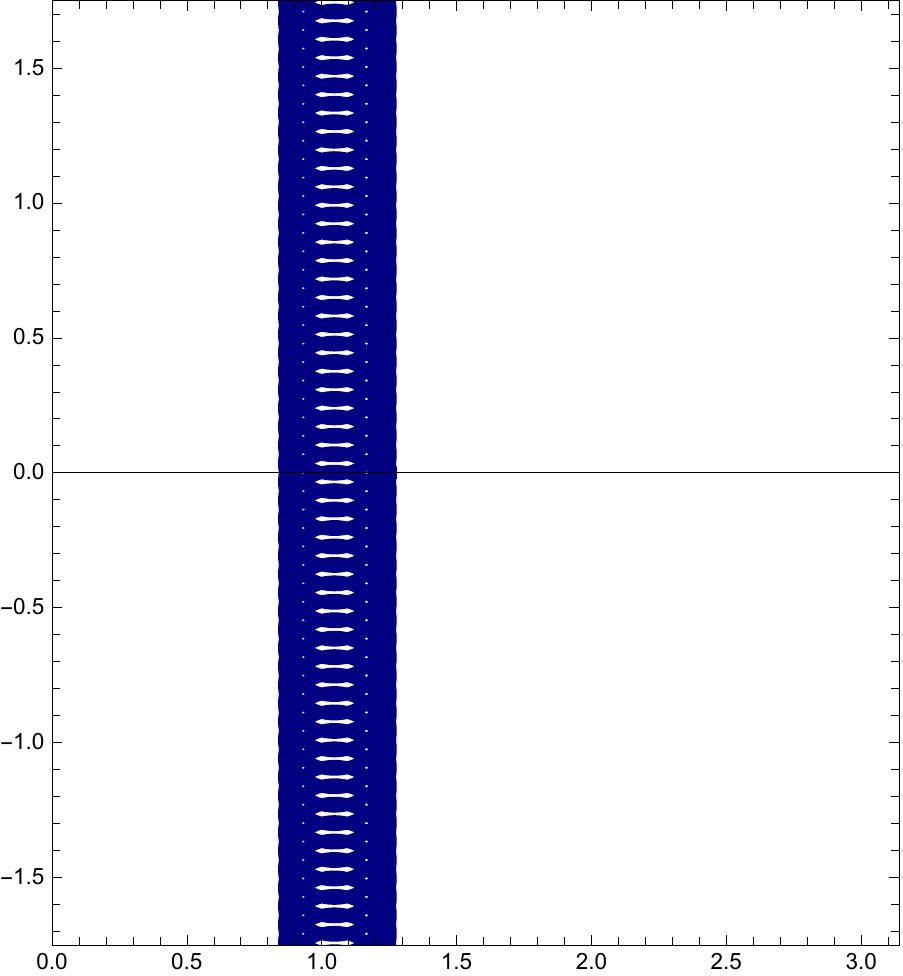} \\
		\includegraphics[width=\textwidth]{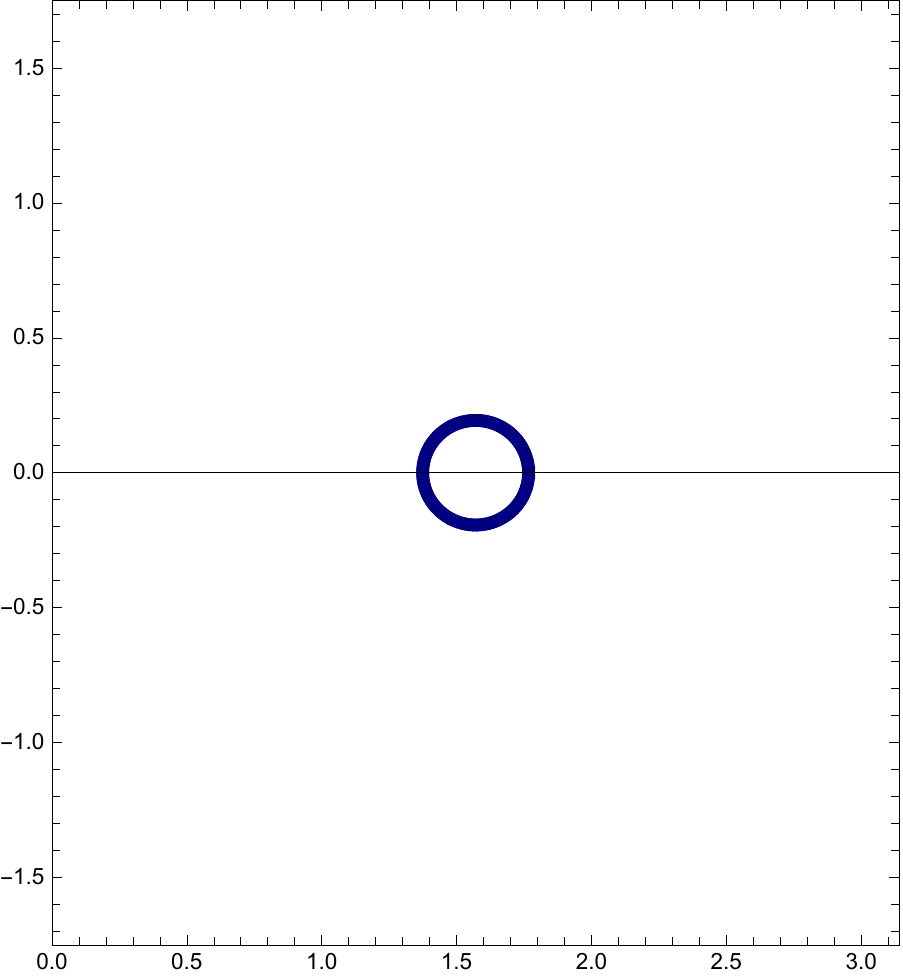}
	\end{minipage}
	\begin{minipage}[t]{0.16\textwidth}
		\includegraphics[width=\textwidth]{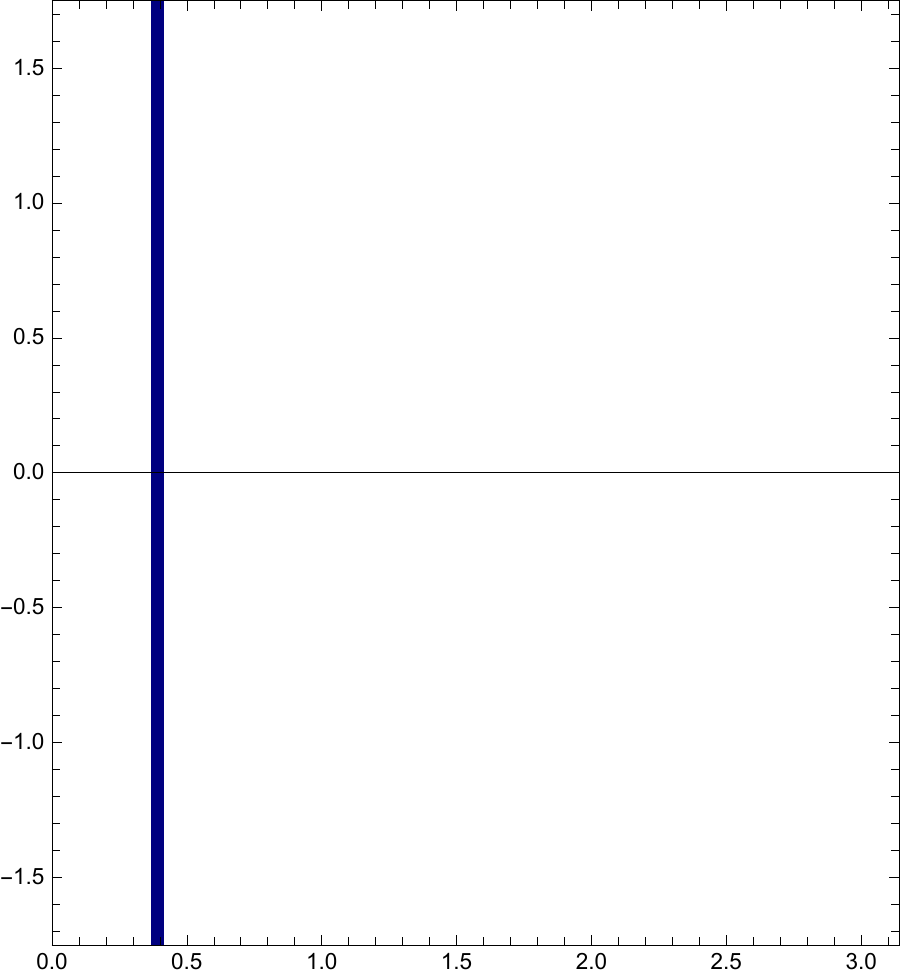} \\
		\includegraphics[width=\textwidth]{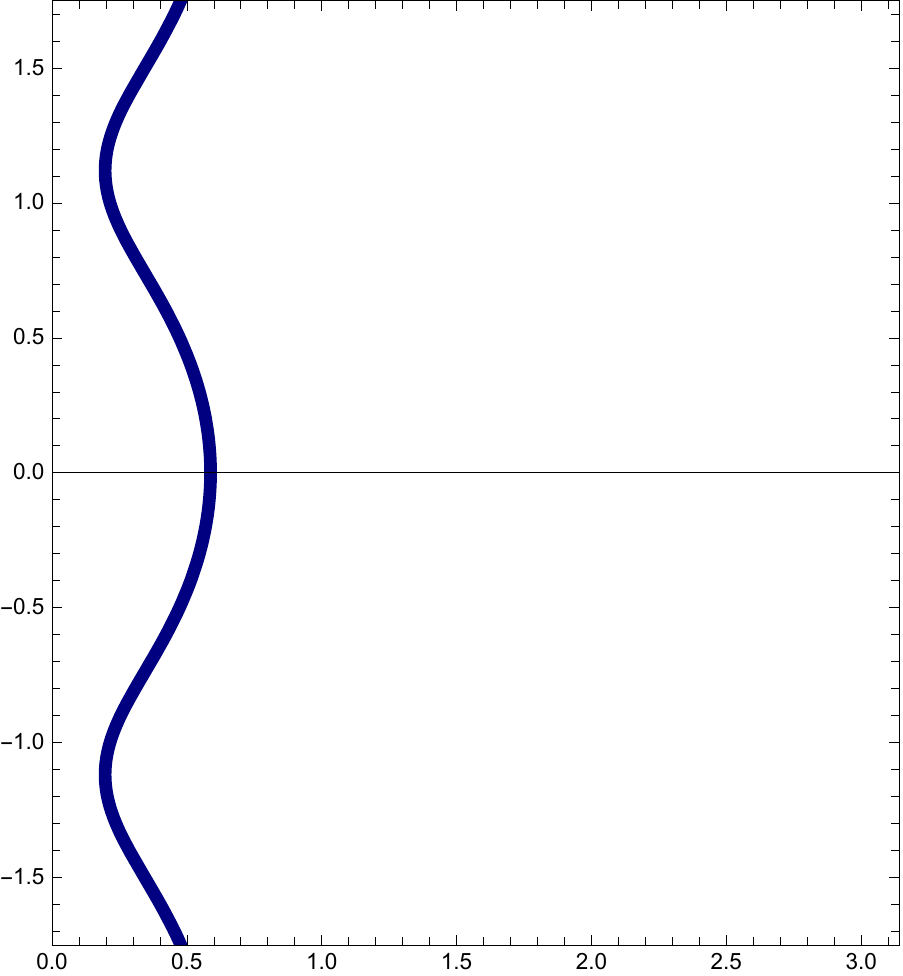} \\
		\includegraphics[width=\textwidth]{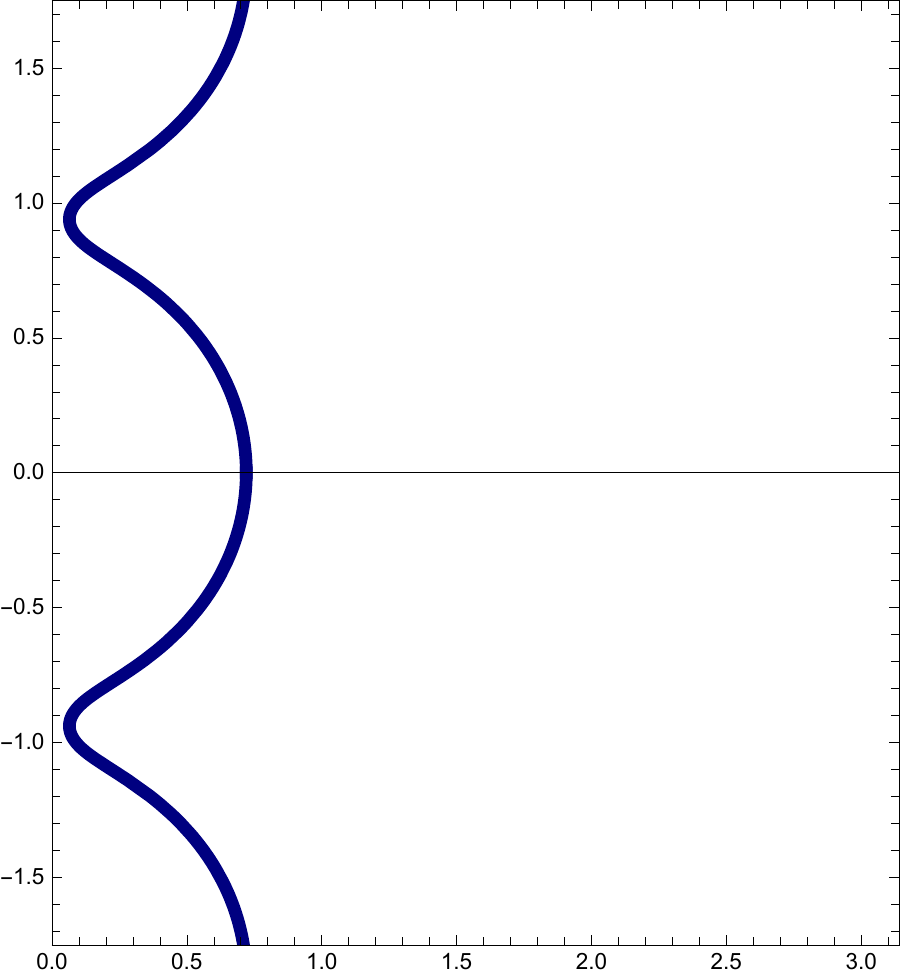} \\
		\includegraphics[width=\textwidth]{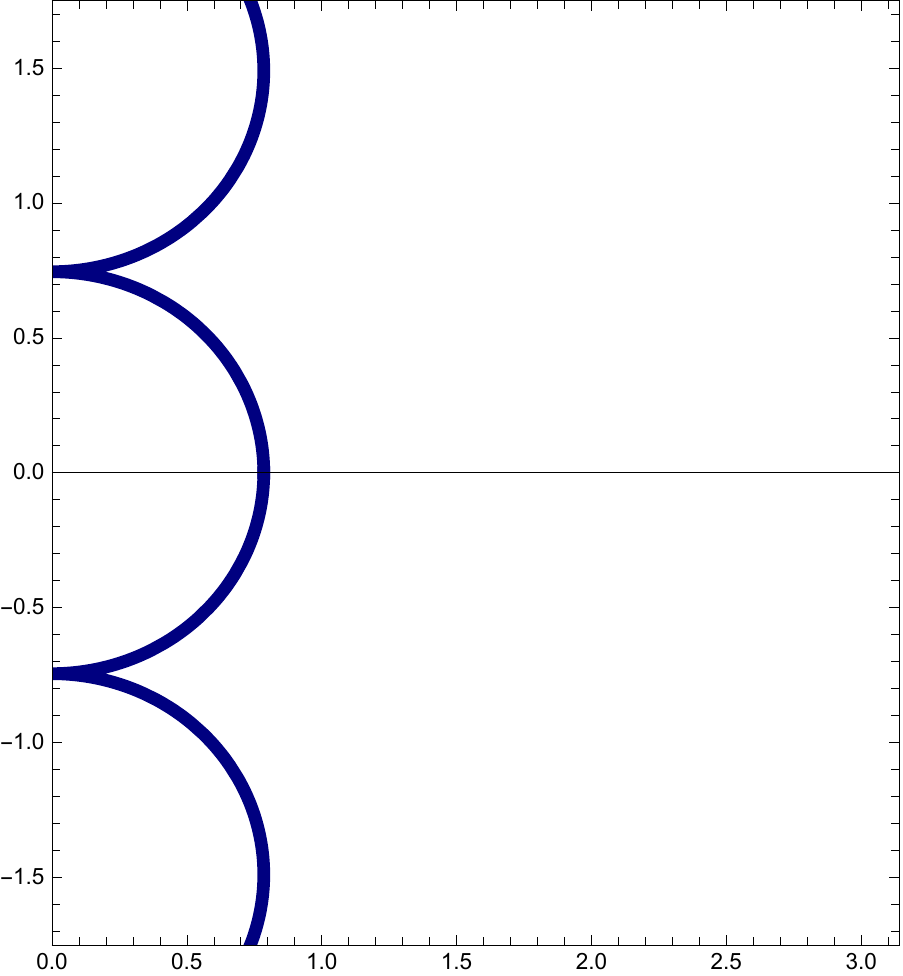} \\
		\includegraphics[width=\textwidth]{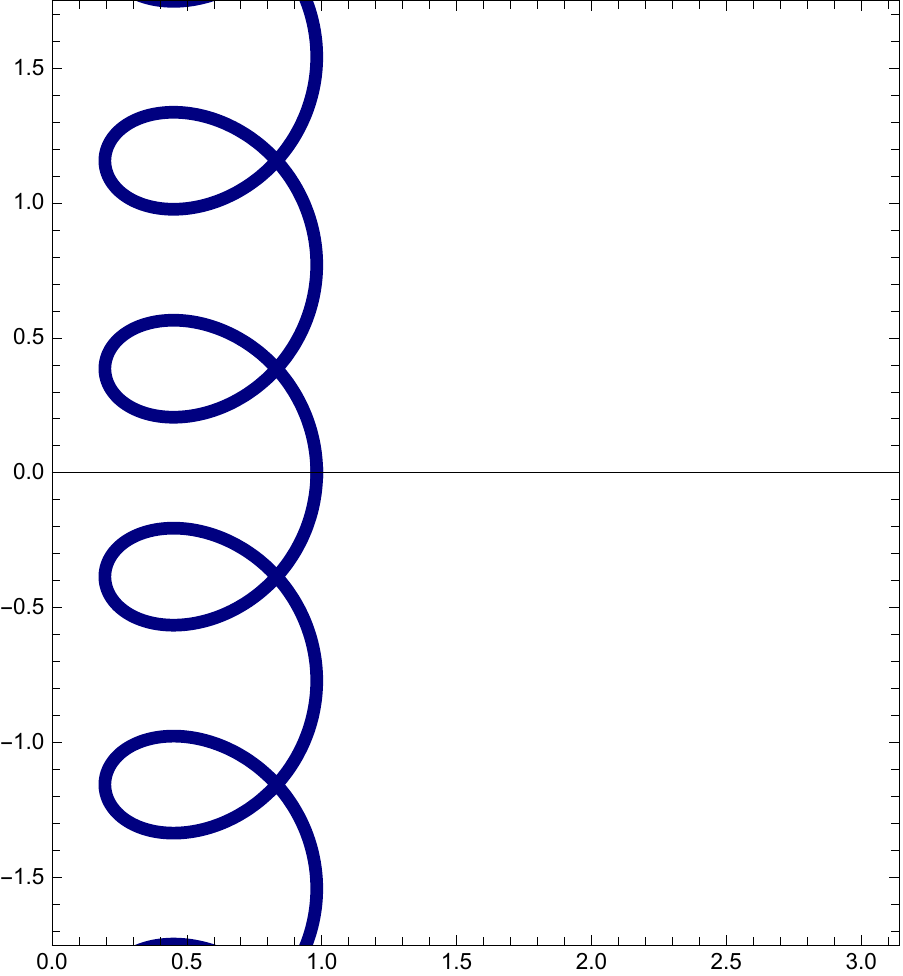} \\
		\includegraphics[width=\textwidth]{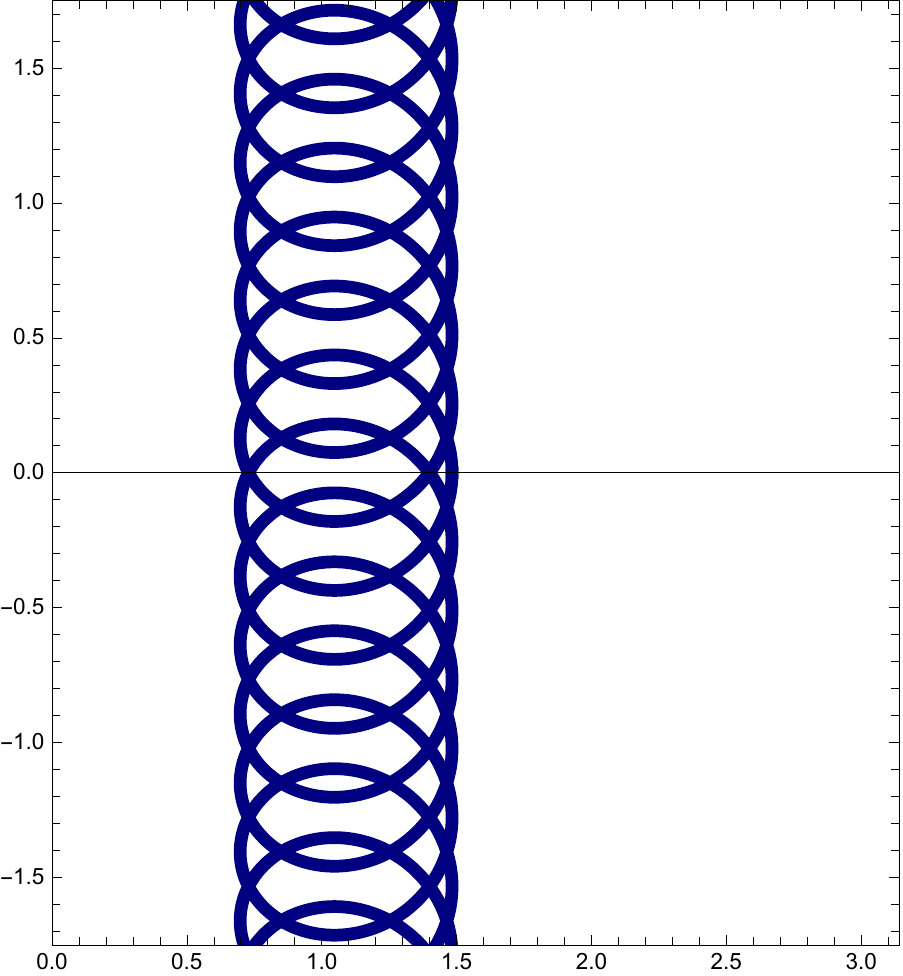} \\
		\includegraphics[width=\textwidth]{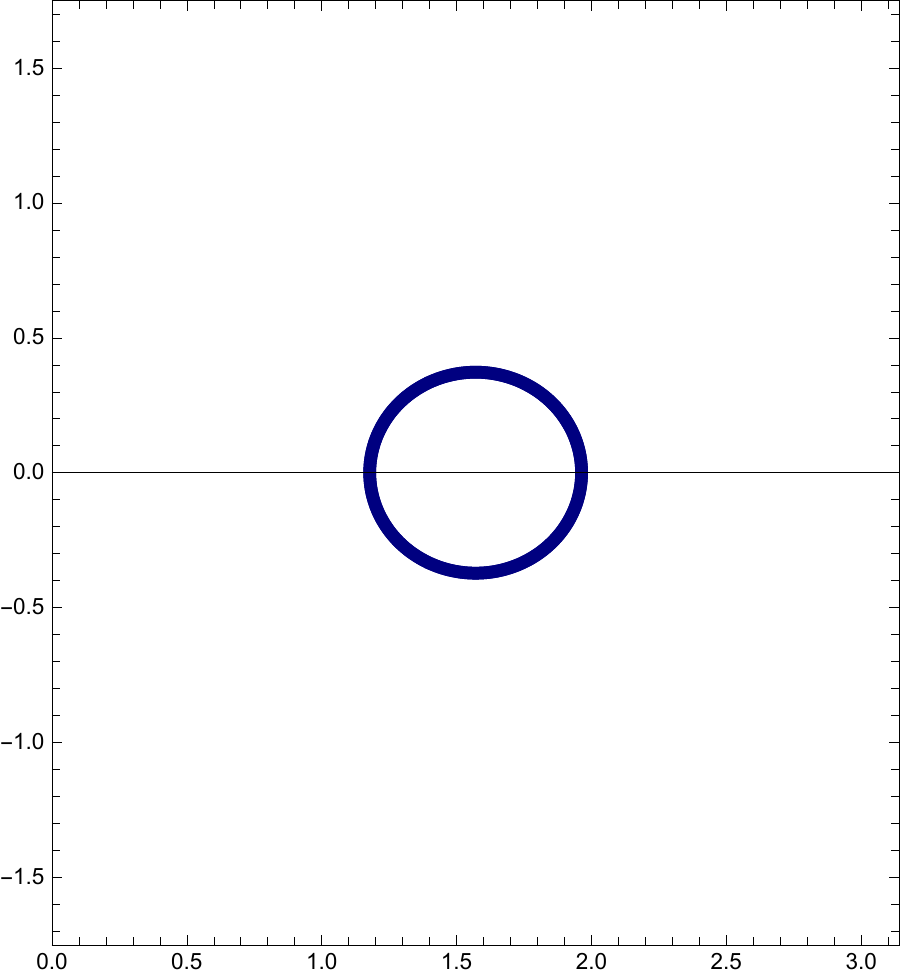}
	\end{minipage}
	\begin{minipage}[t]{0.16\textwidth}
		\includegraphics[width=\textwidth]{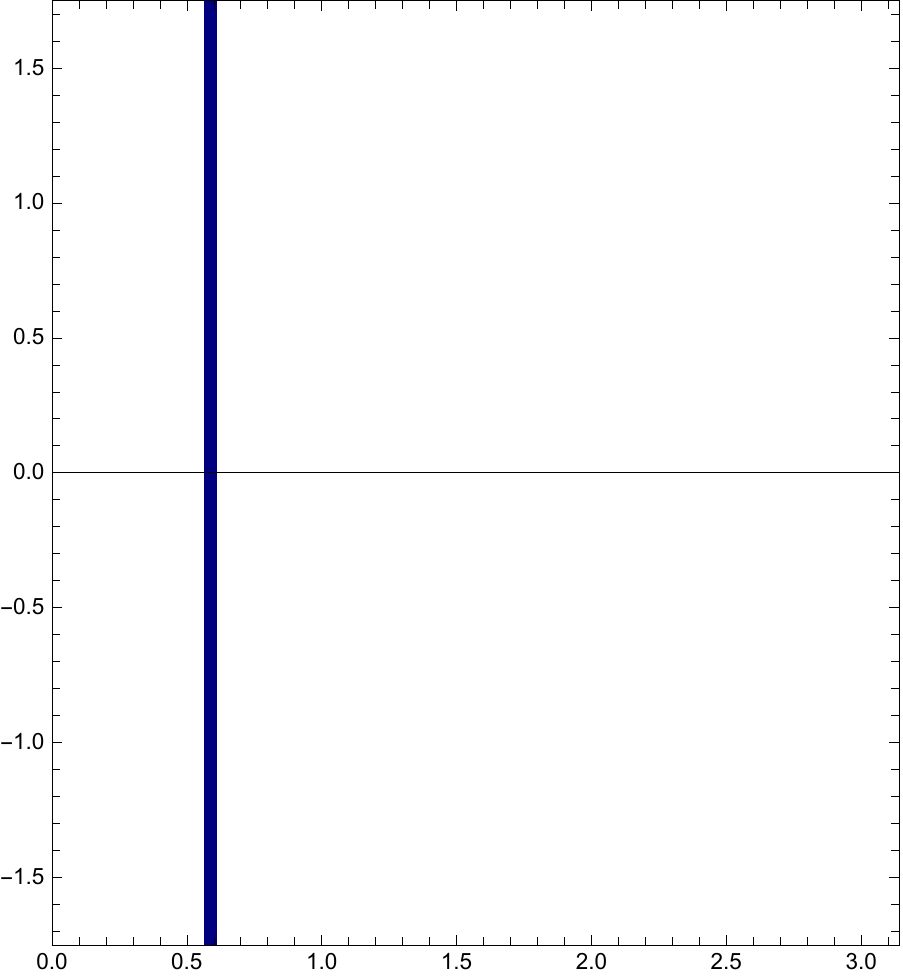} \\
		\includegraphics[width=\textwidth]{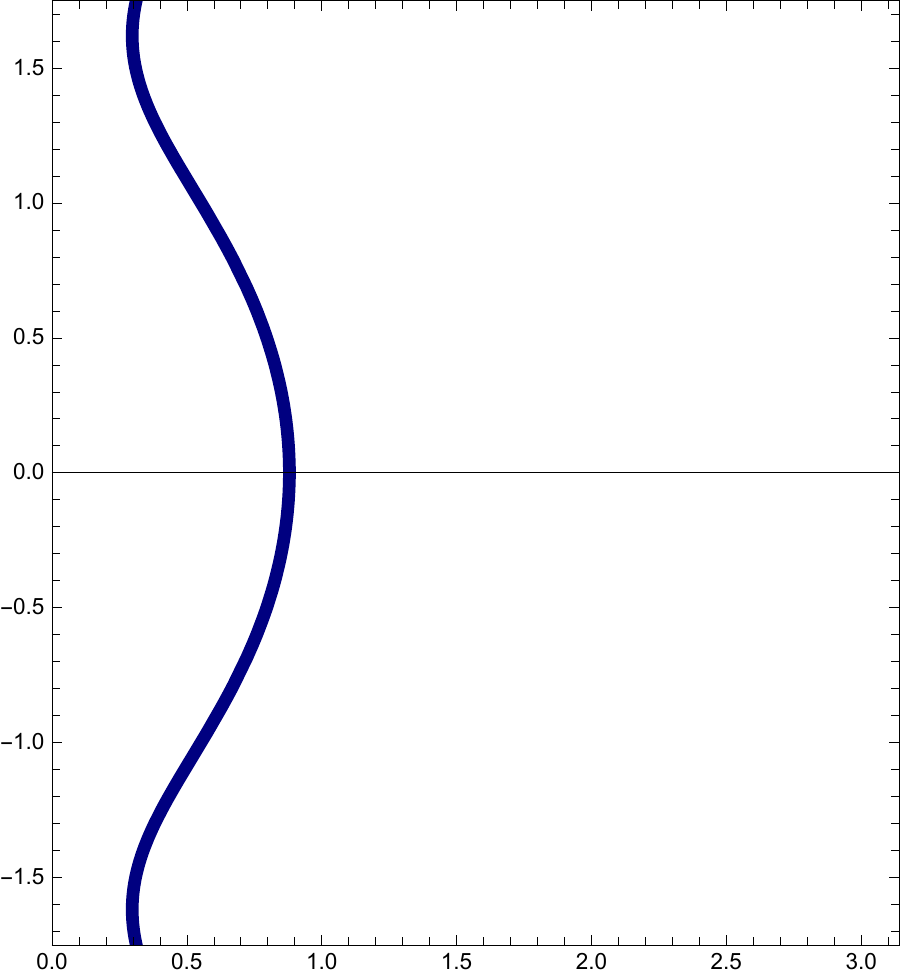}
		\includegraphics[width=\textwidth]{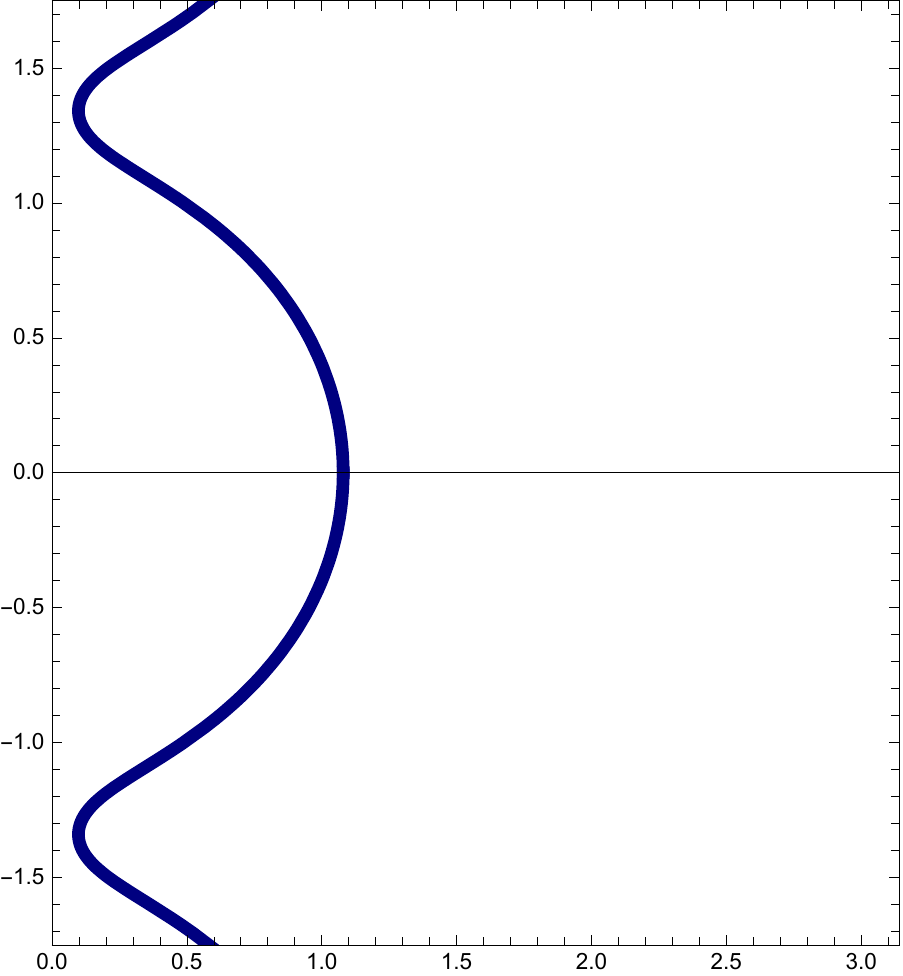} \\
		\includegraphics[width=\textwidth]{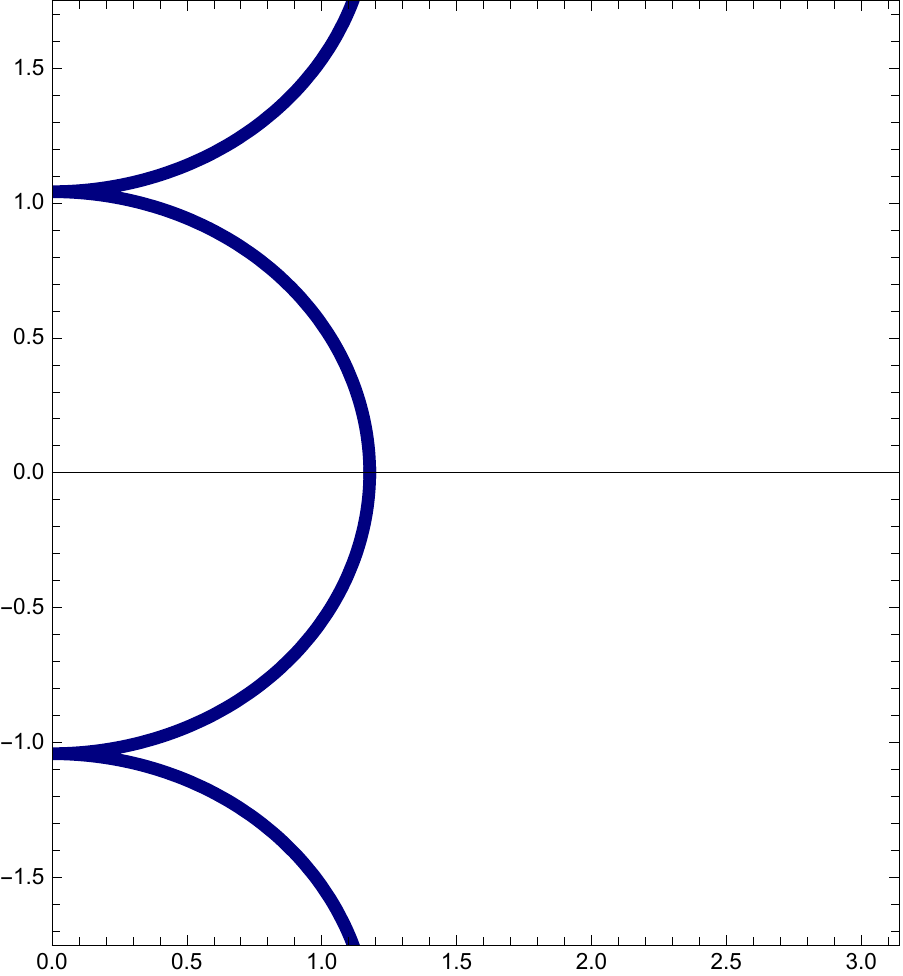} \\
		\includegraphics[width=\textwidth]{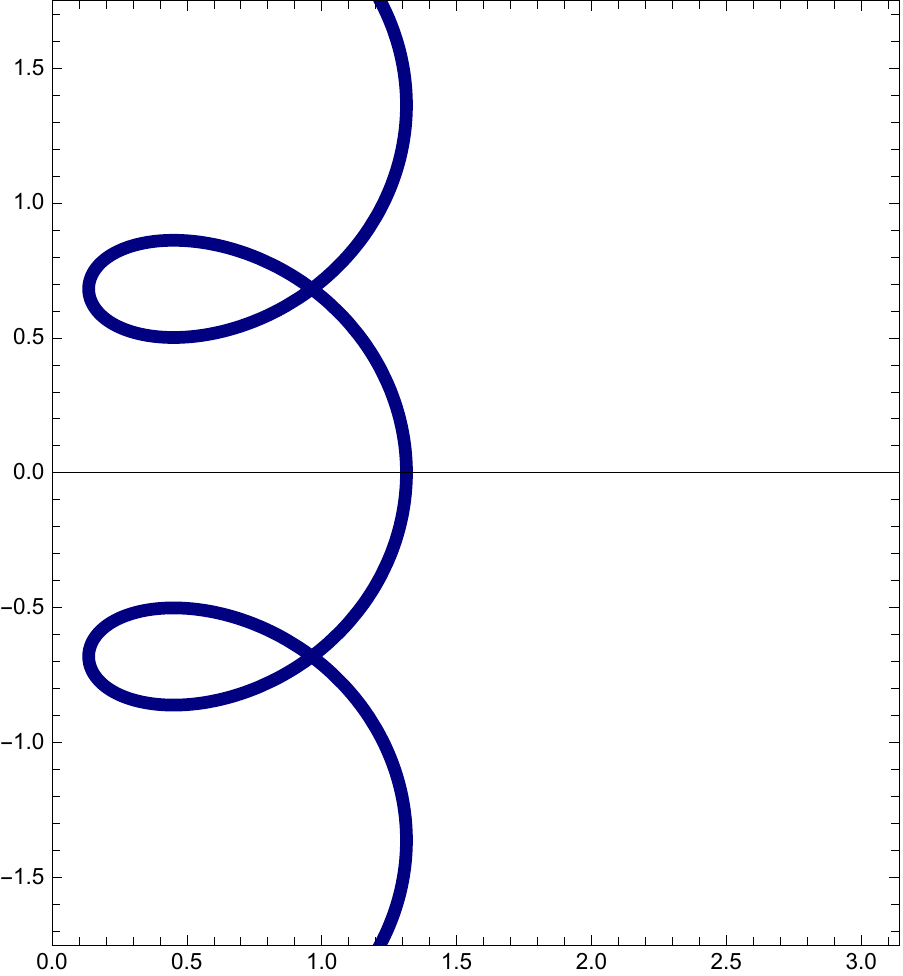} \\
		\includegraphics[width=\textwidth]{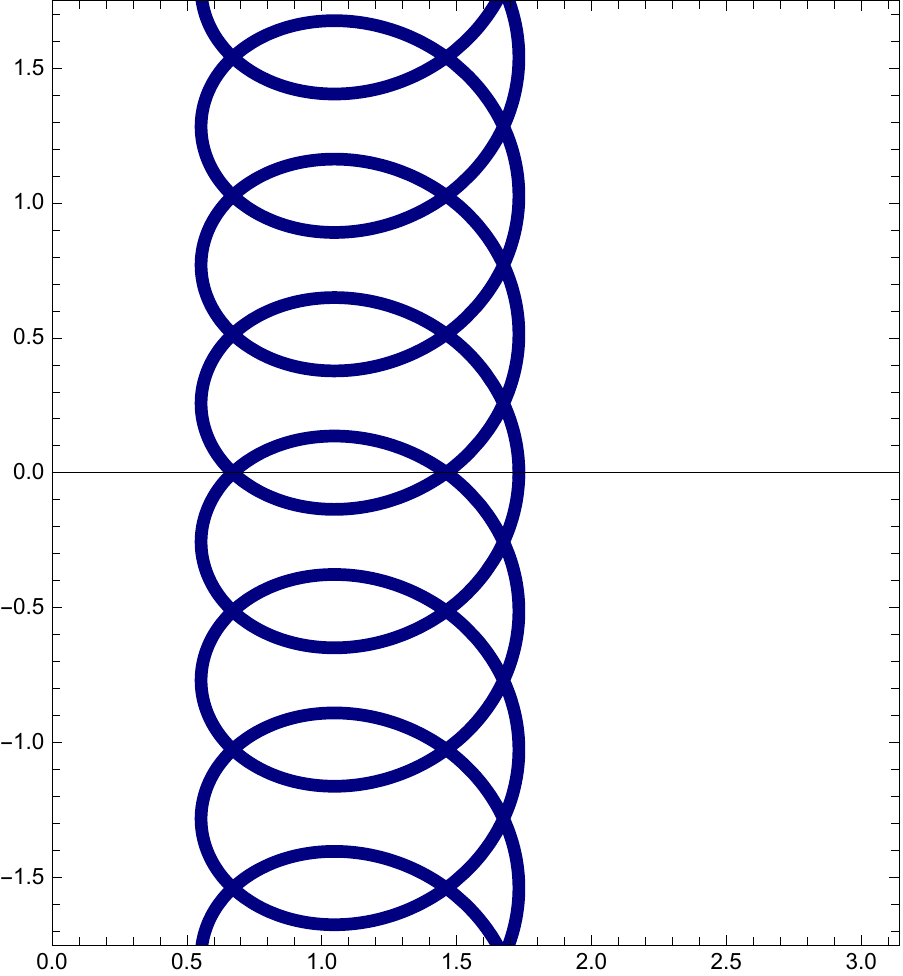} \\
		\includegraphics[width=\textwidth]{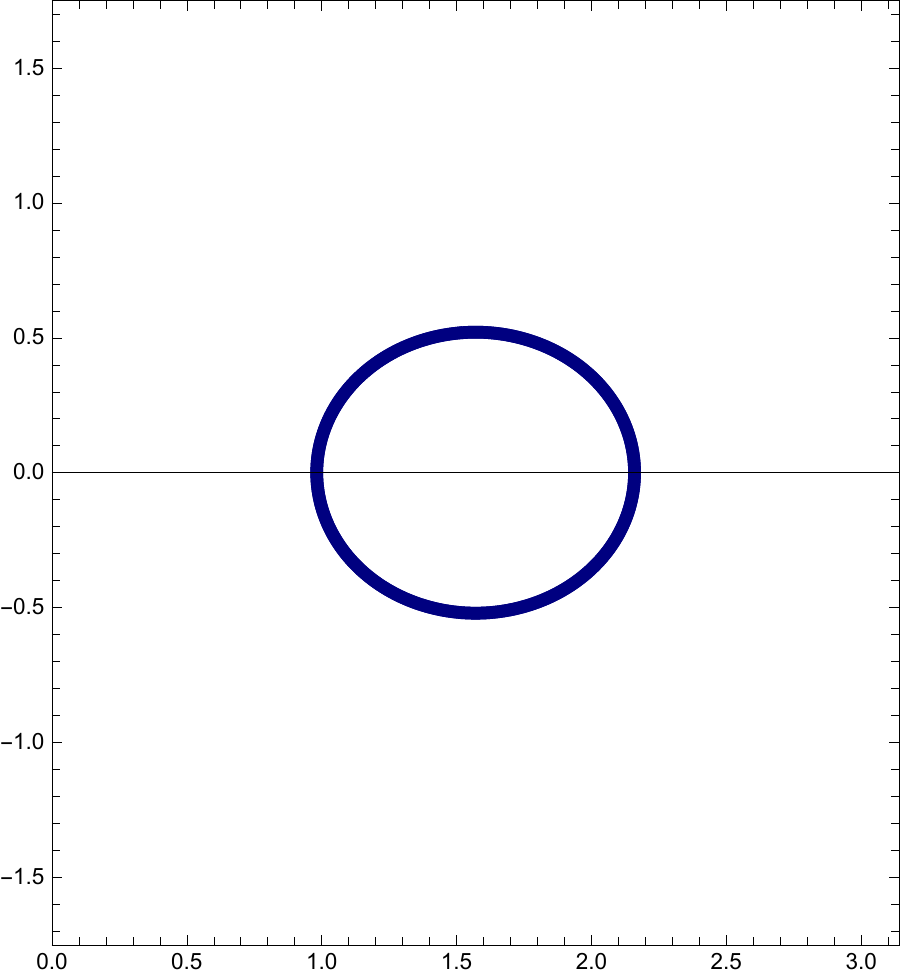}
	\end{minipage}
	\begin{minipage}[t]{0.16\textwidth}
		\includegraphics[width=\textwidth]{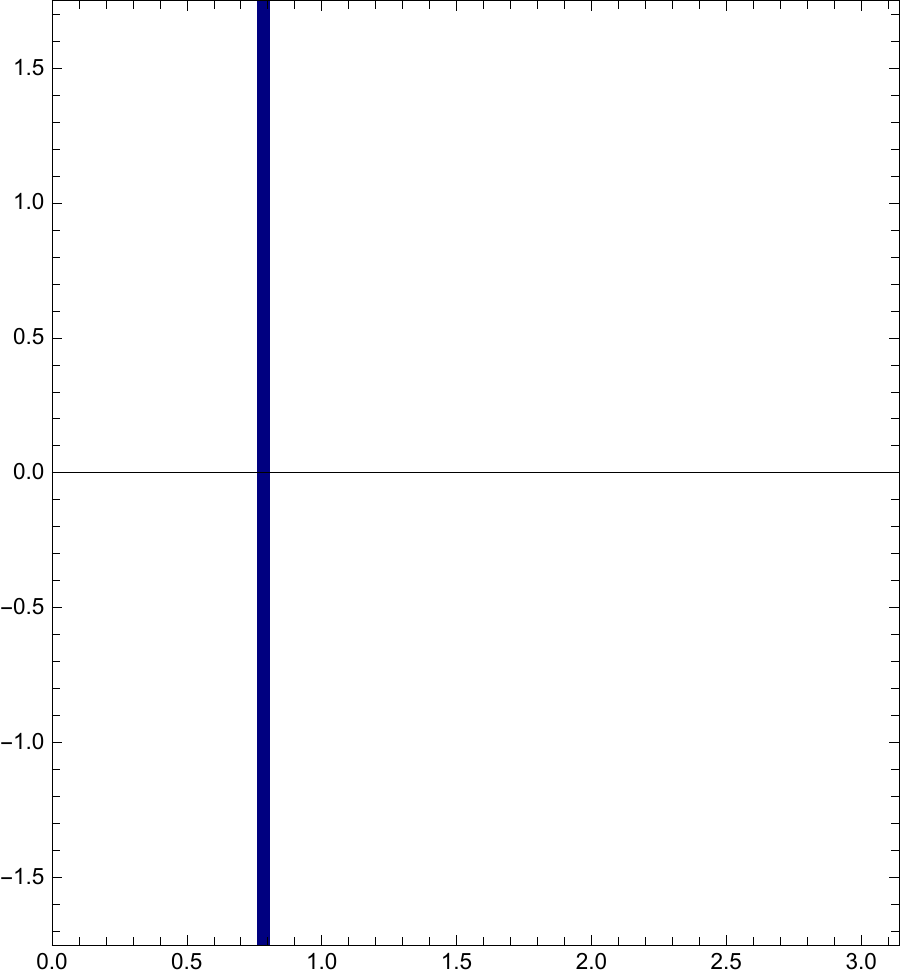} \\
		\includegraphics[width=\textwidth]{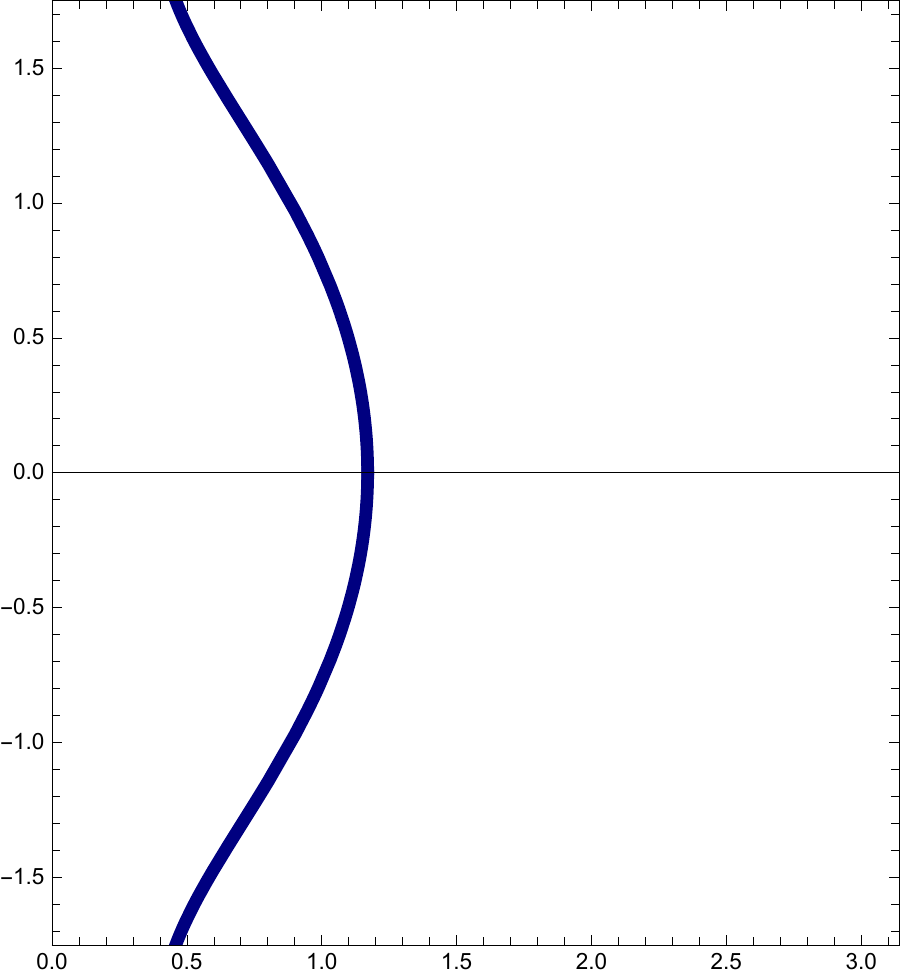} \\		\includegraphics[width=\textwidth]{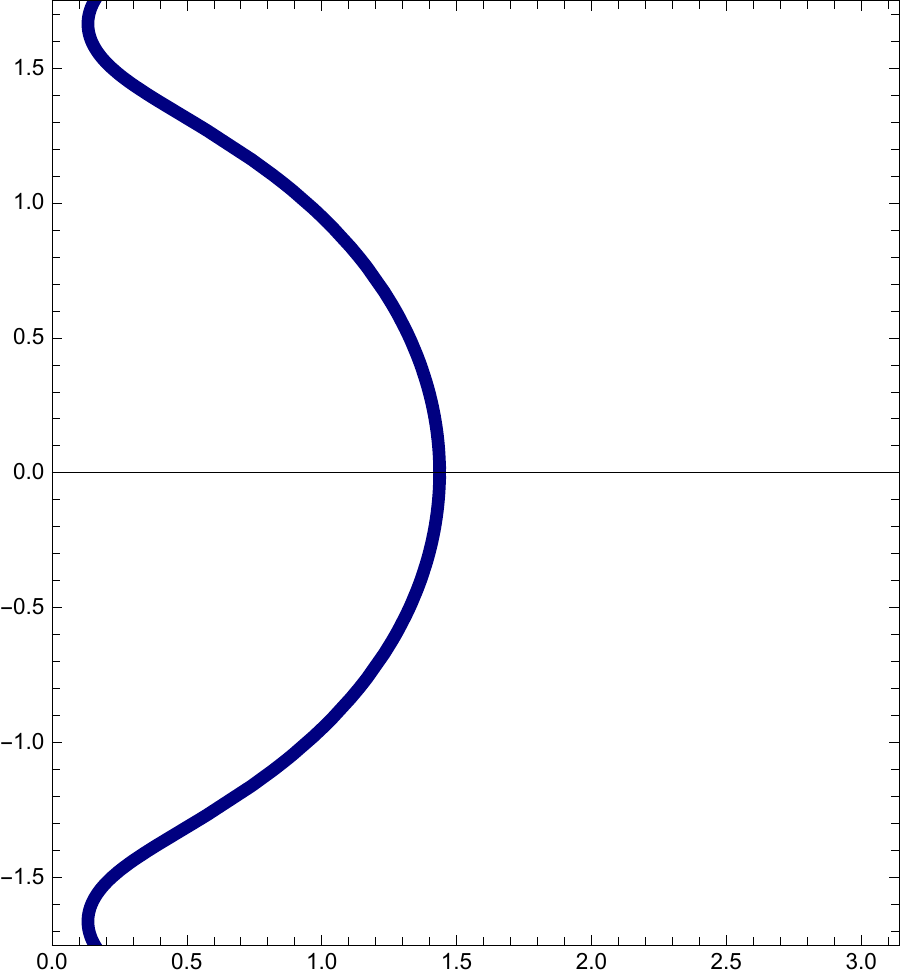} \\
		\includegraphics[width=\textwidth]{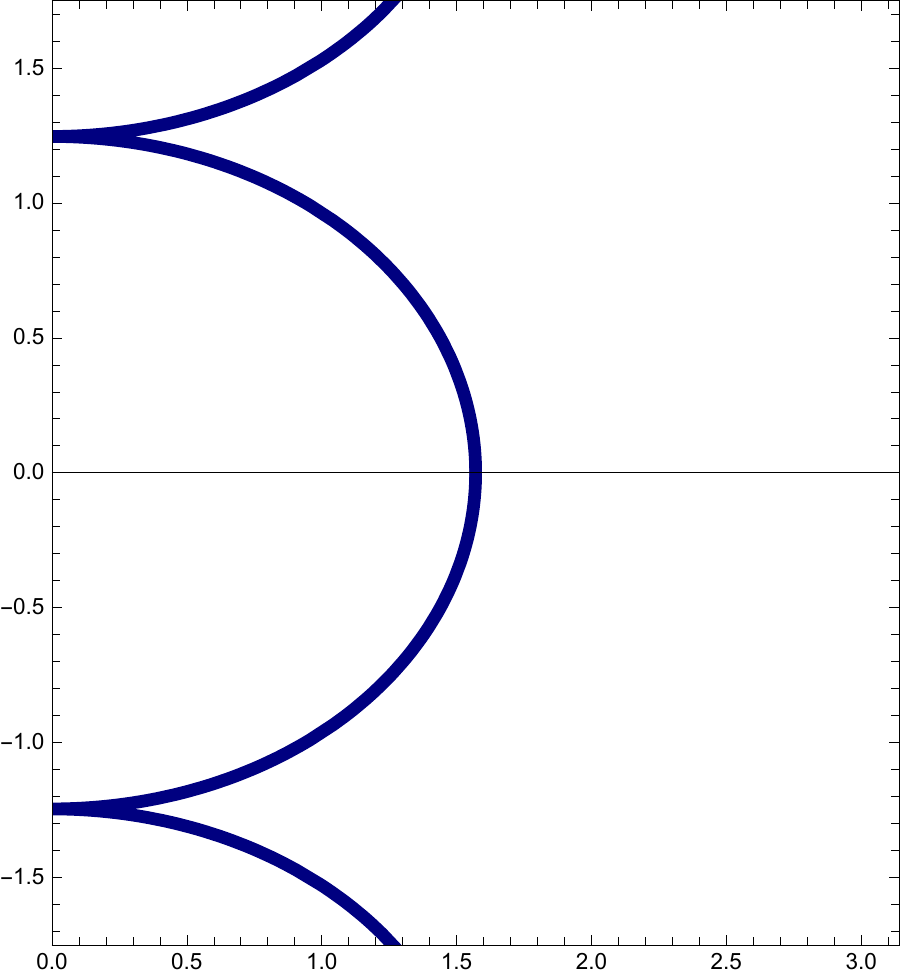} \\
		\includegraphics[width=\textwidth]{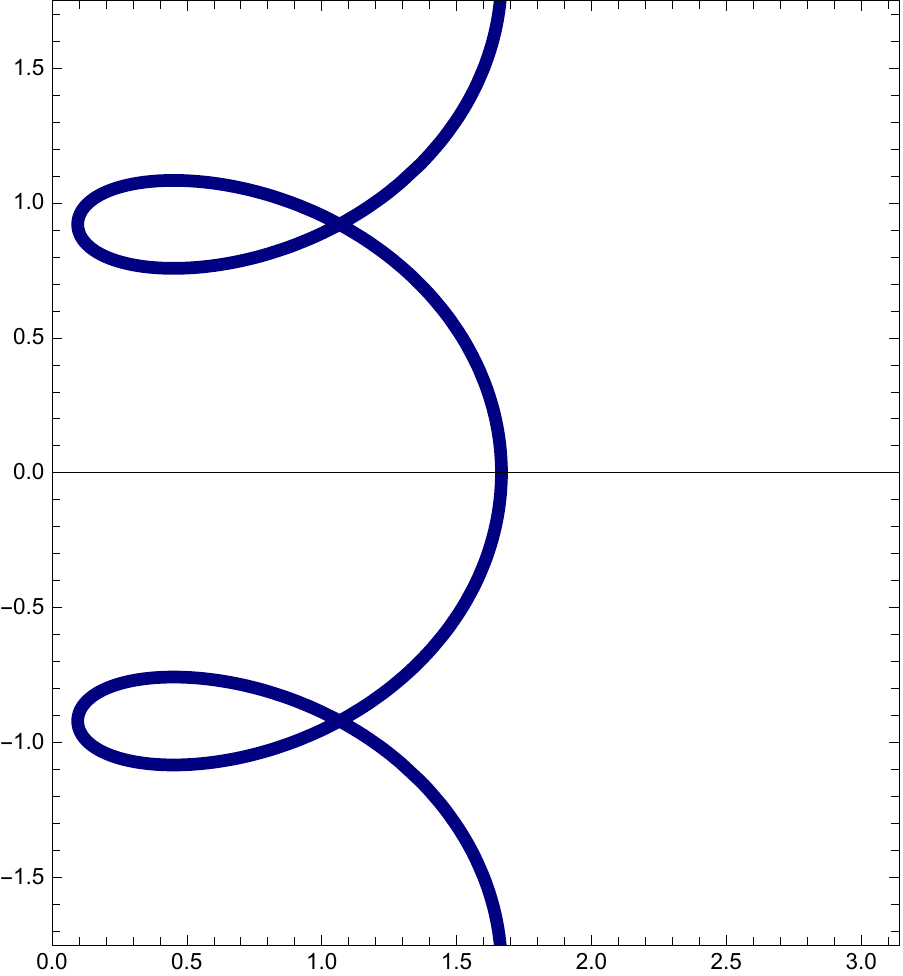} \\
		\includegraphics[width=\textwidth]{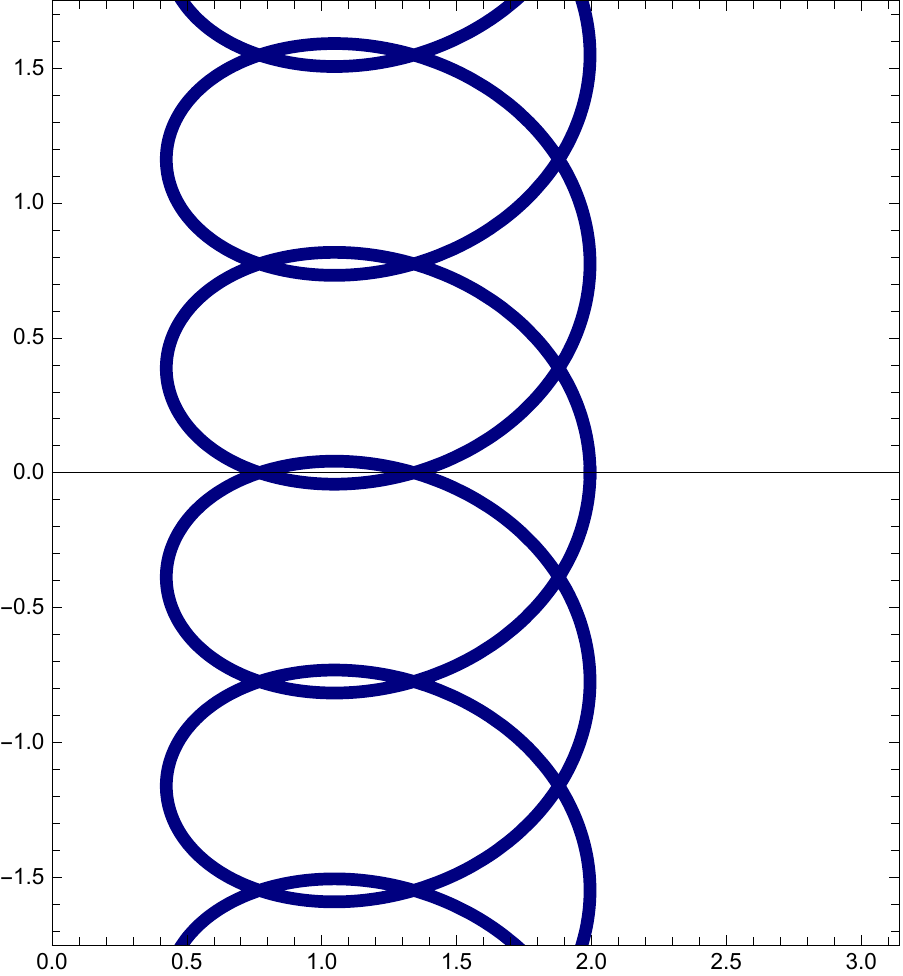} \\
		\includegraphics[width=\textwidth]{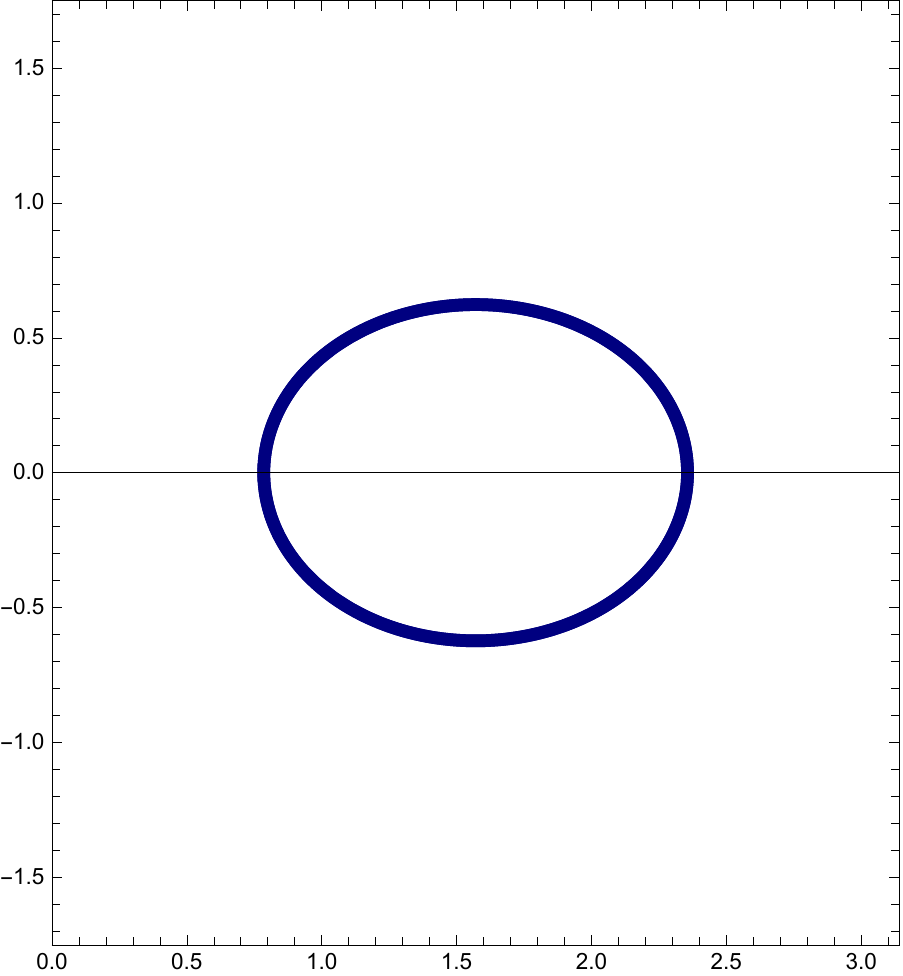}
	\end{minipage}
	\begin{minipage}[t]{0.16\textwidth}
		\includegraphics[width=\textwidth]{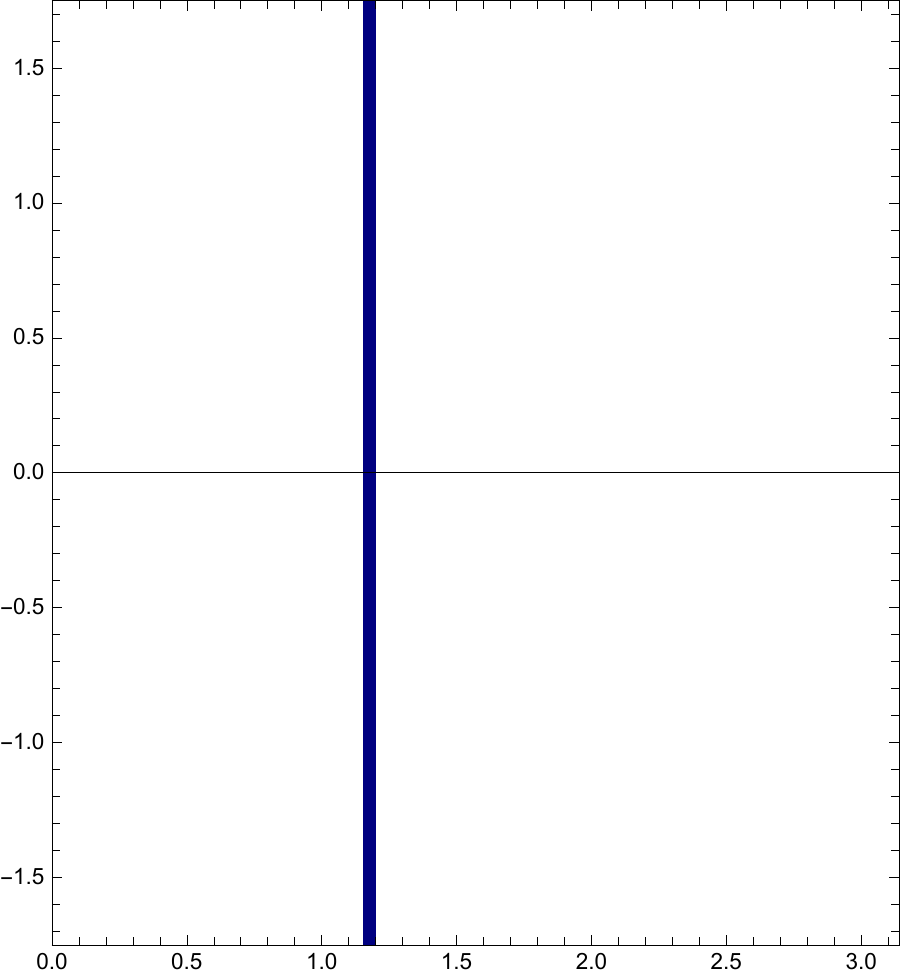} \\
		\includegraphics[width=\textwidth]{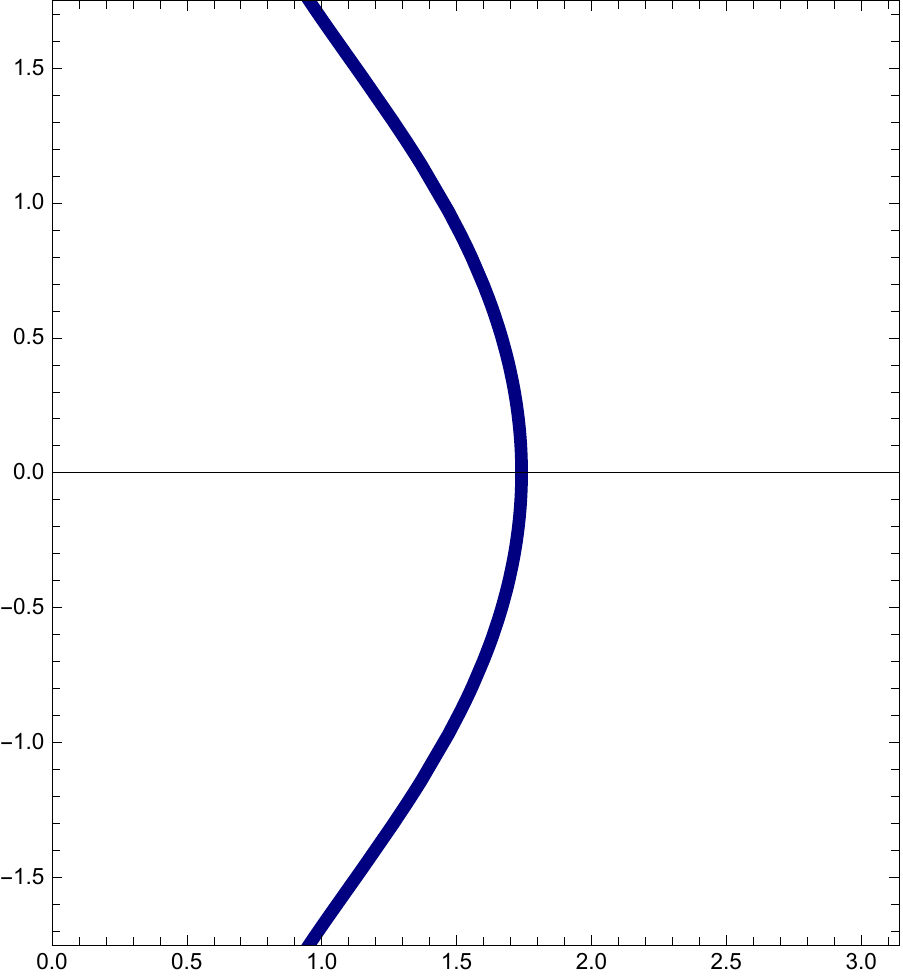} \\
		\includegraphics[width=\textwidth]{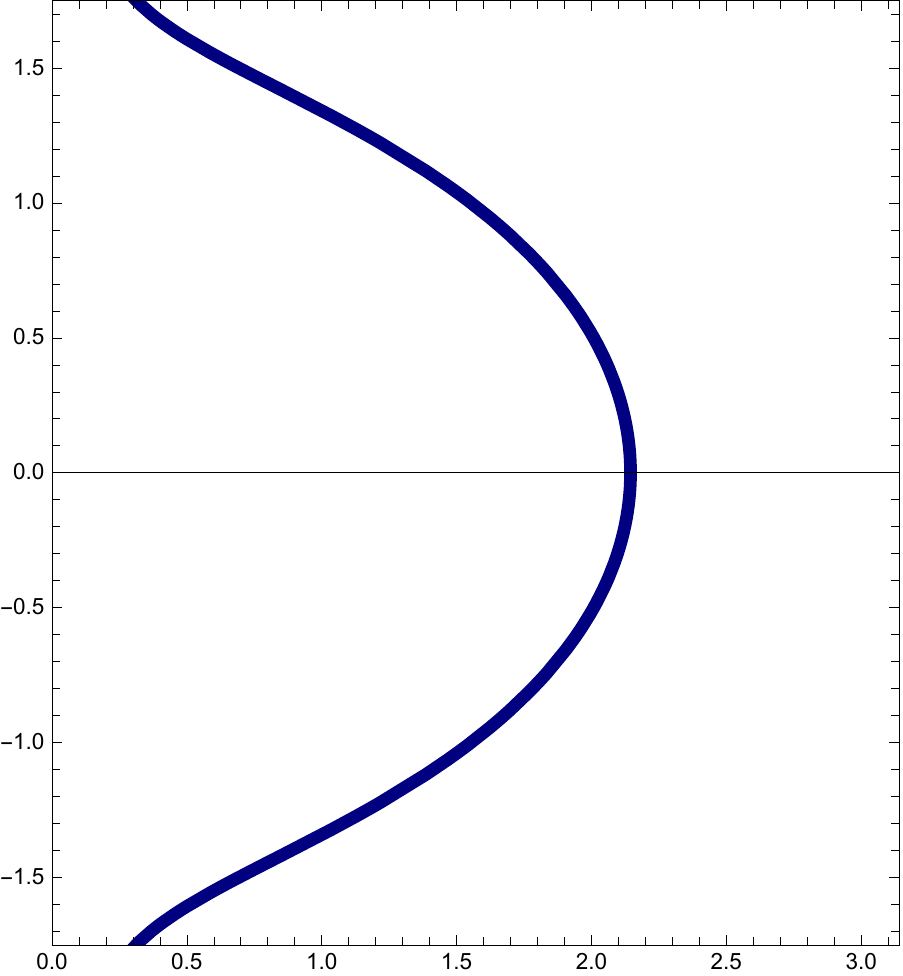} \\
		\includegraphics[width=\textwidth]{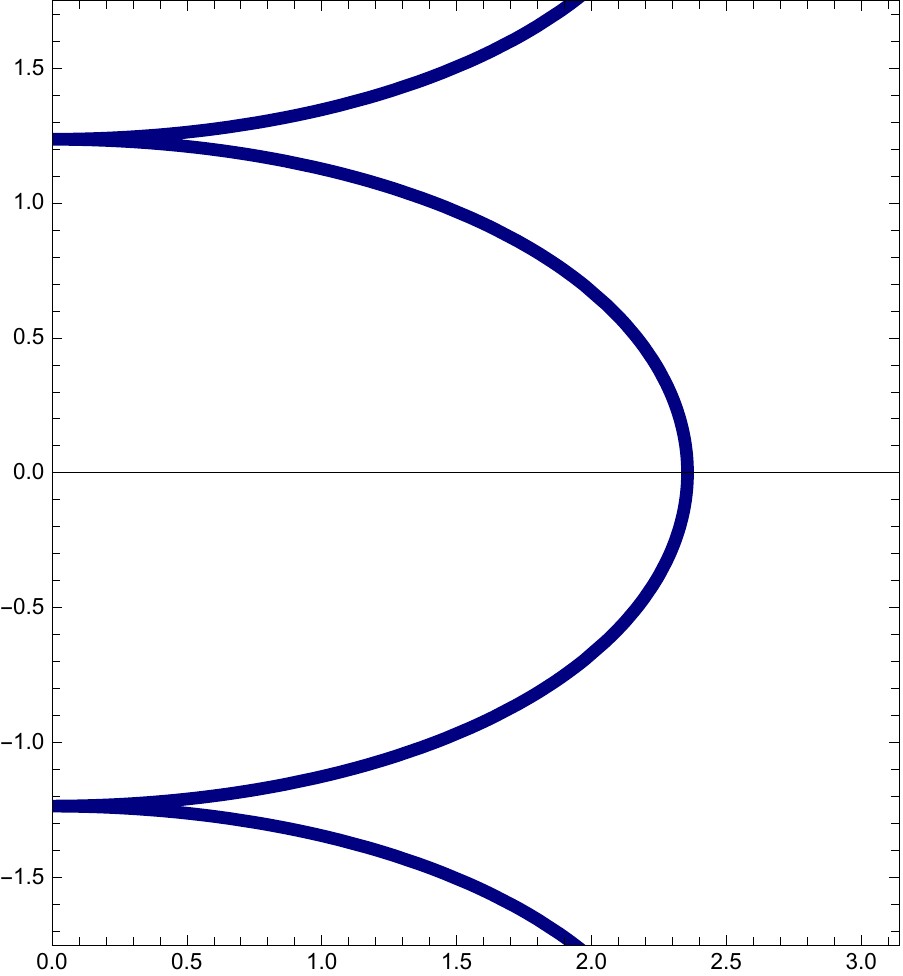} \\
		\includegraphics[width=\textwidth]{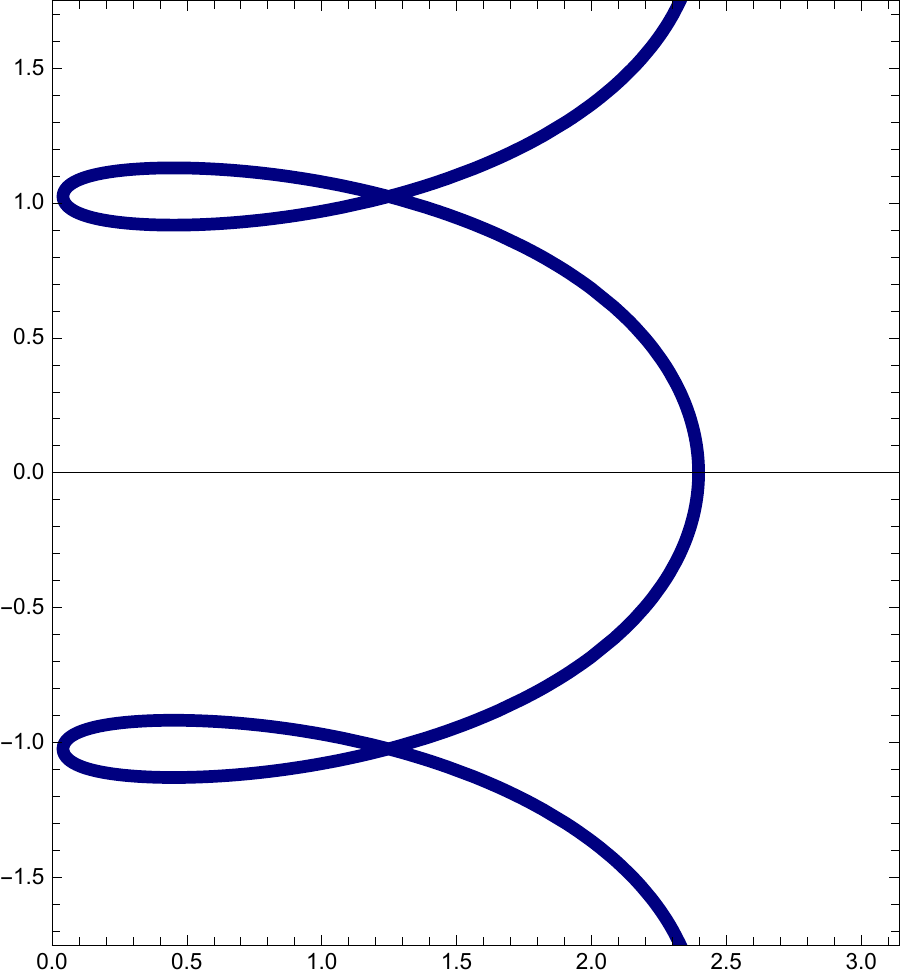} \\
		\includegraphics[width=\textwidth]{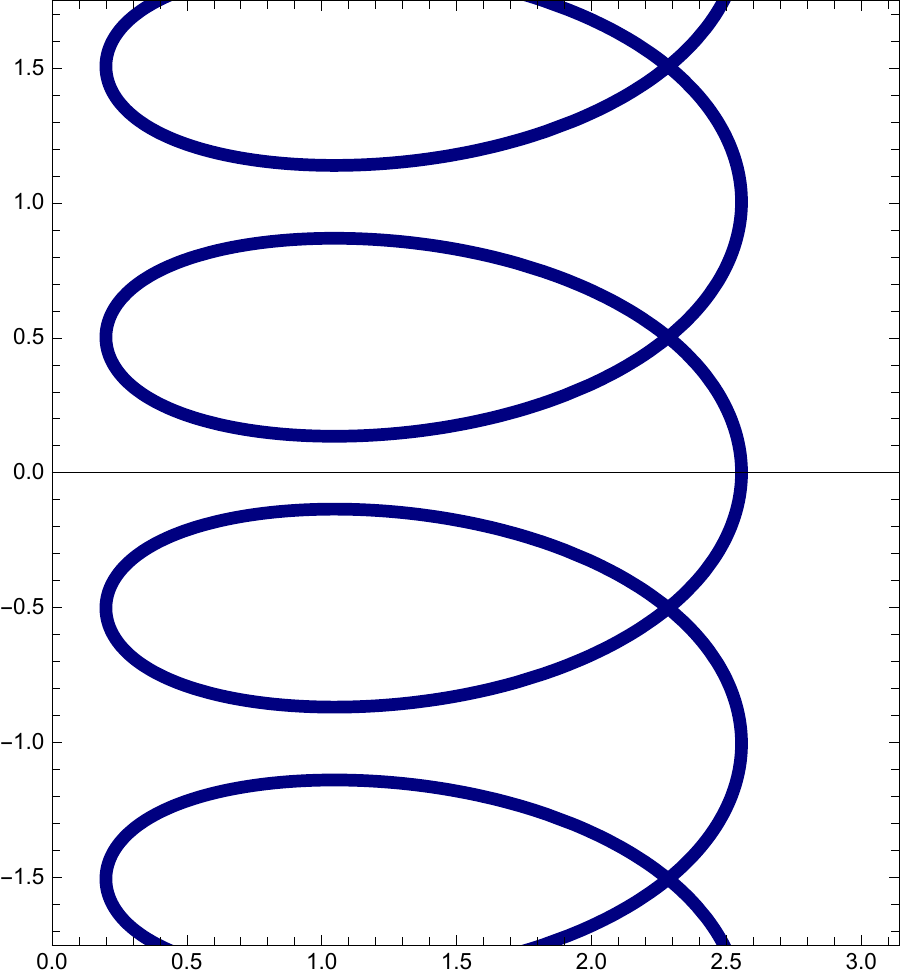} \\
		\includegraphics[width=\textwidth]{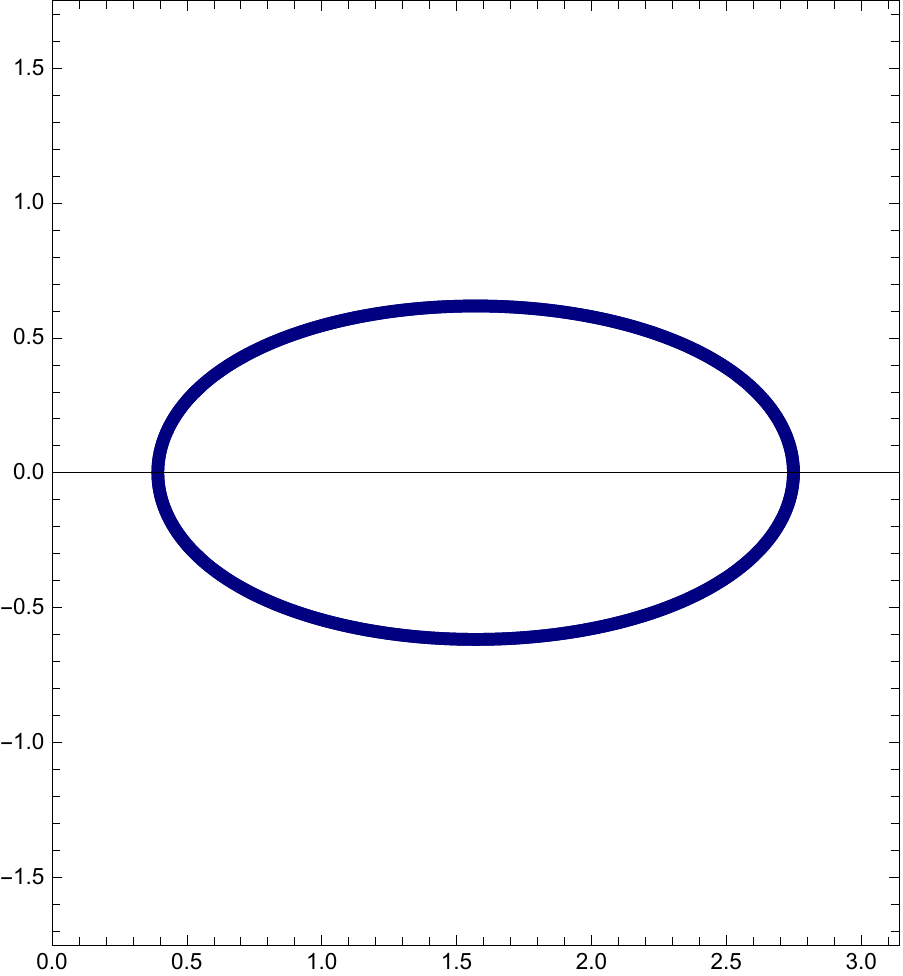}
	\end{minipage}
	\begin{minipage}[t]{0.16\textwidth}
		\includegraphics[width=\textwidth]{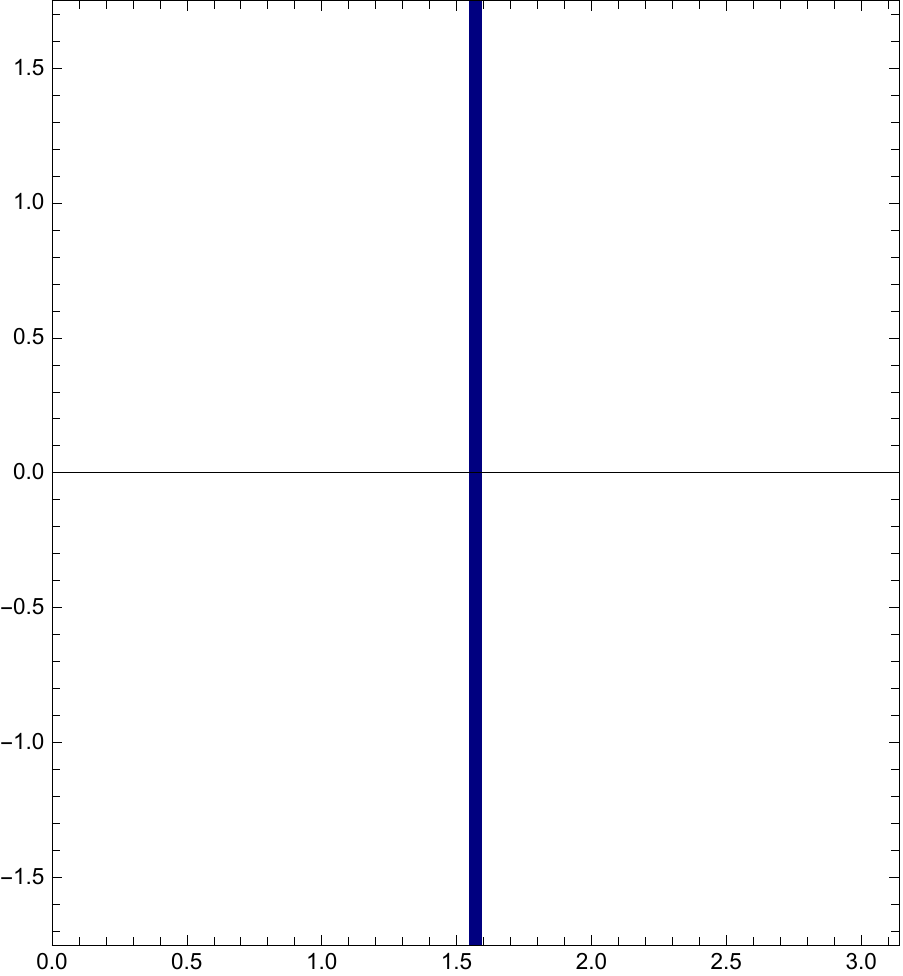} \\
		\includegraphics[width=\textwidth]{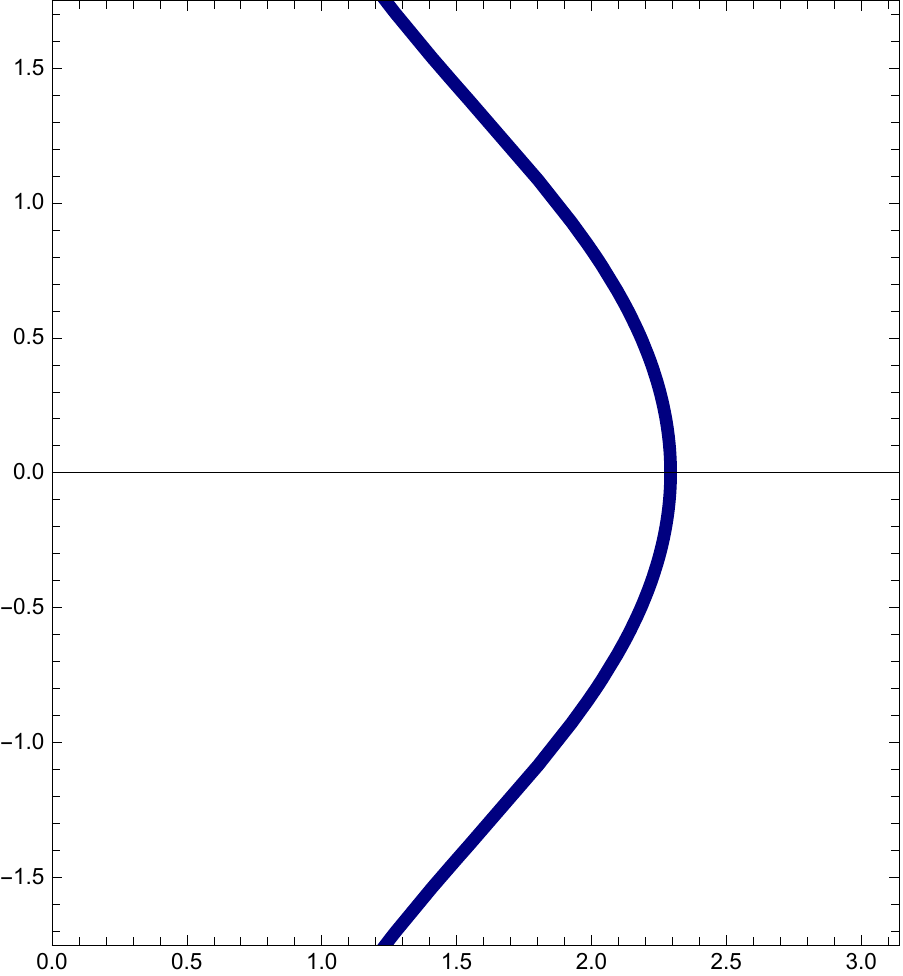} \\
		\includegraphics[width=\textwidth]{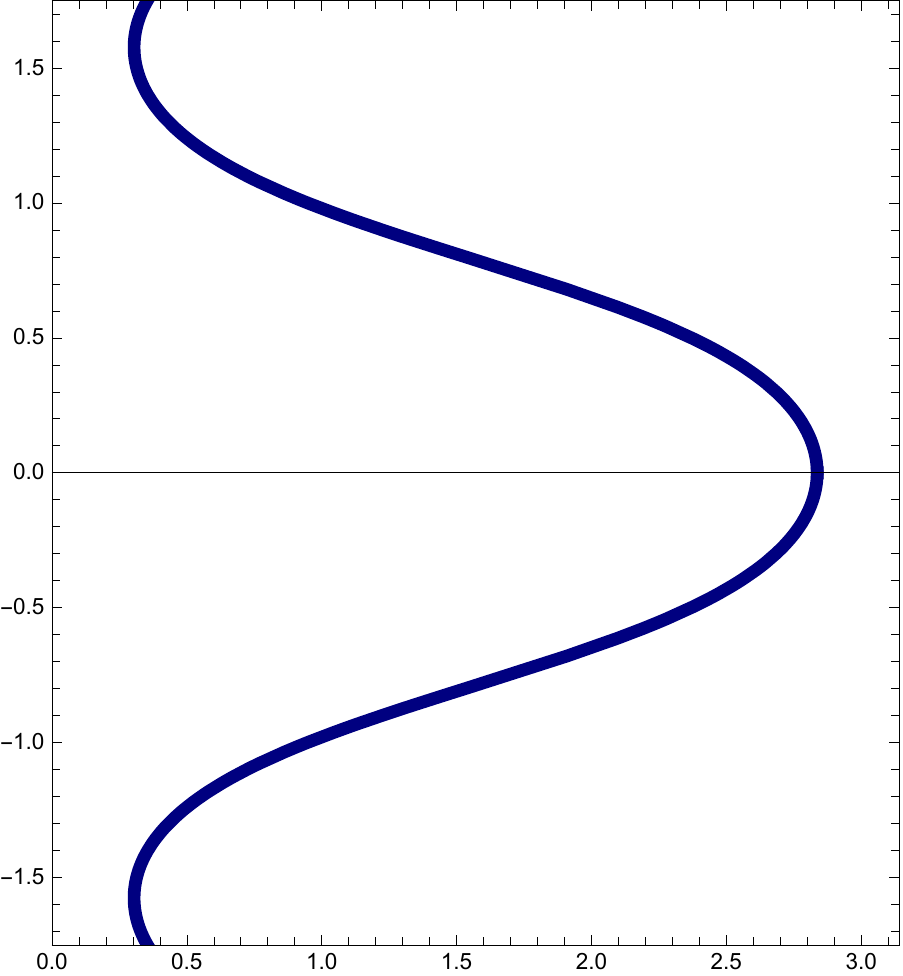} \\
		\includegraphics[width=\textwidth]{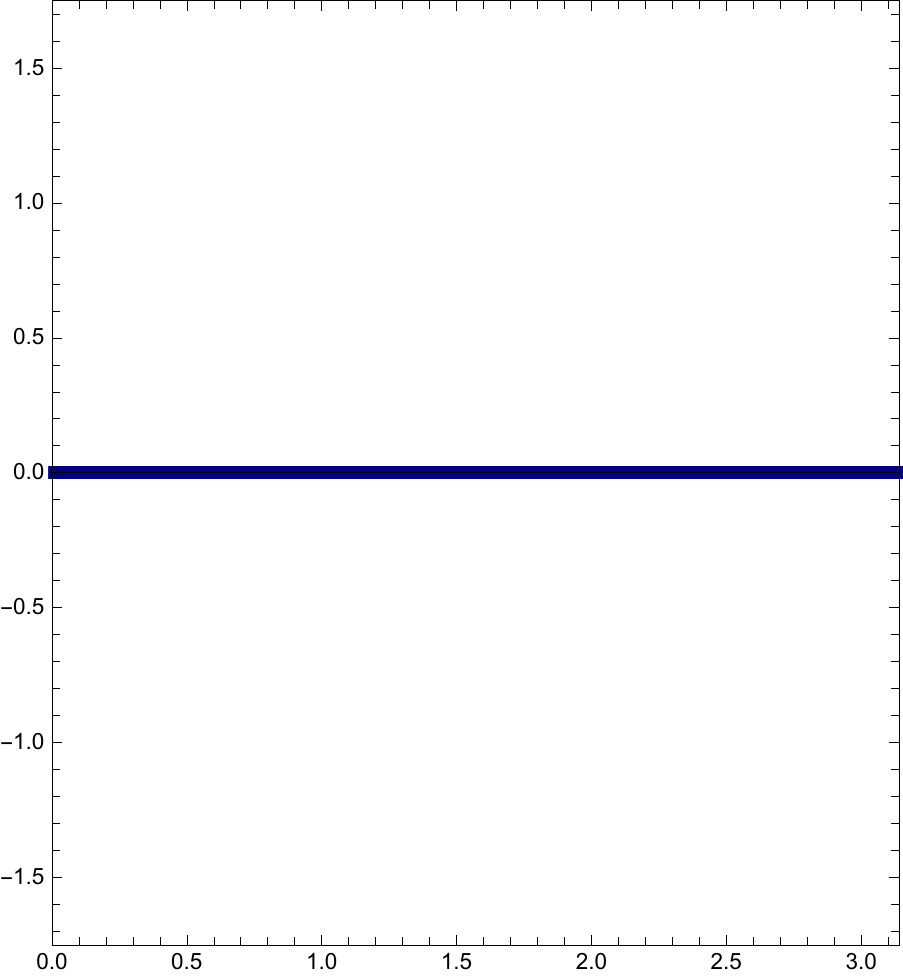} \\
		\includegraphics[width=\textwidth]{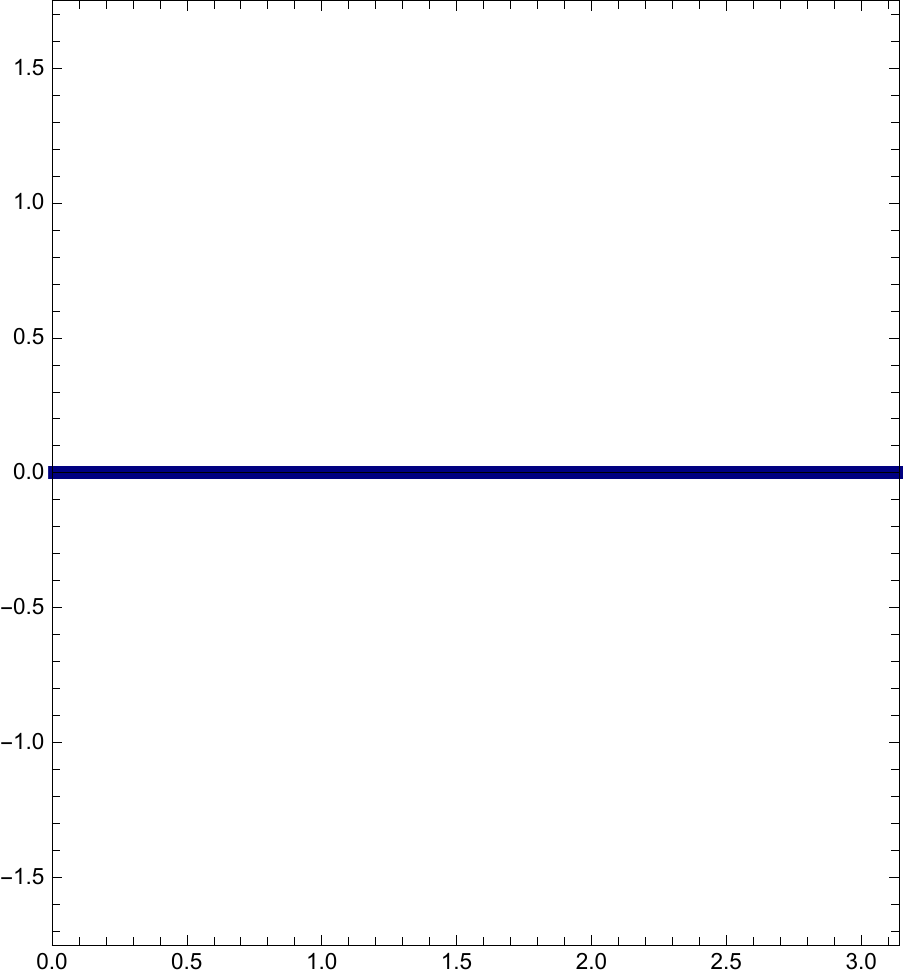} \\
		\includegraphics[width=\textwidth]{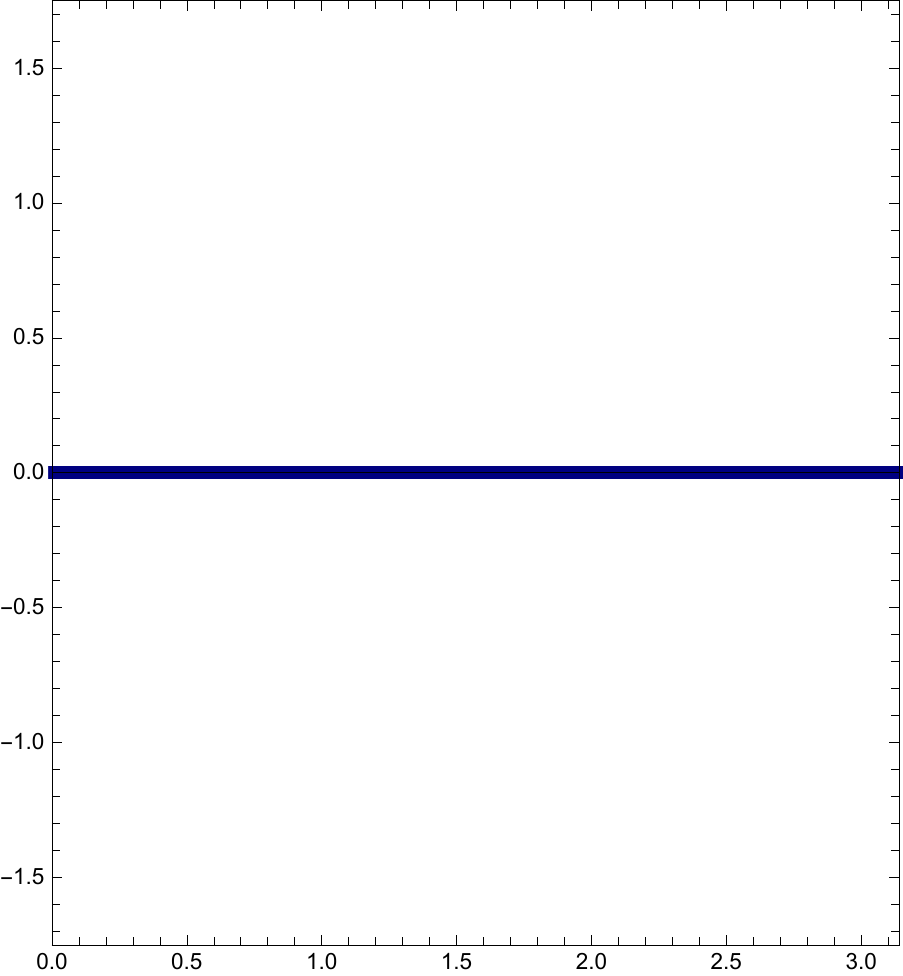} \\
		\includegraphics[width=\textwidth]{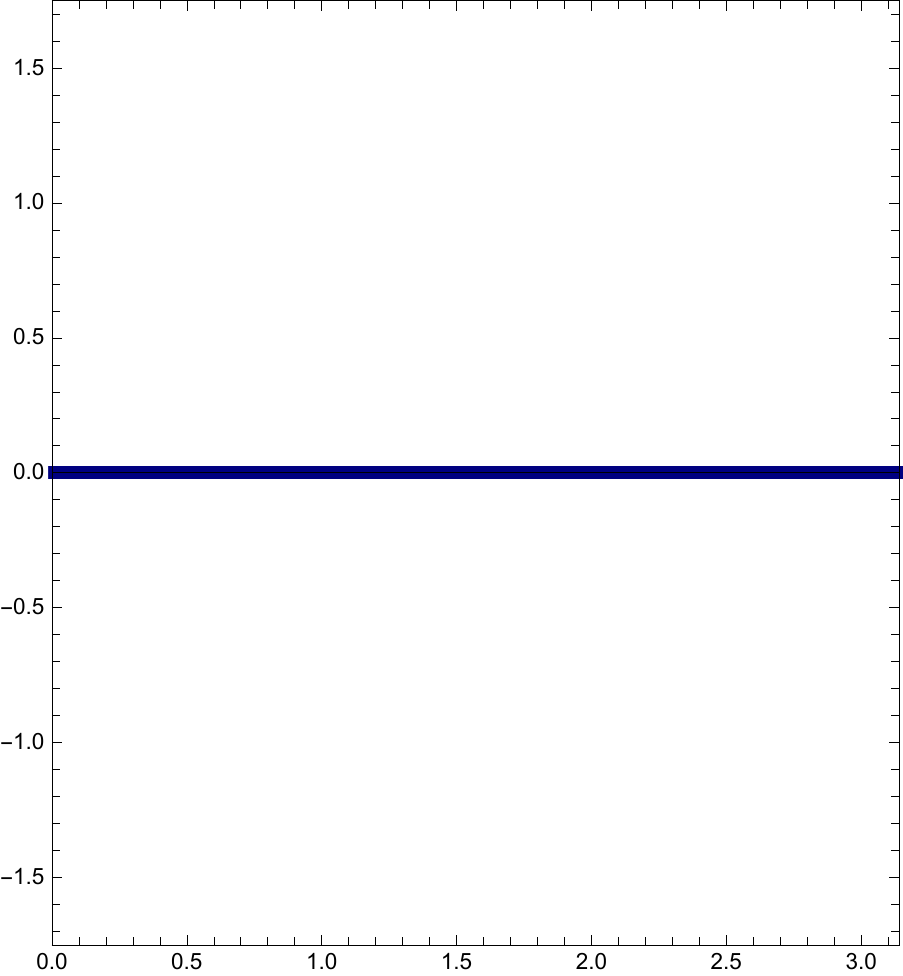}
	\end{minipage}
	\caption{Numerically computed profile curves of the family $\family$ in $\SR$ for the rotational case $a=0$. The energy decreases from $\Jmax$ in the first row (vertical cylinder) to $J=-\frac{2H}{\kappa}$ in the last row (tube). The intermediate rows display surfaces of unduloid type, sphere type, and nodoid type. Each column represents a fixed value of the mean curvature $H$, from large $H$ in the left column to the minimal case in the right column, where all curves of non-positive energy coincide with constant height. The energy here is degenerated.}
	\label{fig:curves_skr}
\end{figure}

\newpage

\begin{figure}[H]
	\centering
	\begin{minipage}[t]{0.16\textwidth}
		\includegraphics[width=\textwidth]{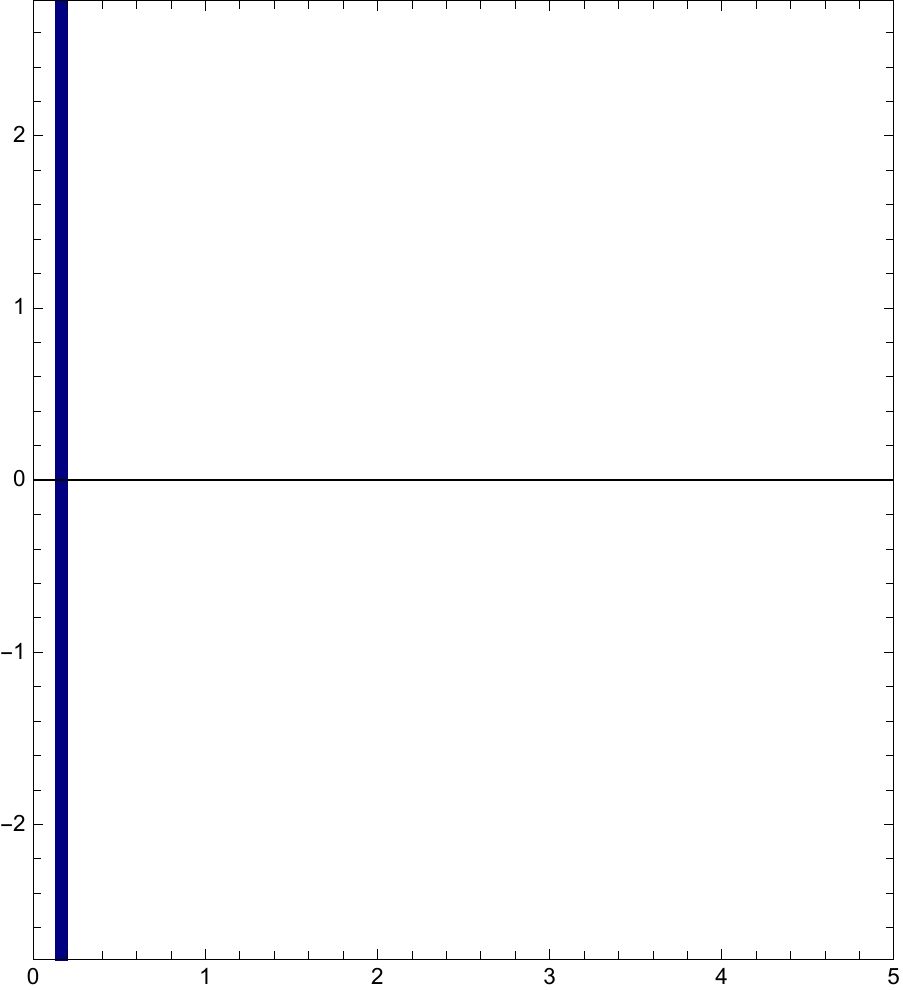} \\
		\includegraphics[width=\textwidth]{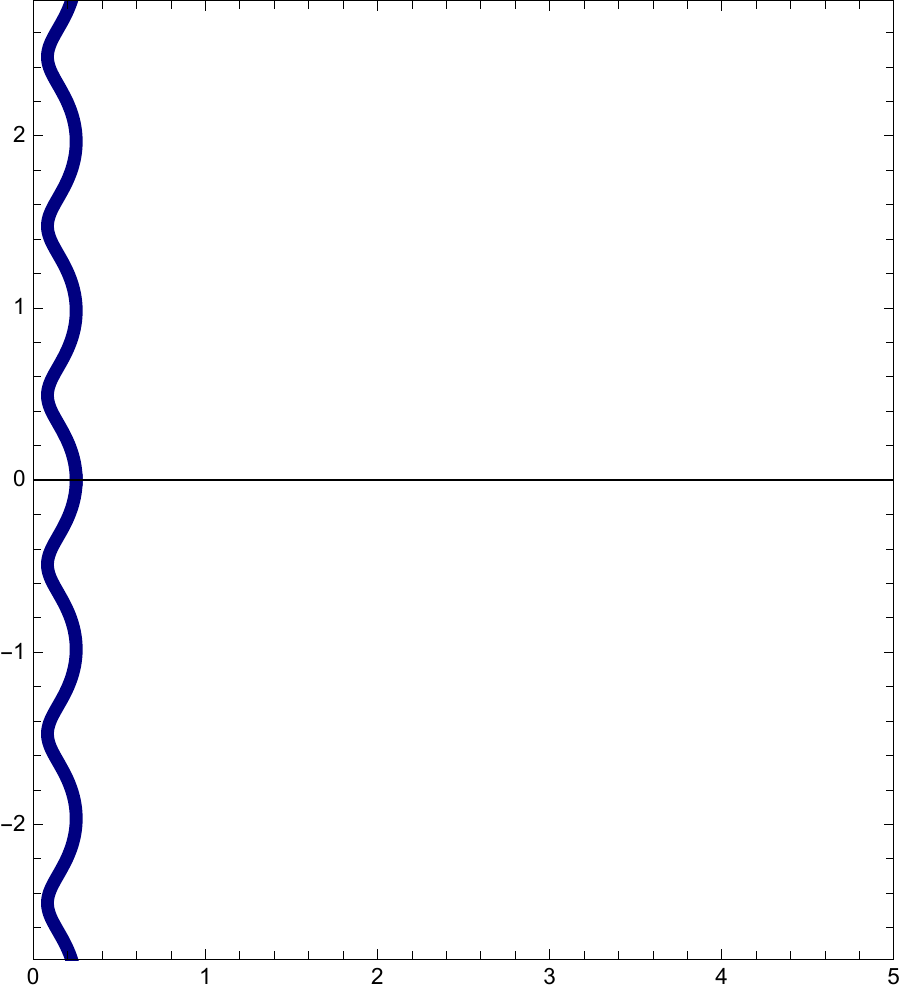} \\
		\includegraphics[width=\textwidth]{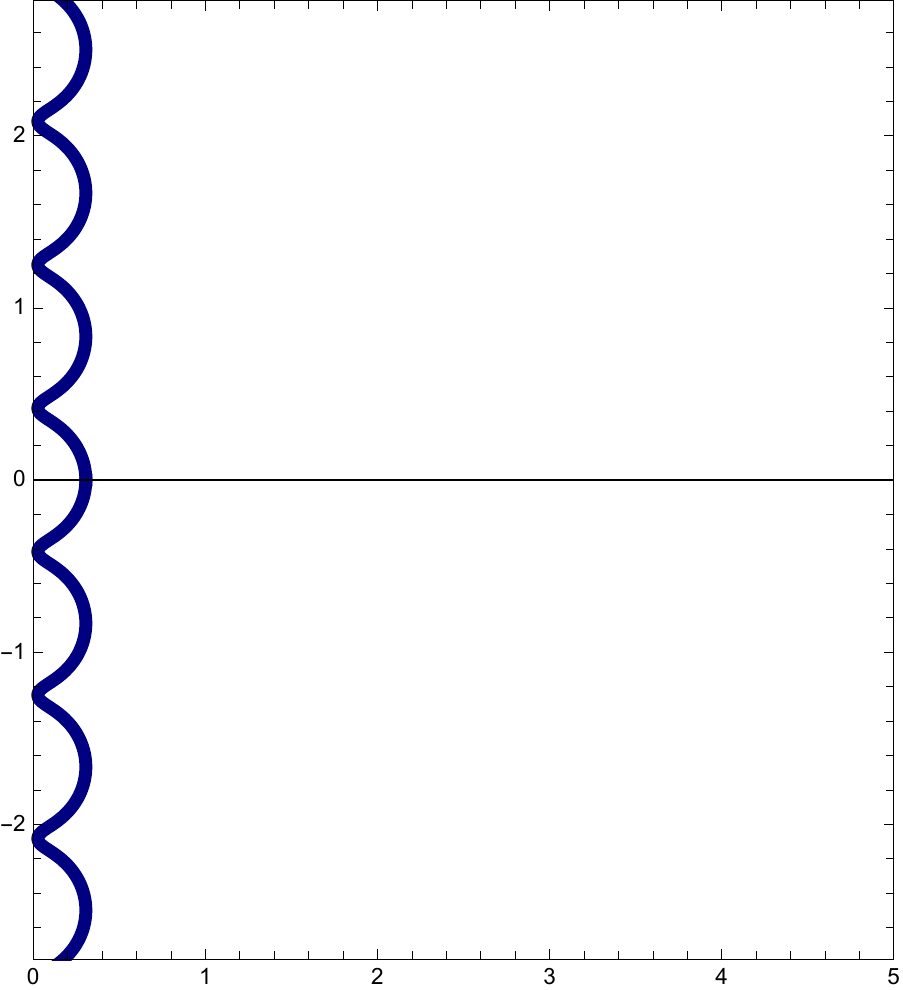} \\
		\includegraphics[width=\textwidth]{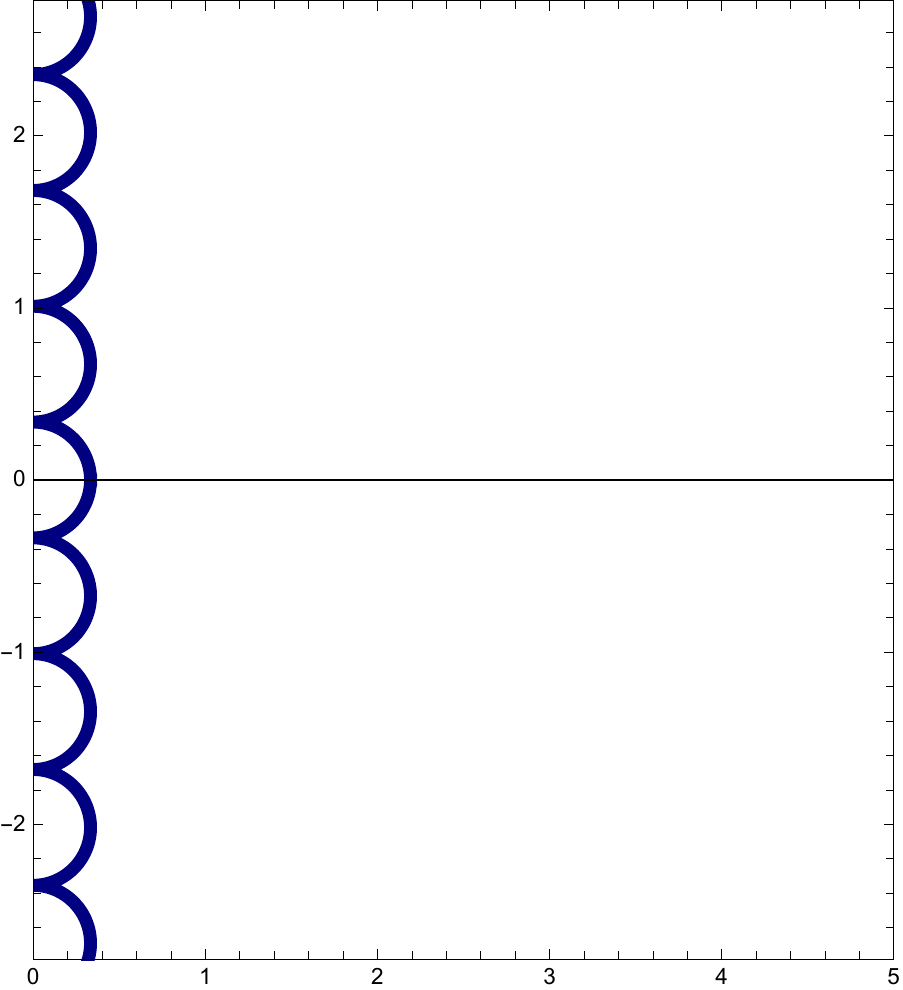} \\
		\includegraphics[width=\textwidth]{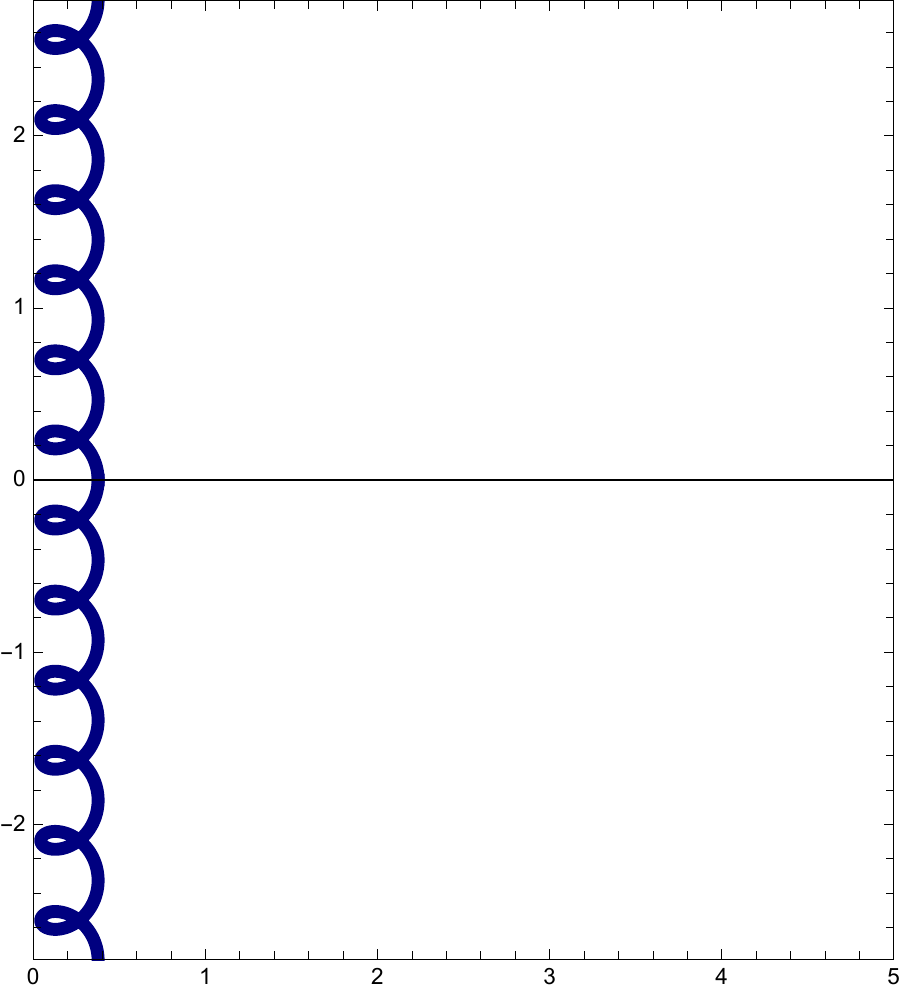} \\
		\includegraphics[width=\textwidth]{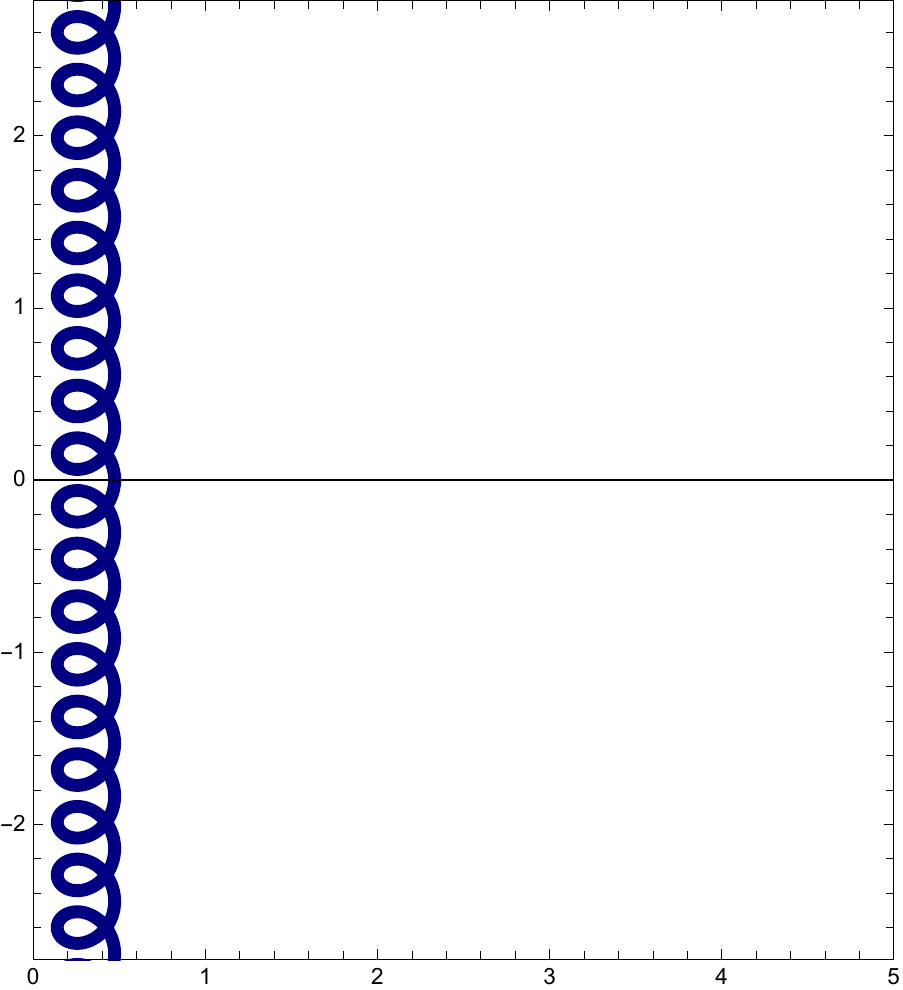} \\
		\includegraphics[width=\textwidth]{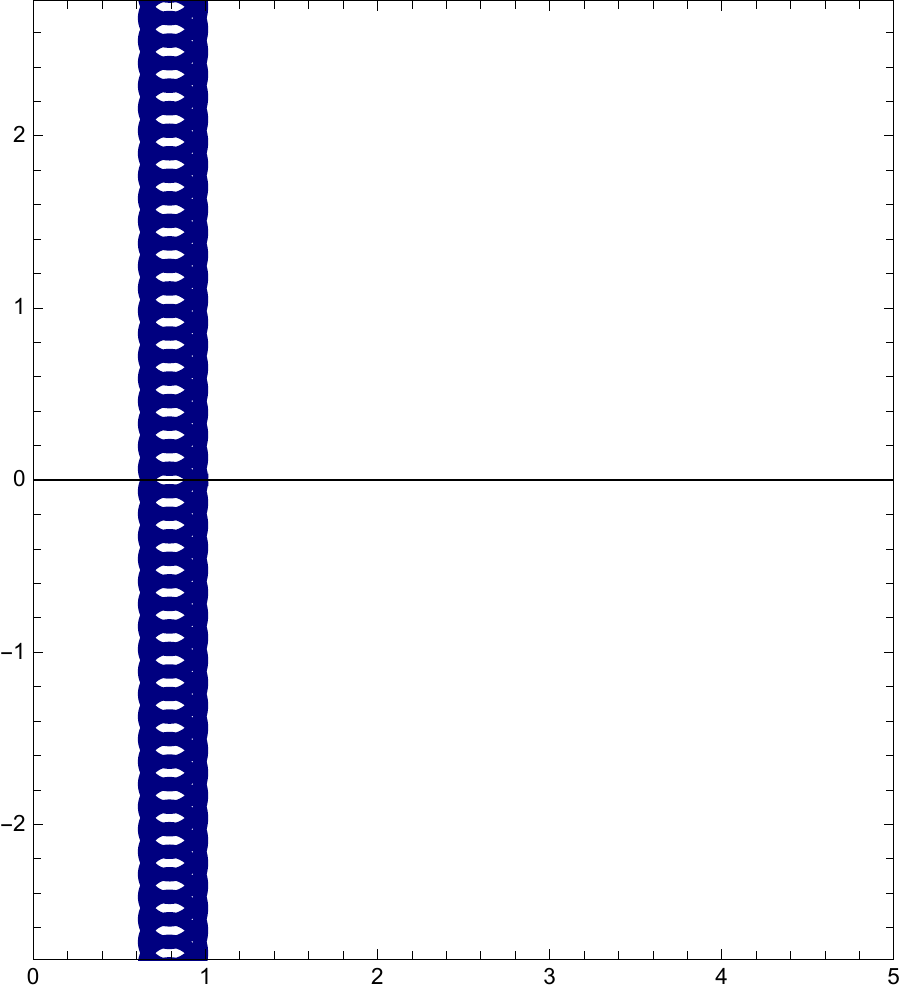}
	\end{minipage}
	\begin{minipage}[t]{0.16\textwidth}
		\includegraphics[width=\textwidth]{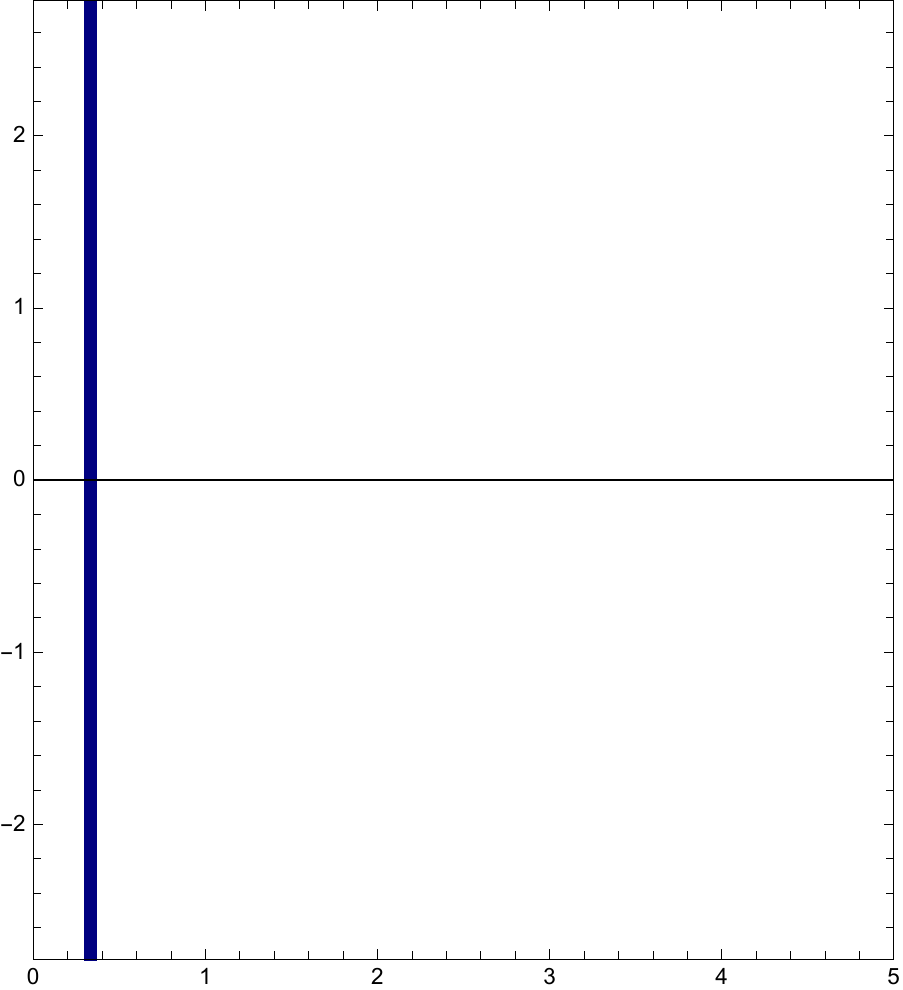} \\
		\includegraphics[width=\textwidth]{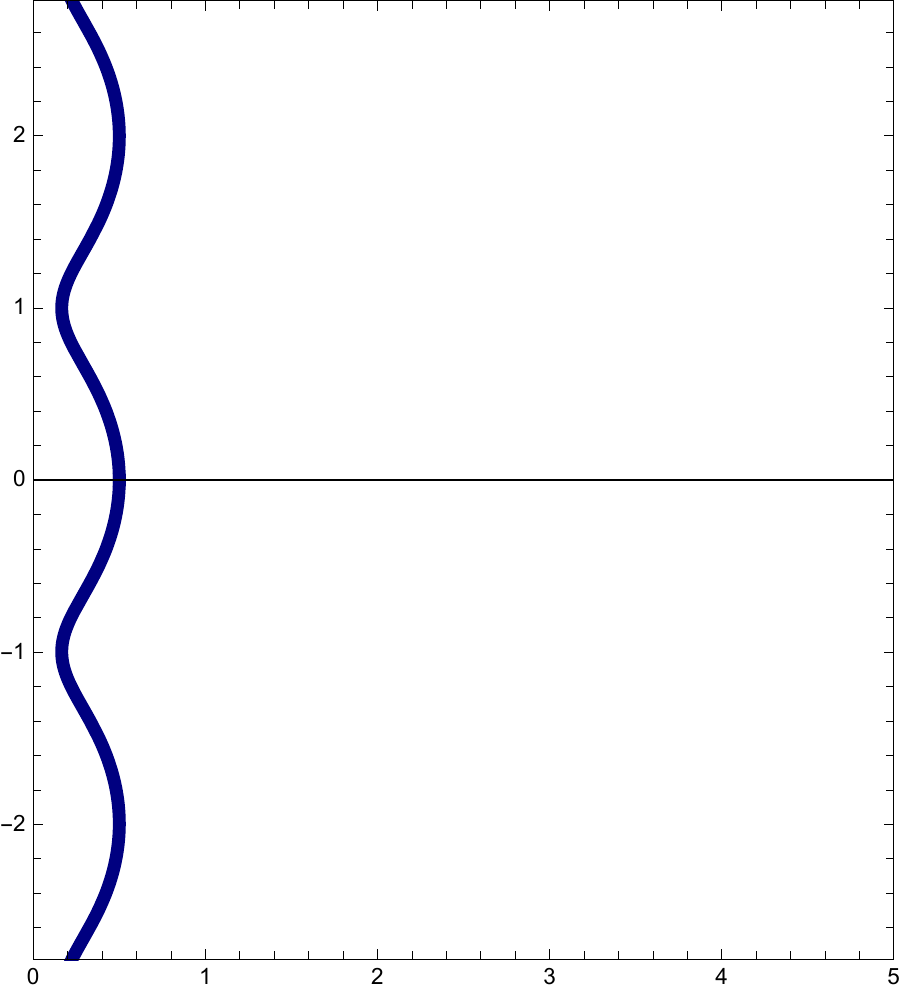} \\
		\includegraphics[width=\textwidth]{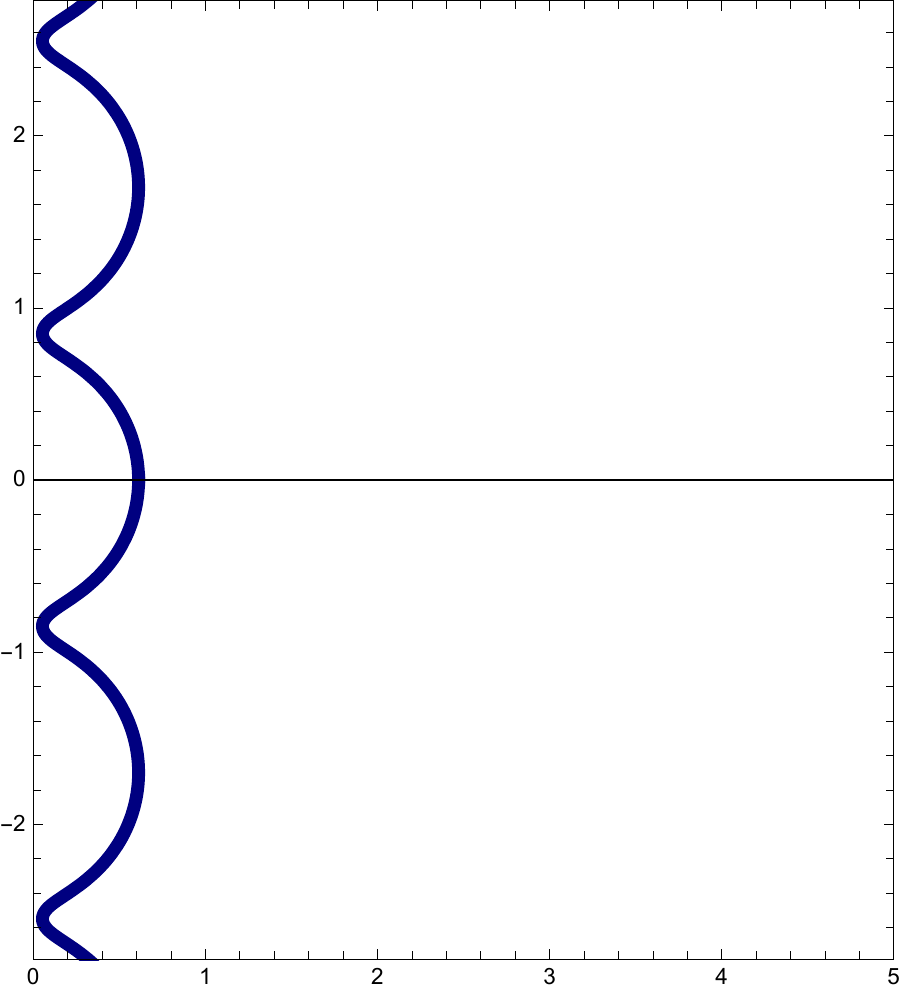} \\
		\includegraphics[width=\textwidth]{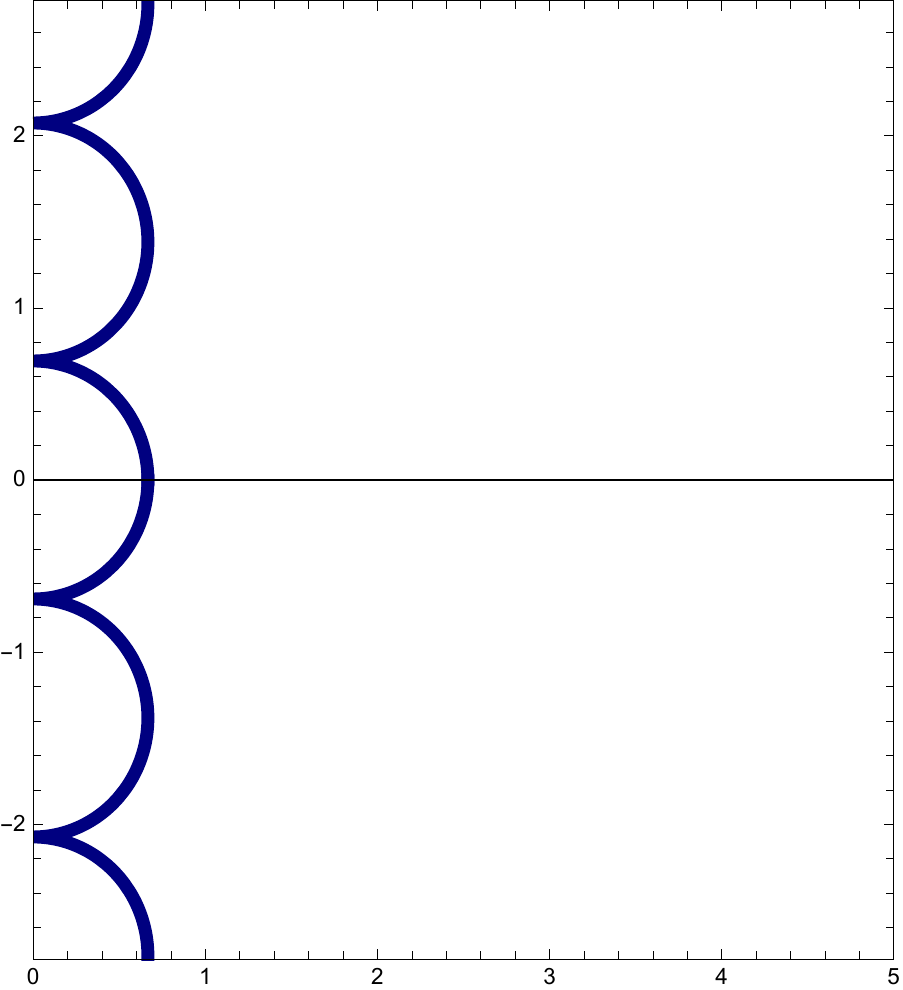} \\
		\includegraphics[width=\textwidth]{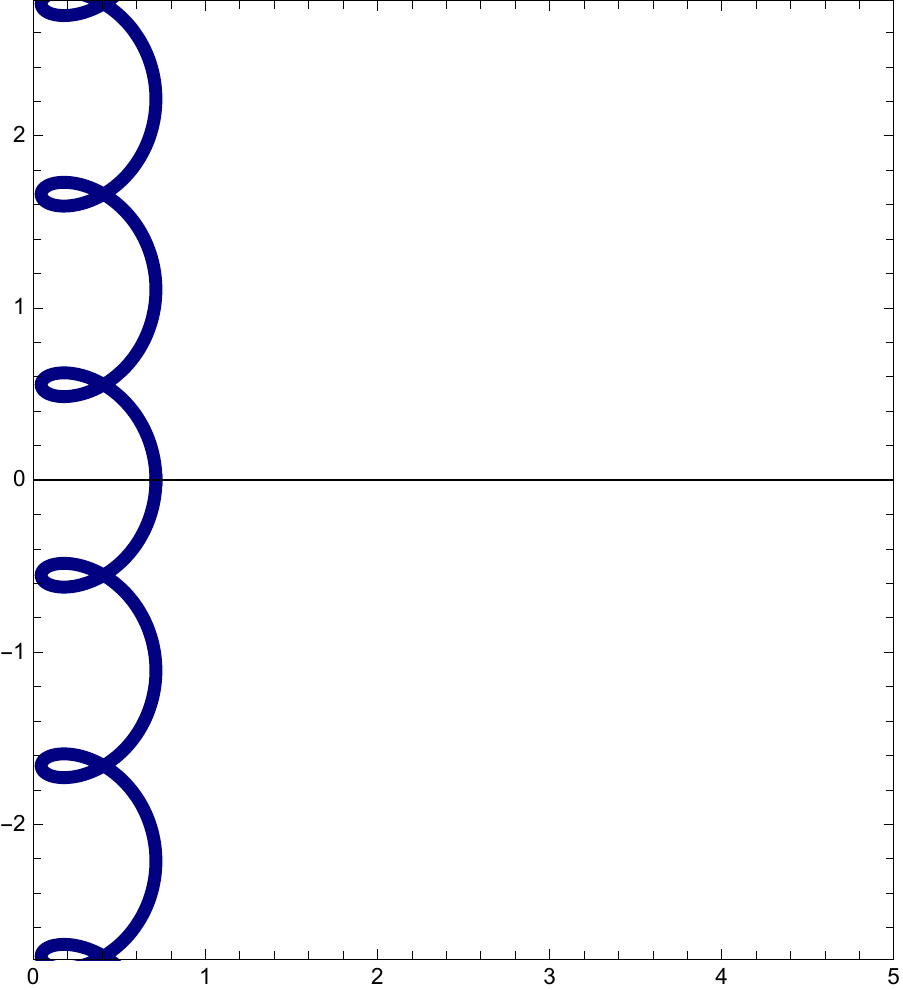} \\
		\includegraphics[width=\textwidth]{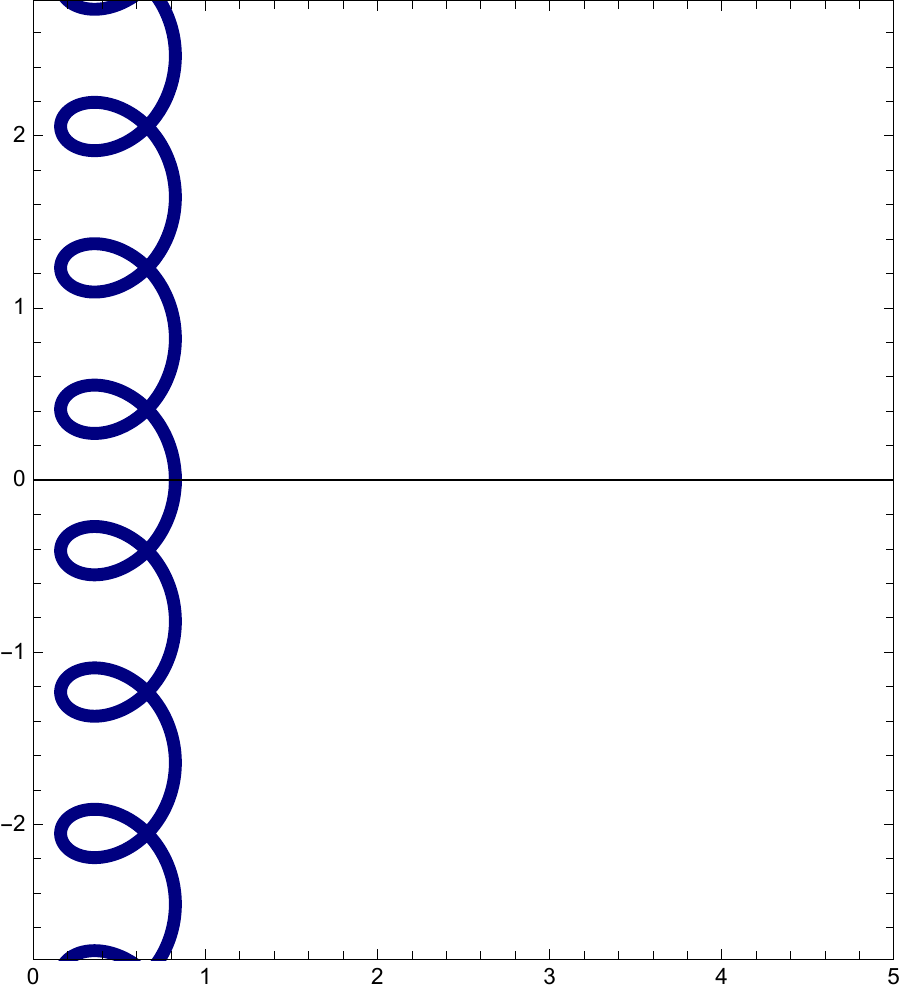} \\
		\includegraphics[width=\textwidth]{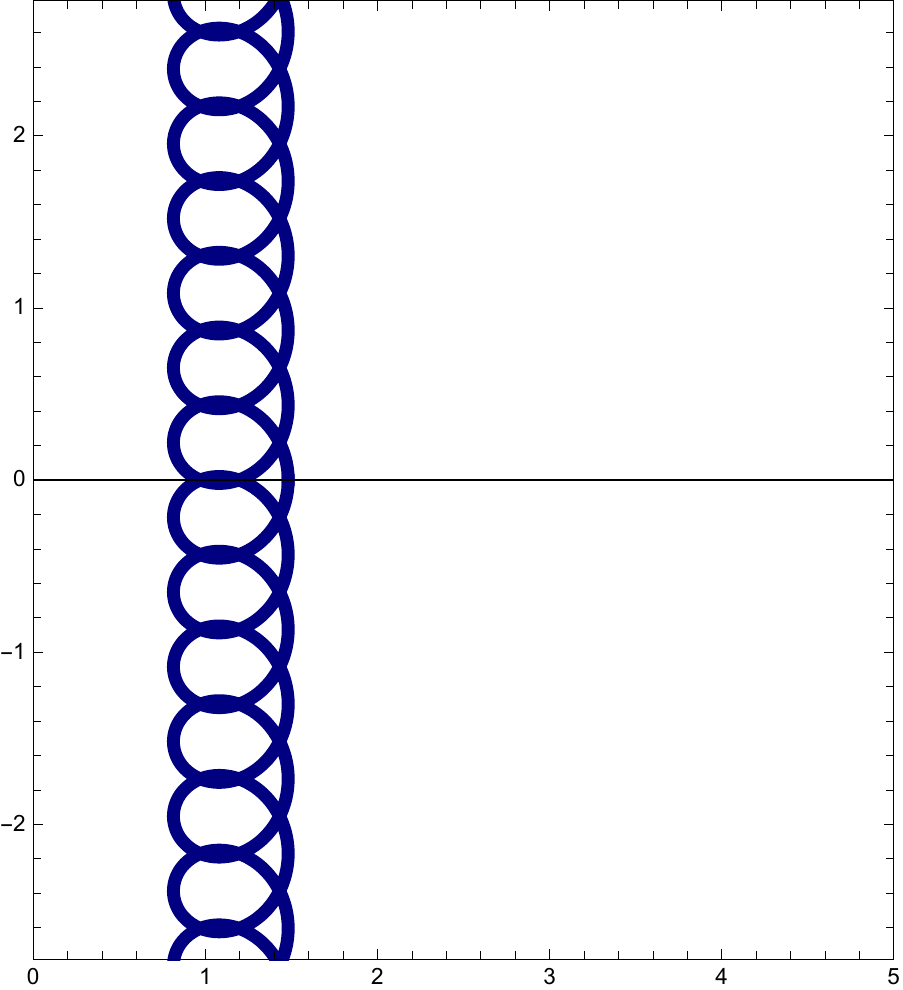}
	\end{minipage}
	\begin{minipage}[t]{0.16\textwidth}
		\includegraphics[width=\textwidth]{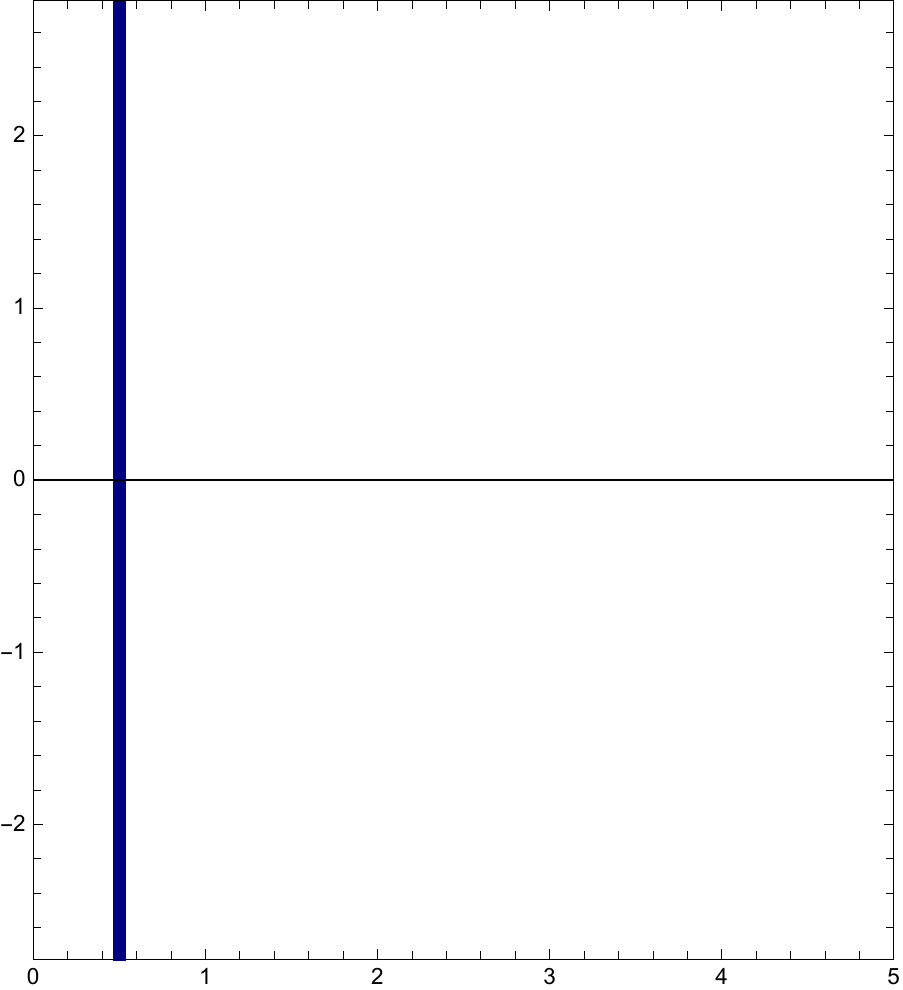} \\
		\includegraphics[width=\textwidth]{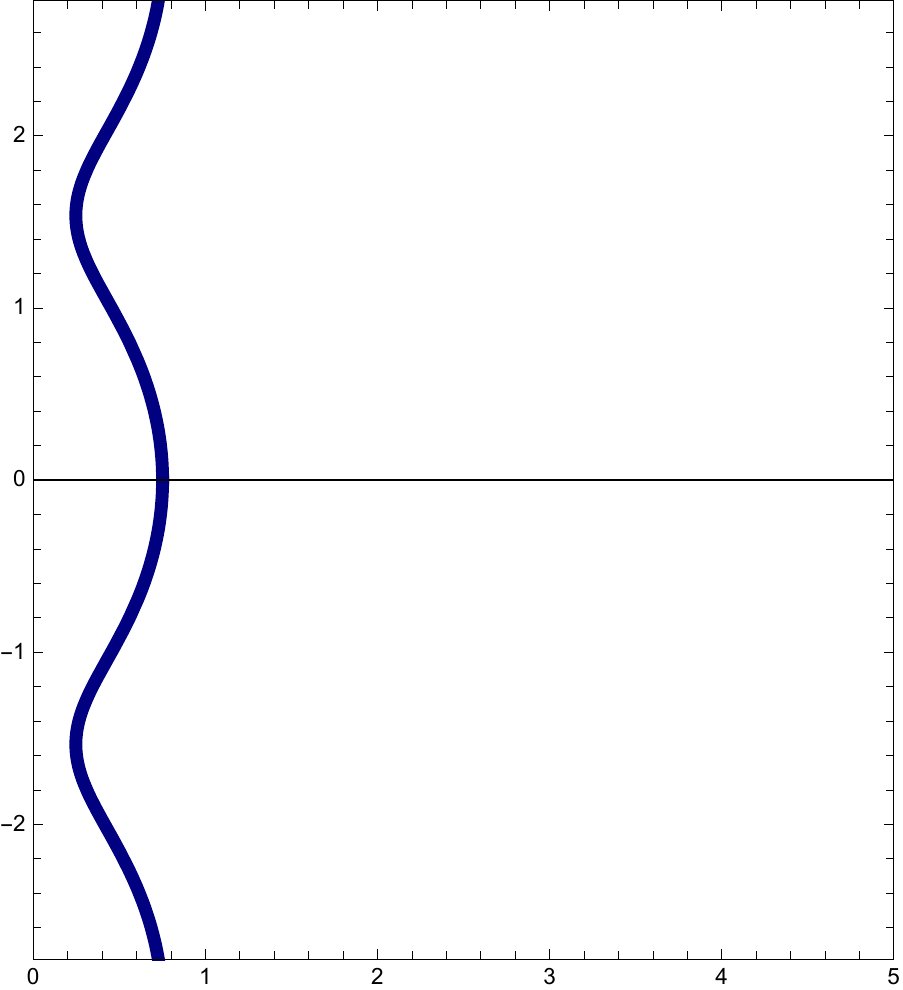}
		\includegraphics[width=\textwidth]{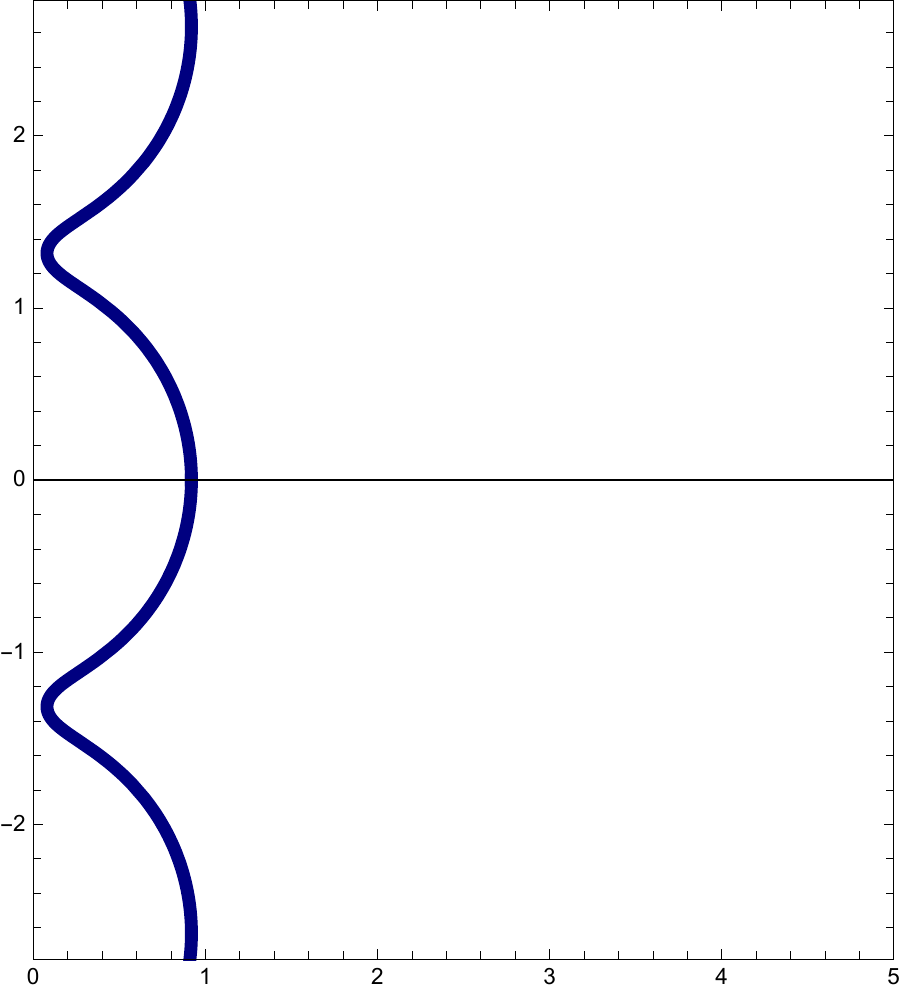} \\
		\includegraphics[width=\textwidth]{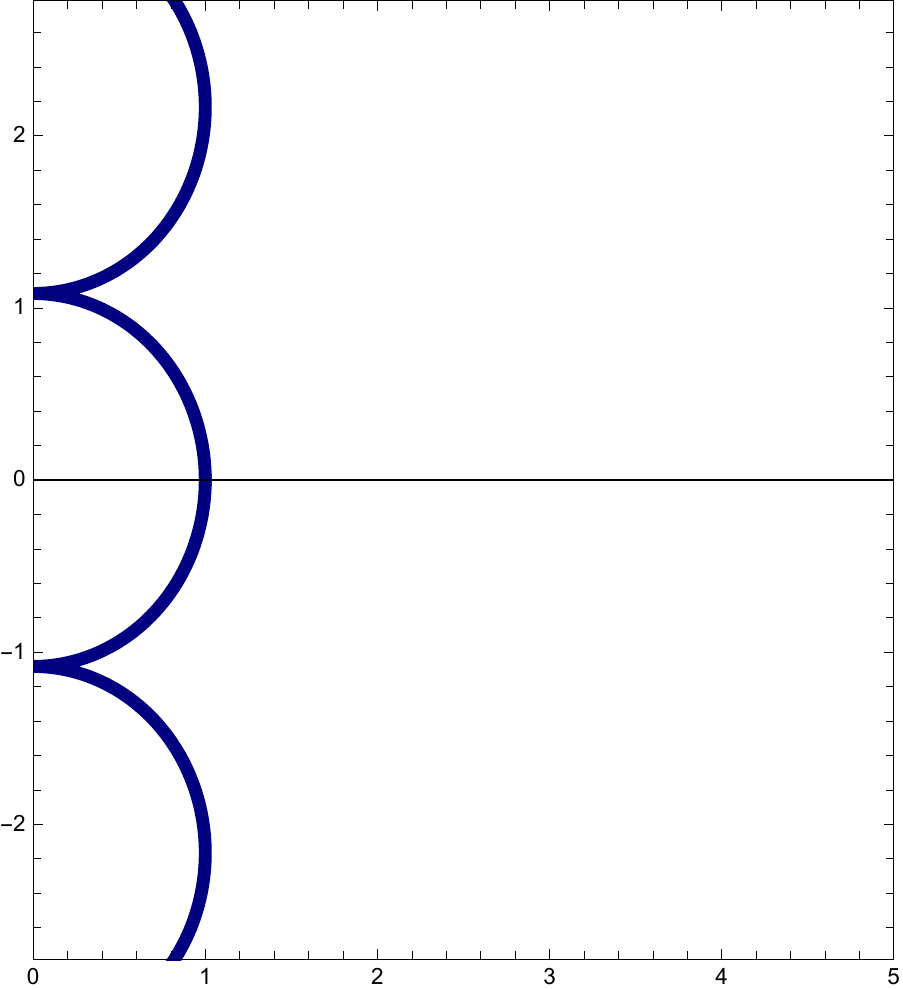} \\
		\includegraphics[width=\textwidth]{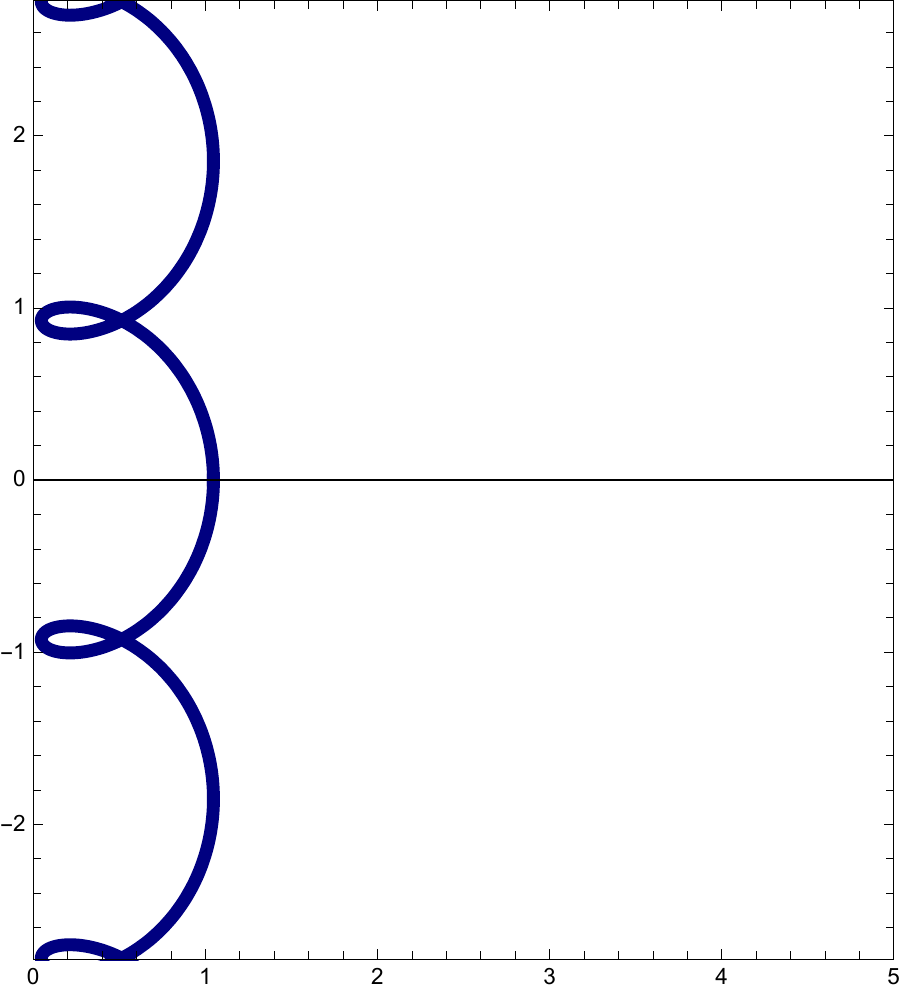} \\
		\includegraphics[width=\textwidth]{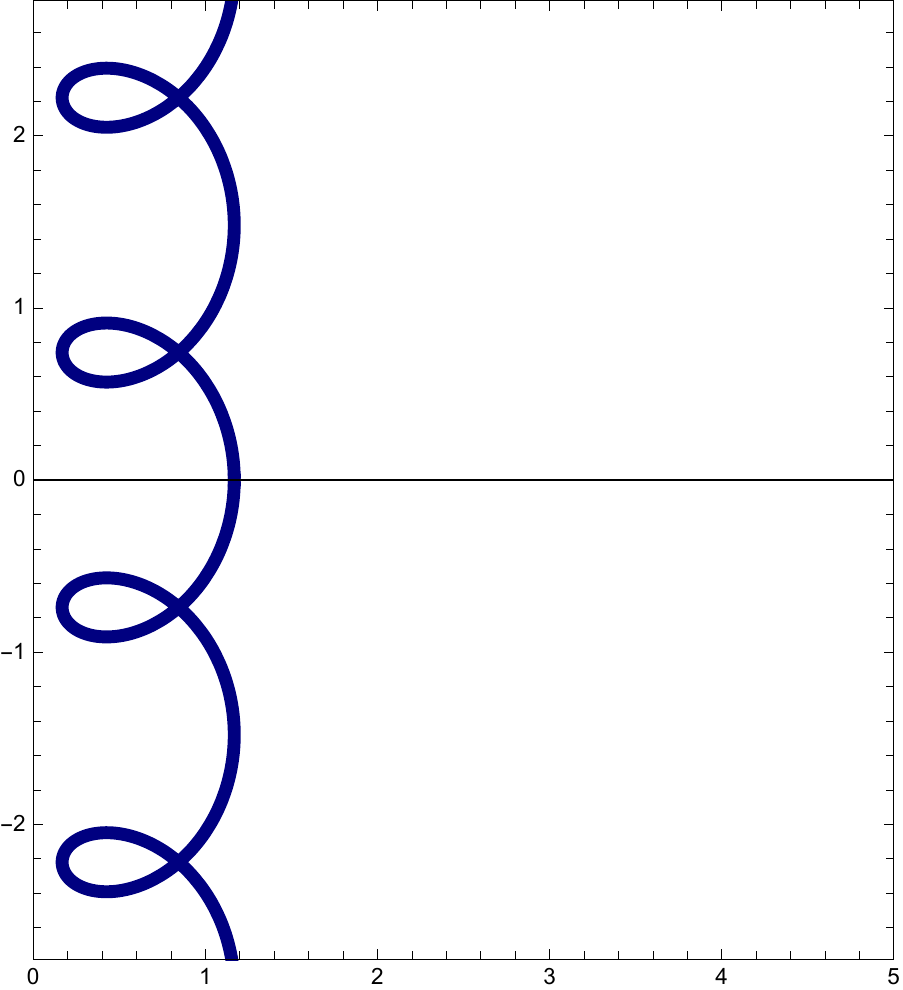} \\
		\includegraphics[width=\textwidth]{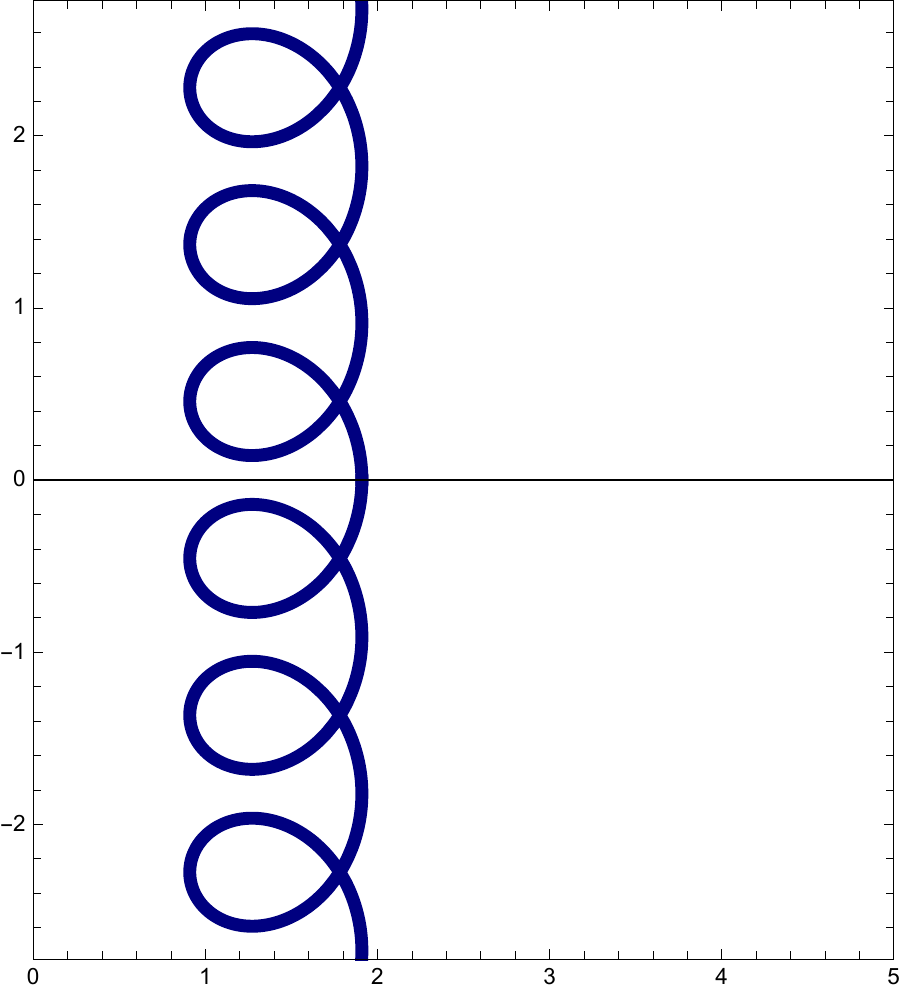}
	\end{minipage}
	\begin{minipage}[t]{0.16\textwidth}
		\includegraphics[width=\textwidth]{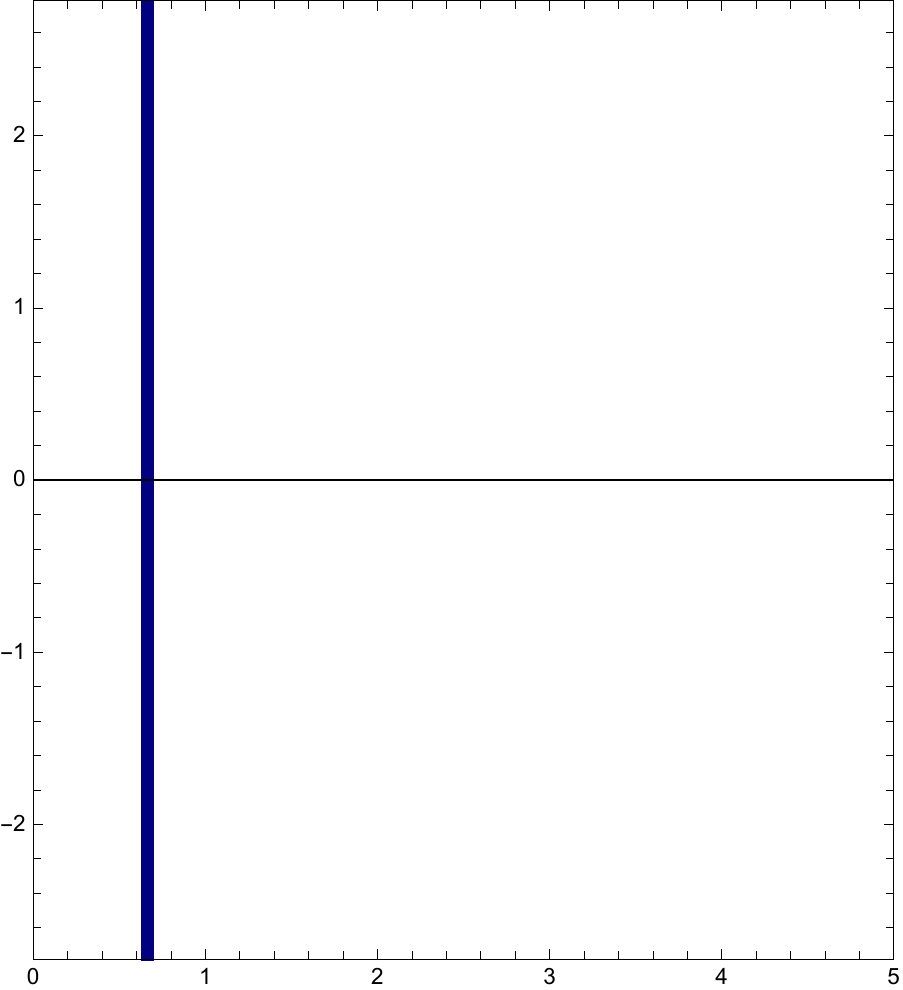} \\
		\includegraphics[width=\textwidth]{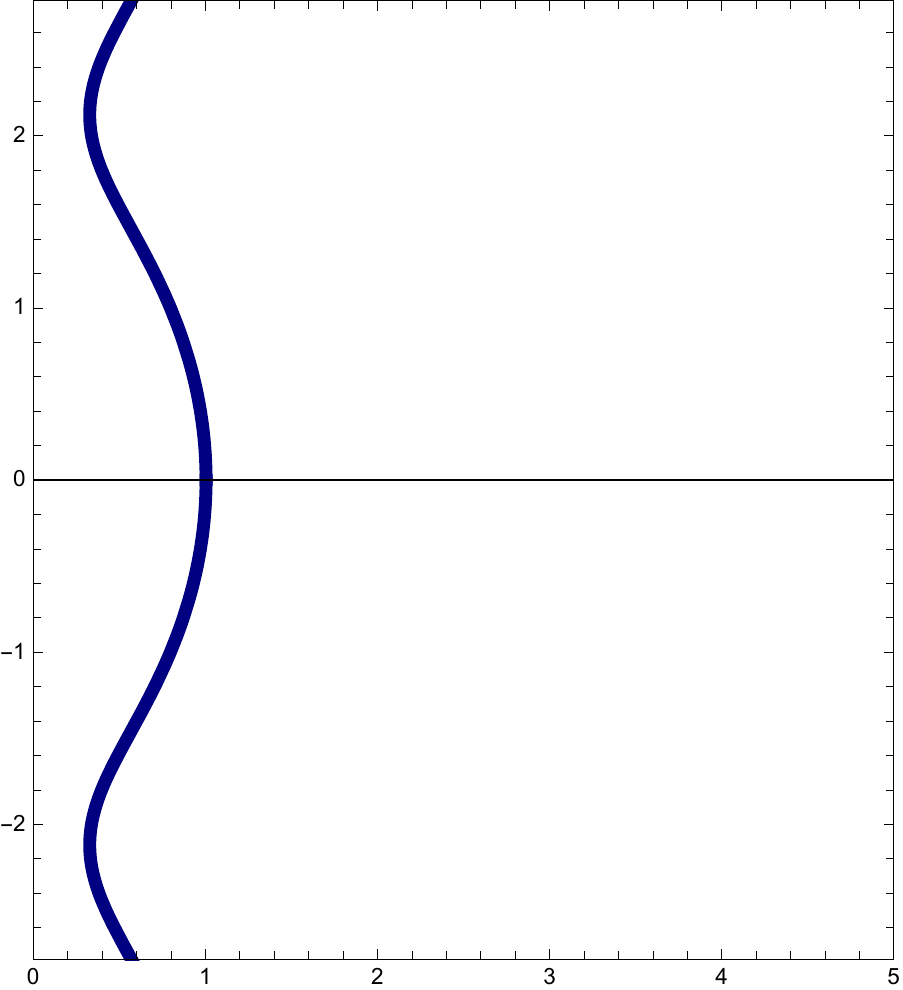} \\		\includegraphics[width=\textwidth]{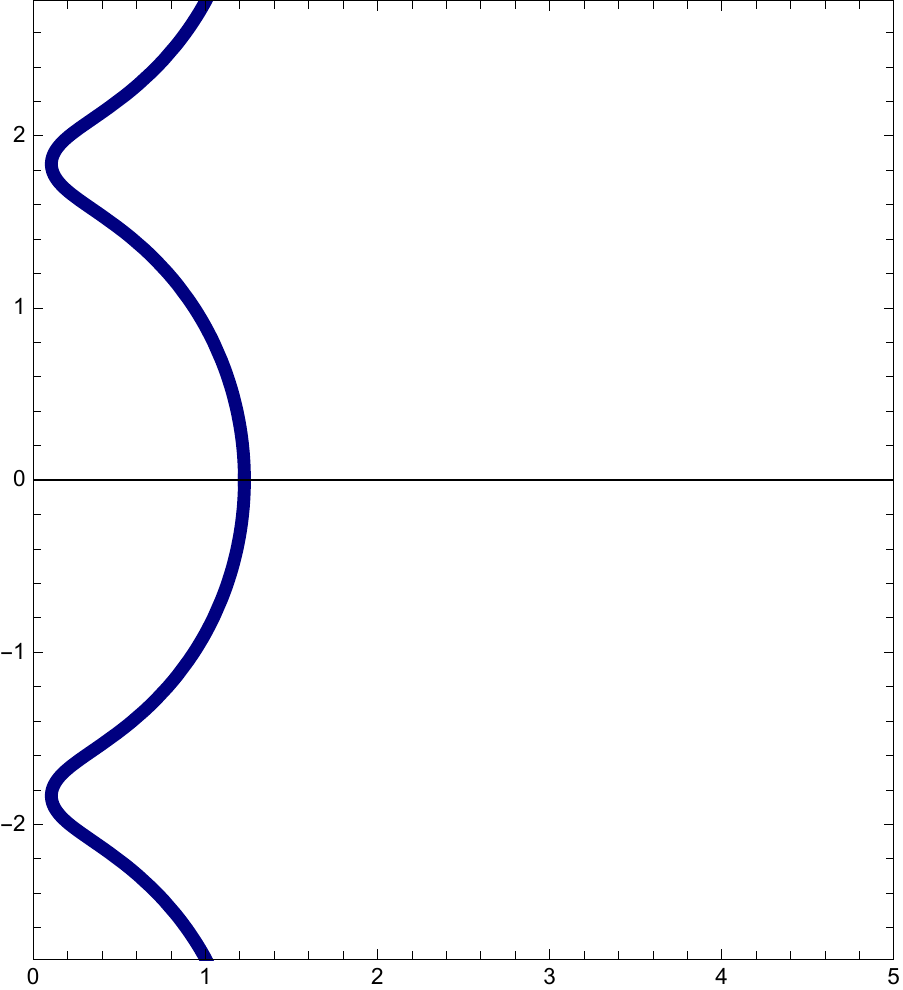} \\
		\includegraphics[width=\textwidth]{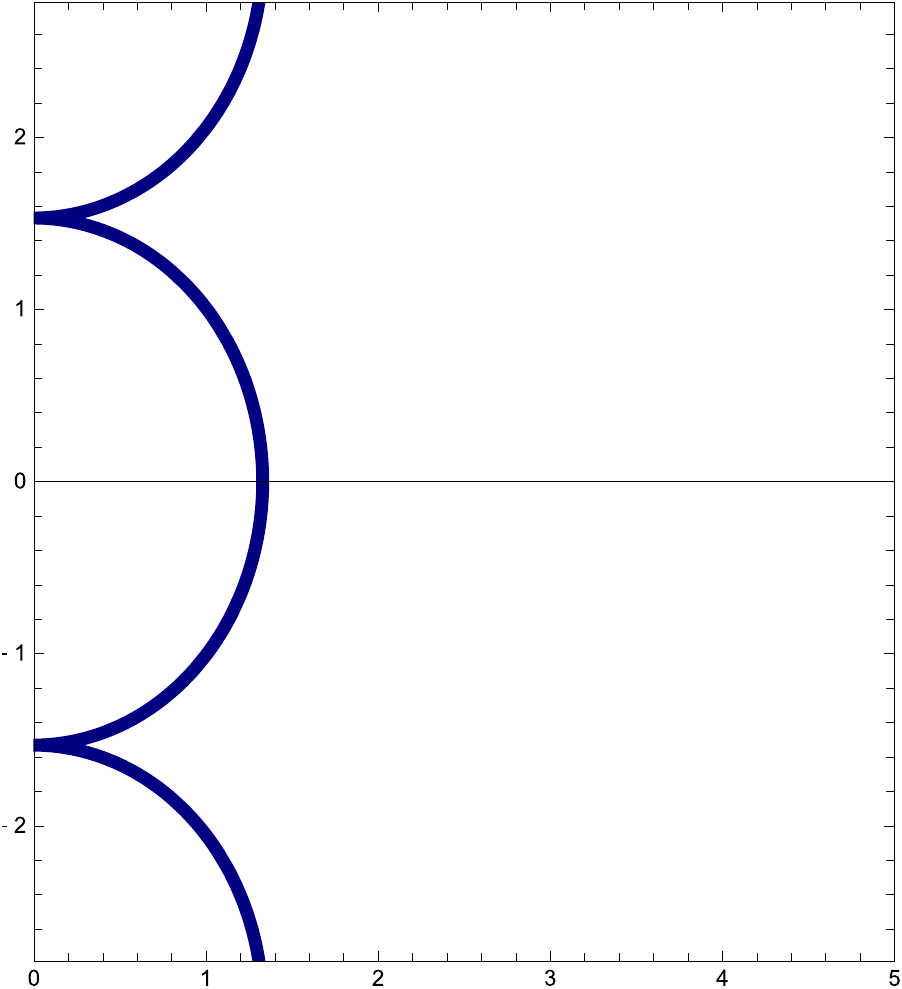} \\
		\includegraphics[width=\textwidth]{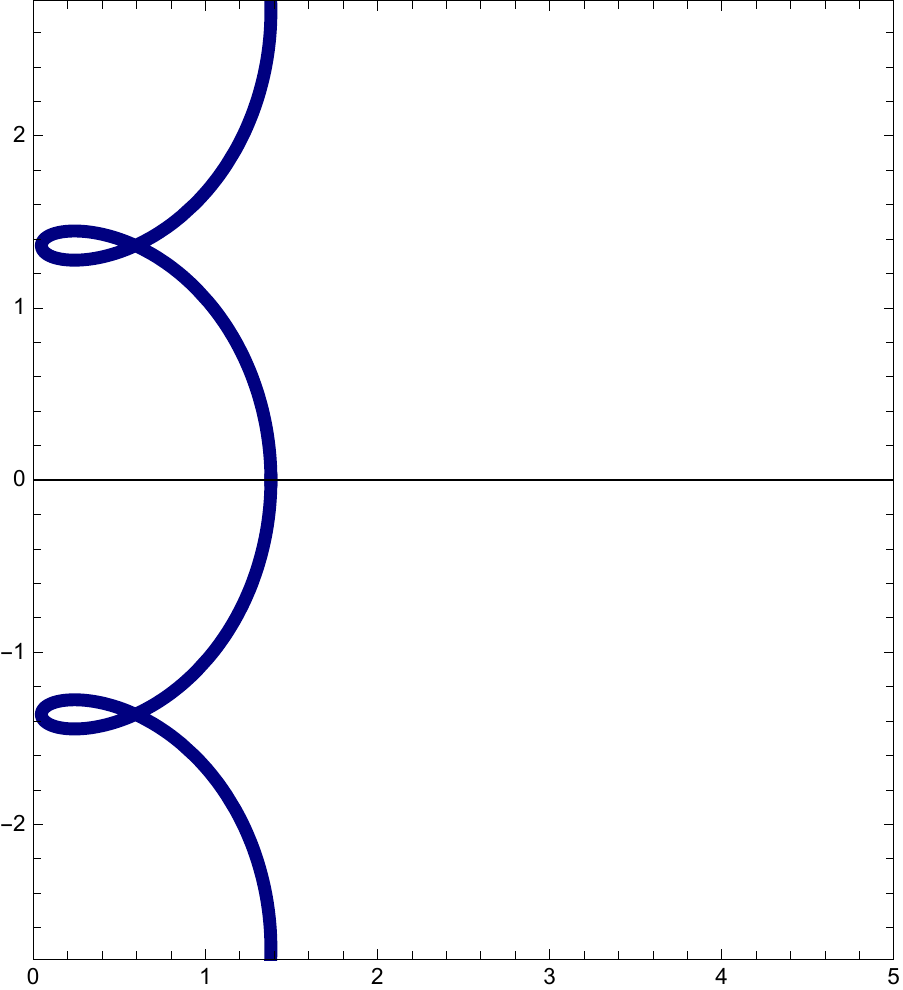} \\
		\includegraphics[width=\textwidth]{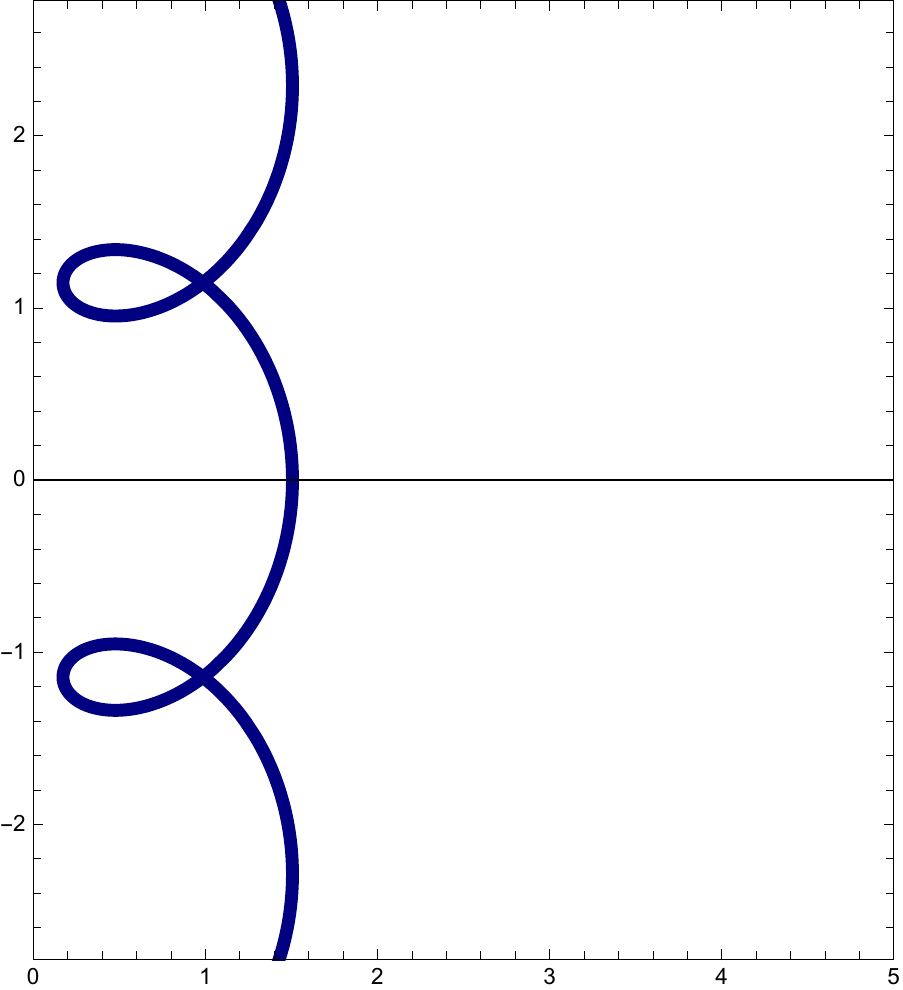} \\
		\includegraphics[width=\textwidth]{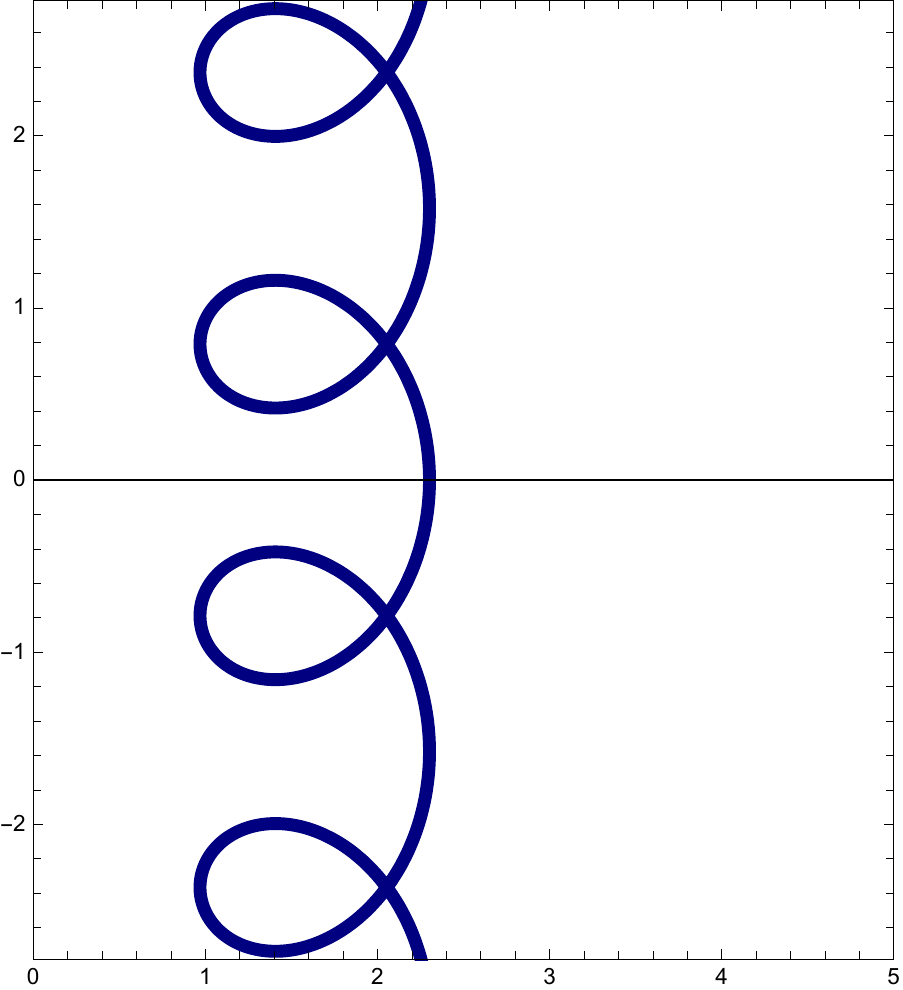}
	\end{minipage}
	\begin{minipage}[t]{0.16\textwidth}
		\includegraphics[width=\textwidth]{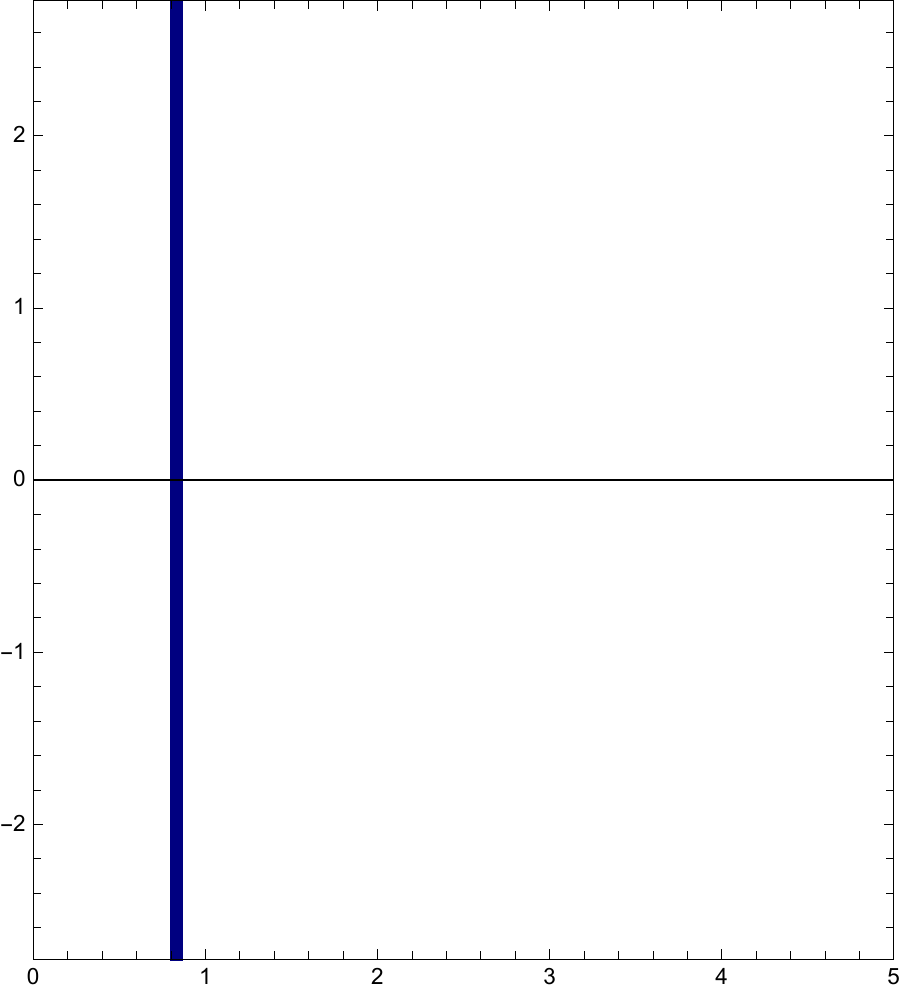} \\
		\includegraphics[width=\textwidth]{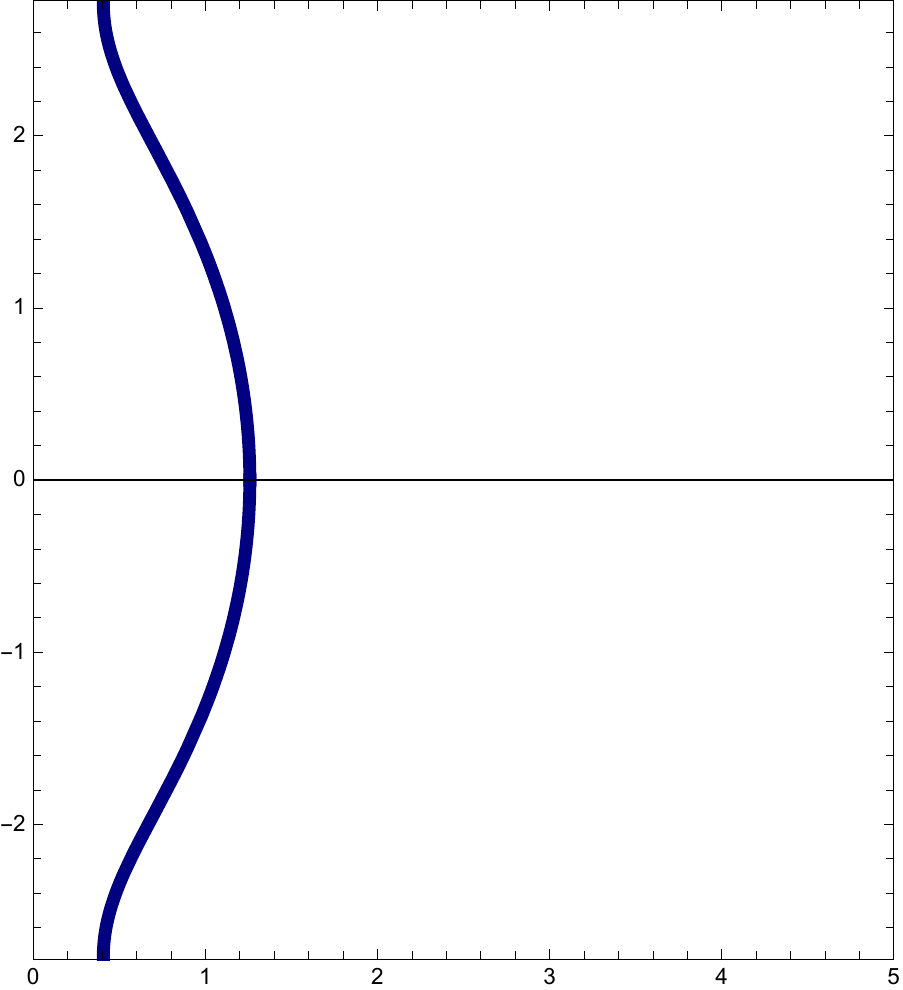} \\
		\includegraphics[width=\textwidth]{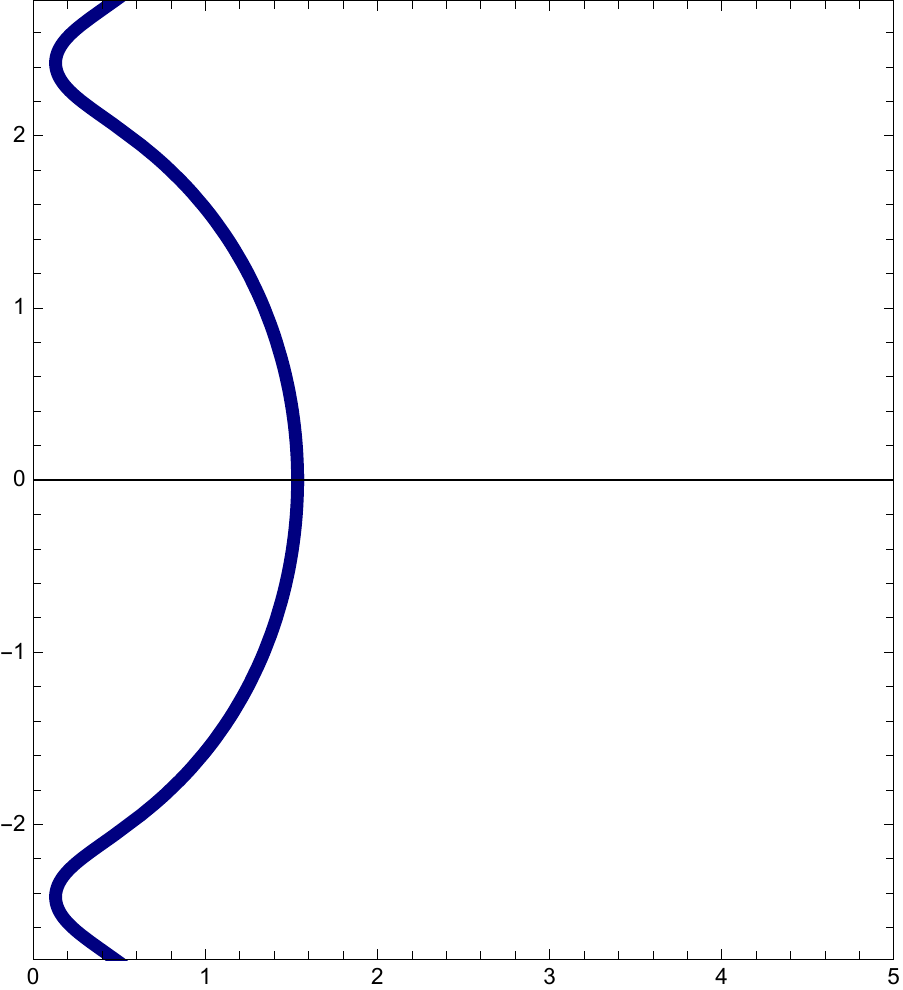} \\
		\includegraphics[width=\textwidth]{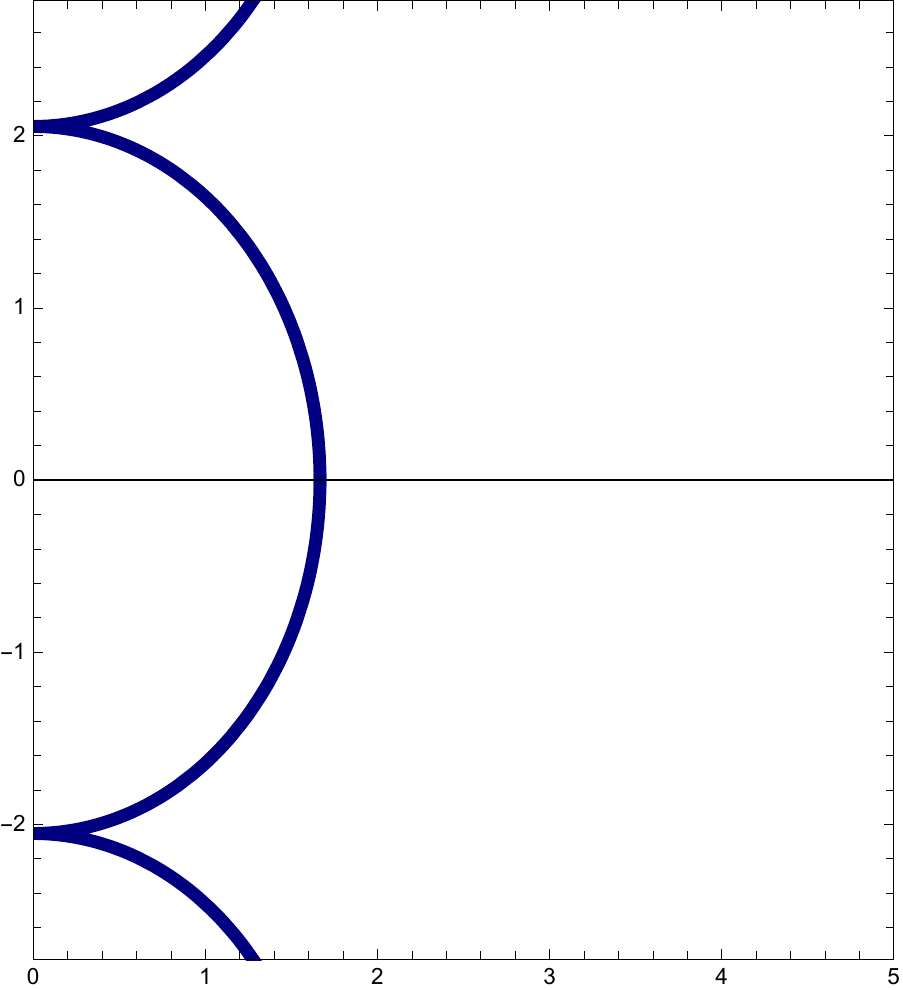} \\
		\includegraphics[width=\textwidth]{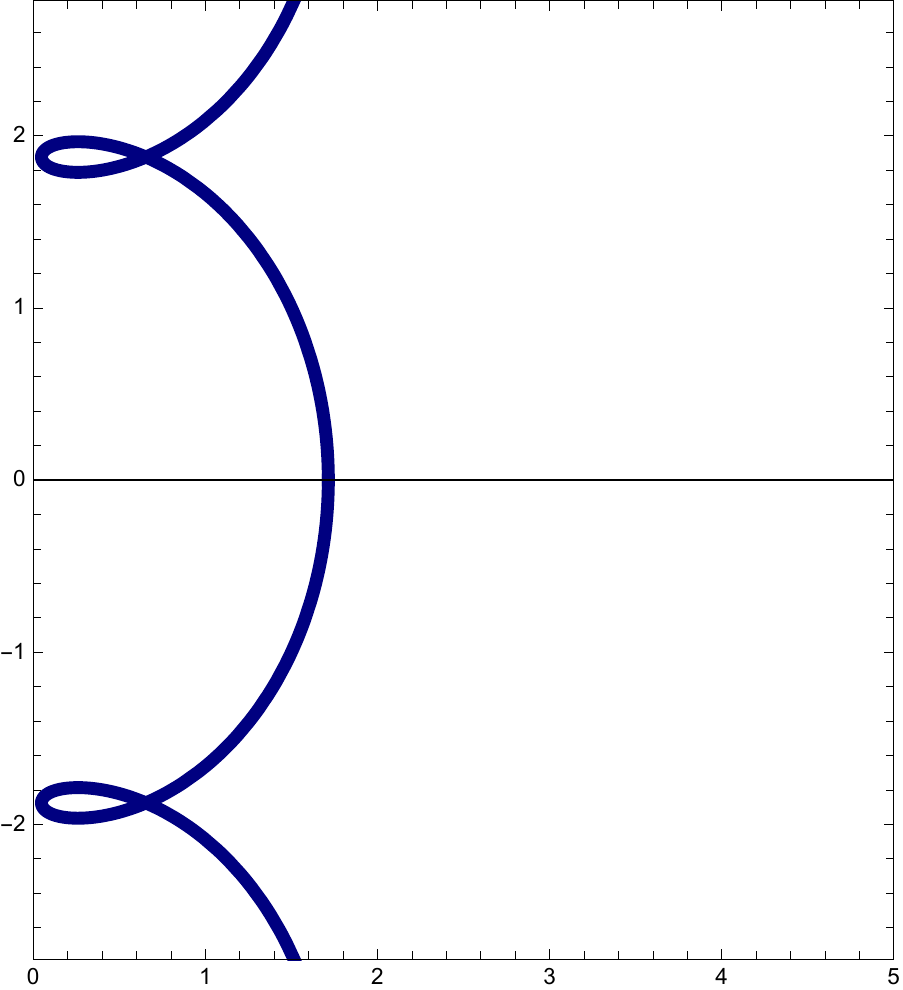} \\
		\includegraphics[width=\textwidth]{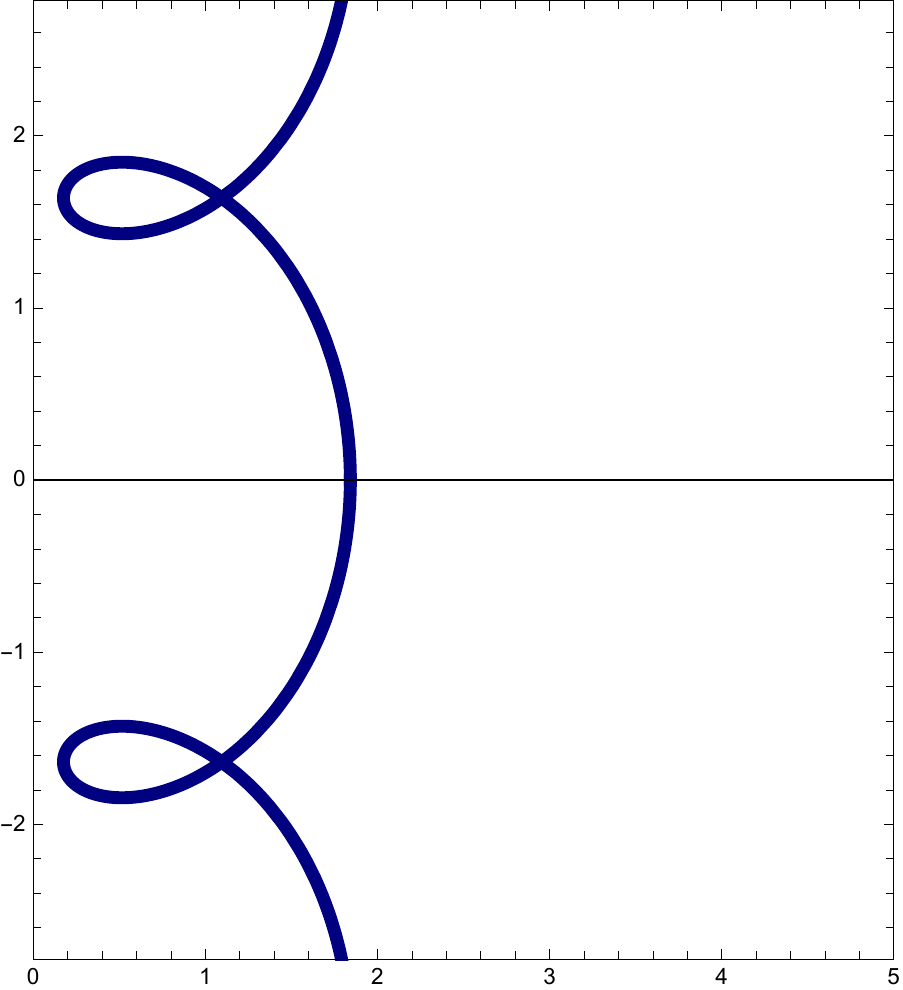} \\
		\includegraphics[width=\textwidth]{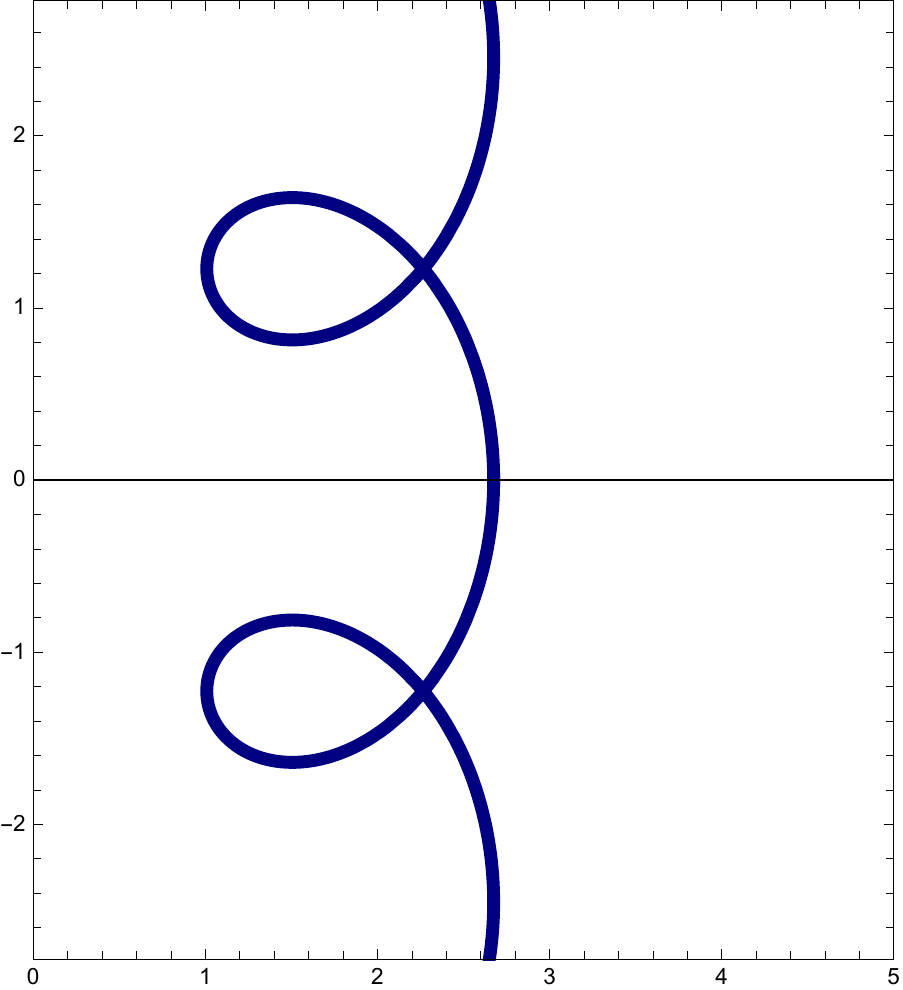}
	\end{minipage}
	\begin{minipage}[t]{0.16\textwidth}
		\includegraphics[width=\textwidth]{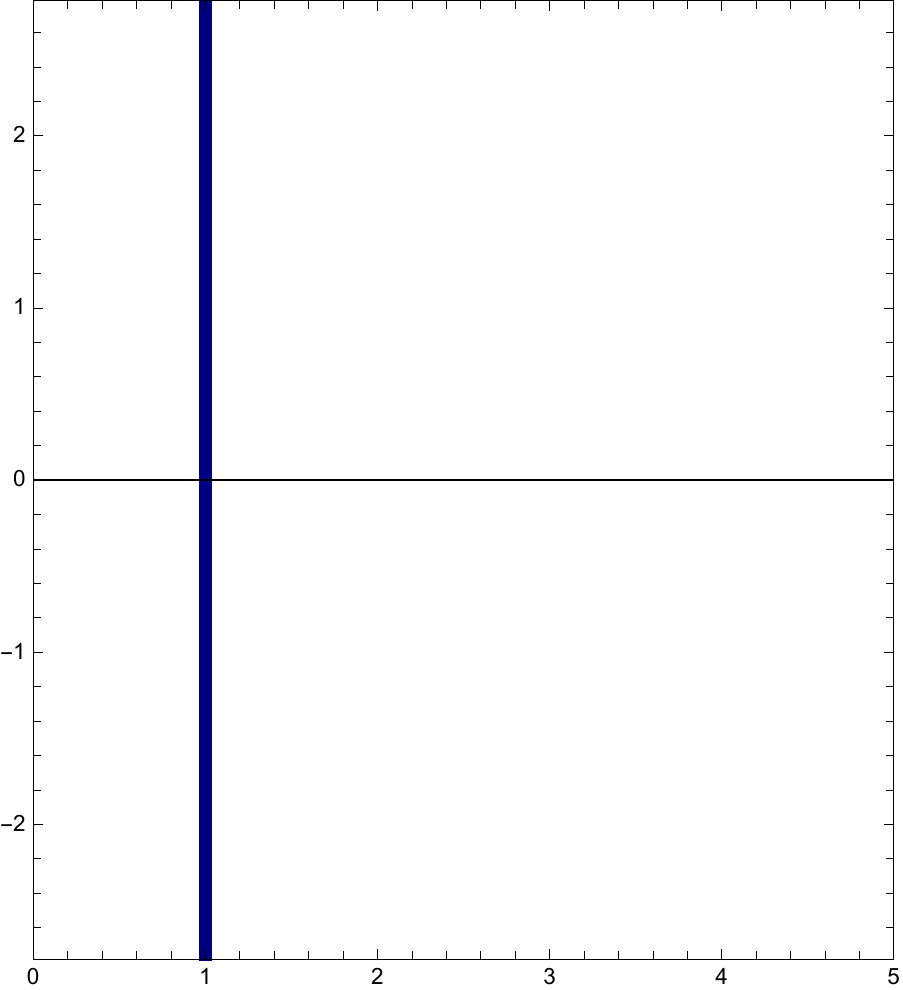} \\
		\includegraphics[width=\textwidth]{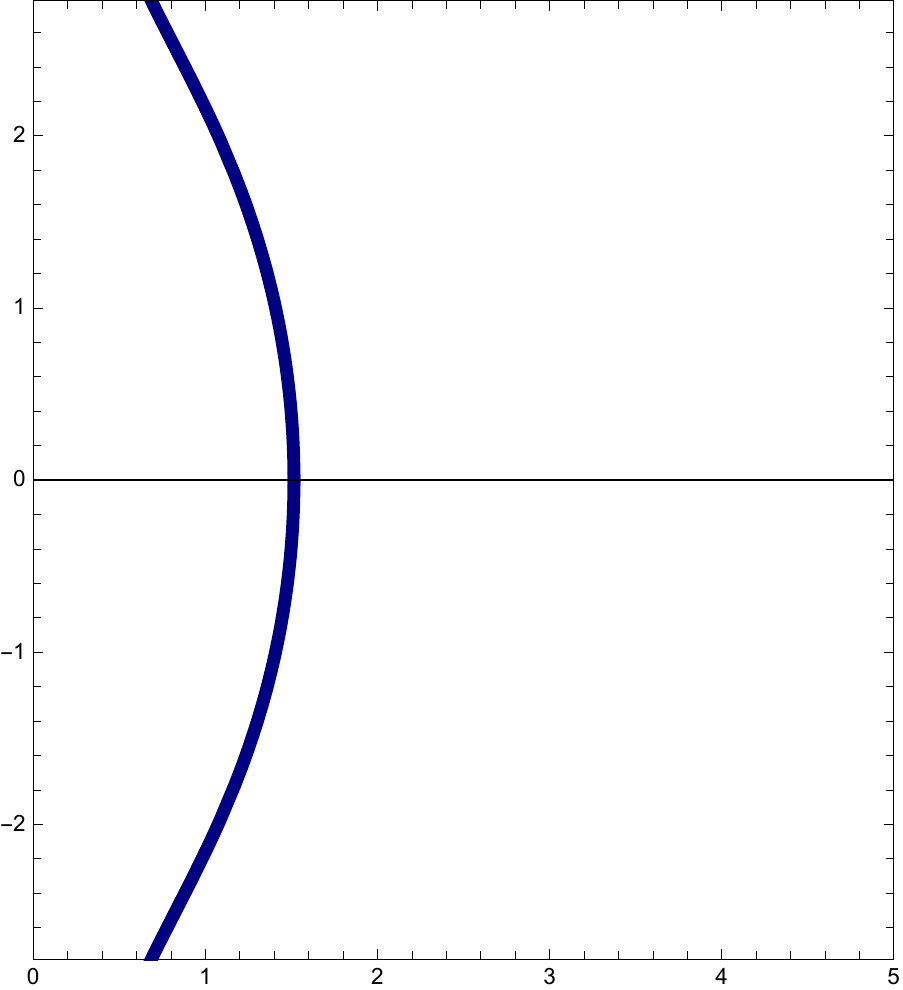} \\
		\includegraphics[width=\textwidth]{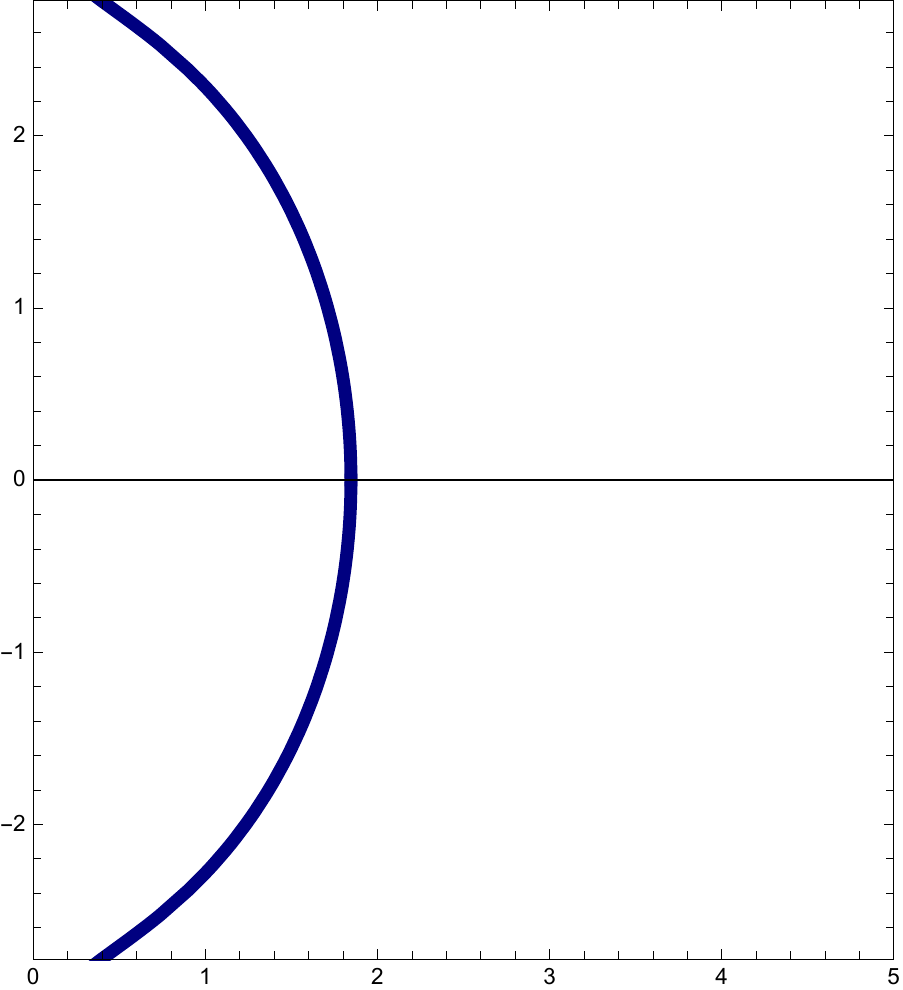} \\
		\includegraphics[width=\textwidth]{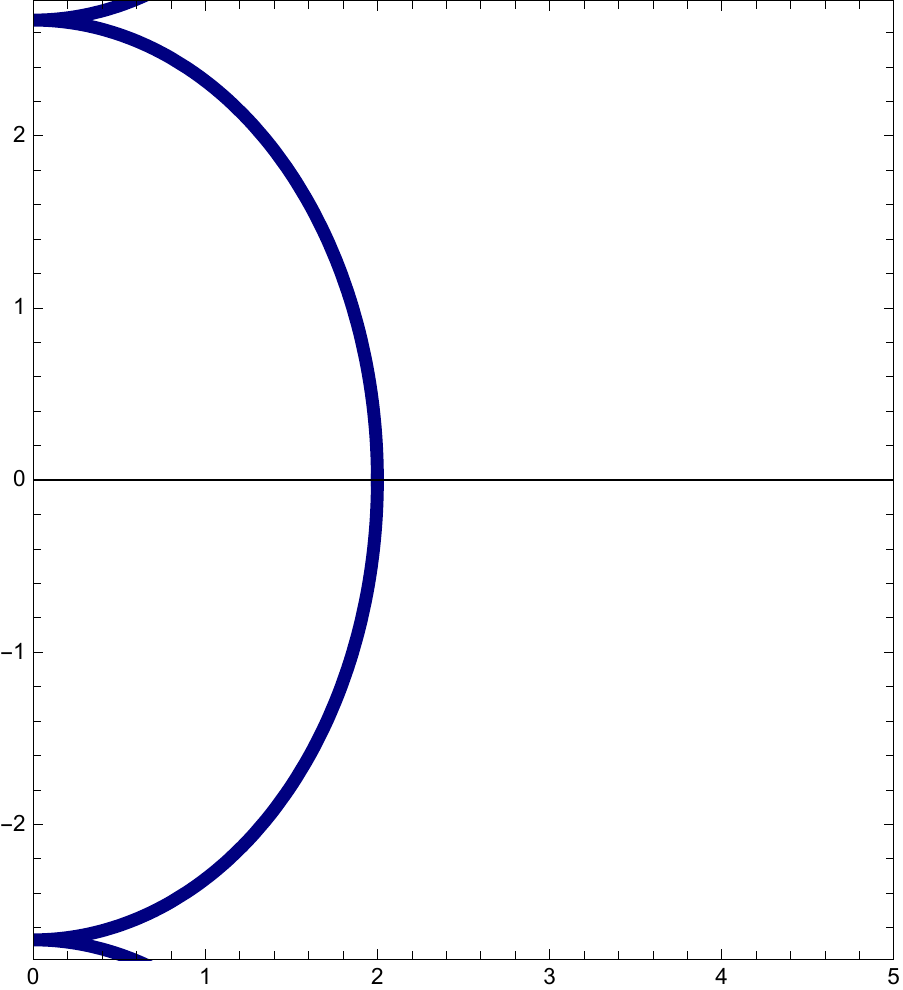} \\
		\includegraphics[width=\textwidth]{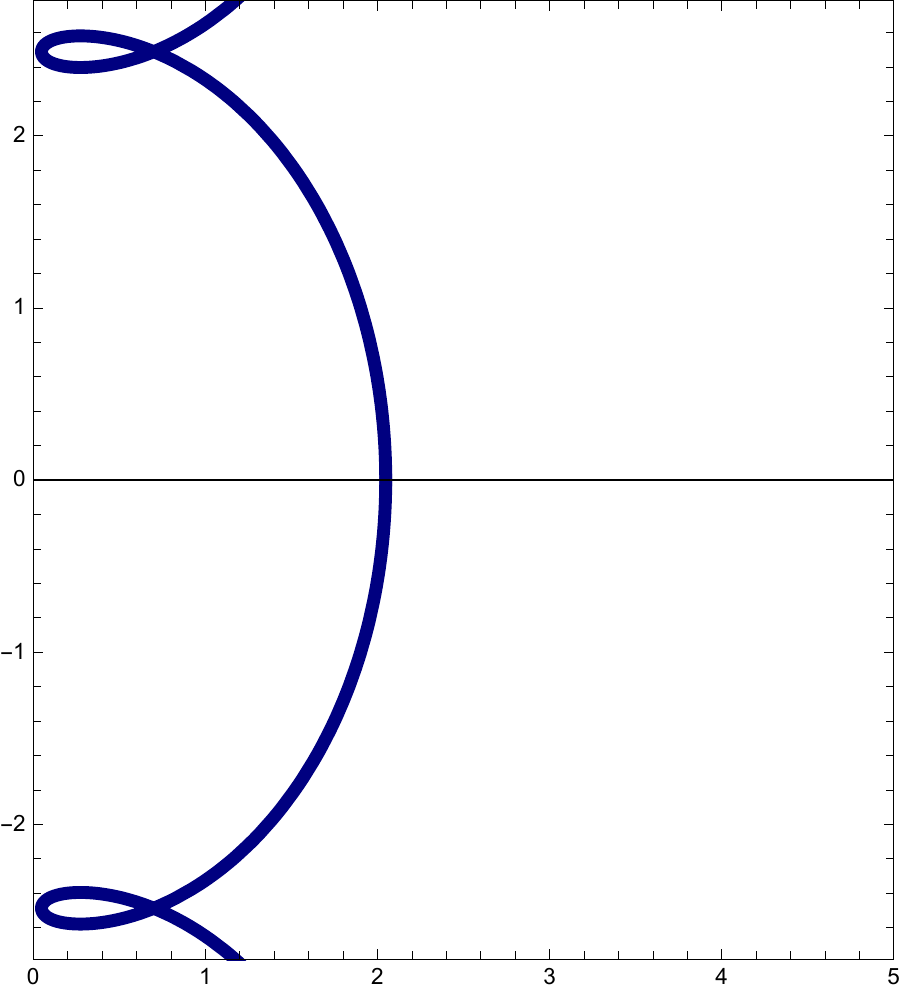} \\
		\includegraphics[width=\textwidth]{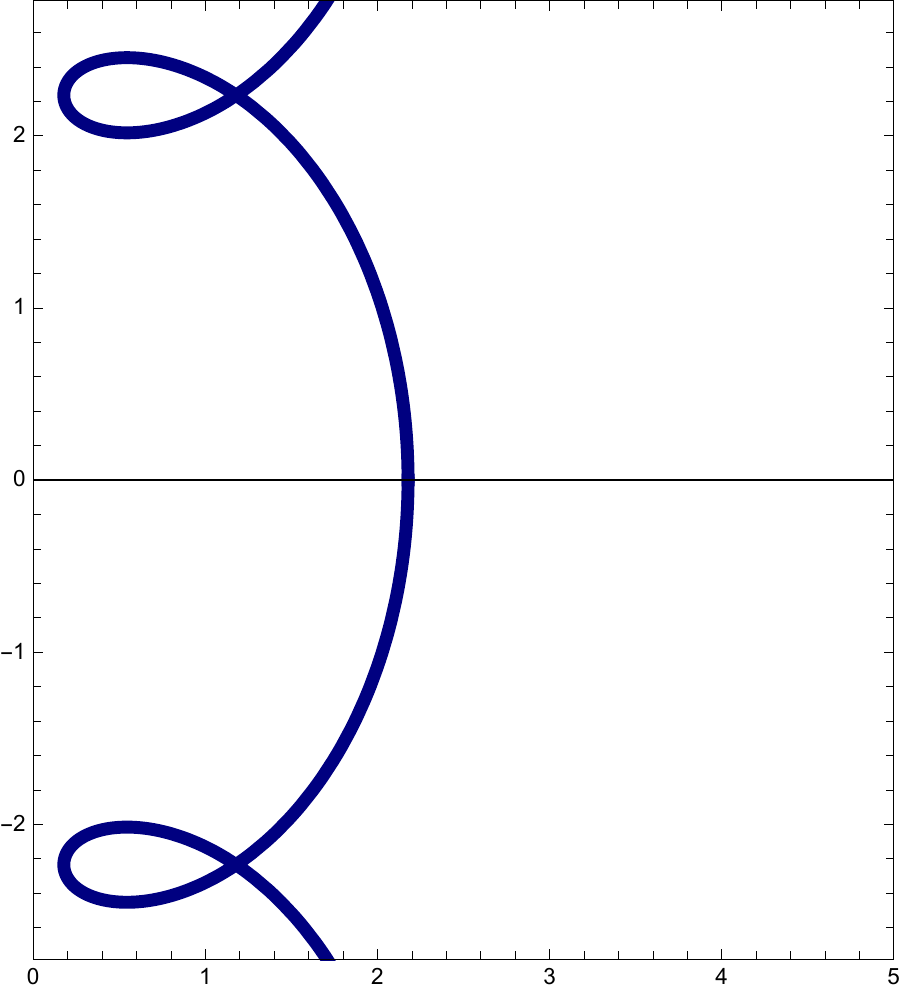} \\
		\includegraphics[width=\textwidth]{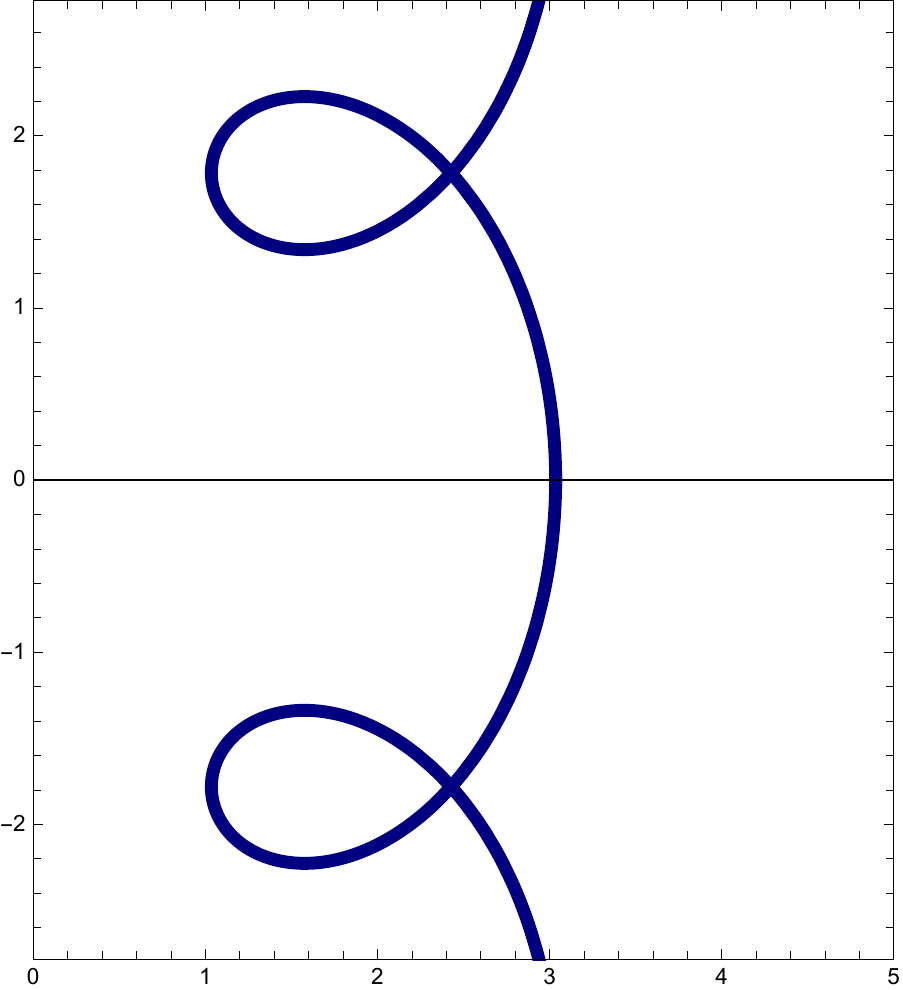}
	\end{minipage}
	\caption{Numerically computed profile curves of the family $\family$ in $\HR$ for the rotational case $a=0$. The energy decreases from $\Jmax$ in the first row (vertical cylinder) to $J<0$ in the lower rows. The intermediate rows display surfaces of unduloid type, sphere type, and nodoid type. Tubes do not appear in this family. Each column represents a fixed value of the mean curvature $H$, from large $H$ in the left column to small $H>\frac{1}{2}$ in the right column.}
	\label{fig:curves_hkr}
\end{figure}

\newpage

\begin{figure}[H]
	\centering
	\begin{minipage}[t]{0.16\textwidth}
		\includegraphics[width=\textwidth]{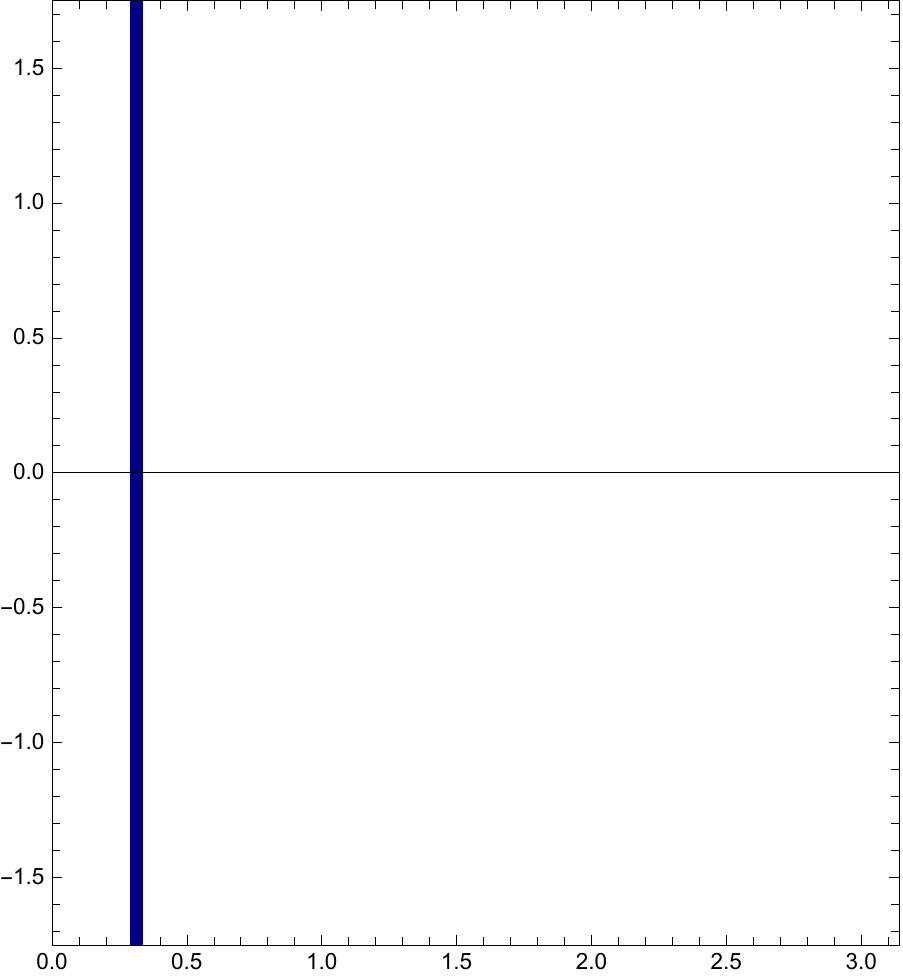} \\
		\includegraphics[width=\textwidth]{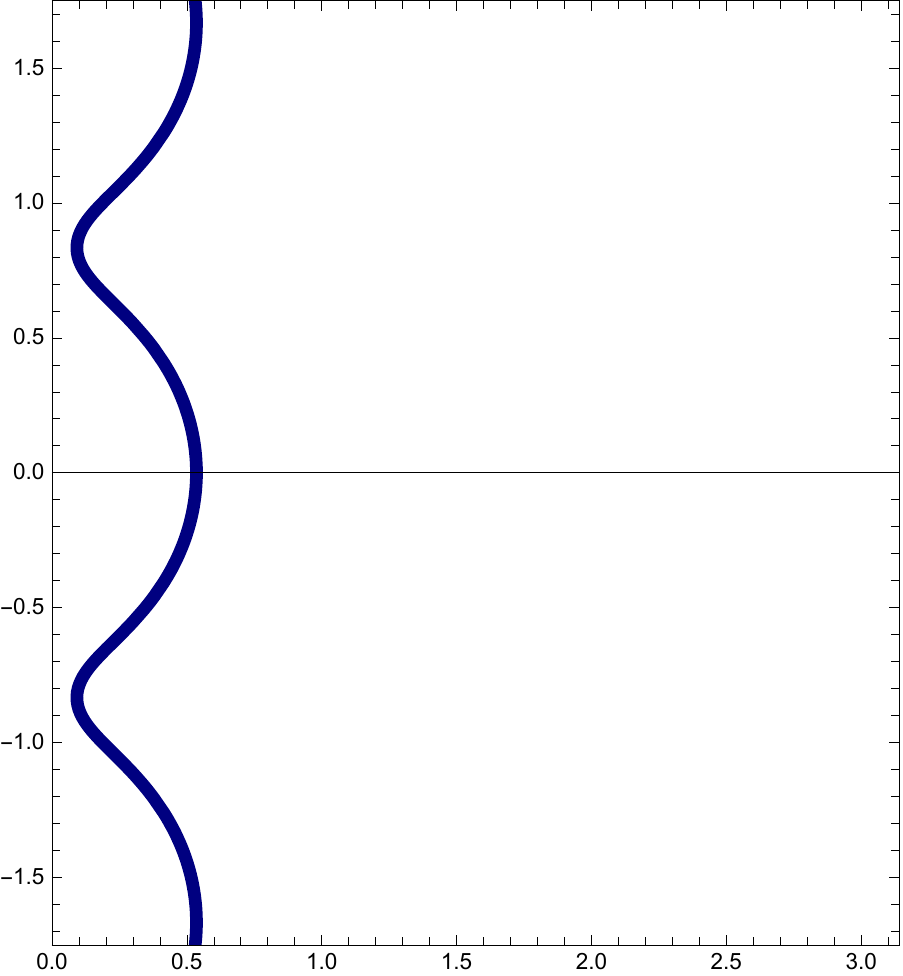} \\
		\includegraphics[width=\textwidth]{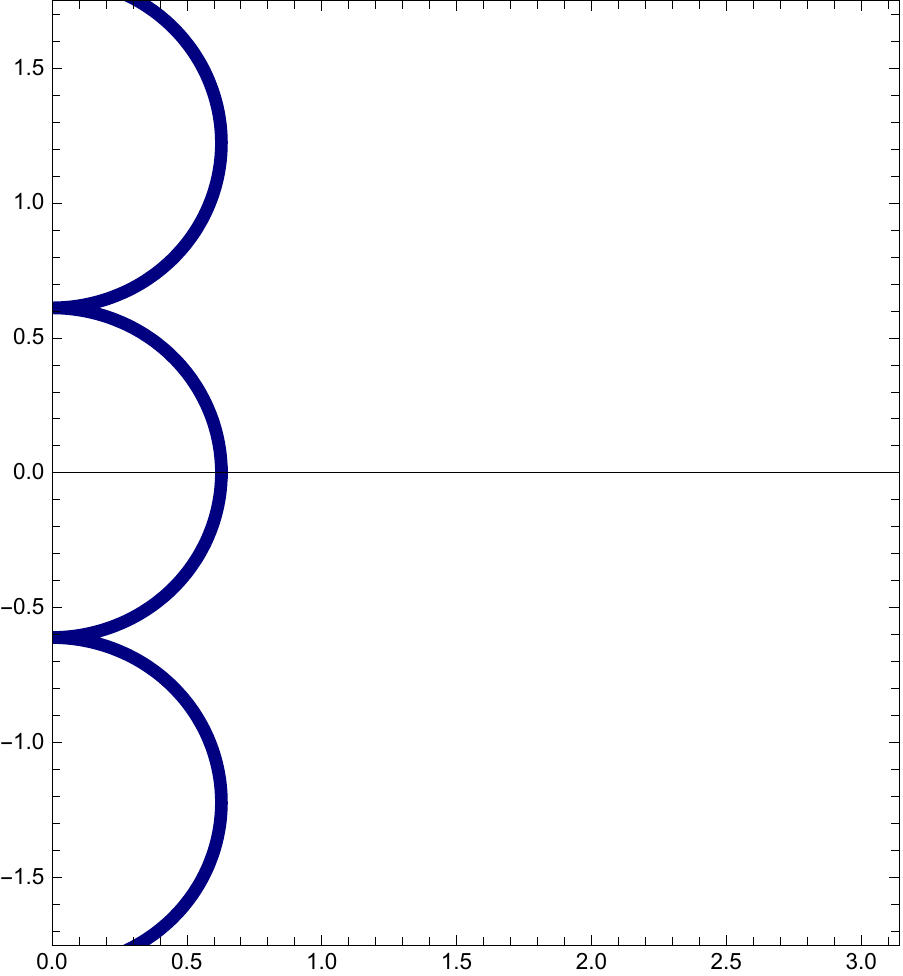} \\
		\includegraphics[width=\textwidth]{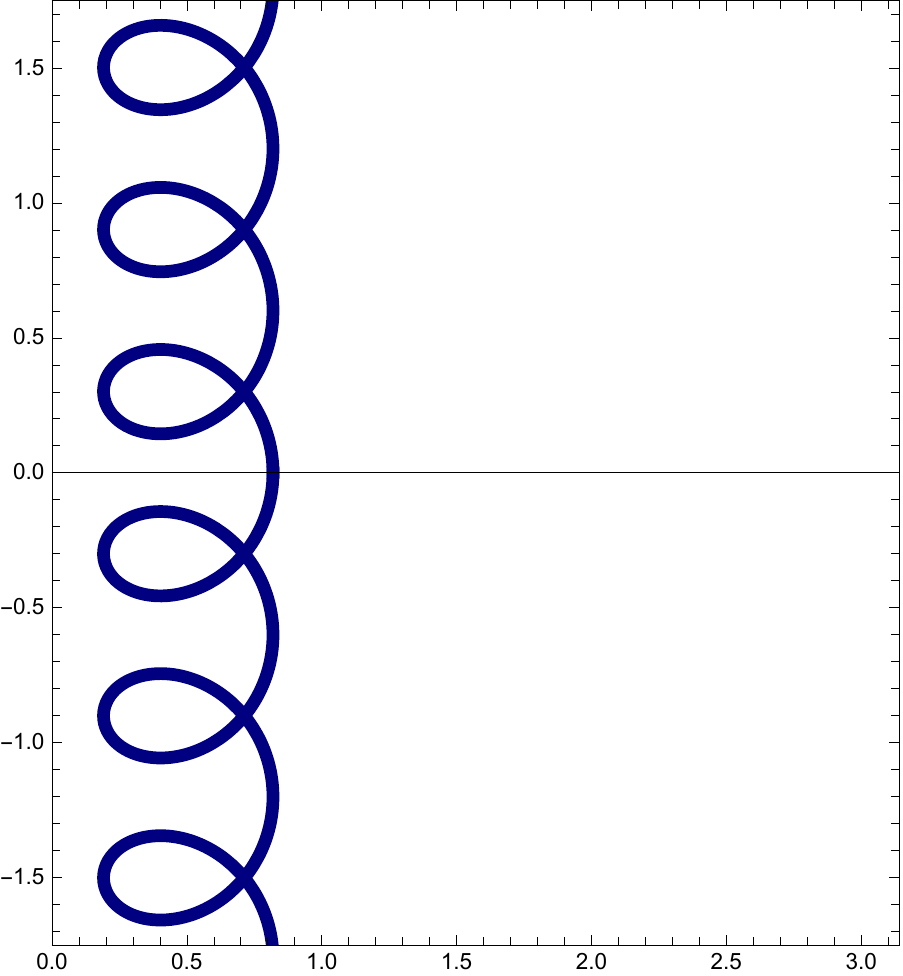} \\
		\includegraphics[width=\textwidth]{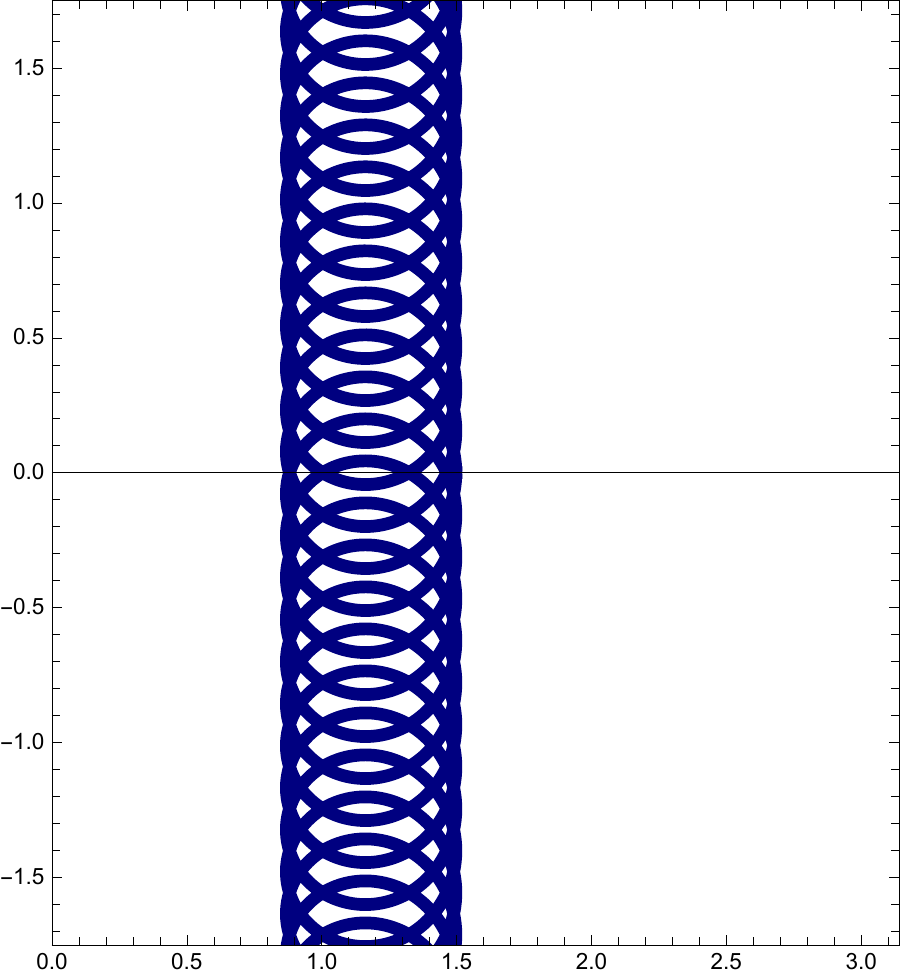} \\
		\includegraphics[width=\textwidth]{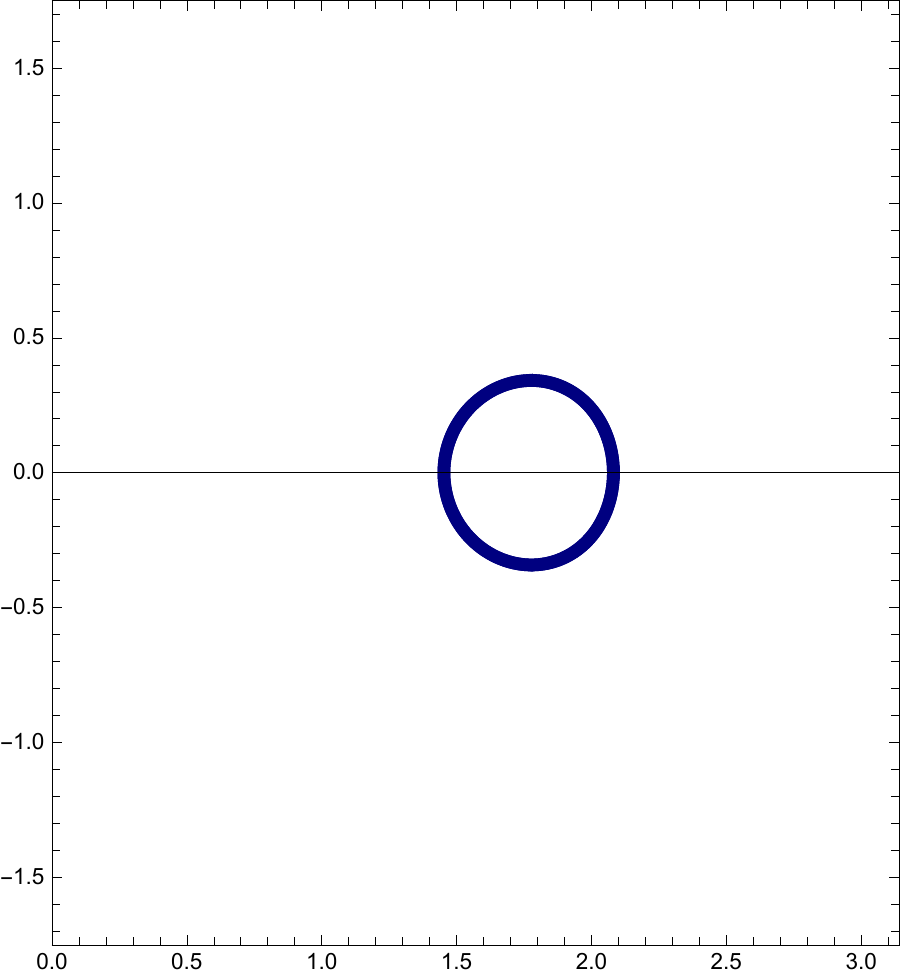} \\
		\includegraphics[width=\textwidth]{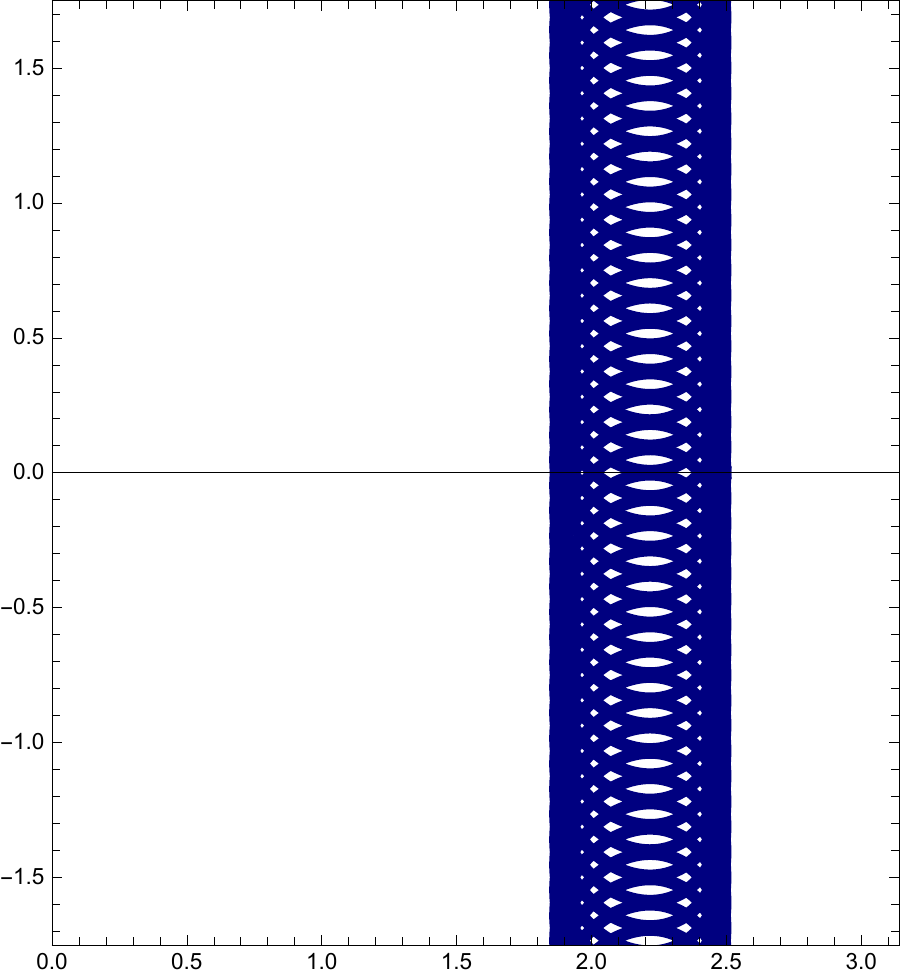}
	\end{minipage}
	\begin{minipage}[t]{0.16\textwidth}
		\includegraphics[width=\textwidth]{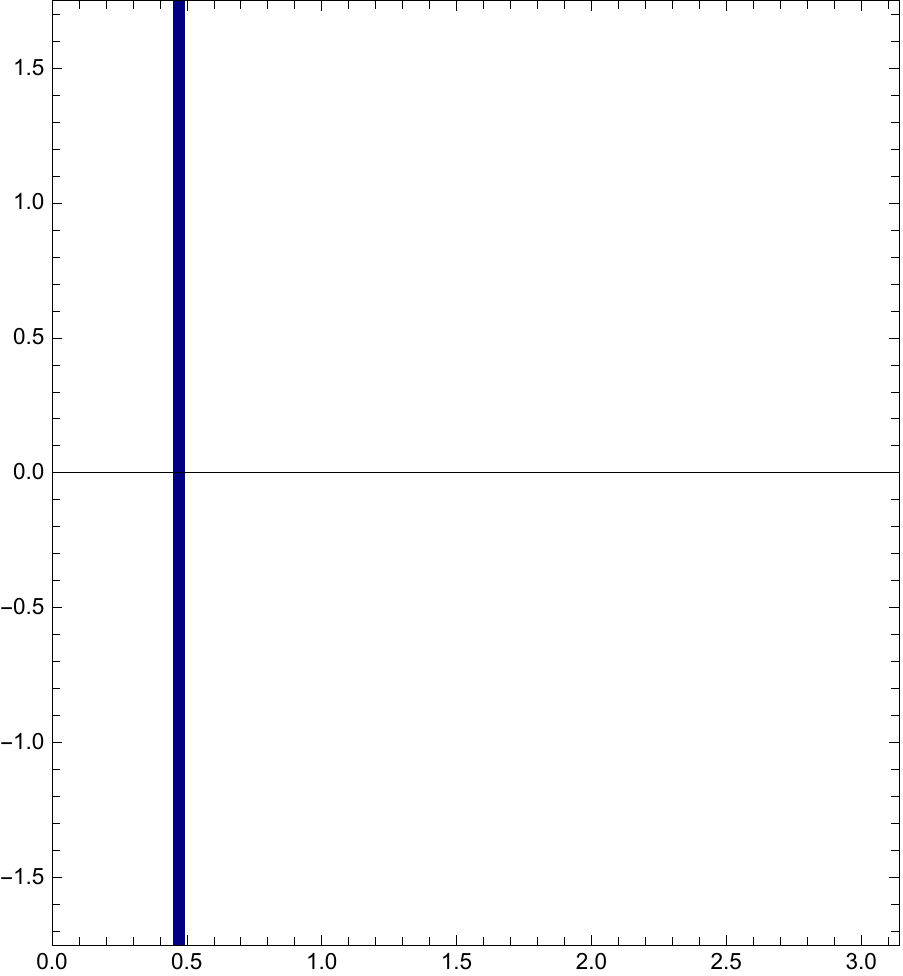} \\
		\includegraphics[width=\textwidth]{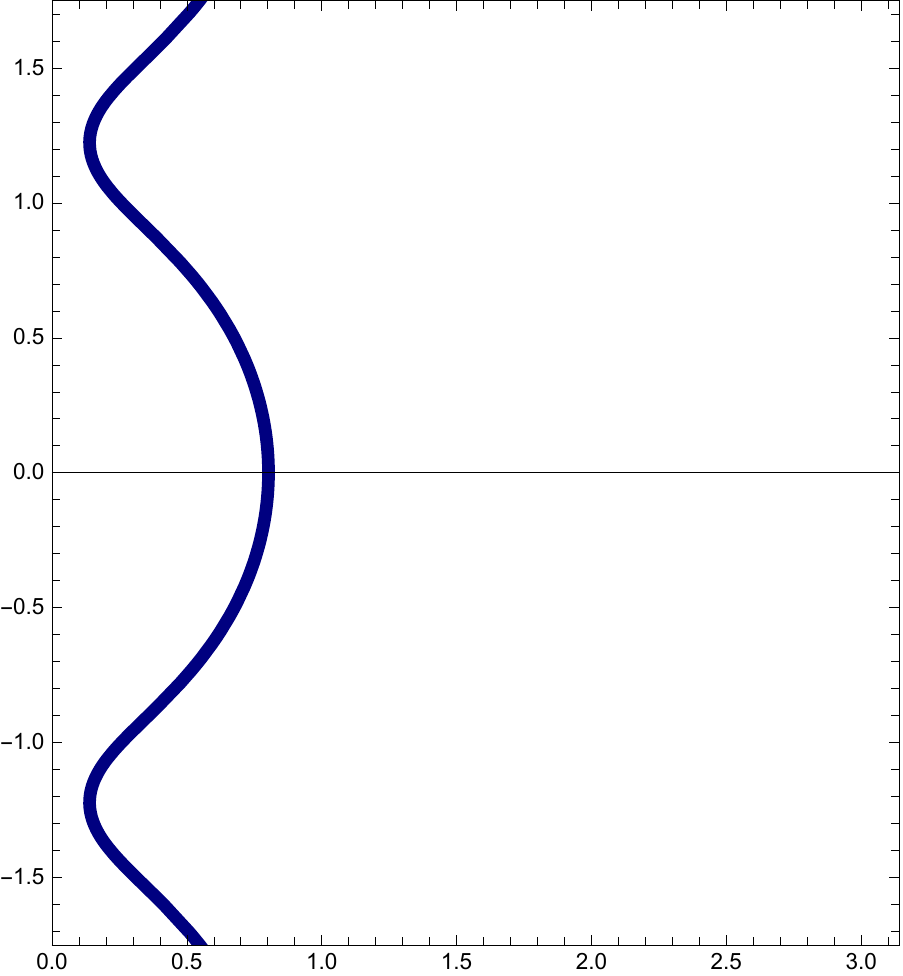} \\
		\includegraphics[width=\textwidth]{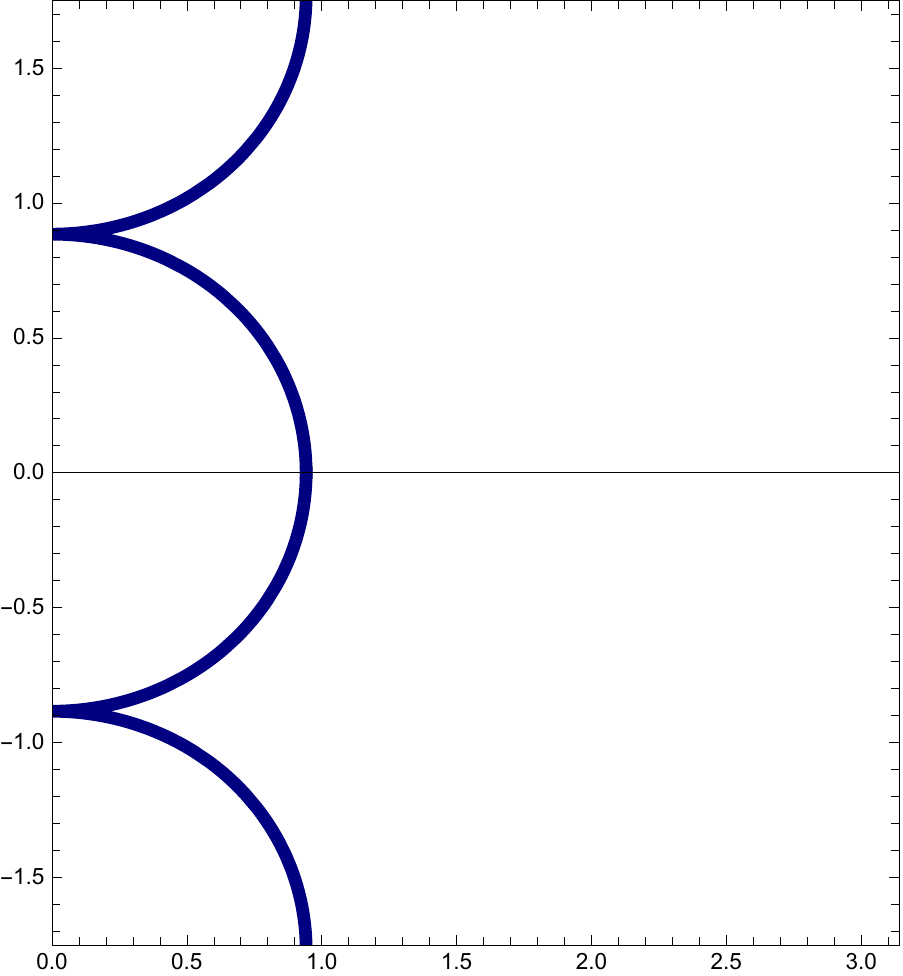} \\
		\includegraphics[width=\textwidth]{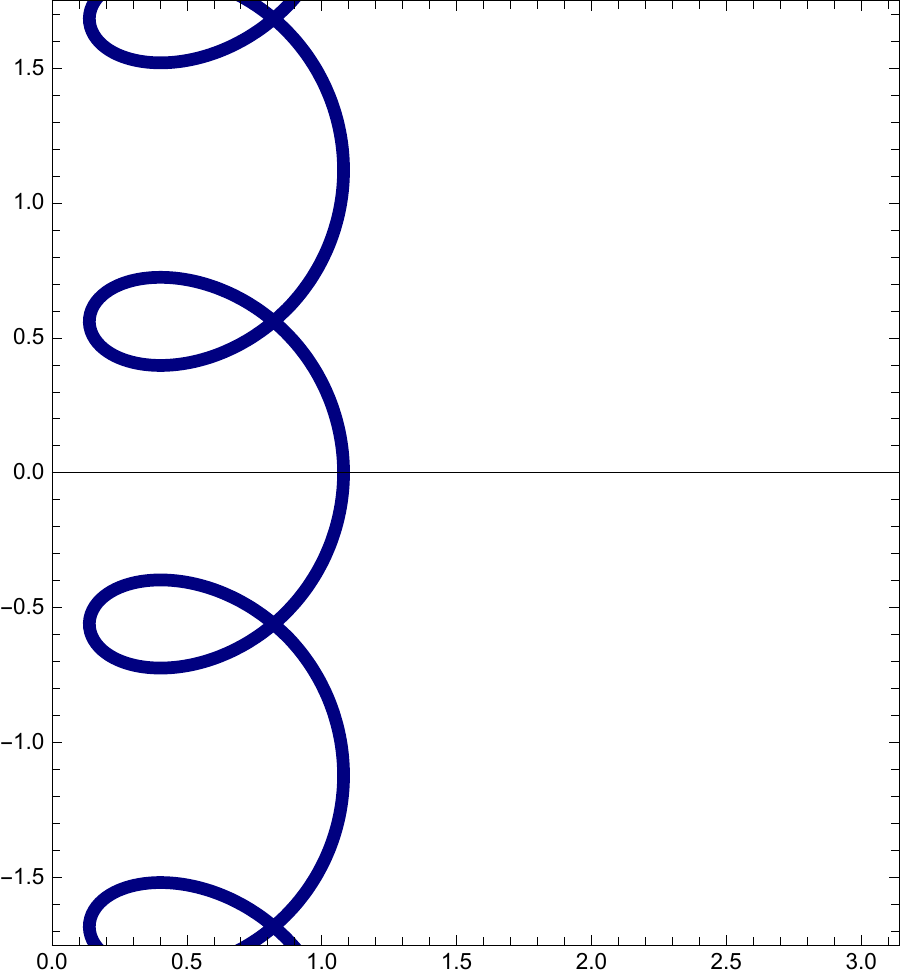} \\
		\includegraphics[width=\textwidth]{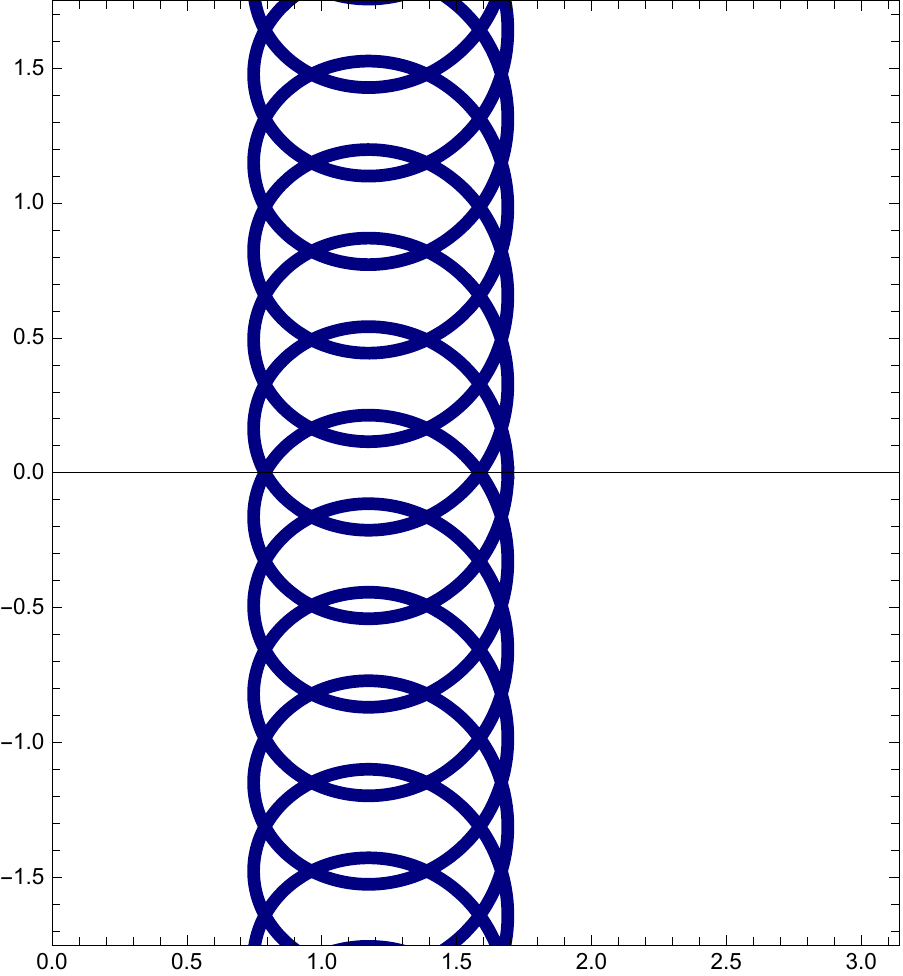} \\
		\includegraphics[width=\textwidth]{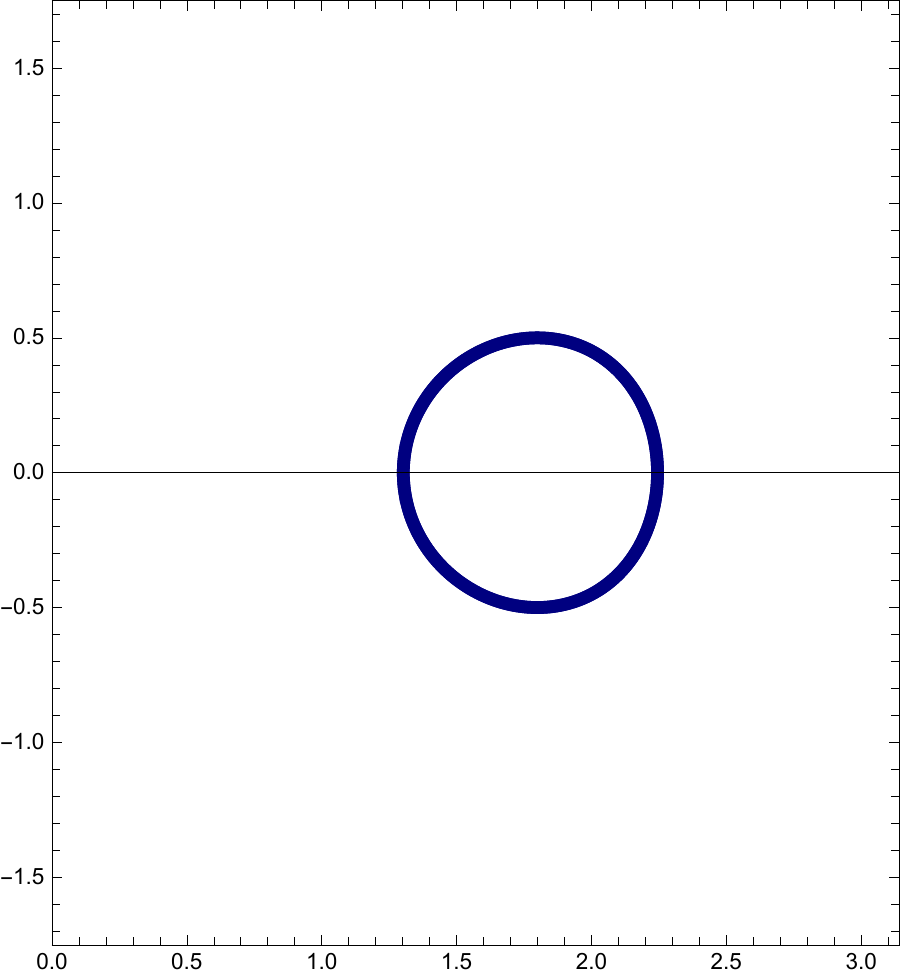} \\
		\includegraphics[width=\textwidth]{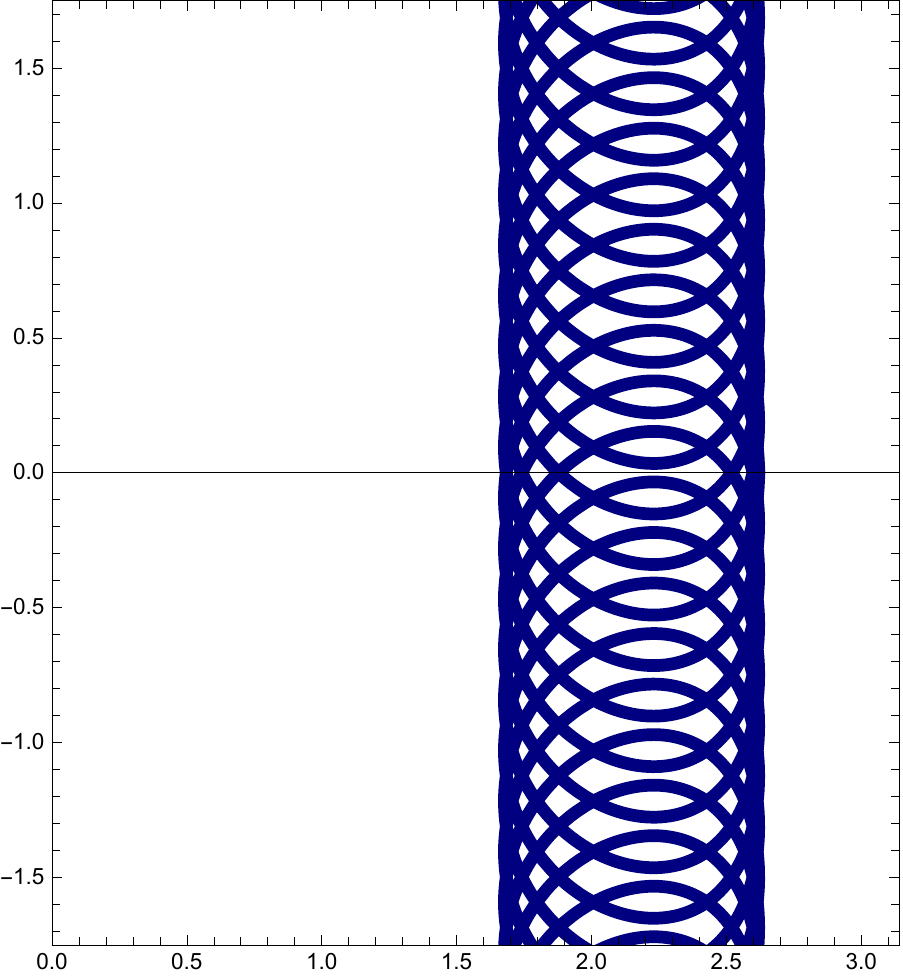}
	\end{minipage}
	\begin{minipage}[t]{0.16\textwidth}
		\includegraphics[width=\textwidth]{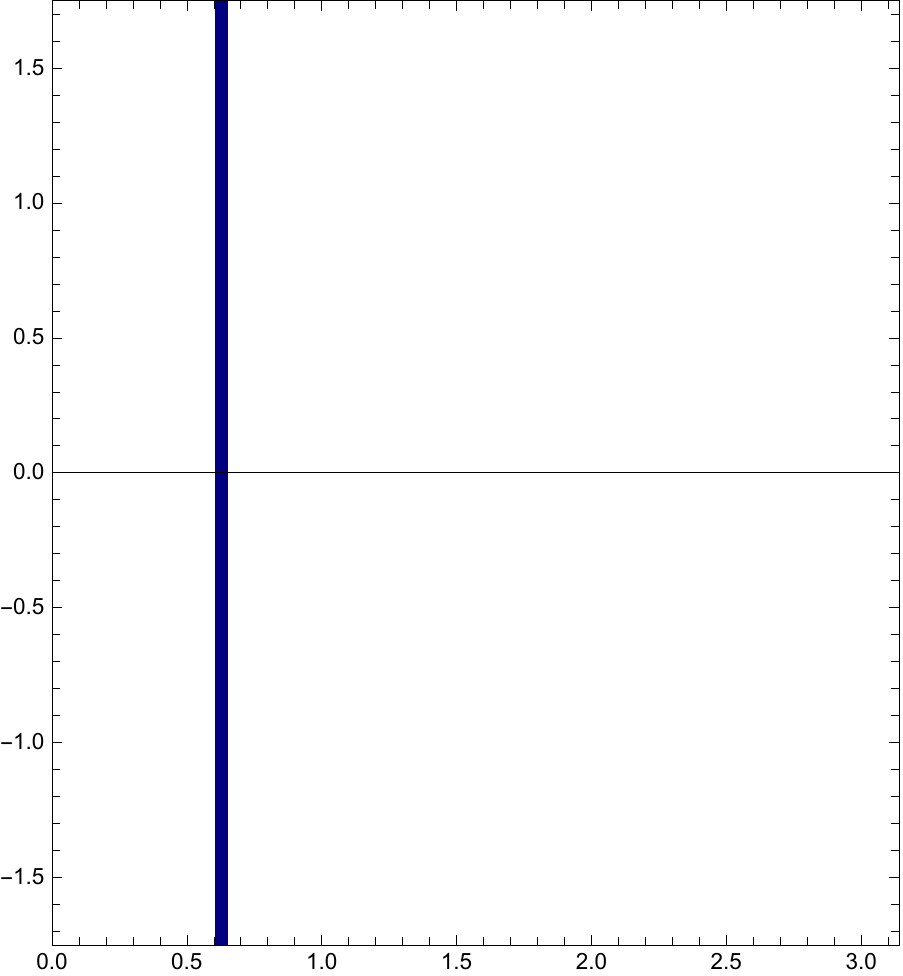} \\
		\includegraphics[width=\textwidth]{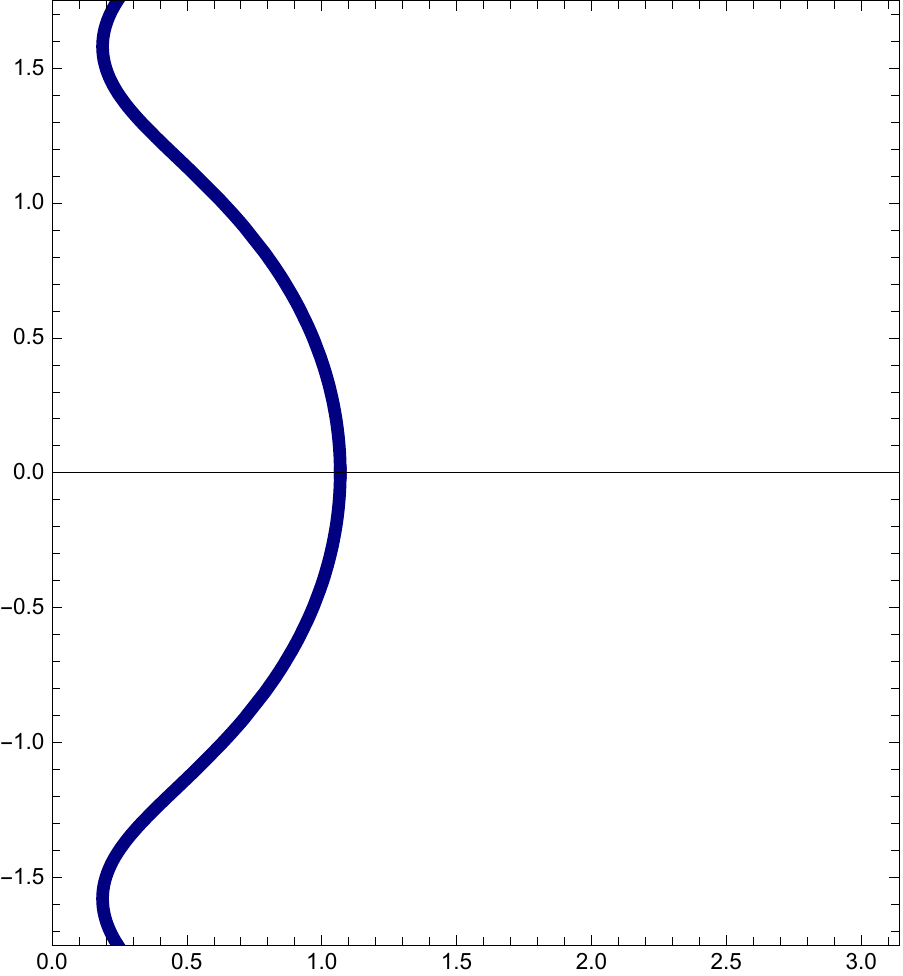}
		\includegraphics[width=\textwidth]{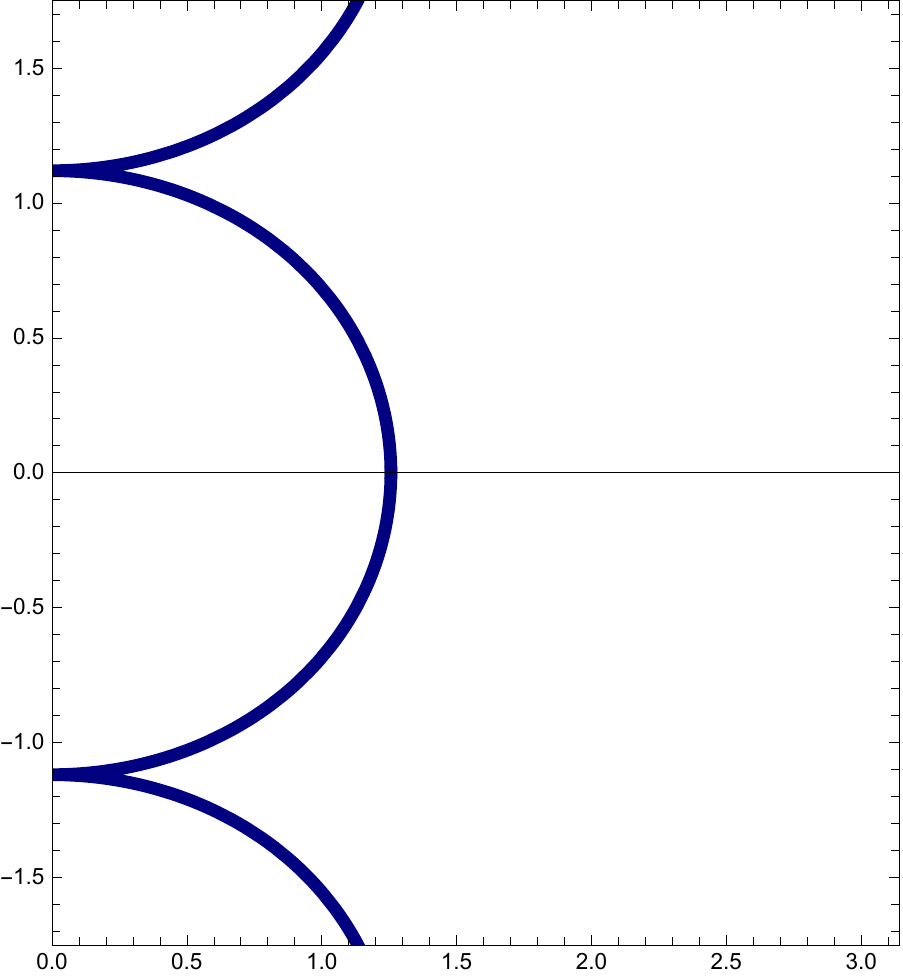} \\
		\includegraphics[width=\textwidth]{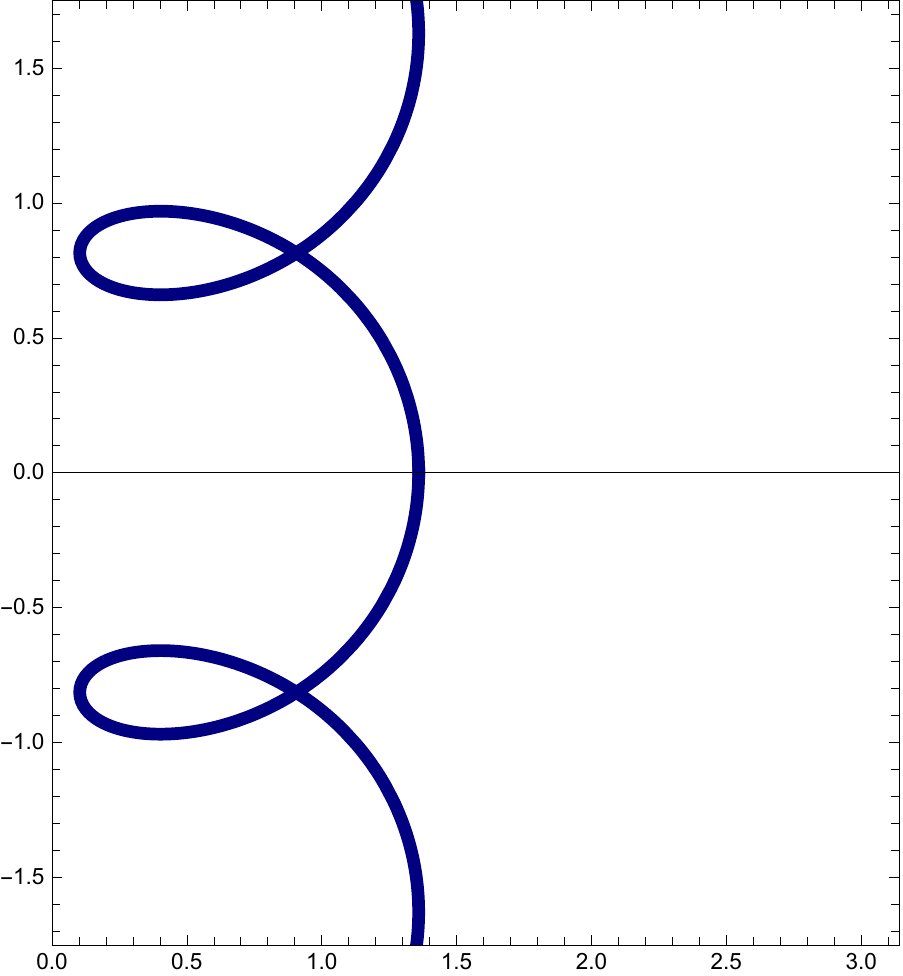} \\
		\includegraphics[width=\textwidth]{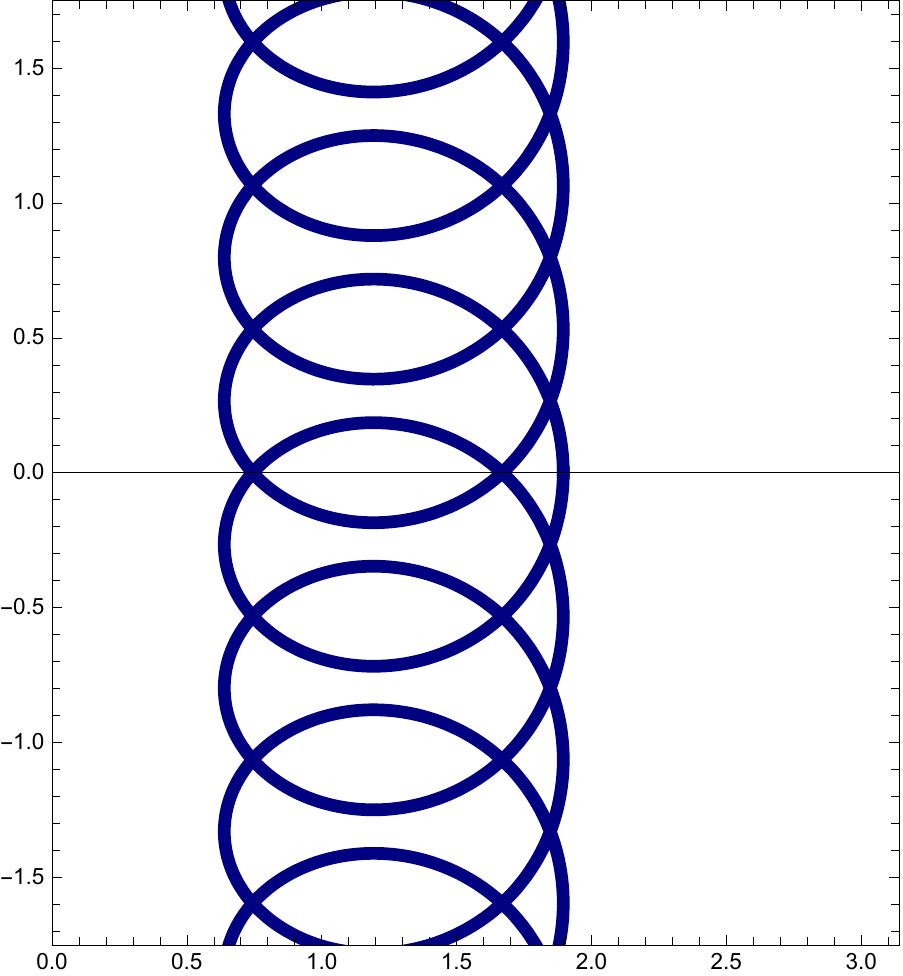} \\
		\includegraphics[width=\textwidth]{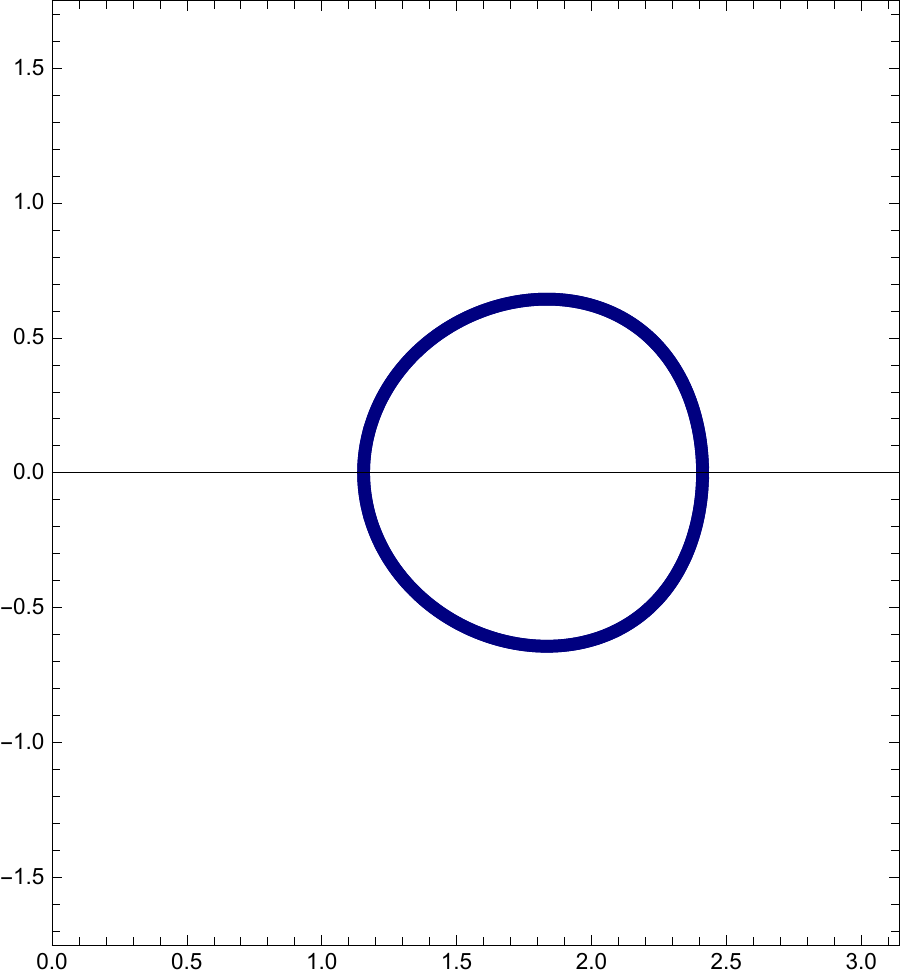} \\
		\includegraphics[width=\textwidth]{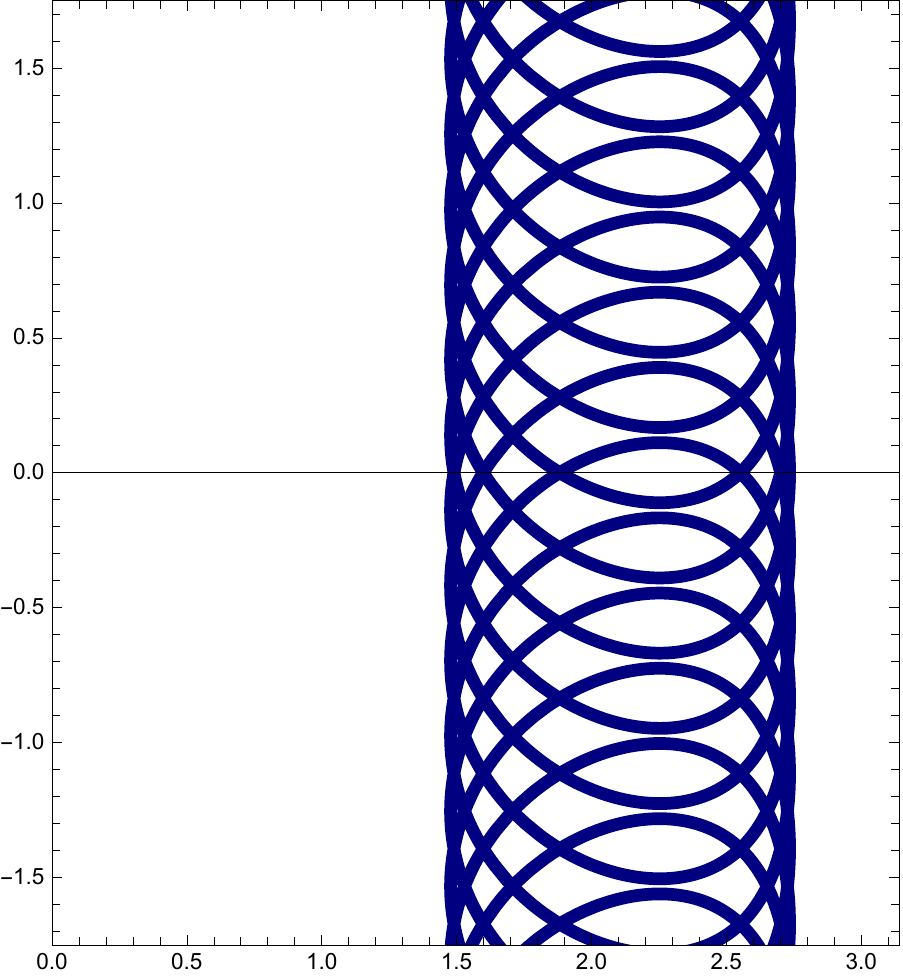}
	\end{minipage}
	\begin{minipage}[t]{0.16\textwidth}
		\includegraphics[width=\textwidth]{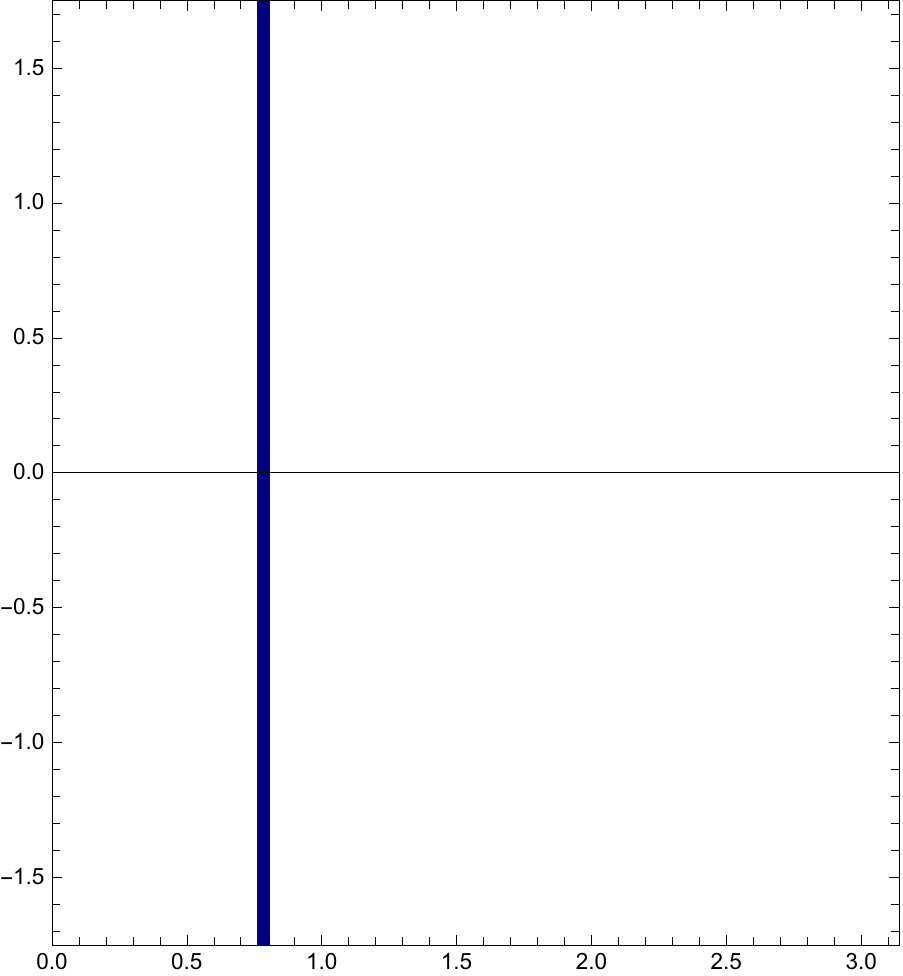} \\
		\includegraphics[width=\textwidth]{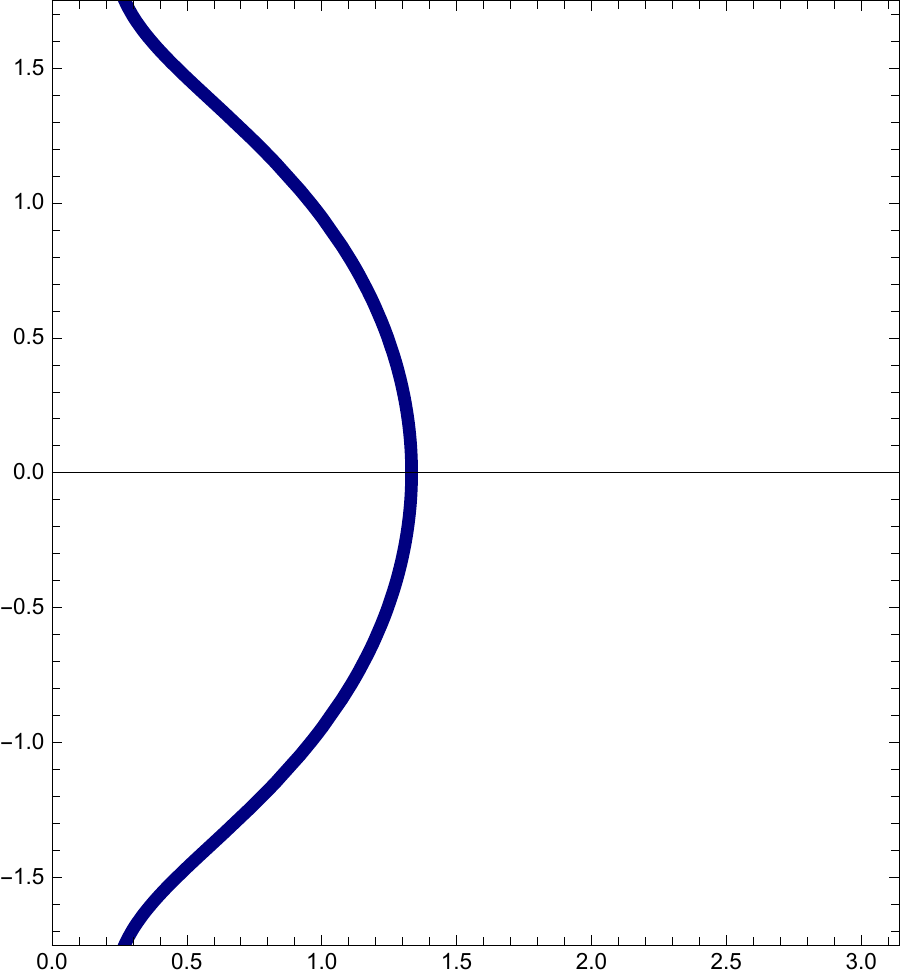} \\		\includegraphics[width=\textwidth]{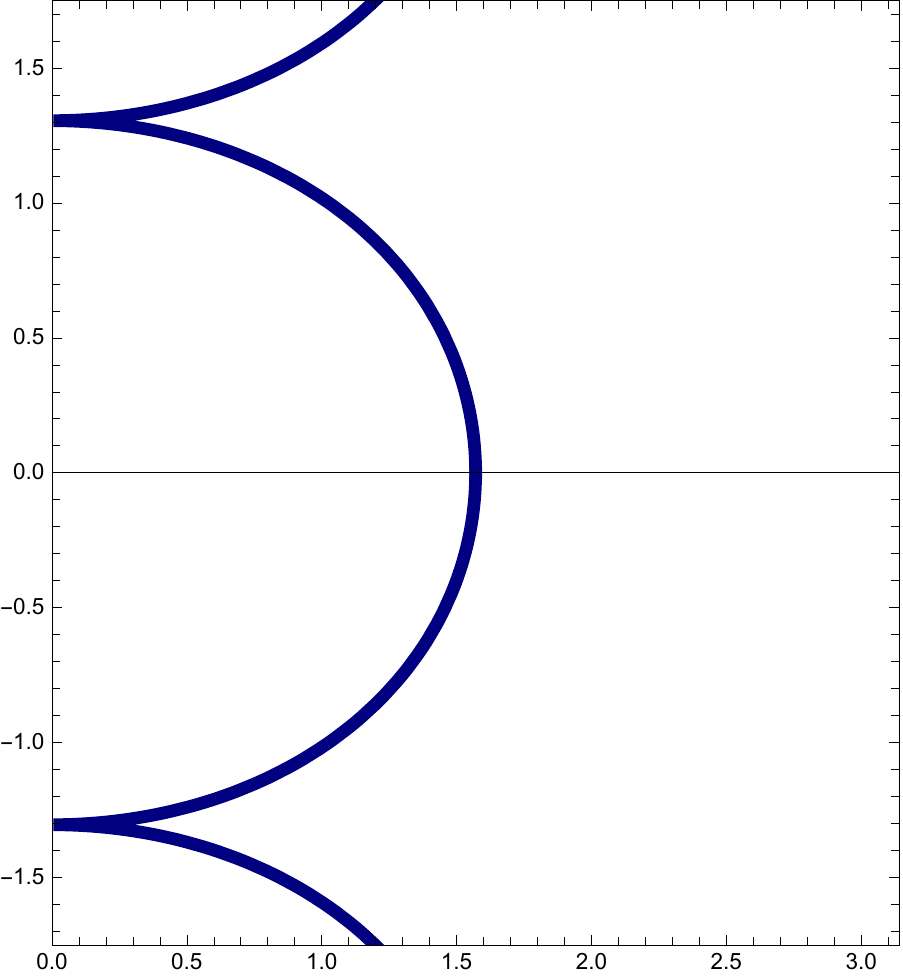} \\
		\includegraphics[width=\textwidth]{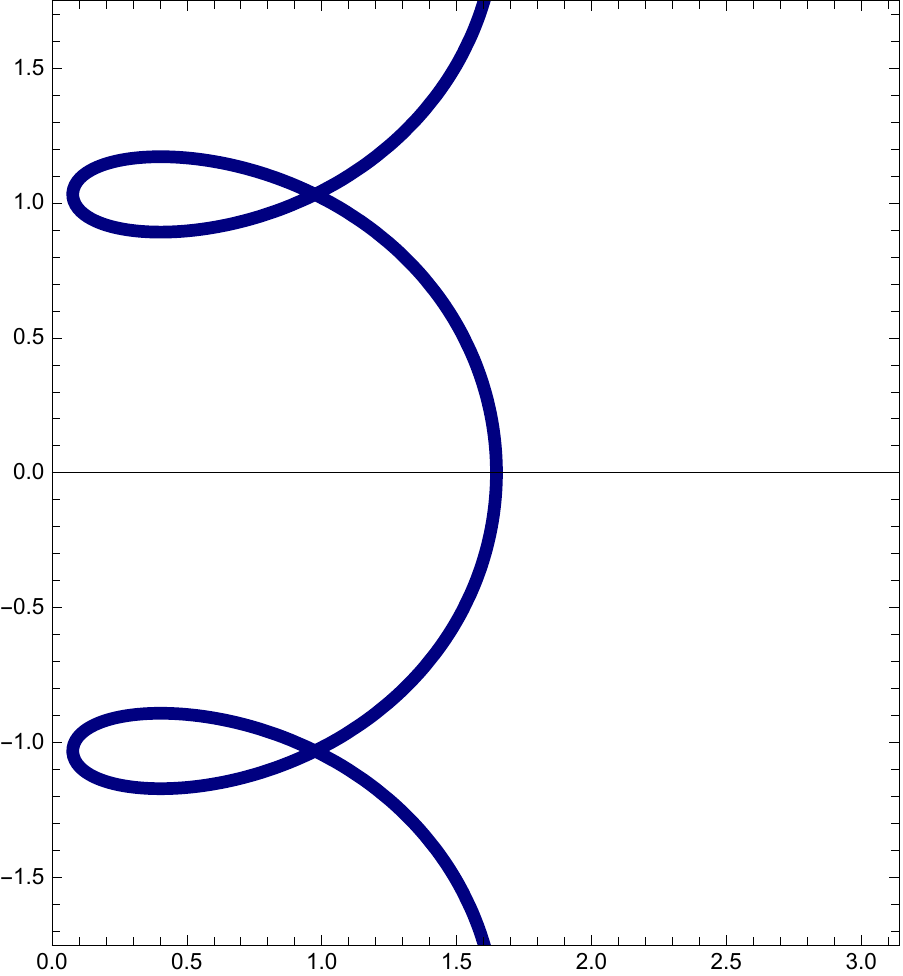} \\
		\includegraphics[width=\textwidth]{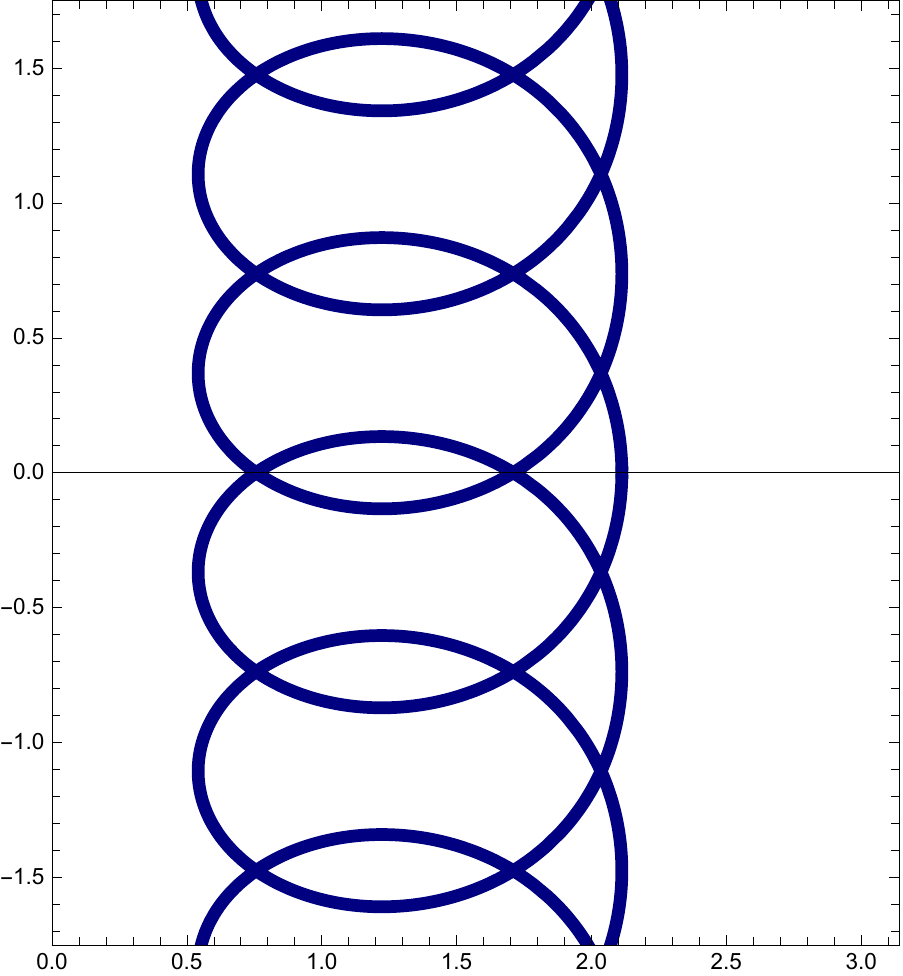} \\
		\includegraphics[width=\textwidth]{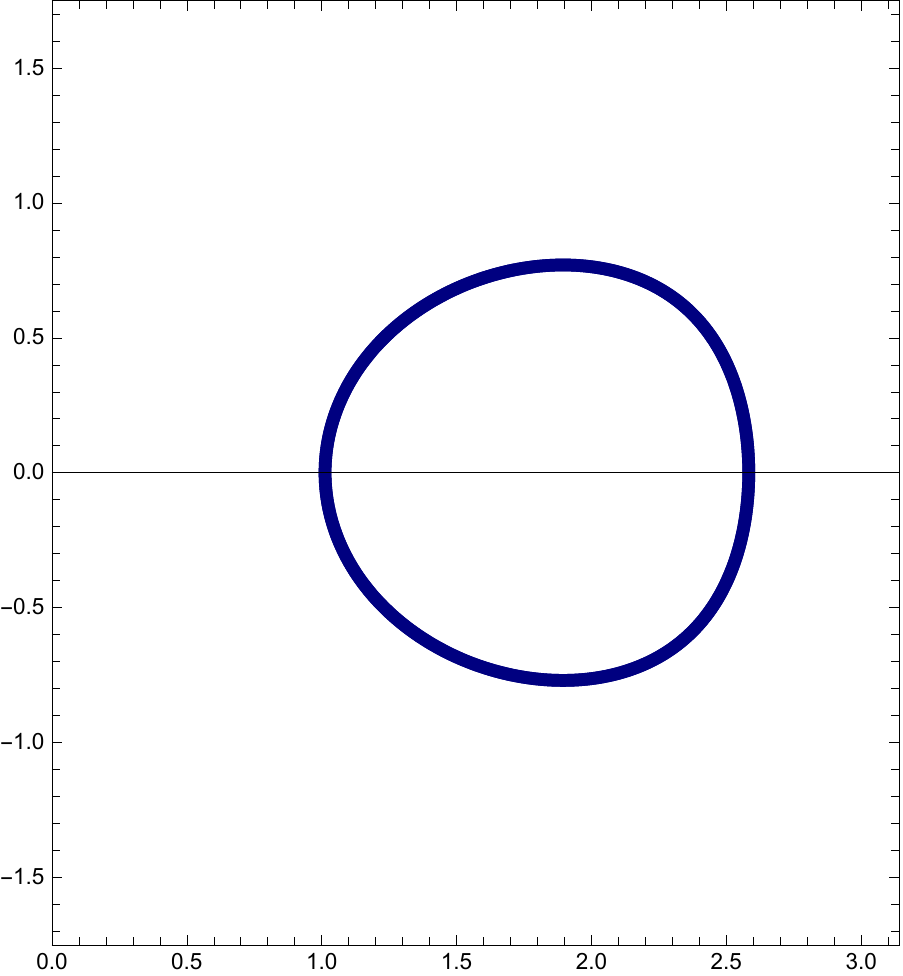} \\
		\includegraphics[width=\textwidth]{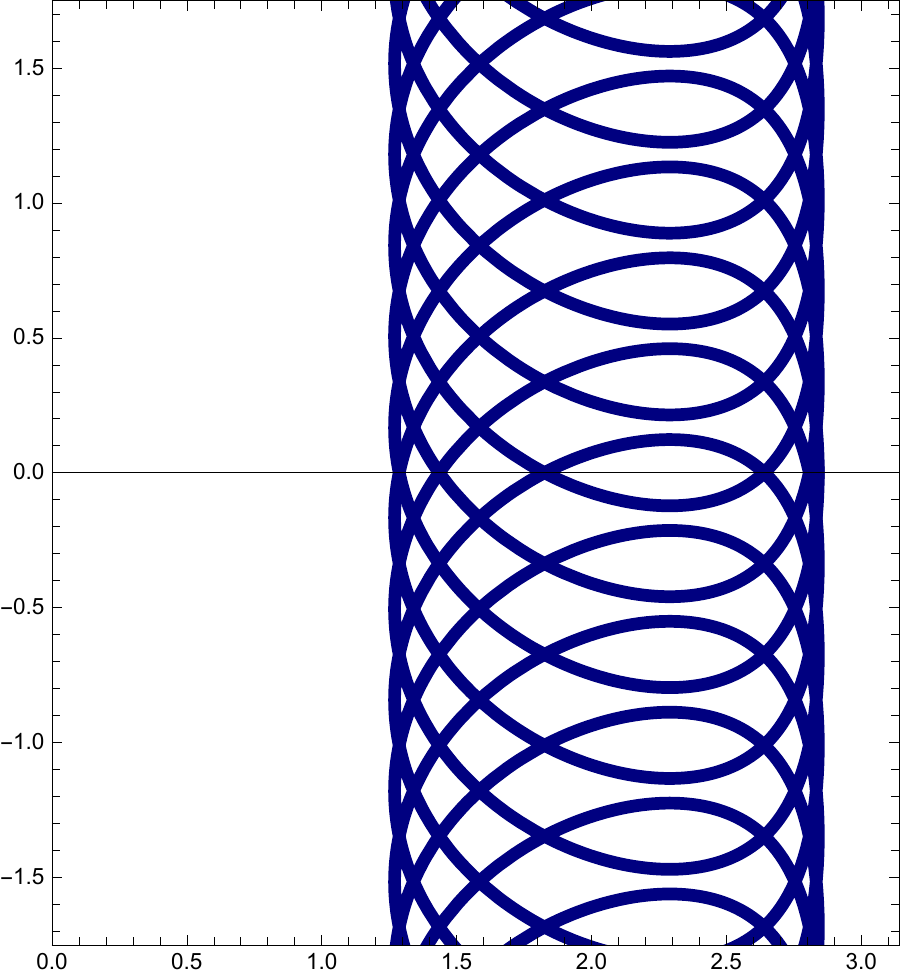}
	\end{minipage}
	\begin{minipage}[t]{0.16\textwidth}
		\includegraphics[width=\textwidth]{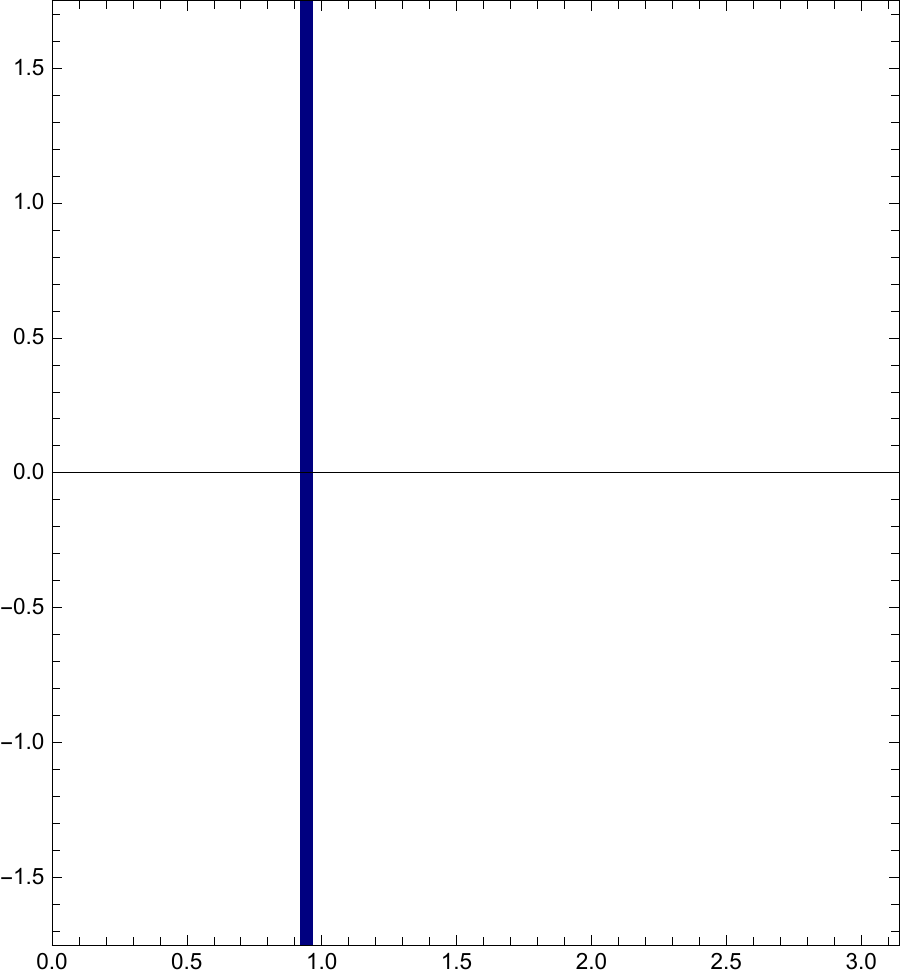} \\
		\includegraphics[width=\textwidth]{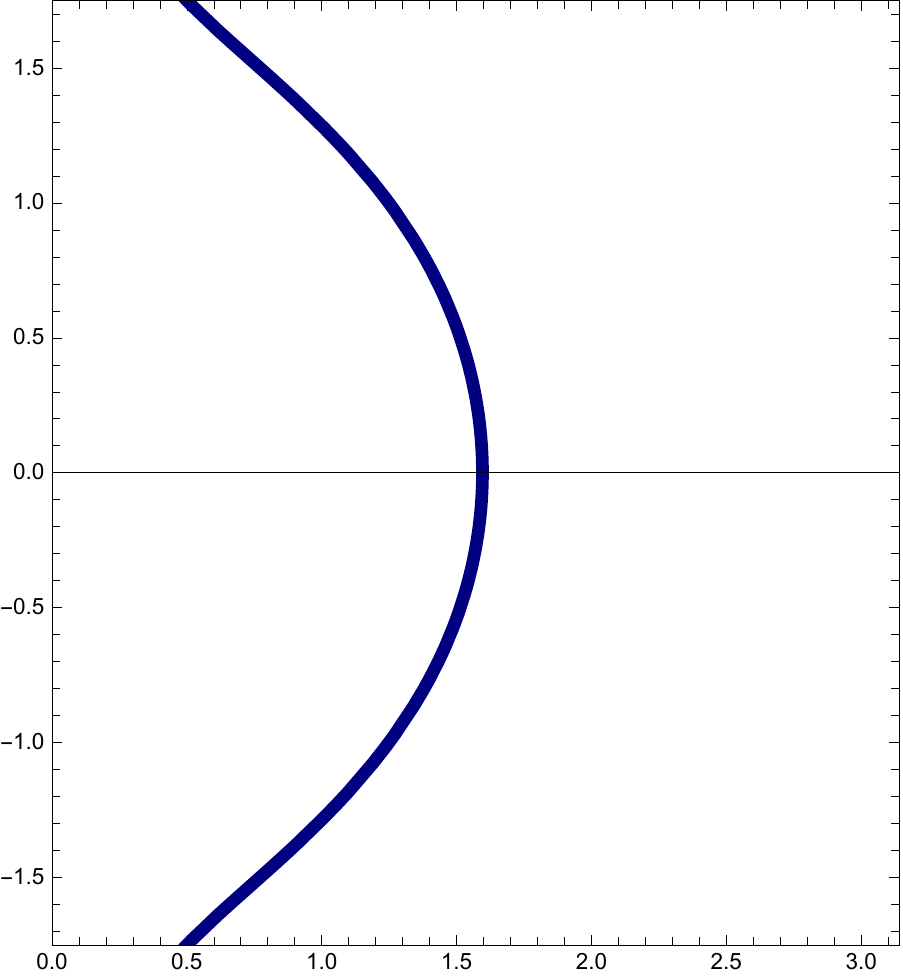} \\
		\includegraphics[width=\textwidth]{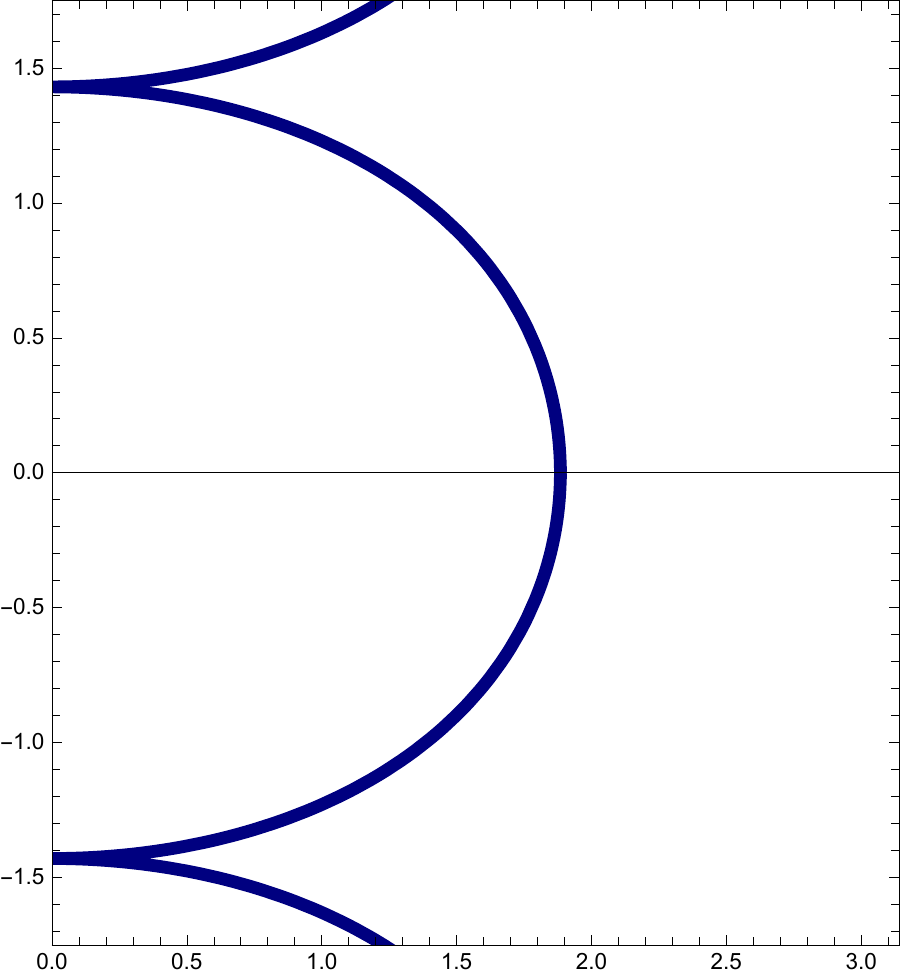} \\
		\includegraphics[width=\textwidth]{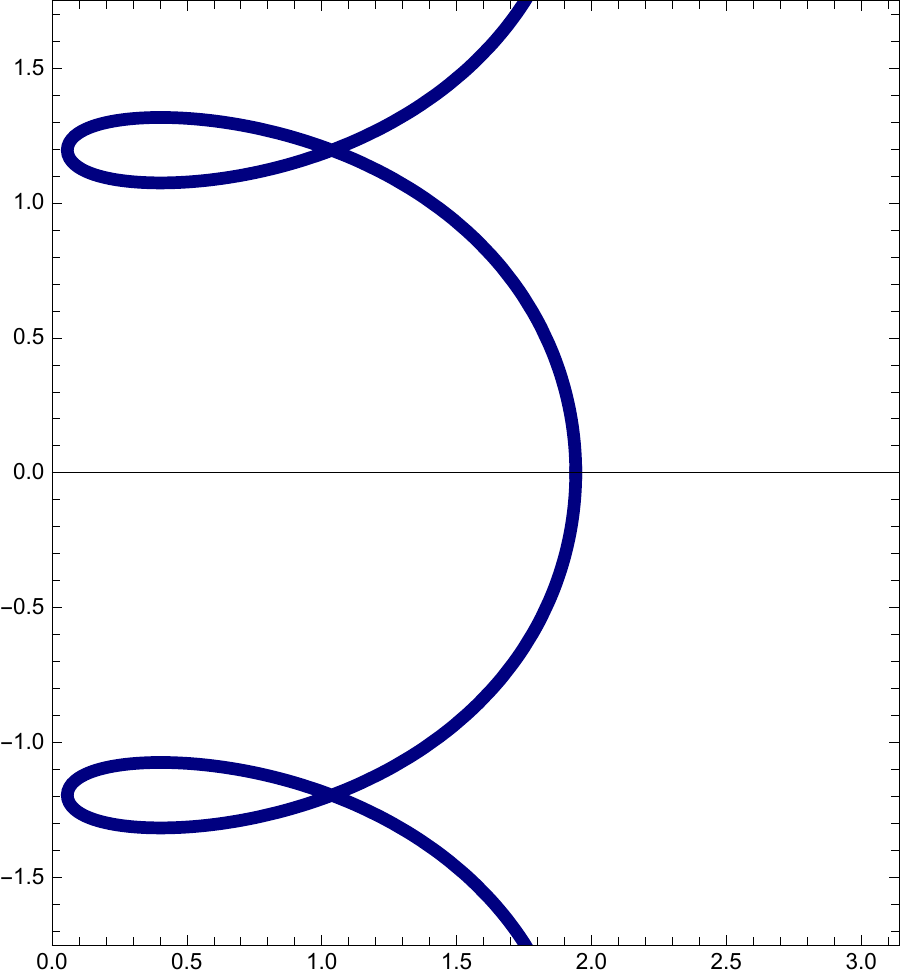} \\
		\includegraphics[width=\textwidth]{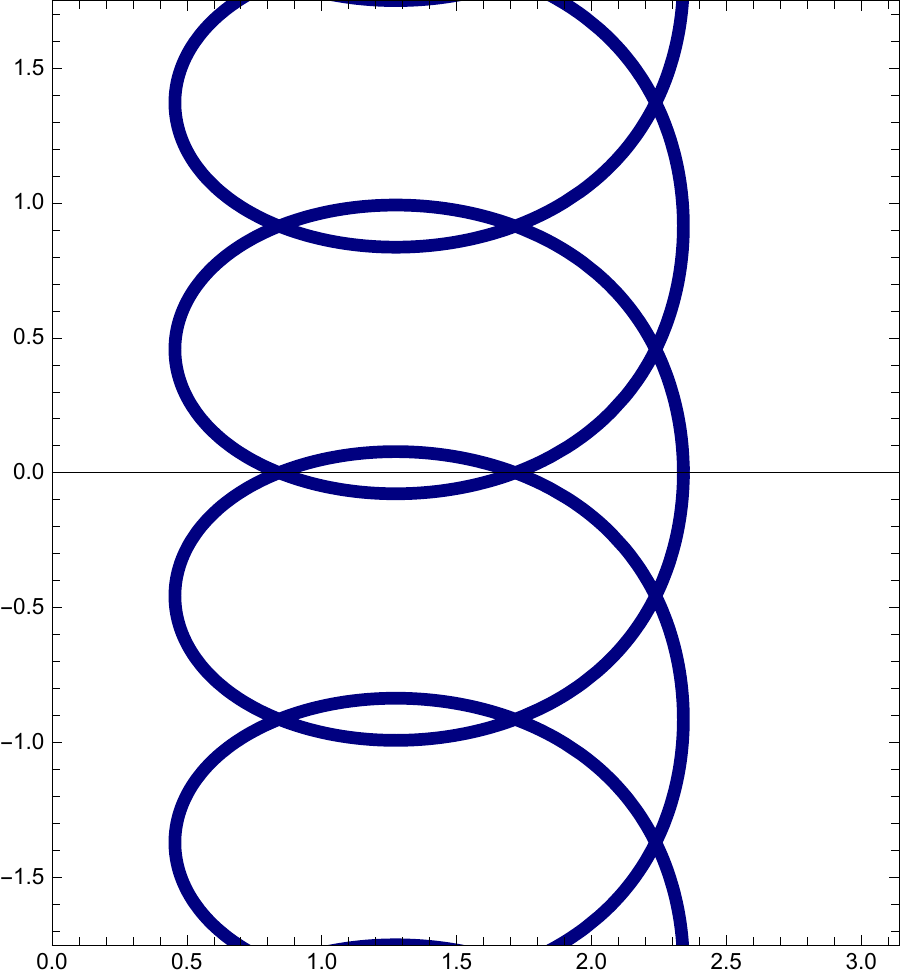} \\
		\includegraphics[width=\textwidth]{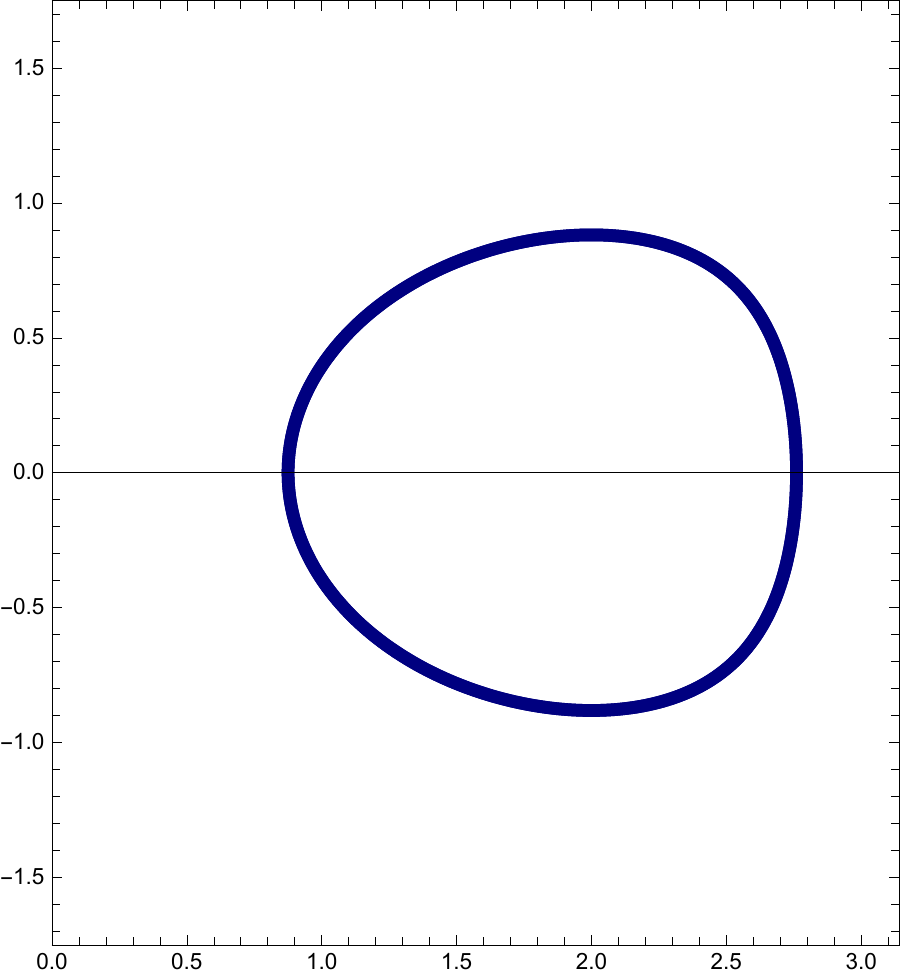} \\
		\includegraphics[width=\textwidth]{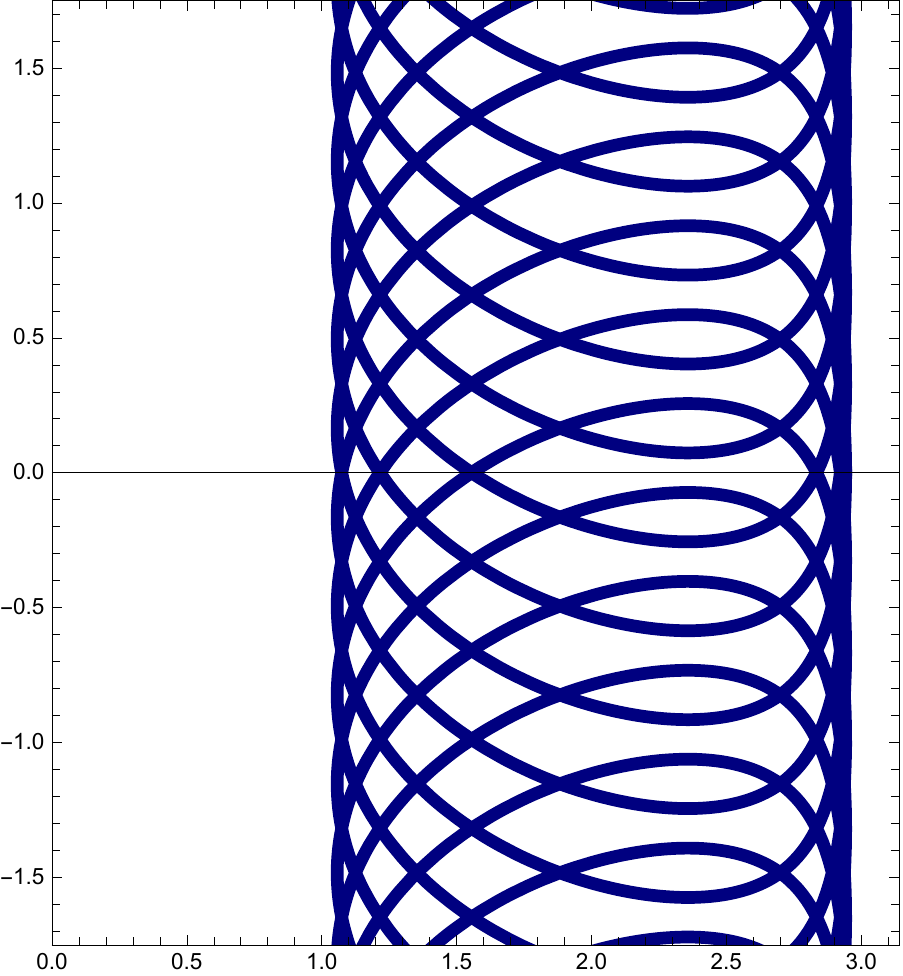}
	\end{minipage}
	\begin{minipage}[t]{0.16\textwidth}
		\includegraphics[width=\textwidth]{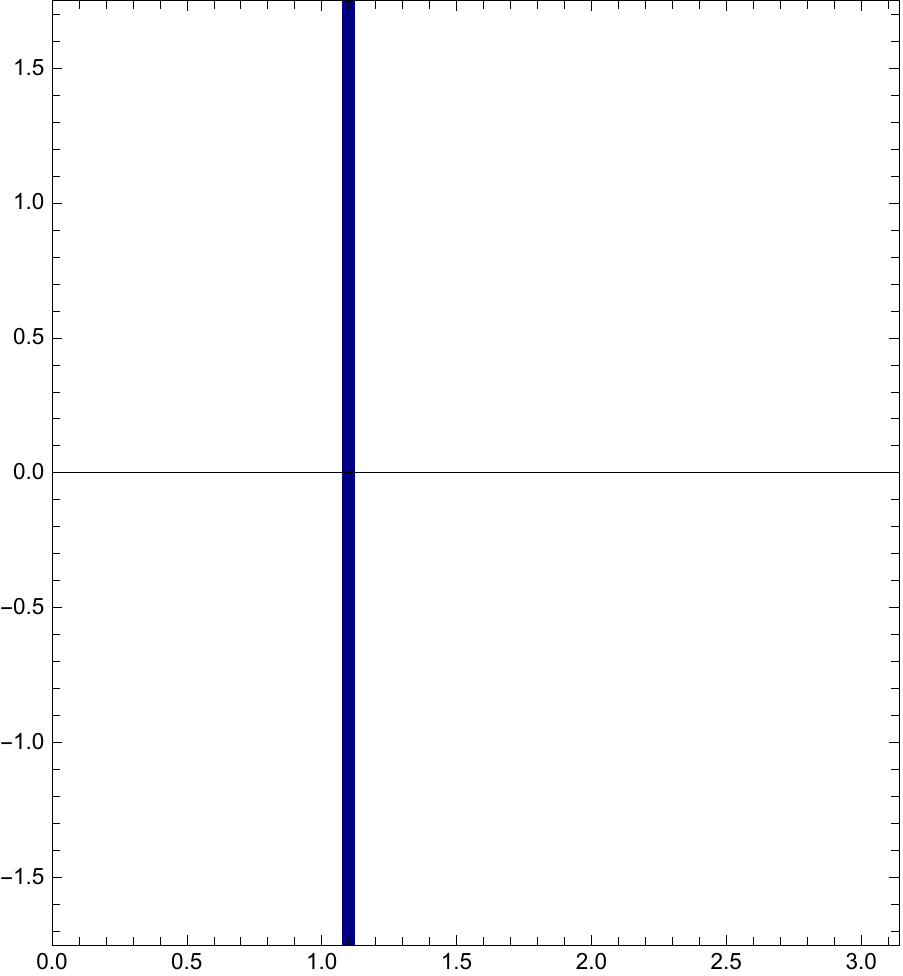} \\
		\includegraphics[width=\textwidth]{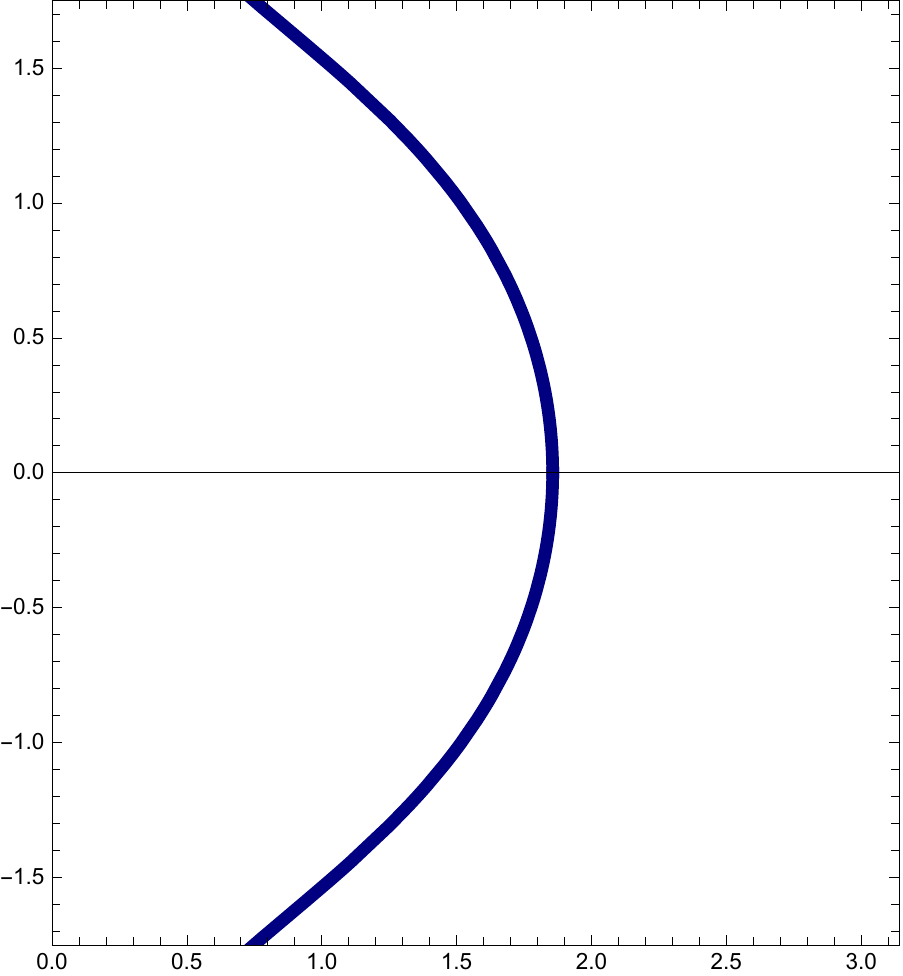} \\
		\includegraphics[width=\textwidth]{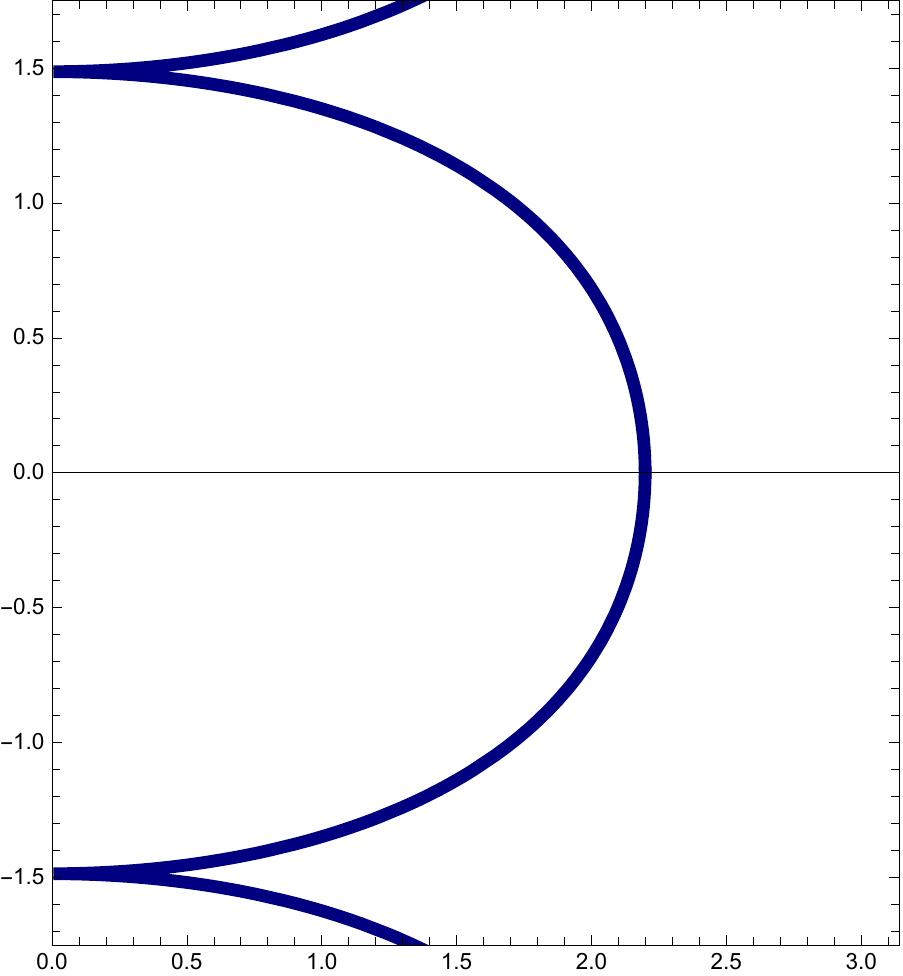} \\
		\includegraphics[width=\textwidth]{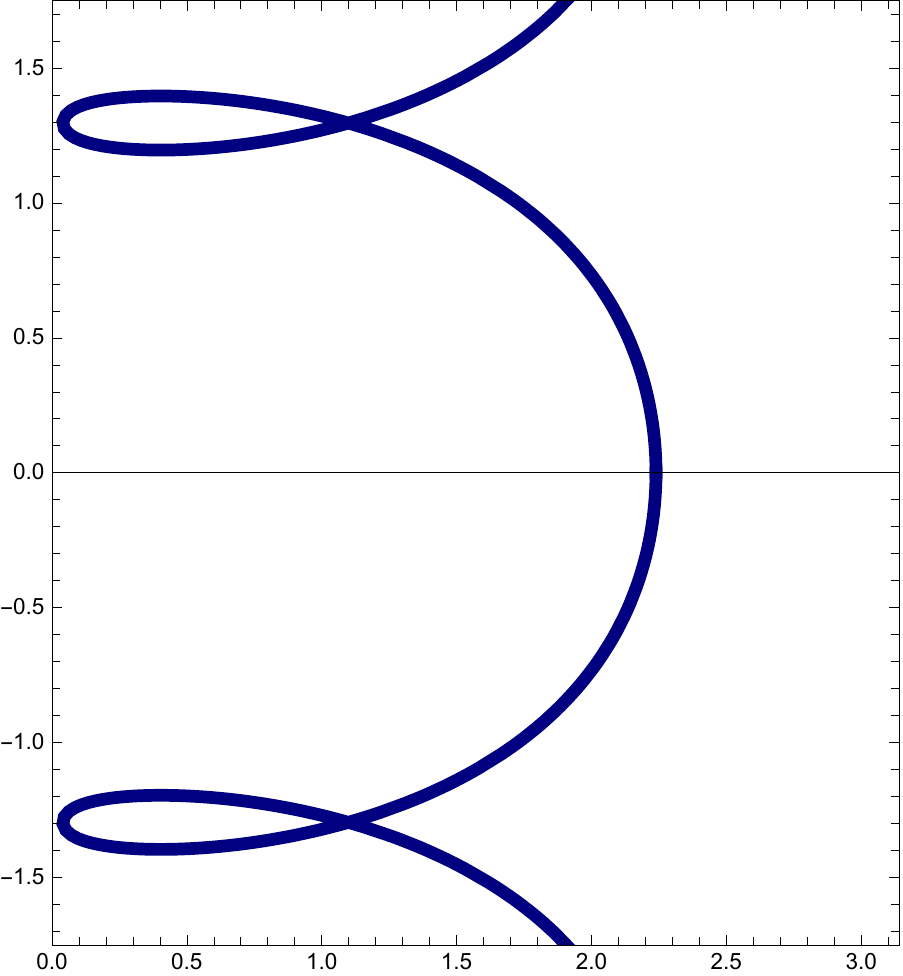} \\
		\includegraphics[width=\textwidth]{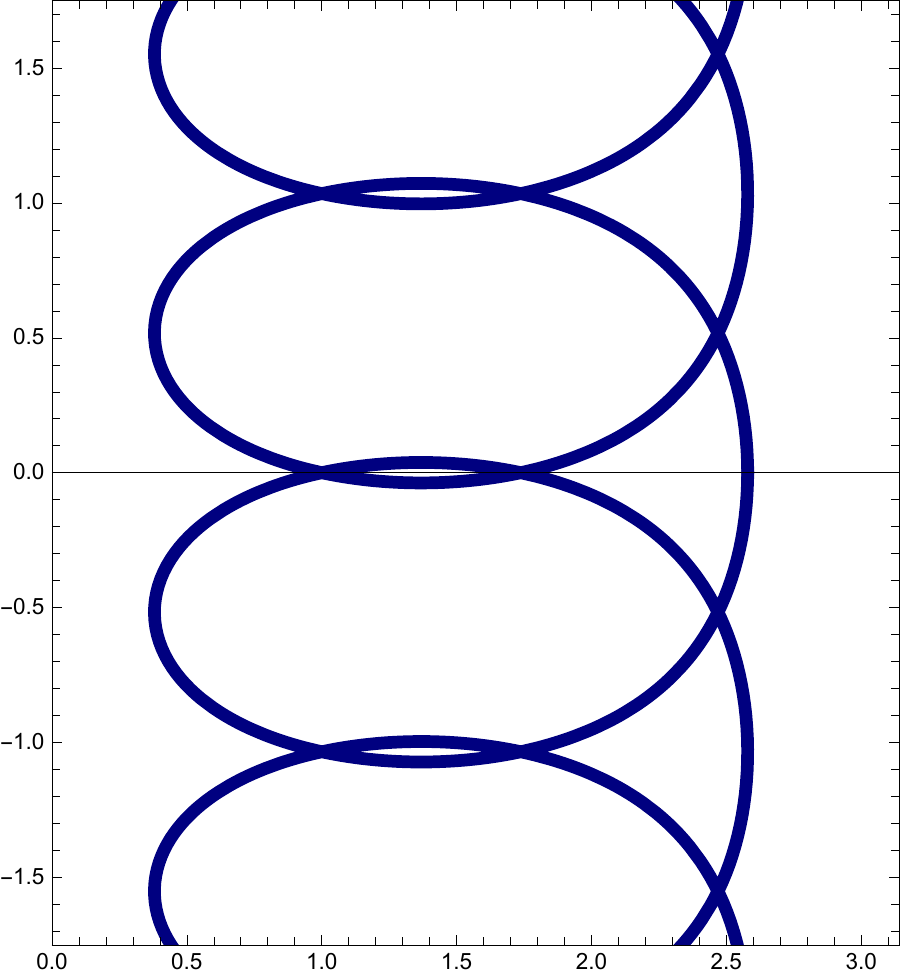} \\
		\includegraphics[width=\textwidth]{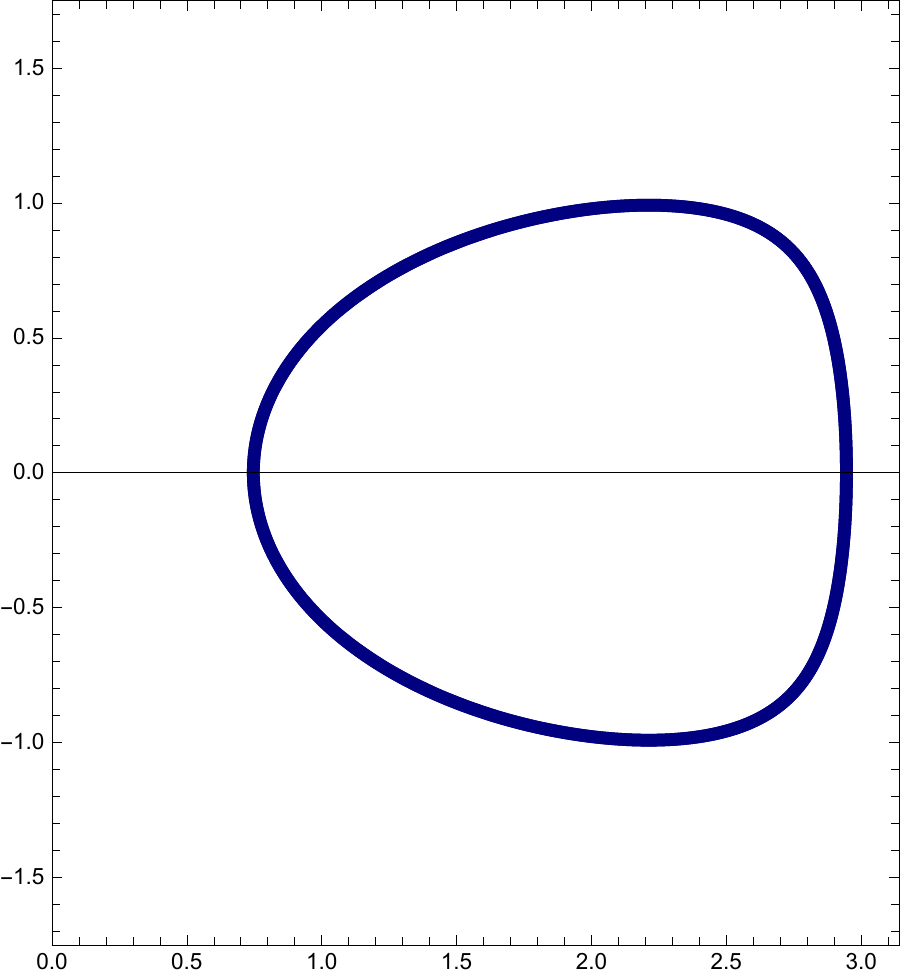} \\
		\includegraphics[width=\textwidth]{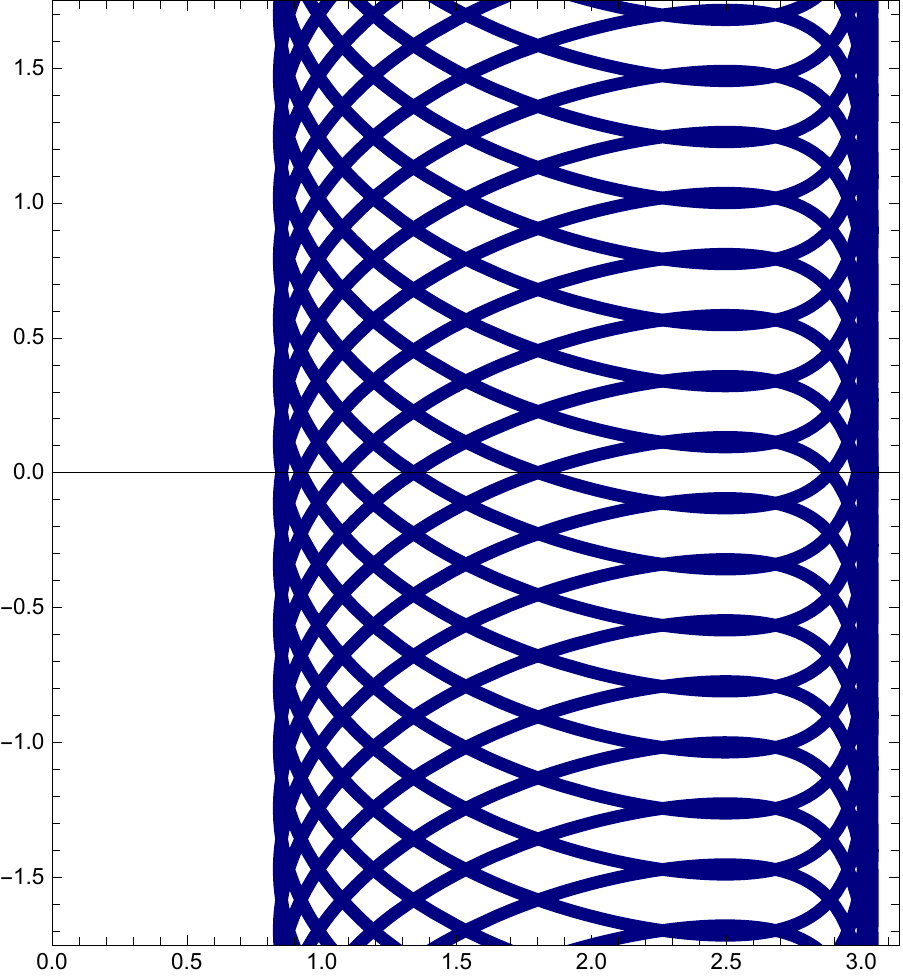}
	\end{minipage}
	\caption{Numerically computed profile curves of the family $\family$ in the Berger sphere $\E(1,0.2)$ for the rotational case $a=0$. The energy decreases from $\Jmax$ in the first row (vertical cylinder) to $J<0$ in the lower rows. The intermediate rows display surfaces of unduloid type, sphere type, nodoid type I, tube, and nodoid type II. Each column represents a fixed value of the mean curvature $H$, from large $H$ in the left column to small $H>0$ in the right column.}
	\label{fig:curves_berger}
\end{figure}

\newpage
\printbibliography		

\end{document}